\documentclass[letterpaper]{amsart}
\usepackage[english]{babel}
\usepackage{amsmath}
\usepackage{amssymb}
\usepackage{amsfonts}
\usepackage{amsthm}
\usepackage{mathtools}
\usepackage{pstricks}
\usepackage{pst-node}
\usepackage{pst-plot}
\usepackage{pst-3dplot}
\usepackage{booktabs}
\usepackage{multirow}
\usepackage{subfigure}
\usepackage[figuresright]{rotating}
\usepackage{url}
\usepackage[all,cmtip]{xy}
\allowdisplaybreaks[1]
\newcommand{\ds}{\displaystyle}
\renewcommand{\subset}{\subseteq}
\renewcommand{\supset}{\supseteq}
\renewcommand{\subsetneq}{\varsubsetneq}

\newcommand{\overbar}[1]{\mkern 1.5mu\overline{\mkern-1.5mu#1\mkern-1.5mu}\mkern 1.5mu}
\newcommand{\R}{\ensuremath{\mathbf{R}}}

\newcommand{\Rplus}{\ensuremath{\mathbf{R}_{\geqslant0}}}
\newcommand{\Rplusbar}{\ensuremath{\overbar{\mathbf{R}}_{\geqslant0}}}

\newcommand{\Rplusnul}{\ensuremath{\mathbf{R}_{>0}}}
\newcommand{\Rplusnulbar}{\ensuremath{\overbar{\mathbf{R}}_{>0}}}

\newcommand{\Rn}{\ensuremath{\mathbf{R}^n}}
\newcommand{\Rplusn}{\ensuremath{\mathbf{R}_{\geqslant0}^n}}
\newcommand{\Rplusbarn}{\ensuremath{\overbar{\mathbf{R}}_{\geqslant0}^n}}
\newcommand{\Rplusnuln}{\ensuremath{\mathbf{R}_{>0}^n}}
\newcommand{\Rplusnulbarn}{\ensuremath{\overbar{\mathbf{R}}_{>0}^n}}
\newcommand{\N}{\ensuremath{\mathbf{N}}}

\newcommand{\Z}{\ensuremath{\mathbf{Z}}}
\newcommand{\Zn}{\ensuremath{\mathbf{Z}^n}}

\newcommand{\Zplus}{\ensuremath{\mathbf{Z}_{\geqslant0}}}
\newcommand{\Zplusbar}{\ensuremath{\overbar{\mathbf{Z}}_{\geqslant0}}}
\newcommand{\Zplusn}{\ensuremath{\mathbf{Z}_{\geqslant0}^n}}

\newcommand{\Zplusnul}{\ensuremath{\mathbf{Z}_{>0}}}

\newcommand{\Zplusnuln}{\ensuremath{\mathbf{Z}_{>0}^n}}
\newcommand{\Zplusnulbarn}{\ensuremath{\overbar{\mathbf{Z}}_{>0}^n}}

\newcommand{\Q}{\ensuremath{\mathbf{Q}}}

\newcommand{\Qplusnul}{\ensuremath{\mathbf{Q}_{>0}}}

\newcommand{\C}{\ensuremath{\mathbf{C}}}
\newcommand{\Ccross}{\ensuremath{\mathbf{C}^{\times}}}
\newcommand{\Ccrossn}{\ensuremath{(\mathbf{C}^{\times})^n}}
\newcommand{\Qp}{\ensuremath{\mathbf{Q}_p}}

\newcommand{\Qpn}{\ensuremath{\mathbf{Q}_p^n}}
\newcommand{\Qpxn}{\ensuremath{(\mathbf{Q}_p^{\times})^n}}
\newcommand{\Zp}{\ensuremath{\mathbf{Z}_p}}
\newcommand{\Zpx}{\ensuremath{\mathbf{Z}_p^{\times}}}

\newcommand{\Zpn}{\ensuremath{\mathbf{Z}_p^n}}
\newcommand{\Zpxn}{\ensuremath{(\mathbf{Z}_p^{\times})^n}}

\newcommand{\Fp}{\ensuremath{\mathbf{F}_p}}

\newcommand{\Fpcross}{\ensuremath{\mathbf{F}_p^{\times}}}
\newcommand{\Fpcrossn}{\ensuremath{(\mathbf{F}_p^{\times})^n}}

\newcommand{\LL}{\ensuremath{\mathbb{L}}}
\newcommand{\MC}{\ensuremath{\mathcal{M}_{\mathbf{C}}}}
\renewcommand{\N}{\Zplus}

\renewcommand{\phi}{\varphi}

\DeclareMathOperator{\aff}{aff}

\DeclareMathOperator{\ac}{ac}
\DeclareMathOperator{\supp}{supp}
\DeclareMathOperator{\mult}{mult}
\DeclareMathOperator{\ord}{ord}

\DeclareMathOperator{\cone}{cone}
\DeclareMathOperator{\adj}{adj}
\DeclareMathOperator{\lcm}{lcm}
\DeclareMathOperator{\Res}{Res}

\newcommand{\Zf}{\ensuremath{Z_f}}
\newcommand{\Zof}{\ensuremath{Z^0_f}}

\newcommand{\Zofs}{\ensuremath{Z^0_f(s)}}

\newcommand{\Ztopof}{\ensuremath{Z^{\mathrm{top},0}_f}}

\newcommand{\Zmotof}{\ensuremath{Z^{\mathrm{mot},0}_f}}

\newcommand{\Zmotoft}{\ensuremath{Z^{\mathrm{mot},0}_f(T)}}

\newcommand{\Gf}{\ensuremath{\Gamma_f}}

\newcommand{\Gglf}{\ensuremath{\Gamma^{\mathrm{gl}}_f}}

\newcommand{\ft}{\ensuremath{f_{\tau}}}

\newcommand{\fbart}{\ensuremath{\overline{f_{\tau}}}}

\newcommand{\Dtu}{\ensuremath{\Delta_{\tau}}}
\newcommand{\Dti}{\ensuremath{\Delta_{\tau}^{\infty}}}

\newcommand{\barchi}{\ensuremath{\overline{\chi}}}

\providecommand{\abs}[1]{\lvert#1\rvert}

\theoremstyle{plain}
\newtheorem{theorem}{Theorem}[section]
\newtheorem{lemma}[theorem]{Lemma}
\newtheorem{proposition}[theorem]{Proposition}
\newtheorem{corollary}[theorem]{Corollary}
\newtheorem{conjecture}[theorem]{Conjecture}

\theoremstyle{definition}
\newtheorem{definition}[theorem]{Definition}

\newtheorem{example}[theorem]{Example}

\theoremstyle{remark}
\newtheorem{remark}[theorem]{Remark}
\newtheorem{remarks}[theorem]{Remarks}

\newtheorem{notation}[theorem]{Notation}

\newcommand{\aA}{\ensuremath{\alpha_A}}
\newcommand{\aB}{\ensuremath{\alpha_B}}
\newcommand{\aC}{\ensuremath{\alpha_C}}
\newcommand{\bA}{\ensuremath{\beta_A}}
\newcommand{\bB}{\ensuremath{\beta_B}}
\newcommand{\bC}{\ensuremath{\beta_C}}

\newcommand{\dA}{\ensuremath{\delta_A}}
\newcommand{\DA}{\ensuremath{\Delta_A}}
\newcommand{\dB}{\ensuremath{\delta_B}}
\newcommand{\DB}{\ensuremath{\Delta_B}}
\newcommand{\dC}{\ensuremath{\delta_C}}

\newcommand{\Dtnul}{\ensuremath{\Delta_{\tau_0}}}
\newcommand{\Dteen}{\ensuremath{\Delta_{\tau_1}}}
\newcommand{\DAB}{\ensuremath{\Delta_{[AB]}}}
\newcommand{\DAC}{\ensuremath{\Delta_{[AC]}}}
\newcommand{\DAD}{\ensuremath{\Delta_{[AD]}}}
\newcommand{\DBC}{\ensuremath{\Delta_{[BC]}}}
\newcommand{\DBD}{\ensuremath{\Delta_{[BD]}}}
\newcommand{\DCD}{\ensuremath{\Delta_{[CD]}}}

\newcommand{\fAB}{\ensuremath{\phi_{AB}}}
\newcommand{\fBC}{\ensuremath{\phi_{BC}}}
\newcommand{\fAtwee}{\ensuremath{\phi_{A2}}}
\newcommand{\fBtwee}{\ensuremath{\phi_{B2}}}
\newcommand{\fCtwee}{\ensuremath{\phi_{C2}}}

\newcommand{\vereen}{\ensuremath{\{0,\ldots,\mu_1-1\}}}
\newcommand{\vertwee}{\ensuremath{\{0,\ldots,\mu_2-1\}}}
\newcommand{\verdrie}{\ensuremath{\{0,\ldots,\mu_3-1\}}}
\newcommand{\verA}{\ensuremath{\{0,\ldots,\mu_A-1\}}}
\newcommand{\verB}{\ensuremath{\{0,\ldots,\mu_B-1\}}}

\newcommand{\Fnuls}{\bigl(p^{\sigma_0+m_0s}-1\bigr)}
\newcommand{\Feens}{\bigl(p^{\sigma_1+m_1s}-1\bigr)}
\newcommand{\Ftwees}{\bigl(p^{\sigma_2+m_2s}-1\bigr)}
\newcommand{\Fdries}{\bigl(p^{\sigma_3+m_3s}-1\bigr)}
\newcommand{\Feen}{\bigl(p^{\sigma_1+m_1s_0}-1\bigr)}
\newcommand{\Ftwee}{\bigl(p^{\sigma_2+m_2s_0}-1\bigr)}
\newcommand{\Fdrie}{\bigl(p^{\sigma_3+m_3s_0}-1\bigr)}
\providecommand{\dds}[1]{\frac{d}{ds}\left.\left[#1\right]\right\rvert_{s=s_0}}
\newcommand{\barpsi}{\ensuremath{\overline{\psi}}}

\begin{document}
\title[Monodromy Conjecture for non-degenerated surface singularities]{Igusa's $p$-adic local zeta function and the Monodromy Conjecture for non-degenerated surface singularities}
\author{Bart Bories}
\address{Department of Mathematics, KU Leuven, Celestijnenlaan 200b -- box 2400, 3001 Leuven, Belgium}
\email{bart.bories@wis.kuleuven.be}
\author{Willem Veys}
\email{wim.veys@wis.kuleuven.be}
\date{\today}

\hyphenation{po-ly-he-dra}

\begin{abstract}
In 2011 Lemahieu and Van Proeyen proved the Monodromy Conjecture for the local topological zeta function of a non-de\-gen\-er\-ated surface singularity. We start from their work and obtain the same result for Igusa's $p$-adic and the motivic zeta function. In the $p$-adic case, this is, for a polynomial $f\in\Z[x,y,z]$ satisfying $f(0,0,0)=0$ and non-degenerated with respect to its Newton polyhedron, we show that every pole of the local $p$-adic zeta function of $f$ induces an eigenvalue of the local monodromy of $f$ at some point of $f^{-1}(0)\subset\C^3$ close to the origin.

Essentially the entire paper is dedicated to proving that, for $f$ as above, certain candidate poles of Igusa's $p$-adic zeta function of $f$, arising from so-called $B_1$-facets of the Newton polyhedron of $f$, are actually not poles. This turns out to be much harder than in the topological setting. The combinatorial proof is preceded by a study of the integral points in three-dimensional fundamental parallelepipeds. Together with the work of Lemahieu and Van Proeyen, this main result leads to the Monodromy Conjecture for the $p$-adic and motivic zeta function of a non-degenerated surface singularity.
\end{abstract}

\maketitle

\setcounter{tocdepth}{2}
\tableofcontents

\setcounter{section}{-1}
\section{Introduction}
\subsection{Igusa's zeta function and the Monodromy Conjecture}
For a prime $p$, we denote by \Qp\ the field of $p$-adic numbers and by \Zp\ its subring of $p$-adic integers. We denote by $|\cdot|$ the $p$-adic norm on \Qp. Let $n\in\Zplusnul$ and denote by $|dx|=|dx_1\wedge\cdots\wedge dx_n|$ the Haar measure on \Qpn, so normalized that \Zpn\ has measure one.

\begin{definition}[Igusa's $p$-adic local zeta function] Let $p$ be a prime number, $f(x)=f(x_1,\ldots,x_n)$ a polynomial in $\Qp[x_1,\ldots,x_n]$, and $\Phi$ a Schwartz--Bruhat function on \Qpn, i.e., a locally constant function $\Phi:\Qpn\to\C$ with compact support. Igusa's $p$-adic local zeta function associated to $f$ and $\Phi$ is defined as
\begin{equation*}
Z_{f,\Phi}:\{s\in\C\mid\Re(s)>0\}\to\C:s\mapsto\int_{\Qpn}|f(x)|^s\Phi(x)|dx|.
\end{equation*}

We will mostly consider the case where $\Phi$ is the characteristic function of either \Zpn\ or $p\Zpn=(p\Zp)^n$. By Igusa's $p$-adic zeta function \Zf\ of $f$ (without mentioning $\Phi$), we mean $Z_{f,\Phi}$, where $\Phi=\chi(\Zpn)$ is the characteristic function of \Zpn. By the local Igusa zeta function \Zof\ of $f$, we mean $Z_{f,\Phi}$, where $\Phi=\chi(p\Zpn)$ is the characteristic function of $p\Zpn$.
\end{definition}

Using resolution of singularities, Igusa \cite{Igu74} proves in 1974 that $Z_{f,\Phi}$ is a rational function in the variable $t=p^{-s}$; more precisely, he shows that there exists a rational function $\widetilde{Z}_{f,\Phi}\in\Q(t)$, such that $Z_{f,\Phi}(s)=\widetilde{Z}_{f,\Phi}(p^{-s})$ for all $s\in\C$ with $\Re(s)>0$. Denoting the meromorphic continuation of $Z_{f,\Phi}$ to the whole complex plane again with $Z_{f,\Phi}$, he also obtains a set of candidate poles for $Z_{f,\Phi}$ in terms of numerical data associated to an embedded resolution of singularities of the locus $f^{-1}(0)\subset\Qpn$. In 1984 Denef \cite{Den84} proves the rationality of $Z_{f,\Phi}$ in an entirely different way, using $p$-adic cell decomposition.

For a prime number $p$ and $f(x)=f(x_1,\ldots,x_n)\in\Zp[x_1,\ldots,x_n]$, Igusa's zeta function \Zf\ is closely related to the numbers $N_l$ of solutions in $(\Zp/p^l\Zp)^n$ of the polynomial congruences $f(x)\equiv0\bmod p^l$ for $l\geqslant1$. For instance, the poles of \Zf\ determine the behavior of the numbers $N_l$ for $l$ big enough.

The poles of Igusa's zeta function are also the subject of the Monodromy Conjecture, formulated by Igusa in 1988. It predicts a remarkable connection between the poles of \Zf\ and the eigenvalues of the local monodromy of $f$. The conjecture is motivated by analogous results for Archimedean local zeta functions (over \R\ or \C\ instead of \Qp) and---of course---by all known examples supporting it. If the Monodromy Conjecture were true, it would explain why generally only few of the candidate poles arising from an embedded resolution of singularities, are actually poles.

\begin{conjecture}[Monodromy Conjecture for Igusa's $p$-adic zeta function over \Qp]\label{mcigusa1}
\textup{\cite{Igu88}}. Let $f(x_1,\ldots,x_n)$ be a polynomial in $\Z[x_1,\ldots,x_n]$. For almost all\,\footnote{By \lq almost all\rq\ we always mean \lq all, except finitely many\rq, unless expressly stated otherwise.} prime numbers $p$, we have the following. If $s_0$ is a pole of Igusa's local zeta function \Zf\ of $f$, then $e^{2\pi i\Re(s_0)}$ is an eigenvalue of the local monodromy operator acting on some cohomology group of the Milnor fiber of $f$ at some point of the hypersurface $f^{-1}(0)\subset\C^n$.
\end{conjecture}

There is a local version of this conjecture considering Igusa's zeta function on a small enough neighborhood of $0\in\Qpn$ and local monodromy only at points of $f^{-1}(0)\subset\C^n$ close to the origin.

\begin{conjecture}[Local version of Conjecture~\ref{mcigusa1}]\label{mcigusa1lokaal}
Let $f(x_1,\ldots,x_n)$ be a polynomial in $\Z[x_1,\ldots,x_n]$ with $f(0)=0$. For almost all prime numbers $p$ and for $k$ big enough, we have the following. If $s_0$ is a pole of $Z_{f,\chi(p^k\Zpn)}$, then $e^{2\pi i\Re(s_0)}$ is an eigenvalue of the local monodromy of $f$ at some point of the hypersurface $f^{-1}(0)\subset\C^n$ close to the origin.
\end{conjecture}
%

There exists a stronger, related conjecture, also due to Igusa and also inspired by the analogous theorem in the Archimedean case.

\begin{conjecture}\label{Smcigusa1}
\textup{\cite{Igu88}}. Let $f(x_1,\ldots,x_n)\in\Z[x_1,\ldots,x_n]$. For almost all prime numbers $p$, we have the following. If $s_0$ is a pole of \Zf, then $\Re(s_0)$ is a root of the Bernstein--Sato polynomial $b_f(s)$ of $f$.

There is a local version of this conjecture considering $Z_{f,\chi(p^k\Zpn)}$ for $k$ big enough and the local Bernstein--Sato polynomial $b_f^0(s)$ of $f$.
\end{conjecture}
%
%

Malgrange \cite{malgrange} proved in 1983 that if $f(x_1,\ldots,x_n)\in\C[x_1,\ldots,x_n]$ and $s_0$ is a root of the Bernstein--Sato polynomial of $f$, then $e^{2\pi is_0}$ is a monodromy eigenvalue of $f$. Therefore Conjecture~\ref{Smcigusa1} implies Conjecture~\ref{mcigusa1}.

The above conjectures were verified by Loeser for polynomials in two variables \cite{Loe88} and for non-degenerated polynomials in several variables subject to extra non-natural technical conditions (see \cite{Loe90} or Theorem~\ref{theoloesernondeg}). In higher dimension or in a more general setting, there are various partial results, e.g., \cite{ACLM02,ACLM05,BorTVBN,BMTmcha,HMY07,LVmcndss,LV09,Loe90,NV10bis,VVmcid2,Vey93,Vey06}.

In \cite{LVmcndss} Lemahieu and Van Proeyen prove the Monodromy Conjecture for the local topological zeta function (a kind of limit of Igusa zeta functions) of a non-degenerated surface singularity. Hence they achieve the result of Loeser for the topological zeta function in dimension three without the extra conditions.

\subsection{Statement of the main theorem}
The (first) goal of this paper is to obtain the result of Lemahieu and Van Proeyen for the original local Igusa zeta function, i.e., to prove Conjecture~\ref{mcigusa1lokaal} for a polynomial in three variables that is non-degenerated over \C\ and \Fp\ with respect to its Newton polyhedron. Before formulating our theorem precisely, let us first define the Newton polyhedron of a polynomial and the notion of non-degeneracy.

\begin{definition}[Newton polyhedron]\label{def_NPad}
Let $R$ be a ring. For $\omega=(\omega_1,\ldots,\omega_n)\in\Zplusn$, we denote by $x^{\omega}$ the corresponding monomial $x_1^{\omega_1}\cdots x_n^{\omega_n}$ in $R[x_1,\ldots,x_n]$. Let $f(x)=f(x_1,\ldots,x_n)=\sum_{\omega\in\Zplusn}a_{\omega}x^{\omega}$ be a nonzero polynomial over $R$ satisfying $f(0)=0$. Denote the support of $f$ by $\supp(f)=\{\omega\in\Zplusn\mid a_{\omega}\neq0\}$. The Newton polyhedron \Gf\ of $f$ is then defined as the convex hull in \Rplusn\ of the set
\begin{equation*}
\bigcup_{\omega\in\supp(f)}\omega+\Rplusn.
\end{equation*}
The global Newton polyhedron $\Gglf$ of $f$ is defined as the convex hull of $\supp(f)$. Clearly we have $\Gf=\Gglf+\Rplusn$.
\end{definition}

\begin{notation}\label{notftauart3}
Let $f$ be as in Definition \ref{def_NPad}. For every face\footnote{By a face of $\Gf$ we mean $\Gf$ itself or one of its proper faces, which are the intersections of $\Gf$ with a supporting hyperplane. See, e.g., \cite{Roc70}.} $\tau$ of the Newton polyhedron $\Gf$ of $f$, we put
\begin{equation*}
\ft(x)=\sum_{\omega\in\tau}a_{\omega}x^{\omega}.
\end{equation*}
\end{notation}

\begin{definition}[Non-degenerated over \C]
Let $f(x)=f(x_1,\ldots,x_n)$ be a nonzero polynomial in $\C[x_1,\ldots,x_n]$ satisfying $f(0)=0$. We say that $f$ is non-degenerated over \C\ with respect to all the faces of its Newton polyhedron \Gf, if for every\footnote{Thus also for \Gf.} face $\tau$ of \Gf, the zero locus $\ft^{-1}(0)\subset\C^n$ of \ft\ has no singularities in \Ccrossn.

We say that $f$ is non-degenerated over \C\ with respect to all the compact faces of its Newton polyhedron, if the same condition is satisfied, but only for the compact faces $\tau$ of \Gf.
\end{definition}

\addtocounter{footnote}{-1}
\begin{definition}[Non-degenerated over \Qp]\label{def_non-degenerated2}
Let $f(x)=f(x_1,\ldots,x_n)$ be a nonzero polynomial in $\Qp[x_1,\ldots,x_n]$ satisfying $f(0)=0$. We say that $f$ is non-degenerated over \Qp\ with respect to all the faces of its Newton polyhedron \Gf, if for every\footnotemark\ face $\tau$ of \Gf, the zero locus $\ft^{-1}(0)\subset\Qpn$ of \ft\ has no singularities in \Qpxn.

We say that $f$ is non-degenerated over \Qp\ with respect to all the compact faces of its Newton polyhedron, if we have the same condition, but only for the compact faces $\tau$ of \Gf.
\end{definition}

\begin{notation}\label{notftaubarart3}
For $f\in\Zp[x_1,\ldots,x_n]$, we denote by $\overline{f}$ the polynomial over \Fp, obtained from $f$, by reducing each of its coefficients modulo $p\Zp$.
\end{notation}

\addtocounter{footnote}{-1}
\begin{definition}[Non-degenerated over \Fp]\label{def_non-degenerated3}
Let $f(x)=f(x_1,\ldots,x_n)$ be a non\-zero polynomial in $\Zp[x_1,\ldots,x_n]$ satisfying $f(0)=0$. We say that $f$ is non-degenerated over \Fp\ with respect to all the faces of its Newton polyhedron \Gf, if for every\footnotemark\ face $\tau$ of \Gf, the zero locus of the polynomial \fbart\ has no singularities in \Fpcrossn, or, equivalently, the system of polynomial congruences
\begin{equation*}
\left\{
\begin{aligned}
\ft(x)&\equiv0\bmod p,\\
\frac{\partial \ft}{\partial x_i}(x)&\equiv0\bmod p;\quad i=1,\ldots,n;
\end{aligned}
\right.
\end{equation*}
has no solutions in \Zpxn.

We say that $f$ is non-degenerated over \Fp\ with respect to all the compact faces of its Newton polyhedron, if the same condition is satisfied, but only for the compact faces $\tau$ of \Gf.
\end{definition}

\begin{remarks}\label{verndcndfp}
\begin{enumerate}
\item Let $f(x_1,\ldots,x_n)\in\Z[x_1,\ldots,x_n]$ be a nonzero polynomial satisfying $f(0)=0$. Suppose that $f$ is non-degenerated over \C\ with respect to all the (compact) faces of its Newton polyhedron \Gf. Then $f$ is non-degenerated over \Fp\ with respect to all the (compact) faces of \Gf, for almost all $p$. This is a consequence of the Weak Nullstellensatz.
\item The condition of non-degeneracy is a generic condition in the following sense. Let $\Gamma\subset\Rplusn$ be a Newton polyhedron. Then almost all\footnote{By \lq almost all\rq\ we mean the following. Let $B$ be any bounded subset of \Rplusn\ that contains all vertices of $\Gamma$. Put $N=\#\Zn\cap\Gamma\cap B$, and associate to every $f(x)\in\C[x_1,\ldots,x_n]$ with $\Gf=\Gamma$ and $\supp(f)\subset B$ an $N$-tuple containing its coefficients. Then the set of $N$-tuples corresponding to a non-degenerated polynomial, is Zariski-dense in $\C^N$.} pol\-y\-no\-mi\-als $f(x)\in\C[x_1,\ldots,x_n]$ with $\Gf=\Gamma$ are non-degenerated over \C\ with respect to all the faces of $\Gamma$. (The same is true if we replace \C\ by \Qp.)
\end{enumerate}
\end{remarks}

We can now state our main theorem.

\begin{theorem}[Monodromy Conjecture for Igusa's $p$-adic local zeta function of a non-degenerated surface singularity]\label{mcigusandss}
Let $f(x,y,z)\in\Z[x,y,z]$ be a nonzero polynomial in three variables satisfying $f(0,0,0)=0$, and let $U\subset\C^3$ be a neighborhood of the origin. Suppose that $f$ is non-degenerated over \C\ with respect to all the compact faces of its Newton polyhedron, and let $p$ be a prime number such that $f$ is also non-degenerated over \Fp\ with respect to the same faces.\footnote{By Remark~\ref{verndcndfp}(i) this is the case for almost all prime numbers $p$.} Suppose that $s_0$ is a pole of the local Igusa zeta function \Zof\ associated to $f$. Then $e^{2\pi i\Re(s_0)}$ is an eigenvalue of the local monodromy of $f$ at some point of $f^{-1}(0)\cap U$.
\end{theorem}

Next we want to state two results of Denef and Hoornaert on Igusa's zeta function for non-degenerated polynomials. To do so, we need some notions that are closely related to Newton polyhedra. We introduce them in the following subsection (see also \cite{BorIZFs,DH01,DL92,HMY07}).

\subsection{Preliminaries on Newton polyhedra}\label{premartdrie}
We gave the definition of a Newton polyhedron in Definition~\ref{def_NPad}. Now we introduce some related notions.

\begin{definition}[$m(k)$]\label{def_mfad}
Let $R$ be a ring, and let $f(x)=f(x_1,\ldots,x_n)$ be a nonzero polynomial over $R$ satisfying $f(0)=0$. For $k\in\Rplusn$, we define
\begin{equation*}
m(k)=\inf_{x\in\Gf}k\cdot x,
\end{equation*}
where $k\cdot x$ denotes the scalar product of $k$ and $x$.
\end{definition}

The infimum in the definition above is actually a minimum, where the minimum can as well be taken over the global Newton polyhedron \Gglf\ of $f$, which is a compact set, or even over the finite set $\supp(f)$.

\begin{definition}[First meet locus]\label{def_firstmeetlocusad}
Let $f$ be as in Definition \ref{def_mfad} and $k\in\Rplusn$. We define the first meet locus of $k$ as the set
\begin{equation*}
F(k)=\{x\in\Gf\mid k\cdot x=m(k)\},
\end{equation*}
which is always a face of \Gf.
\end{definition}

\begin{definition}[Primitive vector]
A vector $k\in\Rn$ is called primitive if the components of $k$ are integers whose greatest common divisor is one.
\end{definition}

\begin{definition}[\Dtu]\label{def_Dfad}
Let $f$ be as in Definition \ref{def_mfad}. For a face $\tau$ of \Gf, we call
\begin{equation*}
\Dtu=\{k\in\Rplusn\mid F(k)=\tau\}
\end{equation*}
the cone associated to $\tau$. The \Dtu\ are the equivalence classes of the equivalence relation $\sim$ on \Rplusn, defined by
\begin{equation*}
k\sim k'\qquad\textrm{if and only if}\qquad F(k)=F(k').
\end{equation*}
The \lq cones\rq\ \Dtu\ thus form a partition of \Rplusn:
\begin{equation*}
\{\Dtu\mid \tau\ \mathrm{is\ a\ face\ of}\ \Gf\}=\Rplusn/\sim.
\end{equation*}
\end{definition}

The \Dtu\ are in fact relatively open\footnote{A subset of \Rplusn\ is called relatively open if it is open in its affine closure.} convex cones\footnote{A subset $C$ of \Rn\ is called a convex cone if it is a convex set and $\lambda x\in C$ for all $x\in C$ and all $\lambda\in\Rplusnul$.} with a very specific structure, as stated in the following lemma.

\begin{lemma}[Structure of the \Dtu]\label{lemma_struc_Dftad}
\textup{\cite[Lemma 2.6]{DH01}}. Let $f$ be as in Definition~\ref{def_mfad}. Let $\tau$ be a proper face of \Gf\ and let $\tau_1,\ldots,\tau_r$ be the facets\footnote{A facet is a face of codimension one.} of \Gf\ that contain $\tau$. Let $v_1,\ldots,v_r$ be the unique primitive vectors in $\Zplusn\setminus\{0\}$ that are perpendicular to $\tau_1,\ldots,\tau_r$, respectively. Then the cone \Dtu\ associated to $\tau$ is the convex cone
\begin{equation*}
\Dtu=\{\lambda_1v_1+\lambda_2v_2+\cdots+\lambda_rv_r\mid \lambda_j\in\Rplusnul\},
\end{equation*}
and its dimension\footnote{The dimension of a convex cone is the dimension of its affine hull.} equals $n-\dim\tau$.
\end{lemma}

\begin{definition}[Rational, simplicial, simple]\label{def_rationalconead}
For $v_1,\ldots,v_r\in\Rn\setminus\{0\}$, we call
\begin{equation*}
\Delta=\cone(v_1,\ldots,v_r)=\{\lambda_1v_1+\lambda_2v_2+\cdots+\lambda_rv_r\mid \lambda_j\in\Rplusnul\}
\end{equation*}
the cone strictly positively spanned by the vectors $v_1,\ldots,v_r$. When the $v_1,\ldots,v_r$ can be chosen from \Zn, we call it a rational cone. If we can choose $v_1,\ldots,v_r$ linearly independent over \R, then $\Delta$ is called a simplicial cone. If $\Delta$ is rational and $v_1,\ldots,v_r$ can be chosen from a \Z-module basis of \Zn, we call $\Delta$ a simple cone.
\end{definition}

It follows from Lemma~\ref{lemma_struc_Dftad} that the topological closures $\overbar{\Dtu}$\footnote{$\overbar{\Dtu}=\{\lambda_1v_1+\lambda_2v_2+\cdots+\lambda_rv_r\mid \lambda_j\in\Rplus\}=\{k\in\Rplusn\mid F(k)\supset\tau\}$.} of the cones \Dtu\ form a fan\footnote{A fan $\mathcal{F}$ is a finite set of rational polyhedral cones such that every face of a cone in $\mathcal{F}$ is contained in $\mathcal{F}$ and the intersection of each two cones $C$ and $C'$ in $\mathcal{F}$ is a face of both $C$ and $C'$.} of rational polyhedral cones\footnote{A rational polyhedral cone is a closed convex cone, generated by a finite subset of \Zn.}.

\begin{remark}
The function $m$ from Definition~\ref{def_mfad} is linear on each $\overbar{\Dtu}$.
\end{remark}

We state without proofs the following two lemmas (see, e.g., \cite{DH01}).

\begin{lemma}[Simplicial decomposition]
Let $\Delta$ be the cone strictly positively spanned by the vectors $v_1,\ldots,v_r\allowbreak\in\Rplusn\setminus\{0\}$. Then there exists a finite partition of $\Delta$ into cones $\delta_i$, such that each $\delta_i$ is strictly positively spanned by a \R-linearly independent subset of $\{v_1,\ldots,v_r\}$. We call such a decomposition a simplicial decomposition of $\Delta$ without introducing new rays.
\end{lemma}

\begin{lemma}[Simple decomposition]
Let $\Delta$ be a rational simplicial cone. Then there exists a finite partition of $\Delta$ into simple cones. (In general, such a decomposition requires the introduction of new rays.)
\end{lemma}

Finally, we need the following notion, which is related to the notion of a simple cone.

\begin{definition}[Multiplicity]\label{def_multad}
Let $v_1,\ldots,v_r$ be \Q-linearly independent vectors in \Zn. The multiplicity of $v_1,\ldots,v_r$, denoted by $\mult(v_1,\ldots,v_r)$, is defined as the index of the lattice $\Z v_1+\cdots+\Z v_r$ in the group of points with integral coordinates in the subspace spanned by $v_1,\ldots,v_r$ of the \Q-vector space~$\Q^n$.

If $\Delta$ is the cone strictly positively spanned by $v_1,\ldots,v_r$, then we define the multiplicity of $\Delta$ as the multiplicity of $v_1,\ldots,v_r$, and we denote it by $\mult\Delta$.
\end{definition}

The following is well-known (see, e.g., \cite[\S 5.3, Thm.~3.1]{Adkins}).

\begin{proposition}
Let $v_1,\ldots,v_r$ be \Q-linearly independent vectors in \Zn. The multiplicity of $v_1,\ldots,v_r$ equals the cardinality of the set
\begin{equation*}
\Zn\cap\left\{\sum\nolimits_{j=1}^rh_jv_j\;\middle\vert\;h_j\in[0,1)\text{ for }j=1,\ldots,r\right\}.
\end{equation*}
Moreover, this number is the greatest common divisor of the absolute values of the determinants of all $(r\times r)$-submatrices of the $(r\times n)$-matrix whose rows contain the coordinates of $v_1,\ldots,v_r$.
\end{proposition}

\begin{remark}\label{premartdrieeinde}
Let $\Delta$ be as in Definition~\ref{def_multad}. Note that $\Delta$ is simple if and only if $\mult\Delta=\mult(v_1,\ldots,v_r)=1$.
\end{remark}

\subsection{Theorems of Denef and Hoornaert}
\begin{notation}\label{notsigmakartdrieintro}
For $k=(k_1,\ldots,k_n)\in\Rn$, we denote $\sigma(k)=k_1+\cdots+k_n$.
\end{notation}

\begin{theorem}\label{theodenef1}
\textup{\cite{Den18,Den95}}.\footnote{The theorem was announced in \cite{Den18} and a proof is written down in \cite{Den95}.} Let $f(x_1,\ldots,x_n)\in\Qp[x_1,\ldots,x_n]$ be a nonzero polynomial with $f(0,\ldots,0)=0$, and $\Phi$ a Schwartz--Bruhat function on \Qpn. Let $\tau_1,\ldots,\tau_r$ be all the facets of \Gf, and let $v_1,\ldots,v_r$ be the unique primitive vectors in $\Zplusn\setminus\{0\}$ that are perpendicular to $\tau_1,\ldots,\tau_r$, respectively. Suppose that $f$ is non-degenerated over \Qp\ with respect to all the compact faces of its Newton polyhedron, and suppose that the support of $\Phi$ is contained in a small enough neighborhood of the origin. If $s_0$ is a pole of $Z_{f,\Phi}$, then
\begin{equation}\label{scpsvolgdenef}
\begin{aligned}
s_0&=-1+\frac{2k\pi i}{\log p}\qquad\text{for some $k\in\Z$, or}\\
s_0&=-\frac{\sigma(v_j)}{m(v_j)}+\frac{2k\pi i}{m(v_j)\log p}
\end{aligned}
\end{equation}
for some $j\in\{1,\ldots,r\}$ with $m(v_j)\neq0$ and some $k\in\Z$.
\end{theorem}

Essential in the proof of Theorem~\ref{mcigusandss} is the following combinatorial formula for \Zof\ for non-degenerated polynomials due to Denef and Hoornaert.

\begin{theorem}\label{formdenhoor}
\textup{\cite[Thm.~4.2]{DH01}}. Let $f(x)=f(x_1,\ldots,x_n)$ be a nonzero pol\-y\-no\-mi\-al in $\Zp[x_1,\ldots,x_n]$ satisfying $f(0)=0$. Suppose that $f$ is non-de\-gen\-er\-ated over \Fp\ with respect to all the compact faces of its Newton polyhedron \Gf. Then the local Igusa $p$-adic zeta function associated to $f$ is the meromorphic complex function
\begin{equation}\label{cfDHinlartdrie}
\Zof=\sum_{\substack{\tau\mathrm{\ compact}\\\mathrm{face\ of\ }\Gf}}L_{\tau}S(\Dtu),
\end{equation}
with
\begin{gather*}
L_{\tau}:s\mapsto L_{\tau}(s)=\left(\frac{p-1}{p}\right)^n-\frac{N_{\tau}}{p^{n-1}}\frac{p^s-1}{p^{s+1}-1},\\
N_{\tau}=\#\left\{x\in\Fpcrossn\;\middle\vert\;\fbart(x)=0\right\},
\end{gather*}
and
\begin{equation*}
S(\Dtu):s\mapsto S(\Dtu)(s)=\sum_{k\in\Zn\cap\Dtu}p^{-\sigma(k)-m(k)s}
\end{equation*}
for every compact face $\tau$ of \Gf.

The $S(\Dtu)$ can be calculated as follows. Choose a decomposition $\{\delta_i\}_{i\in I}$ of the cone \Dtu\ into simplicial cones $\delta_i$ without introducing new rays. Then clearly
\begin{equation}\label{deelsomform2}
S(\Dtu)=\sum_{i\in I}S(\delta_i),
\end{equation}
in which
\begin{equation*}
S(\delta_i):s\mapsto S(\delta_i)(s)=\sum_{k\in\Zn\cap\delta_i}p^{-\sigma(k)-m(k)s}.
\end{equation*}
Suppose that the cone $\delta_i$ is strictly positively spanned by the linearly independent primitive vectors $v_j$, $j\in J_i$, in $\Zplusn\setminus\{0\}$. Then we have
\begin{equation*}
S(\delta_i)(s)=\frac{\Sigma(\delta_i)(s)}{\prod_{j\in J_i}(p^{\sigma(v_j)+m(v_j)s}-1)},
\end{equation*}
with $\Sigma(\delta_i)$ the function
\begin{equation}\label{defSigmadiThDH}
\Sigma(\delta_i):s\mapsto\Sigma(\delta_i)(s)=\sum_hp^{\sigma(h)+m(h)s},
\end{equation}
where $h$ runs through the elements of the set
\begin{equation*}
H(v_j)_{j\in J_i}=\Z^n\cap\lozenge(v_j)_{j\in J_i},
\end{equation*}
with
\begin{equation*}
\lozenge(v_j)_{j\in J_i}=\left\{\sum\nolimits_{j\in J_i}h_jv_j\;\middle\vert\;h_j\in[0,1)\text{ for all }j\in J_i\right\}
\end{equation*}
the fundamental parallelepiped spanned by the vectors $v_j$, $j\in J_i$.
\end{theorem}

\begin{remark}
There exists a global version of this formula for \Zf; the condition is that $f$ is non-degenerated over \Fp\ with respect to \underline{all} the faces of its Newton polyhedron, and the sum \eqref{cfDHinlartdrie} should be taken over \underline{all} the faces of \Gf\ as well (including \Gf\ itself). In the few definitions that follow, we state everything for the local Igusa zeta function \Zof, since this zeta function is the subject of our theorem. Nevertheless, all notions and results have straightforward analogues for \Zf\ (see \cite{DH01}).
\end{remark}

The formula for \Zof\ in the theorem confirms (under slightly different conditions) the result of Denef that if $s_0$ is a pole of \Zof, it must be one of the numbers \eqref{scpsvolgdenef} from Theorem~\ref{theodenef1}. We call these numbers the candidate poles of \Zof.

\subsection{Expected order and contributing faces}
\begin{definition}[Expected order of a candidate pole]
Let $f$ be as in Theorem~\ref{formdenhoor}, and suppose that $s_0$ is a candidate pole of \Zof. We define the expected order of the candidate pole $s_0$ (as a pole of \Zof\ with respect to the formula in Theorem~\ref{formdenhoor}) as
\begin{equation}\label{defexpordpole}
\max\{\text{order of $s_0$ as a pole of $L_{\tau}S(\Dtu)$}\mid\text{$\tau$ face of \Gf}\}.
\end{equation}
Hereby we agree that the order of $s_0$ as a pole of $L_{\tau}S(\Dtu)$ equals zero, if $s_0$ is not a pole of $L_{\tau}S(\Dtu)$. Note that if $\Re(s_0)\neq-1$, we may omit $L_{\tau}$ in \eqref{defexpordpole}.
\end{definition}

\begin{remark}
Clearly the expected order of a candidate pole $s_0$ of \Zof\ is an upper bound for the actual order of $s_0$ as a pole of \Zof.
\end{remark}

\begin{definition}[Contributing vector/face/cone]\label{defcontrinlad}
Let $f$ be as in Theorem~\ref{formdenhoor}, and suppose that $s_0$ is a candidate pole of \Zof. We say that a primitive vector $v\in\Zplusn\setminus\{0\}$ contributes to $s_0$ if $p^{\sigma(v)+m(v)s_0}=1$, or, equivalently, if
\begin{equation*}
s_0=-\frac{\sigma(v)}{m(v)}+\frac{2k\pi i}{m(v)\log p}
\end{equation*}
for some $k\in\Z$. We say that a facet $\tau$ of \Gf\ contributes to the candidate pole $s_0$, if the unique primitive vector $v\in\Zplusn\setminus\{0\}$ that is perpendicular to $\tau$, contributes to $s_0$. More generally, a face of \Gf\ is said to contribute to $s_0$, if it is contained in one or more contributing facets of \Gf. Finally, we say that a cone $\delta=\cone(v_1,\ldots,v_r)$, minimally\footnote{By \lq minimally\rq\ we mean that $\delta\neq\cone(v_1,\ldots,v_{j-1},v_{j+1},\ldots,v_r)$ for all $j\in\{1,\ldots,r\}$.} strictly positively spanned by the primitive vectors $v_1,\ldots,v_r\in\Zplusn\setminus\{0\}$, contributes to $s_0$, if one or more of the vectors $v_j$ contribute to $s_0$. Note that in this way a face $\tau$ of \Gf\ contributes to $s_0$ if and only if its associated cone \Dtu\ does so.
\end{definition}

Let $f$ be as above, and suppose that $s_0$ is a candidate pole of \Zof\ with $\Re(s_0)\neq-1$. From Theorem~\ref{formdenhoor} it should be clear that if we want to investigate whether $s_0$ is actually a pole or not, we only need to consider the sum $\sum L_{\tau}S(\Dtu)$ over the contributing compact faces $\tau$ of \Gf. Furthermore, if, for a contributing compact face $\tau$, in order to deal with $S(\Dtu)$, we consider a simplicial subdivision $\{\delta_i\}_i$ of the cone \Dtu, we only need to take into account the terms of \eqref{deelsomform2} corresponding to the contributing simplicial cones in $\{\delta_i\}_i$, in order to decide whether $s_0$ is a pole or not.

\begin{remark}
Let $f$ be as in Theorem~\ref{formdenhoor}. Let $\tau$ be a facet of \Gf\ and $s_0$ a candidate pole of \Zof. One easily checks that if $\tau$ contributes to $s_0$, then $\Re(s_0)=-1/t_0$, where $(t_0,\ldots,t_0)$ denotes the intersection point of the affine support $\aff\tau$ of $\tau$ with the diagonal $\{(t,\ldots,t)\mid t\in\R\}\subset\Rn$ of the first orthant (see also \cite[Prop.~5.1]{DH01}).
\end{remark}

\subsection{$B_1$-facets and the structure of the proof of the main theorem}
The proof of Theorem~\ref{mcigusandss} consists of three results, namely Theorem~\ref{theoAenL}, Proposition~\ref{propAenL}, and Theorem~\ref{maintheoartdrie}, all stated below. The first two results have been proved by Lemahieu and Van Proeyen in \cite{LVmcndss}; the last one is the subject of the current paper; its proof covers Sections~\ref{fundpar}--\ref{secgeval7art3} (pp.~\pageref{fundpar}--\pageref{eindegrbew}). In order to state the theorems, we have one last important notion to introduce: that of a $B_1$-facet.

\begin{definition}[$B_1$-facet]\label{defbeenfacetad}
Let $R$ be a ring and $n\in\Z_{\geqslant2}$. Let $f(x)=f(x_1,\ldots,x_n)$ be a nonzero polynomial over $R$ satisfying $f(0)=0$. We call a facet $\tau$ of \Gf\ a $B_1$-simplex for a variable $x_i\in\{x_1,\ldots,x_n\}$, if $\tau$ is a simplex with $n-1$ vertices in the coordinate hyperplane $\{x_i=0\}$ and one vertex in the hyperplane $\{x_i=1\}$. We call a facet of \Gf\ a $B_1$-simplex, if it is a $B_1$-simplex for some variable $x_i$.

A facet $\tau$ of \Gf\ is called non-compact for a variable $x_j\in\{x_1,\ldots,x_n\}$, if for every point $(x_1,\ldots,x_n)\in\tau$, we have $(x_1,\ldots,x_{j-1},x_j+1,x_{j+1},\ldots,x_n)\in\tau$. For $j\in\{1,\ldots,n\}$, we shall denote by $\pi_j$ the projection
\begin{equation*}
\pi_j:\Rn\to\R^{n-1}:(x_1,\ldots,x_n)\mapsto(x_1,\ldots,x_{j-1},x_{j+1},\ldots,x_n).
\end{equation*}

Suppose that $n\geqslant3$. We call a facet $\tau$ of \Gf\ a non-compact $B_1$-facet for a variable $x_i$, if $\tau$ is non-compact for precisely one variable $x_j\neq x_i$ and $\pi_j(\tau)$ is a $B_1$-simplex in $\R^{n-1}$ for the variable $x_i$. A facet of \Gf\ is called a non-compact $B_1$-facet, if it is a non-compact $B_1$-facet for some variable $x_i$.

Finally, we call a facet of \Gf\ a $B_1$-facet (or $B_1$ for short) for a variable $x_i$, if it is either a $B_1$-simplex for $x_i$ or a non-compact $B_1$-facet for $x_i$; we call it a $B_1$-facet when it is $B_1$ for some variable $x_i$.
\end{definition}

The first step in the proof of Theorem~\ref{mcigusandss} is the fact that \lq almost all\rq\ candidate poles of \Zof\ induce monodromy eigenvalues; \lq almost all\rq\ means all, except---possibly---those that are only contributed by $B_1$-facets.

\begin{theorem}[On the candidate poles of \Zof\ contributed by non-$B_1$-facets]\label{theoAenL}
Cfr.\ \textup{\cite[Theorem~10]{LVmcndss}}. Let $f$ and $p$ be as in Theorem~\ref{mcigusandss}. Let $s_0$ be a candidate pole of \Zof\ and suppose that $s_0$ is contributed by a facet of \Gf\ that is not a $B_1$-facet. Then $e^{2\pi i\Re(s_0)}$ is an eigenvalue of the local monodromy of $f$ at some point of the surface $f^{-1}(0)\subset\C^3$ close to the origin.
\end{theorem}

The proof of the theorem above relies on Varchenko's formula \cite{Varchenko} for the zeta function of monodromy of $f$ (at the origin) in terms of the Newton polyhedron of $f$, which in turn relies on A'Campo's formula \cite{AC75} for the same zeta function in terms of an embedded resolution of singularities of $f^{-1}(0)\subset\C^3$.

In this context, we would also like to mention the results of Denef--Sperber \cite{denefsperber} and Cluckers \cite{CluckersDUKE,CluckersTAMS} on exponential sums associated to non-degenerated polynomials. Here one also obtains nice results when imposing certain conditions on the faces of the Newton polyhedron that are similar to the one in the theorem above.

This is probably also a good place to state the result of Loeser on the Monodromy Conjecture for non-degenerated singularities. Loeser proves (in general dimension) a result similar to Theorem~\ref{theoAenL}, imposing several, rather technical conditions on the Newton polyhedron's facets.

\begin{theorem}\label{theoloesernondeg}
\textup{\cite{Loe90}}. Let $f(x_1,\ldots,x_n)\in\C[x_1,\ldots,x_n]$ be a nonzero pol\-y\-no\-mial with $f(0,\ldots,0)=0$. Suppose that $f$ is non-degenerated over \C\ with respect to all the compact faces of its Newton polyhedron. Let $\tau_0$ be a compact facet of \Gf, and let $\tau_1,\ldots,\tau_r$ be all the facets of \Gf\ that are different, but not disjoint from $\tau_0$. Denote by $v_0,v_1,\ldots,v_r$ the unique primitive vectors in $\Zplusn\setminus\{0\}$ that are perpendicular to $\tau_0,\tau_1,\ldots,\tau_r$, respectively. Suppose that
\begin{enumerate}
\item ${\ds\frac{\sigma(v_0)}{m(v_0)}<1}$\quad and that
\item ${\ds\frac{1}{\mult(v_0,v_j)}\left(\sigma(v_j)-\frac{\sigma(v_0)}{m(v_0)}m(v_j)\right)\not\in\Z}$\quad for all $j\in\{1,\ldots,r\}$.
\end{enumerate}
Then $-\sigma(v_0)/m(v_0)$ is a root of the local Bernstein--Sato polynomial $b_f^0$ of $f$. Hereby $\mult(v_0,v_j)$ denotes the multiplicity of $v_0$ and $v_j$ (cfr.\ Def\-i\-ni\-tion~\ref{def_multad}).
\end{theorem}

By the result of Malgrange \cite{malgrange} we mentioned earlier, under the conditions of the theorem, we also have that $e^{-2\pi i\sigma(v_0)/m(v_0)}$ is an eigenvalue of the local monodromy of $f$ at some point of $f^{-1}(0)\subset\C^n$ close to the origin. Loeser proves that this remains true if we replace Condition~(i) by $\sigma(v_0)/m(v_0)\not\in\Z$.

Let us go back to Theorem~\ref{theoAenL} and the $B_1$-facets. What can we say about the candidate poles of \Zof\ that are exclusively contributed by $B_1$-facets? In 1984 Denef announced the following theorem (in general dimension) on candidate poles of $Z_{f,\Phi}$ that are contributed by a single $B_1$-simplex.
%
%

\begin{theorem}\label{deneftheonp1}
Cfr.\ \textup{\cite{Den18,DenefSargos92}}.\footnote{The theorem was announced in \cite{Den18} and a proof is sketched in the real case in \cite{DenefSargos92}. This proof is adaptable to the $p$-adic case, but except for dimension three, a complete detailed proof has not been written down yet.} Let $f(x_1,\ldots,x_n)\in\Qp[x_1,\ldots,x_n]$ be a non\-zero polynomial with $f(0,\ldots,0)=0$, and $\Phi$ a Schwartz--Bruhat function on \Qpn. Suppose that $f$ is non-degenerated over \Qp\ with respect to all the compact faces of its Newton polyhedron, and suppose that the support of $\Phi$ is contained in a small enough neighborhood of the origin. Let $\tau_0,\tau_1,\ldots,\tau_r$ be all the facets of \Gf, and let $v_0,v_1,\ldots,v_r$ be the unique primitive vectors in $\Zplusn\setminus\{0\}$ that are perpendicular to $\tau_0,\tau_1,\ldots,\tau_r$, respectively. Suppose that $\tau_0$ is a $B_1$-simplex, that $\sigma(v_0)/m(v_0)\neq1$, and that $\sigma(v_0)/m(v_0)\neq\sigma(v_j)/m(v_j)$ for all $j\in\{1,\ldots,r\}$. Then there is no pole $s_0$ of $Z_{f,\Phi}$ with $\Re(s_0)=-\sigma(v_0)/m(v_0)$.
\end{theorem}

We can restate Denef's theorem as follows.

\begin{theorem}\label{deneftheonp1bis}
Cfr.\ \textup{\cite{Den18,DenefSargos92}}. Let $f(x_1,\ldots,x_n)\in\Qp[x_1,\ldots,x_n]$ be a non\-zero polynomial with $f(0,\ldots,0)=0$, and $\Phi$ a Schwartz--Bruhat function on \Qpn. Suppose that $f$ is non-degenerated over \Qp\ with respect to all the compact faces of its Newton polyhedron, and suppose that the support of $\Phi$ is contained in a small enough neighborhood of the origin. Let $s_0\neq-1$ be a real candidate pole of $Z_{f,\Phi}$. Suppose that exactly one facet of \Gf\ contributes to $s_0$ and that this facet is a $B_1$-simplex. Then there exists no pole of $Z_{f,\Phi}$ with real part $s_0$.
\end{theorem}
%
%

Denef noticed that one cannot expect this theorem to be generally true for candidate poles that are contributed by several $B_1$-simplices. He gave the following counterexample\footnote{Denef in fact showed that for $f=x^n+xy+y^m+z^2$, the candidate pole $-3/2$ (which is contributed by two $B_1$-simplices) is an actual pole of $Z_{f,\Phi}$ for $n,m$ big enough.} in dimension three. We will discuss the example in detail, as it also illustrates Denef and Hoornaert's formula.

%
%
\begin{figure}
\centering
\psset{unit=.0888\textwidth}
\subfigure[Newton polyhedron \Gf\ of $f=x^3+xy+y^2+z^2$ and its faces]{
\psset{Beta=20}\psset{Alpha=65}
\begin{pspicture}(-2.85,-2.6)(4.82,5)
\pstThreeDCoor[xMin=0,yMin=0,zMin=0,xMax=6,yMax=5,zMax=5,labelsep=5pt,linecolor=black,linewidth=1.0pt]
{
\psset{linecolor=black,linewidth=.5pt,linestyle=dashed,subticks=1}
\pstThreeDLine(1,0,0)(1,1,0)\pstThreeDLine(1,1,0)(0,1,0)
}
{
\psset{dotstyle=none,dotscale=1,drawCoor=false}
\psset{linecolor=lightgray,opacity=.6,linewidth=1.2pt,linejoin=1}
\pstThreeDLine*(0,0,5)(0,0,2)(3,0,0)(6,0,0)(6,0,5)(0,0,5)
\pstThreeDLine*(0,0,5)(0,0,2)(0,2,0)(0,5,0)(0,5,5)(0,0,5)
\pstThreeDLine*(6,0,0)(3,0,0)(1,1,0)(0,2,0)(0,5,0)(6,5,0)(6,0,0)
}
{
\psset{dotstyle=none,dotscale=1,drawCoor=false}
\psset{linecolor=black,linewidth=1.2pt,linejoin=1,arrows=->}
\psset{fillstyle=none}
\pstThreeDLine(0,0,2)(0,0,5)
\pstThreeDLine(3,0,0)(6,0,0)
\pstThreeDLine(0,2,0)(0,5,0)
}
{
\psset{dotstyle=none,dotscale=1,drawCoor=false}
\psset{linecolor=black,linewidth=1.2pt,linejoin=1}
\psset{fillcolor=lightgray,opacity=.6,fillstyle=solid}
\pstThreeDLine(0,0,2)(1,1,0)(3,0,0)(0,0,2)(0,2,0)(1,1,0)
}
\pstThreeDPut[pOrigin=t](3,0,-0.2){$A$}
\pstThreeDPut[pOrigin=t](3,0,-0.53){\scriptsize$(3,0,0)$}
\pstThreeDPut[pOrigin=t](1,1,-0.2){$B$}
\pstThreeDPut[pOrigin=t](1,1,-0.53){\scriptsize$(1,1,0)$}
\pstThreeDPut[pOrigin=t](0,2,-0.2){$C$}
\pstThreeDPut[pOrigin=t](0,2,-0.53){\scriptsize$(0,2,0)$}
\pstThreeDPut[pOrigin=bl](0,0.05,2.15){$D${\scriptsize\,$(0,0,2)$}}
\pstThreeDPut[pOrigin=c](1.33,0.33,0.66){$\tau_0$}
\pstThreeDPut[pOrigin=c](0.33,1,0.66){$\tau_1$}
\pstThreeDPut[pOrigin=c](0,2.5,2.5){$\tau_2$}
\pstThreeDPut[pOrigin=c](3,0,2.5){$\tau_3$}
\pstThreeDPut[pOrigin=c](4,3.333,0){$\tau_4$}
\pstThreeDPut[pOrigin=bl](6,0,0.35){$l_x$}
\pstThreeDPut[pOrigin=br](0,5,0.2){$l_y$}
\pstThreeDPut[pOrigin=tl](0,0.15,4.99){$l_z$}
\end{pspicture}
}
\\[+2.25ex]
\subfigure[Cones \Dtu\ associated to the faces $\tau$ of \Gf]{
\psset{unit=.058\textwidth}
\psset{Beta=10}\psset{Alpha=35}
\begin{pspicture}(-9.75,-2)(7.22,10.6)
\pstThreeDCoor[xMin=0,yMin=0,zMin=0,xMax=10,yMax=10,zMax=10,nameX={},nameY={},nameZ={},labelsep=5pt,linecolor=gray,linewidth=1.0pt]
{
\psset{linecolor=gray,linewidth=.5pt,linejoin=1,linestyle=dashed,fillcolor=lightgray,fillstyle=none}
}
{
\psset{linecolor=black,linewidth=1.1pt,linejoin=1,fillcolor=lightgray,fillstyle=none,arrows=->}
\pstThreeDLine(0,0,0)(10,0,0)
\pstThreeDLine(0,0,0)(0,10,0)
\pstThreeDLine(0,0,0)(0,0,10)
}
{
\psset{labelsep=2.5pt}
\uput[90](0,1.2){\psframebox*[framesep=1.5pt,framearc=0]{\darkgray$v_4$}}
}
{
\psset{linecolor=black,linewidth=1.1pt,linejoin=1,fillcolor=lightgray,fillstyle=none}
\pstThreeDLine(0,0,0)(2.22,4.44,3.33)
\pstThreeDLine(0,0,0)(3.33,3.33,3.33)
}
{
\psset{linecolor=darkgray,linewidth=1.2pt,linejoin=1,arrows=->,arrowscale=1,fillcolor=lightgray,fillstyle=none}
\pstThreeDLine(0,0,0)(2,4,3)
\pstThreeDLine(0,0,0)(1.2,1.2,1.2)
\pstThreeDLine(0,0,0)(1.2,0,0)
\pstThreeDLine(0,0,0)(0,1.2,0)
\pstThreeDLine(0,0,0)(0,0,1.2)
}
{
\psset{linecolor=white,linewidth=4pt,linejoin=1,fillcolor=lightgray,fillstyle=none}
\pstThreeDLine(3.10,3.55,3.33)(2.45,4.22,3.33)
\pstThreeDLine(3,4,3)(1.66,6.66,1.66)
}
{
\psset{linecolor=black,linewidth=.4pt,linejoin=1,fillcolor=lightgray,fillstyle=none}
\pstThreeDLine(10,0,0)(3.33,3.33,3.33)
\pstThreeDLine(0,10,0)(2.22,4.44,3.33)
\pstThreeDLine(0,0,10)(0,10,0)(10,0,0)(0,0,10)
\pstThreeDLine(0,0,10)(2.22,4.44,3.33)(3.33,3.33,3.33)(0,0,10)
}
{
\psset{linecolor=black,linewidth=.4pt,linejoin=1,linestyle=dashed,fillcolor=lightgray,fillstyle=none}
\pstThreeDLine(3.33,3.33,3.33)(1.665,6.665,1.665)\pstThreeDLine(1.665,6.665,1.665)(0,10,0)
}
{
\psset{labelsep=3pt}
\uput[-30](.33,1){\darkgray$v_0$}
\uput[-90](-1,-.16){\darkgray$v_2$}
\uput[-90](.7,-.21){\darkgray$v_3$}
}
{
\psset{labelsep=2.5pt}
\uput[180](-.3,.8){\darkgray$v_1$}
}
{
\psset{labelsep=5pt}
\uput[27](2.875,4.215){$\Delta_{l_x}$}
\uput[146](-4.1,4.425){$\Delta_{l_y}$}
\uput[-93](-1.225,-1.21){$\Delta_{l_z}$}
\uput[0](5.75,-1.42){$y,\Delta_{\tau_3}$}
\uput[180](-8.2,-1){$\Delta_{\tau_2},x$}
}
{
\psset{labelsep=4.7pt}
}
\rput(1.9,3.50){$\Delta_A$}
\rput(-.0333,4.4){\psframebox*[framesep=1.3pt,framearc=0]{$\Delta_B$}}
\rput(-2.6,3.50){$\Delta_C$}
\rput(1.05,1.79){$\delta_1$}
\rput(2.475,.53){\psframebox*[framesep=1.3pt,framearc=0]{$\delta_2$}}
\rput(-1.7,.4){$\delta_3$}
%
%
\pstThreeDNode(2.22,4.44,3.33){dt0}
\pstThreeDNode(3.33,3.33,3.33){dt1}
\pstThreeDNode(1.11,2.22,6.65){AB}
\pstThreeDNode(1.11,7.20,1.66){AD}
\pstThreeDNode(1.66,1.66,6.65){BC}
\pstThreeDNode(2.664,3.996,3.33){BD}
\pstThreeDNode(6.65,1.66,1.66){CD}
\rput[l](-9.7,4.6){\rnode{CDlabel}{$\Delta_{[CD]}$}}
\rput[Bl](-9.7,10.18){\rnode{dt1label}{$\Delta_{\tau_1}$}}
\rput[B](-4.8,10.18){\rnode{BClabel}{$\Delta_{[BC]}$}}
\rput[B](0,10.18){$z,\Delta_{\tau_4}$}
\rput[B](3.3,10.18){\rnode{ABlabel}{$\Delta_{[AB]}$}}
\rput[r](7.24,6.3){\rnode{dt0label}{$\Delta_{\tau_0}$}}
\rput[Br](7.24,10.18){\rnode{BDlabel}{$\Delta_{[BD]}$}}
\rput[r](7.24,2.4){\rnode{ADlabel}{$\Delta_{[AD]}$}}
\nccurve[linewidth=.3pt,nodesepB=2.5pt,nodesepA=1pt,angleA=-115,angleB=0]{->}{dt0label}{dt0}
\ncline[linewidth=.3pt,nodesepB=2.5pt,nodesepA=2pt]{->}{dt1label}{dt1}
\ncline[linewidth=.3pt,nodesepB=2pt,nodesepA=1.5pt]{->}{ABlabel}{AB}
\ncline[linewidth=.3pt,nodesepB=2pt,nodesepA=1pt]{->}{ADlabel}{AD}
\ncline[linewidth=.3pt,nodesepB=2pt,nodesepA=1pt]{->}{BClabel}{BC}
\ncline[linewidth=.3pt,nodesepB=2pt,nodesepA=1.5pt]{->}{BDlabel}{BD}
\ncline[linewidth=.3pt,nodesepB=2pt,nodesepA=2.5pt]{->}{CDlabel}{CD}
\end{pspicture}
}
\psset{Beta=30}\psset{Alpha=45}
\caption{Combinatorial data associated to $f=x^3+xy+y^2+z^2$}
\label{figvoorbeeldart3}
\end{figure}
%
%

\begin{example}[Actual pole of \Zof\ only contributed by $B_1$-facets]\label{denefvbart3}
{\normalfont [Denef, 1984]}. Let $p\geqslant3$ be a prime number and consider $f=x^3+xy+y^2+z^2\in\Zp[x,y,z]$. The Newton polyhedron \Gf\ of $f$ and the cones associated to its faces are drawn in Figure~\ref{figvoorbeeldart3}. One checks that $f$ is non-degenerated over \Fp\ with respect to all the compact faces of \Gf\ ($p\neq2$).

Table~\ref{tabelfacetsart3} gives an overview of the facets $\tau_j$ of \Gf, their associated numerical data $(m(v_j),\sigma(v_j))$, and their associated candidate poles of \Zof. Facets $\tau_0$ and $\tau_1$ are $B_1$-simplices, while $\tau_2,\tau_3,\tau_4$ lie in coordinate hyperplanes and hence do not yield any candidate poles. Poles of \Zof\ are therefore among the numbers
\begin{equation*}
s_k=-\frac{3}{2}+\frac{k\pi i}{3\log p},\quad s'_l=-1+\frac{2l\pi i}{\log p};\qquad k,l\in\Z.
\end{equation*}
The candidate poles $s_k$ with $3\nmid k$ are only contributed by $\tau_0$ and have expected order one, while the $s_k$ with $3\mid k$ are contributed by $\tau_0$ and $\tau_1$; the latter have expected order two since the contributing facets $\tau_0$ and $\tau_1$ share the edge $[BD]$.

We will now calculate \Zof\ using Theorem~\ref{formdenhoor} in order to find out which candidate poles are actually poles. Table~\ref{tabelconesart3} provides an overview of \Gf's compact faces and their associated cones and all the data needed to fill in the theorem's formula. The numbers $N_{\tau}$ that appear in the $L_{\tau}(s)$ are listed in the third column. Hereby $N_0$ and $N_1$ represent the numbers
\begin{align*}
N_0&=\#\left\{(x,z)\in(\Fpcross)^2\;\middle\vert\;x^3+z^2=0\right\}\qquad\text{and}\\
N_1&=\#\left\{(y,z)\in(\Fpcross)^2\;\middle\vert\;y^2+z^2=0\right\}.
\end{align*}

The $S(\Dtu)(s)$ can be calculated based on the data on the cones \Dtu\ in the right-hand side of Table~\ref{tabelconesart3}. We find that \Dtu\ is simplicial for every (compact) face $\tau$ of \Gf, except for $\tau=D$. Those cones \Dtu, $\tau\neq D$, are even simple, except for $\Delta_A,\Delta_B,\Delta_{[AB]}$, whose corresponding fundamental parallelepipeds contain besides the origin also the integral point $(1,2,2)$ (see Table~\ref{tabelfacetsart3}). In order to calculate $S(\Delta_D)(s)$, we choose to decompose $\Delta_D$ into the simplicial cones $\delta_1,\delta_2,\delta_3$ that happen to be simple as well (see Table~\ref{tabelconesart3}).

We now obtain \Zof\ as
\begin{equation*}
\Zofs=\sum_{\substack{\tau\mathrm{\ compact}\\\mathrm{face\ of\ }\Gf}}L_{\tau}(s)S(\Dtu)(s)=
\frac{(p-1)(p^{s+3}-1)}{p^3(p^{s+1}-1)(p^{2s+3}-1)}.
\end{equation*}
Note that \Zofs\ does not depend on $N_0$ or $N_1$. We conclude that the candidate poles that are only contributed by a single $B_1$-simplex are not poles. On the other hand we find that the numbers $s_{3k}$, despite being only contributed by $B_1$-simplices, are indeed poles, although their order is lower than expected.
\end{example}

%
%
\begin{table}
\centering
\caption{Numerical data associated to $(1,2,2)$ and the facets of \Gf}\label{tabelfacetsart3}
{
\setlength{\belowrulesep}{.65ex}
\setlength{\aboverulesep}{.4ex}
\setlength{\belowbottomsep}{0ex}
\setlength{\defaultaddspace}{.5em}
\begin{tabular}{*{6}{l}}\toprule
facet $\tau$ & $\tau$ com- & primitive vector & \multirow{2}*{$m(v)$} & \multirow{2}*{$\sigma(v)$} & candidate poles of \Zof\\
of \Gf & pact? & $v\perp\tau$ & & & contributed by $\tau$\\\midrule
$\tau_0$ & yes & $v_0(2,4,3)$ & $6$ & $9$ & ${\ds -\frac32+\frac{k\pi i}{3\log p};\ k\in\Z}$\\\addlinespace
$\tau_1$ & yes & $v_1(1,1,1)$ & $2$ & $3$ & ${\ds -\frac32+\frac{k\pi i}{\log p};\ k\in\Z}$\\\addlinespace
$\tau_2$ & no  & $v_2(1,0,0)$ & $0$ & $1$ & none\\\addlinespace
$\tau_3$ & no  & $v_3(0,1,0)$ & $0$ & $1$ & none\\\addlinespace
$\tau_4$ & no  & $v_4(0,0,1)$ & $0$ & $1$ & none\\\midrule[\heavyrulewidth]
 & & integral point $h$ & $m(h)$ & $\sigma(h)$ & \\\cmidrule{3-5}
 & & $(1,2,2)$ & $3$ & $5$ & \\\bottomrule
\end{tabular}
}
\end{table}
%
%

%
%
\begin{sidewaystable}
\caption{Data associated to the compact faces of \Gf\ and their associated cones}\label{tabelconesart3}
{
\setlength{\belowrulesep}{.8ex}
\setlength{\aboverulesep}{.7ex}
\setlength{\belowbottomsep}{0ex}
\setlength{\defaultaddspace}{.57em}
\begin{tabular}{*{9}{l}}\toprule
face $\tau$ & \multirow{2}*{\fbart} & \multirow{2}*{$N_{\tau}$} & \multirow{2}*{$L_{\tau}(s)$} & cone & $\dim$ & primitive & $\mult$ & \multirow{2}*{$S(\Dtu)(s),S(\delta_i)(s)$}\\
of \Gf & & & & $\Dtu,\delta_i$ & $\Dtu,\delta_i$ & generators & $\Dtu,\delta_i$ & \\\midrule
$A$ & $x^3$ & $0$ & $\bigl(\frac{p-1}{p}\bigr)^3$ & $\Delta_A$ & $3$ & $v_0,v_3,v_4$ & $2$ & $\frac{1+p^{3s+5}}{(p^{6s+9}-1)(p-1)^2}$\\\addlinespace
$B$ & $xy$ & $0$ & $\bigl(\frac{p-1}{p}\bigr)^3$ & $\Delta_B$ & $3$ & $v_0,v_1,v_4$ & $2$ & $\frac{1+p^{3s+5}}{(p^{6s+9}-1)(p^{2s+3}-1)(p-1)}$\\\addlinespace
$C$ & $y^2$ & $0$ & $\bigl(\frac{p-1}{p}\bigr)^3$ & $\Delta_C$ & $3$ & $v_1,v_2,v_4$ & $1$ & $\frac{1}{(p^{2s+3}-1)(p-1)^2}$\\\midrule[.03em]
 & & & & $\delta_1$ & $3$ & $v_0,v_1,v_3$ & $1$ & $\frac{1}{(p^{6s+9}-1)(p^{2s+3}-1)(p-1)}$\\\addlinespace
 & & & & $\delta_2$ & $2$ & $v_1,v_3$     & $1$ & $\frac{1}{(p^{2s+3}-1)(p-1)}$\\\addlinespace
 & & & & $\delta_3$ & $3$ & $v_1,v_2,v_3$ & $1$ & $\frac{1}{(p^{2s+3}-1)(p-1)^2}$\\\addlinespace
$D$ & $z^2$ & $0$ & $\bigl(\frac{p-1}{p}\bigr)^3$ & $\Delta_D$ & $3$ & $v_0,v_1,v_2,v_3$ & -- & $\frac{p^{6s+10}-1}{(p^{6s+9}-1)(p^{2s+3}-1)(p-1)^2}$\\\midrule[.03em]
$[AB]$ & $x^3+xy$ & $(p-1)^2$ & $\bigl(\frac{p-1}{p}\bigr)^3-\bigl(\frac{p-1}{p}\bigr)^2\frac{p^s-1}{p^{s+1}-1}$ & $\Delta_{[AB]}$ & $2$ & $v_0,v_4$ & $2$ & $\frac{1+p^{3s+5}}{(p^{6s+9}-1)(p-1)}$\\\addlinespace
$[BC]$ & $xy+y^2$ & $(p-1)^2$ & $\bigl(\frac{p-1}{p}\bigr)^3-\bigl(\frac{p-1}{p}\bigr)^2\frac{p^s-1}{p^{s+1}-1}$ & $\Delta_{[BC]}$ & $2$ & $v_1,v_4$ & $1$ & $\frac{1}{(p^{2s+3}-1)(p-1)}$\\\addlinespace
$[AD]$ & $x^3+z^2$ & $(p-1)N_0$ & $\bigl(\frac{p-1}{p}\bigr)^3-\frac{(p-1)N_0}{p^2}\frac{p^s-1}{p^{s+1}-1}$ & $\Delta_{[AD]}$ & $2$ & $v_0,v_3$ & $1$ & $\frac{1}{(p^{6s+9}-1)(p-1)}$\\\addlinespace
$[BD]$ & $xy+z^2$ & $(p-1)^2$ & $\bigl(\frac{p-1}{p}\bigr)^3-\bigl(\frac{p-1}{p}\bigr)^2\frac{p^s-1}{p^{s+1}-1}$ & $\Delta_{[BD]}$ & $2$ & $v_0,v_1$ & $1$ & $\frac{1}{(p^{6s+9}-1)(p^{2s+3}-1)}$\\\addlinespace
$[CD]$ & $y^2+z^2$ & $(p-1)N_1$ & $\bigl(\frac{p-1}{p}\bigr)^3-\frac{(p-1)N_1}{p^2}\frac{p^s-1}{p^{s+1}-1}$ & $\Delta_{[CD]}$ & $2$ & $v_1,v_2$ & $1$ & $\frac{1}{(p^{2s+3}-1)(p-1)}$\\\addlinespace
$\tau_0$ & $x^3+xy+z^2$ & $(p-1)^2-N_0$ & $\bigl(\frac{p-1}{p}\bigr)^3-\frac{(p-1)^2-N_0}{p^2}\frac{p^s-1}{p^{s+1}-1}$ & $\Delta_{\tau_0}$ & $1$ & $v_0$ & $1$ & $\frac{1}{p^{6s+9}-1}$\\\addlinespace
$\tau_1$ & $xy+y^2+z^2$ & $(p-1)^2-N_1$ & $\bigl(\frac{p-1}{p}\bigr)^3-\frac{(p-1)^2-N_1}{p^2}\frac{p^s-1}{p^{s+1}-1}$ & $\Delta_{\tau_1}$ & $1$ & $v_1$ & $1$ & $\frac{1}{p^{2s+3}-1}$\\\bottomrule
\end{tabular}
}
\end{sidewaystable}
%
%

In situations as in Example~\ref{denefvbart3} that are not covered by Theorem~\ref{theoAenL}, one needs to prove that the pole in question induces a monodromy eigenvalue. This is done by Lemahieu and Van Proeyen in the following proposition and forms the second step in the proof of Theorem~\ref{mcigusandss}. Note that the two $B_1$-simplices in the example are $B_1$ with respect to different variables.

\begin{proposition}\label{propAenL}
Cfr.\ \textup{\cite[Theorem~15]{LVmcndss}}. Let $f$ and $p$ be as in Theorem~\ref{mcigusandss}. Let $s_0$ be a candidate pole of \Zof\ and suppose that $s_0$ is contributed by two $B_1$-facets of \Gf\ that are \underline{not} $B_1$ for a same variable and that have an edge in common. Then $e^{2\pi i\Re(s_0)}$ is an eigenvalue of the local monodromy of $f$ at some point of the surface $f^{-1}(0)\subset\C^3$ close to the origin.
\end{proposition}

The proof of the proposition again uses Varchenko's formula and is part of the proof of Theorem~15 in \cite{LVmcndss}. In this paper one considers the local topological zeta function \Ztopof\ instead of \Zof; however, the candidate poles of \Ztopof\ are precisely the real parts of the candidate poles of \Zof, and whenever a facet of \Gf\ contributes to a candidate pole $s_0$ of \Zof, it contributes to the candidate pole $\Re(s_0)$ of \Ztopof\ as well; therefore Proposition~\ref{propAenL} follows in the same way.

In order to conclude Theorem~\ref{mcigusandss}, we want to prove that the remaining candidate poles, i.e., candidate poles only contributed by $B_1$-facets, but not satisfying the conditions of Proposition~\ref{propAenL}, are actually not poles. The result is---under slightly different conditions and in dimension three---an optimization of Theorem~\ref{deneftheonp1bis}, partially allowing the candidate pole $s_0$ to be contributed by several $B_1$-facets, including non-compact ones. This is the final step of the proof.

\begin{theorem}[On candidate poles of \Zof\ only contributed by $B_1$-facets]\label{maintheoartdrie}
Let $f(x,y,z)\in\Zp[x,y,z]$ be a nonzero polynomial in three variables with $f(0,0,0)=0$. Suppose that $f$ is non-degenerated over \Fp\ with respect to all the compact faces of its Newton polyhedron. Let $s_0$ be a candidate pole of \Zof\ with $\Re(s_0)\neq-1$, and suppose that $s_0$ is only contributed by $B_1$-facets of \Gf. Assume also that for any pair of contributing $B_1$-facets, we have that
\begin{itemize}
\item[-] either they are $B_1$-facets for a same variable,
\item[-] or they have at most one point in common.
\end{itemize}
Then $s_0$ is not a pole of \Zof.
\end{theorem}

This is the key theorem of the paper that we will prove in the next eight sections. The theorem has been proved for the local topological zeta function by Lemahieu and Van Proeyen \cite[Proposition~14]{LVmcndss}. In our proof we will consider the same seven cases as they did, distinguishing all possible configurations of contributing $B_1$-facets. The idea is to calculate in every case the residue(s) of \Zof\ in the candidate pole in question $s_0$, based on Denef and Hoornaert's formula for \Zof\ for non-degenerated $f$ (Theorem~\ref{formdenhoor}).

The main difficulty in comparison with the topological zeta function approach lies in the calculation of $\Sigma(\delta)(s_0)$ and $\Sigma(\delta)'(s_0)$ for different simplicial cones $\delta$ (see Equation~\eqref{defSigmadiThDH} in Theorem~\ref{formdenhoor}). Where for the topological zeta function it is sufficient to consider the multiplicity of a cone, for the $p$-adic zeta function one has to sum over the lattice points that yield this multiplicity. Lemahieu and Van Proeyen used a computer algebra package to manipulate the rational expressions they obtained for the topological zeta function and so achieved their result; in the $p$-adic case this approach is no longer possible.

In order to deal with the aforementioned sums, we study in Section~\ref{fundpar} the integral points in a three-dimensional fundamental parallelepiped. The aim is to achieve an explicit description of the points that we can use in the rest of the proof to calculate the sums $\Sigma(\delta)(s_0)$ and $\Sigma(\delta)'(s_0)$ over those points. These calculations often lead to polynomial expressions with floor or ceil functions in the exponents; dealing with them forms the second main difficulty of the proof. A third complication is due to the existence of imaginary candidate poles; their residues are usually harder to calculate than those of their real colleagues.
%

\begin{remark}
Although everything in this paper is formulated for \Qp, the results can be generalized in a straightforward way to arbitrary $p$-adic fields. The reason is that Denef and Hoornaert's formula for Igusa's zeta function has a very similar form over finite field extensions of \Qp.
\end{remark}

\subsection{Overview of the paper}
As mentioned before, Section~\ref{fundpar} contains an elaborate study of the integral points in three-dimensional fundamental parallelepipeds.

Sections~\ref{secgeval1art3}--\ref{secgeval7art3} cover the proof of Theorem~\ref{maintheoartdrie}; every section treats one possible configuration of $B_1$-facets contributing to a same candidate pole.

In Section~\ref{sectkarakter} we verify the analogue of Theorem~\ref{maintheoartdrie} for Igusa's zeta function of a polynomial $f(x_1,\ldots,x_n)\in\Zp[x_1,\ldots,x_n]$ and a non-trivial character of \Zpx. This leads to the Monodromy Conjecture in this case as well.

In Section~\ref{sectmotivisch} we state and prove the motivic version of our main theorem; i.e., we obtain the motivic Monodromy Conjecture for a non-degenerated surface singularity. This section also contains a detailed proof of the motivic analogue of Denef and Hoornaert's formula. Our objective is to obtain a formula for the motivic zeta function as an element in the ring $\MC[[T]]$, as it is defined, rather than as an element in some localization of $\MC[[T]]$.\footnote{The ring \MC\ denotes the localization of the Grothendieck ring of complex algebraic varieties with respect to the class of the affine line, while $T$ is a formal indeterminate.} This explains the technicality of the formula and its proof.

\section{On the integral points in a three-dimensional fundamental parallelepiped spanned by primitive vectors}\label{fundpar}
\setcounter{subsection}{-1}
\subsection{Introduction}\label{introductionalgfp}
We recall some basic definitions and results.

\begin{definition}[Primitive vector] A vector $w=(a_1,\ldots,a_n)\in\Zn$ is called primitive if $\gcd(a_1,\ldots,a_n)=1$.
\end{definition}

\begin{definition}[Fundamental parallelepiped] Let $w_1,\ldots,w_t$ be \R-linear independent primitive vectors in $\Zplus^n$. We call the set
\begin{equation*}
\lozenge(w_1,\ldots,w_t)=\left\{\sum\nolimits_{j=1}^th_jw_j\;\middle\vert\;h_j\in[0,1);\quad j=1,\ldots,t\right\}\subset\Rplusn
\end{equation*}
the fundamental parallelepiped spanned by the vectors $w_1,\ldots,w_t$.
\end{definition}

\begin{definition}[Multiplicity] The number of integral points (i.e., points with integer coordinates) in a fundamental parallelepiped $\lozenge(w_1,\ldots,w_t)$ is called the multiplicity of the fundamental parallelepiped. We denote it by
\begin{equation*}
\mult\lozenge(w_1,\ldots,w_t)=\#\left(\Zn\cap\lozenge(w_1,\ldots,w_t)\right).
\end{equation*}
\end{definition}

The following result is well-known.

\begin{proposition}\label{multipliciteit}
Let $w_1,\ldots,w_t$ be \R-linear independent primitive vectors in $\Zplus^n$. The multiplicity of the fundamental parallelepiped $\lozenge(w_1,\ldots,w_t)$ equals the greatest common divisor of the absolute values of the determinants of all $(t\times t)$-submatrices of the $(t\times n)$-matrix whose rows contain the coordinates of $w_1,\ldots,w_t$.
\end{proposition}

\begin{notation}\label{notatiematrices}
We will denote the determinant of a square, real matrix $M=(a_{ij})_{ij}$ by $\det M=\lvert a_{ij}\rvert_{ij}$ and its absolute value by $\abs{\det M}=\lVert a_{ij}\rVert_{ij}$.
\end{notation}

For the rest of this section, we fix three linearly independent primitive vectors $w_1,w_2,w_3\in\Zplus^3$ and denote their coordinates by
\begin{equation*}
w_1(a_1,b_1,c_1),\qquad w_2(a_2,b_2,c_2),\qquad\text{and}\qquad w_3(a_3,b_3,c_3).
\end{equation*}
We also fix notations for the following sets and their cardinalities (cfr.\ Figure~\ref{fundpar3D}):
\begin{alignat*}{3}
H&=\Z^3\cap\lozenge(w_1,w_2,w_3),&\qquad\qquad\ \ \mu&=\#H&&=\mult\lozenge(w_1,w_2,w_3),\\
H_1&=\Z^3\cap\lozenge(w_2,w_3)\subset H,&\mu_1&=\#H_1&&=\mult\lozenge(w_2,w_3),\\
H_2&=\Z^3\cap\lozenge(w_1,w_3)\subset H,&\mu_2&=\#H_2&&=\mult\lozenge(w_1,w_3),\\
H_3&=\Z^3\cap\lozenge(w_1,w_2)\subset H,&\mu_3&=\#H_3&&=\mult\lozenge(w_1,w_2).
\end{alignat*}

%
%
\begin{figure}
\psset{unit=.06\textwidth}
\psset{Alpha=46}
\centering
\begin{pspicture}(-4.8,-2.8)(4.9,5.75)
{\darkgray\pstThreeDCoor[linecolor=darkgray,linewidth=.8pt,xMin=0,yMin=0,zMin=0,xMax=6,yMax=6,zMax=6]}
{
\psset{linecolor=lightgray,linewidth=.8pt,opacity=.8}
\pstThreeDLine*(0,0,0)(5,6,2)(10,8,6)(5,2,4)(0,0,0)
\pstThreeDLine*(0,0,0)(5,6,2)(7,10,9)(2,4,7)(0,0,0)
\pstThreeDLine*(0,0,0)(5,2,4)(7,6,11)(2,4,7)(0,0,0)
}
{
\psset{linecolor=black,linewidth=.8pt,arrows=->,arrowscale=1.2}
\pstThreeDLine(0,0,0)(5,6,2)
\pstThreeDLine(0,0,0)(5,2,4)
\pstThreeDLine(0,0,0)(2,4,7)
}
{
\newgray{mygray}{.8}
\psset{linecolor=mygray,linewidth=2.2pt}
\pstThreeDLine(10.6,9.2,8.1)(11,10,9.5)\pstThreeDLine(8,10.4,9.8)(11,11.6,12.2)
}
{
\psset{linecolor=black,linewidth=.4pt,linestyle=dashed}
\pstThreeDLine(5,6,2)(10,8,6)\pstThreeDLine(5,2,4)(10,8,6)
\pstThreeDLine(5,6,2)(7,10,9)\pstThreeDLine(2,4,7)(7,10,9)
\pstThreeDLine(5,2,4)(7,6,11)\pstThreeDLine(2,4,7)(7,6,11)
\pstThreeDLine(10,8,6)(12,12,13)\pstThreeDLine(7,10,9)(12,12,13)\pstThreeDLine(7,6,11)(12,12,13)
}
\pstThreeDPut(5,4,3){$H_1$}\pstThreeDPut(3.5,5,4.5){$H_2$}\pstThreeDPut(3.5,3,5.5){$H_3$}\pstThreeDPut(6,6,6.5){$H$}
\uput{6pt}[65](1.5,4){$w_1$}\uput{6pt}[160](-2,1){$w_2$}\uput{6pt}[-70](.82,-2.15){$w_3$}
\end{pspicture}
\psset{Alpha=45}
\caption{A fundamental parallelepiped spanned by three primitive vectors in $\Zplus^3$. The sets $H,H_1,H_2,H_3$ denote the intersections of the respective fundamental parallelepipeds with $\Z^3$.}
\label{fundpar3D}
\end{figure}
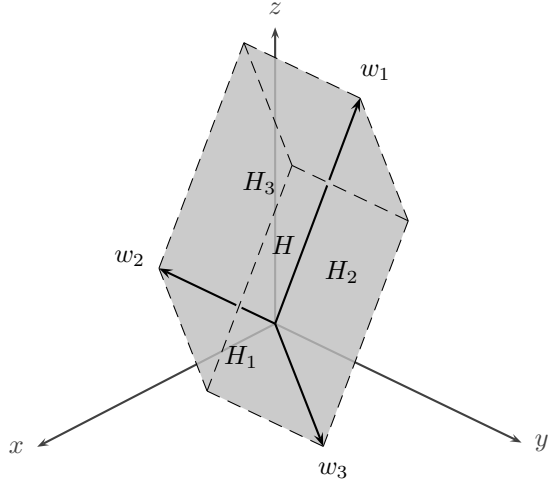
%
%

Throughout this section we consider the matrix
\begin{equation*}
M=
\begin{pmatrix}
a_1&b_1&c_1\\
a_2&b_2&c_2\\
a_3&b_3&c_3
\end{pmatrix}
\in\Zplus^{3\times 3}
\end{equation*}
and its minors
\begin{align*}
M_{11}&=
\begin{pmatrix}
b_2&c_2\\
b_3&c_3
\end{pmatrix},
&
M_{12}&=
\begin{pmatrix}
a_2&c_2\\
a_3&c_3
\end{pmatrix},
&
M_{13}&=
\begin{pmatrix}
a_2&b_2\\
a_3&b_3
\end{pmatrix},\\
M_{21}&=
\begin{pmatrix}
b_1&c_1\\
b_3&c_3
\end{pmatrix},
&
M_{22}&=
\begin{pmatrix}
a_1&c_1\\
a_3&c_3
\end{pmatrix},
&
M_{23}&=
\begin{pmatrix}
a_1&b_1\\
a_3&b_3
\end{pmatrix},\\
M_{31}&=
\begin{pmatrix}
b_1&c_1\\
b_2&c_2
\end{pmatrix},
&
M_{32}&=
\begin{pmatrix}
a_1&c_1\\
a_2&c_2
\end{pmatrix},
&
M_{33}&=
\begin{pmatrix}
a_1&b_1\\
a_2&b_2
\end{pmatrix}.
\end{align*}
Let us denote $d=\det M$ and $d_{ij}=\det M_{ij}$; $i,j=1,2,3$. The matrix
\begin{equation*}
\adj M=
\begin{pmatrix}
d_{11}&-d_{21}&d_{31}\\
-d_{12}&d_{22}&-d_{32}\\
d_{13}&-d_{23}&d_{33}
\end{pmatrix}
=\left((-1)^{i+j}d_{ij}\right)_{ij}^{\mathrm{T}}
\end{equation*}
is called the adjugate matrix of $M$ and has the important property that
\begin{equation*}
(\adj M)M=M(\adj M)=dI,
\end{equation*}
with $I$ the $(3\times 3)$-identity matrix. According to Proposition~\ref{multipliciteit}, we have
\begin{alignat*}{2}
\mu&=\#H&&=\abs{d},\\
\mu_1&=\#H_1&&=\gcd(\abs{d_{11}},\abs{d_{12}},\abs{d_{13}}),\\
\mu_2&=\#H_2&&=\gcd(\abs{d_{21}},\abs{d_{22}},\abs{d_{23}}),\\
\mu_3&=\#H_3&&=\gcd(\abs{d_{31}},\abs{d_{32}},\abs{d_{33}}).
\end{alignat*}

Note that every $h\in H$ can be written in a unique way as
\begin{equation*}
h=h_1w_1+h_2w_2+h_3w_3
\end{equation*}
with $h_j\in[0,1)$; $j=1,2,3$. We shall always denote the coordinates of a point $h\in H$ with respect to the basis $(w_1,w_2,w_3)$ of $\R^3$ over \R, by $(h_1,h_2,h_3)$.

\begin{notation}\label{notatiemodulo}
Let $a\in\R$. We denote by $\lfloor a\rfloor$ the largest integer not greater than $a$ (integer part of $a$) and by $\lceil a\rceil$ the smallest integer not less than $a$. The fractional part of $a$ will be denoted by $\{a\}=a-\lfloor a\rfloor\in[0,1)$. By generalization, we shall denote for any $b\in\Rplusnul$ by $\{a\}_b$ the unique element $\{a\}_b\in[0,b)$ such that $a-\{a\}_b\in b\Z$. Note that
\begin{equation*}
\{a\}_b=b\left\{\frac{a}{b}\right\}.
\end{equation*}
\end{notation}

The aim of this section is to prove the following theorem.

\begin{theorem}\label{algfp}
\begin{enumerate}
\item\label{punteen} The multiplicities $\mu_1,\mu_2,\mu_3$ all divide $\mu$;

\item\label{punttwee} we have even more: for all distinct $i,j\in\{1,2,3\}$ it holds that $\mu_i\mu_j\mid\mu$.

\item\label{puntdrie} For every $h\in H_1$ the coordinates $h_2,h_3$ of $h$ belong to the set
\begin{equation*}
\left\{0,\frac{1}{\mu_1},\frac{2}{\mu_1},\ldots,\frac{\mu_1-1}{\mu_1}\right\},
\end{equation*}
and every element of the above set is the $w_2$-coordinate ($w_3$-coordinate) of exactly one point $h\in H_1$; i.e.,
\begin{equation*}
\{h_2\mid h\in H_1\}=\{h_3\mid h\in H_1\}=\left\{0,\frac{1}{\mu_1},\frac{2}{\mu_1},\ldots,\frac{\mu_1-1}{\mu_1}\right\}.
\end{equation*}

Moreover, there exists a unique $\xi_1\in\vereen$ with $\xi_1+\mu_1\Z$ a generator of the additive group $\Z/\mu_1\Z$ (i.e., with $\gcd(\xi_1,\mu_1)=1$), such that all $\mu_1$ points of $H_1$ are given by
\begin{equation*}
\frac{i}{\mu_1}w_2+\left\{\frac{i\xi_1}{\mu_1}\right\}w_3;\qquad i=0,\ldots,\mu_1-1.
\end{equation*}

Of course, we have analogous results for $H_2$ and $H_3$.

\item\label{puntvier} For every $h\in H$ the coordinate $h_1$ of $h$ belongs to the set
\begin{equation*}
\left\{0,\frac{\mu_1}{\mu},\frac{2\mu_1}{\mu},\ldots,\frac{\mu-\mu_1}{\mu}\right\}.
\end{equation*}
Moreover, every possible coordinate $l\mu_1/\mu$, $l\in\{0,\ldots,\mu/\mu_1-1\}$, occurs precisely $\mu_1$ times. (The set $H$ indeed contains $\mu_1(\mu/\mu_1)=\mu$ points.)

We have of course analogous results for the coordinates $h_2$ and $h_3$ of the points $h\in H$.

\item\label{puntvijf} By (ii) we can write $\mu=\mu_1\mu_2\phi_3$ with $\phi_3\in\Zplusnul$. It then holds that
\begin{equation*}
\gcd(\mu_1,\mu_2)\mid\mu_3\mid\gcd(\mu_1,\mu_2)\phi_3.
\end{equation*}
As a consequence we have that $\gcd(\mu_1,\mu_2,\mu_3)=\gcd(\mu_1,\mu_2)$. (The same result holds, of course, as well for the other two combinations of two out of three multiplicities $\mu_j$.)

\item\label{puntzes} We give an explicit description of the $\mu$ points of $H$.

\item\label{puntzeven} Finally, we explain how the numbers $\xi_1,\xi_2,\xi_3$ (mentioned above), and $\eta,\eta',l_0$ (defined later on) that appear in the several descriptions of points of $H$, can be calculated from the coordinates of $w_1,w_2,w_3$.
\end{enumerate}
\end{theorem}

\subsection{A group structure on $H$}
\begin{notation}
For any $h\in\R^3$ we denote by $\{h\}$ its reduction modulo $\Z w_1+\Z w_2+\Z w_3$; i.e., $\{h\}$ denotes the unique element $\{h\}\in\lozenge(w_1,w_2,w_3)$ such that $h-\{h\}\in\Z w_1+\Z w_2+\Z w_3$. If we write $h$ as $h=h_1w_1+h_2w_2+h_3w_3$ with $h_1,h_2,h_3\in\R$, we have that
\begin{equation*}
\{h\}=\{h_1\}w_1+\{h_2\}w_2+\{h_3\}w_3.
\end{equation*}
\end{notation}

We can make $H$ into a group by considering addition modulo $\Z w_1+\Z w_2+\Z w_3$ as a group law:
\begin{equation*}
\{\cdot +\cdot\}:H\times H\to H:(h,h')\mapsto \{h+h'\}.
\end{equation*}
The operation $\{\cdot +\cdot\}$ makes $H$ into a finite abelian group of order $\mu$. It is easy to verify that the subsets $H_1,H_2,H_3$ of $H$ are in fact subgroups.

Consider the abelian group $\Z^3,+$ and its subgroups
\begin{align*}
\Lambda&=\Z w_1+\Z w_2+\Z w_3,&\Lambda_1&=\Z w_2+\Z w_3,\\
\Lambda_2&=\Z w_1+\Z w_3,&\Lambda_3&=\Z w_1+\Z w_2,
\end{align*}
generated by $\{w_1,w_2,w_3\},\{w_2,w_3\},\{w_1,w_3\},$ and $\{w_1,w_2\}$, respectively. It then holds that
\begin{equation}\label{isomorfismen}
\begin{alignedat}{2}
H&\cong\frac{\Z^3}{\Lambda},&H_1&\cong\frac{\Z^3\cap\left(\R w_2+\R w_3\right)}{\Lambda_1},\\
H_2&\cong\frac{\Z^3\cap\left(\R w_1+\R w_3\right)}{\Lambda_2},&\qquad H_3&\cong\frac{\Z^3\cap\left(\R w_1+\R w_2\right)}{\Lambda_3}.
\end{alignedat}
\end{equation}

\subsection{Divisibility among the multiplicities $\mu,\mu_1,\mu_2,\mu_3$}
Since $H_1,H_2,H_3$ form subgroups of $H$, their orders divide the order of $H$: $\mu_1,\mu_2,\mu_3\mid\mu$ (Theorem~\ref{algfp}(i)).

Consider the subgroups $H_1,H_2$ of $H$. The subgroup $H_1\cap H_2$ precisely contains the integral points in the fundamental parallelepiped
\begin{equation*}
\lozenge(w_3)=\{h_3w_3\mid h_3\in[0,1)\}.
\end{equation*}
Hence since $w_3$ is primitive, $H_1\cap H_2$ is the trivial group (this can also be seen from the isomorphisms \eqref{isomorfismen}). It follows that $H_1+H_2\cong H_1\oplus H_2$ and thus
\begin{equation*}
\abs{H_1+H_2}=\abs{H_1\oplus H_2}=\abs{H_1}\abs{H_2}=\mu_1\mu_2.
\end{equation*}
The fact that $H_1+H_2$ is a subgroup of $H$ now easily implies that $\mu_1\mu_2\mid\mu$. Analogously, we find that $\mu_1\mu_3,\mu_2\mu_3\mid\mu$. This proves Theorem~\ref{algfp}(ii).

From now on, we shall write
\begin{equation*}
\mu=\mu_1\mu_2\phi_3=\mu_1\mu_3\phi_2=\mu_2\mu_3\phi_1
\end{equation*}
with $\phi_1,\phi_2,\phi_3\in\Zplusnul$.

\subsection{On the $\mu_1$ points of $H_1$}
Let $h\in H_1=\Z^3\cap\lozenge(w_2,w_3)$ and write
\begin{equation*}
h=h_2w_2+h_3w_3
\end{equation*}
with $h_2,h_3\in[0,1)$. Because $\abs{H_1}=\mu_1$, the $\mu_1$-th multiple of $h$ in $H$ must equal the identity element:
\begin{equation*}
\{\mu_1h\}=\{\mu_1h_2\}w_2+\{\mu_1h_3\}w_3=(0,0,0);
\end{equation*}
i.e., $\{\mu_1h_2\}=\{\mu_1h_3\}=0$, and thus $h_2,h_3\in(1/\mu_1)\Z$.

Since $h_2,h_3\in(1/\mu_1)\Z$ and $0\leqslant h_2,h_3<1$, the only possible values for $h_2,h_3$ are
\begin{equation*}
0,\frac{1}{\mu_1},\frac{2}{\mu_1},\ldots,\frac{\mu_1-1}{\mu_1}.
\end{equation*}
Moreover, since $w_2,w_3$ are primitive, every $i/\mu_1$, $i\in\vereen$, is the $w_2$-coordinate ($w_3$-coordinate) of at most, and therefore exactly, one point $h\in H_1$:
\begin{equation*}
\{h_2\mid h\in H_1\}=\{h_3\mid h\in H_1\}=\left\{0,\frac{1}{\mu_1},\frac{2}{\mu_1},\ldots,\frac{\mu_1-1}{\mu_1}\right\}.
\end{equation*}

So there exists a unique $\xi_1\in\vereen$ such that
\begin{equation}\label{puntenvanHeen}
h^{\ast}=\frac{1}{\mu_1}w_2+\frac{\xi_1}{\mu_1}w_3\in H_1.
\end{equation}
Consider the cyclic subgroup $\langle h^{\ast}\rangle\subset H_1$ generated by $h^{\ast}$. This subgroup contains the $\mu_1$ distinct elements
\begin{equation*}
\{ih^{\ast}\}=\frac{i}{\mu_1}w_2+\left\{\frac{i\xi_1}{\mu_1}\right\}w_3;\qquad i=0,\ldots,\mu_1-1;
\end{equation*}
of $H_1$, and therefore equals $H_1$. Figure~\ref{fundpar2D} illustrates the situation. This gives us a complete\footnote{In Paragraph~\ref{pardetxieencnul} we explain how to obtain $\xi_1$ from the coordinates of $w_2$ and $w_3$.} description of the points of $H_1$. Besides, since $h_3=\{i\xi_1/\mu_1\}$ runs through $\{0,1/\mu_1,\ldots,(\mu_1-1)/\mu_1\}$ when $i$ runs through $\{0,\ldots,\mu_1-1\}$, we have that $\xi_1+\mu_1\Z$ generates $\Z/\mu_1\Z,+$ and therefore $\gcd(\xi_1,\mu_1)=1$. Obviously, analogous results hold for $H_2$ and $H_3$, concluding Theorem~\ref{algfp}(iii).

%
%
\begin{figure}
\psset{unit=.110864745\textwidth}
\centering
\begin{pspicture}(-2.52,-.92)(6.5,2.92)
{\lightgray\pstThreeDCoor[linecolor=lightgray,linewidth=.7pt,Alpha=135,xMin=0,yMin=0,zMin=0,xMax=2.2,yMax=2.6,zMax=2.97,spotX=0]}
{
\psset{linecolor=black,linewidth=.3pt,linestyle=dashed}
\psline(-2.5,-.91)(-2.5,2.91)\psline(-2.5,2.91)(6.5,2.91)\psline(6.5,2.91)(6.5,-.91)\psline(6.5,-.91)(-2.5,-.91)
}
{
\psset{linecolor=darkgray,linewidth=.3pt,linestyle=dashed}
\psline(.12,.36)(1.8,.6)\psline(.24,.72)(.8,.8)\psline(.36,1.08)(2.6,1.4)\psline(.48,1.44)(1.6,1.6)
\psline(.56,.08)(.8,.8)\psline(1.12,.16)(1.6,1.6)\psline(1.68,.24)(1.8,.6)\psline(2.24,.32)(2.6,1.4)
}
\psdots[dotsize=4pt,dotstyle=o,linecolor=black](-1.8,-.6)(-.8,-.8)(-2,.4)(-1,.2)(0,0)(1,-.2)(2,-.4)(3,-.6)(4,-.8)(-2.2,1.4)(-1.2,1.2)(-.2,1)(.8,.8)(1.8,.6)(2.8,.4)(3.8,.2)(4.8,0)(5.8,-.2)(-2.4,2.4)(-1.4,2.2)(-.4,2)(.6,1.8)(1.6,1.6)(2.6,1.4)(3.6,1.2)(4.6,1)(5.6,.8)(.4,2.8)(1.4,2.6)(2.4,2.4)(3.4,2.2)(4.4,2)(5.4,1.8)(6.4,1.6)(5.2,2.8)(6.2,2.6)
\psdots[dotsize=4pt,linecolor=black](0,0)(.8,.8)(1.8,.6)(1.6,1.6)(2.6,1.4)
{
\psset{linecolor=black,linewidth=1pt,linestyle=dashed}
\psline(.6,1.8)(3.4,2.2)\psline(2.8,.4)(3.4,2.2)
}
{
\psset{linecolor=black,linewidth=1pt,arrows=->,arrowscale=1}
\psline(0,0)(.6,1.8)\psline(0,0)(2.8,.4)
}
\uput[-90](0,0){$0$}\uput[-90](.56,.08){$\frac15$}\uput[-90](1.12,.16){$\frac25$}\uput[-90](1.68,.24){$\frac35$}\uput[-90](2.24,.32){$\frac45$}\uput[180](0,0){$0$}\uput[180](.12,.36){\scriptsize$1/5$}\uput[180](.24,.72){\scriptsize$2/5$}\uput[180](.36,1.08){\scriptsize$3/5$}\uput[180](.48,1.44){\scriptsize$4/5$}\uput{6pt}[-45](2.8,.4){$w_2$}\uput{6pt}[135](.6,1.8){$w_3$}\uput{6pt}[45](3.4,2.2){$w_2+w_3$}\uput{6pt}[135](6.5,-.91){$\R w_2+\R w_3$}\rput{8.4}(2.68,1.864){$\lozenge(w_2,w_3)$}
\end{pspicture}
\caption{Example of a fundamental parallelepiped $\lozenge(w_2,w_3)$ spanned by two primitive vectors $w_2$ and $w_3$ in $\Zplus^3$. The dots represent the integral points in the plane $\R w_2+\R w_3$; the solid dots are the integral points inside the fundamental parallelepiped and make up $H_1$. Observe the coordinates $(h_2,h_3)$ of the points $h\in H_1$ with respect to the basis $(w_2,w_3)$. In this example the multiplicity $\mu_1$ of $\lozenge(w_2,w_3)$ equals $5$ and $\xi_1=2$.}
\label{fundpar2D}
\end{figure}
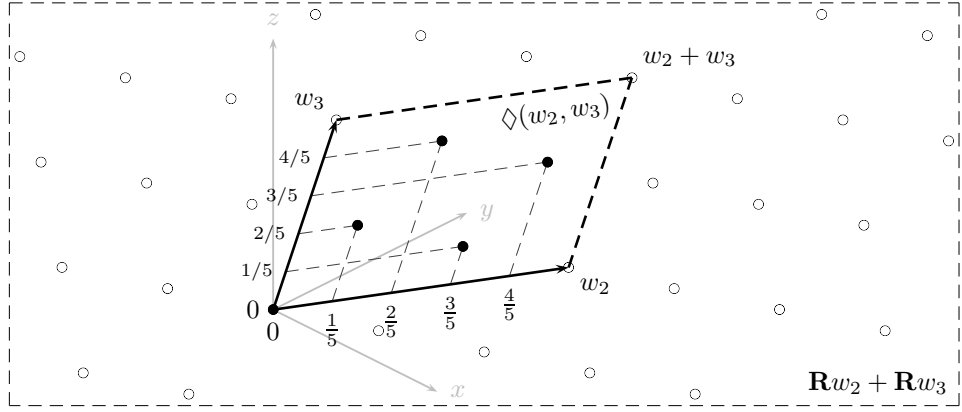
%
%

\subsection{On the $w_1$-coordinates of the points of $H$}
Let $h\in H$ and write
\begin{equation}\label{hgelijkaan}
h=h_1w_1+h_2w_2+h_3w_3
\end{equation}
with $h_1,h_2,h_3\in[0,1)$. Because $\abs{H}=\mu$, it holds that
\begin{equation*}
\{\mu h\}=\{\mu h_1\}w_1+\{\mu h_2\}w_2+\{\mu h_3\}w_3=(0,0,0),
\end{equation*}
and therefore
\begin{equation*}
h_1,h_2,h_3\in\left\{0,\frac{1}{\mu},\frac{2}{\mu},\ldots,\frac{\mu-1}{\mu}\right\}.
\end{equation*}

Let us study the $w_1$-coordinates of the $\mu$ points of $H$ in more detail. Note that the $\mu/\mu_1$ cosets of the subgroup $H_1$ of $H$ form the equivalence classes of the equivalence relation $\sim$ on $H$ defined by
\begin{equation*}
h\sim h'\qquad\text{if and only if}\qquad h_1=h_1';
\end{equation*}
i.e., $H/H_1=H/\sim$ as sets. This implies that there are $\mu/\mu_1$ possible values for the $w_1$-coordinate of a point of $H$, and since every coset of $H_1$ contains $\mu_1$ elements, every possible $w_1$-coordinate occurs precisely $\mu_1$ times.

Moreover, the classes modulo \Z\ of the possible $w_1$-coordinates form a subgroup of $(1/\mu)\Z/\Z$, isomorphic to $H/H_1$. The possible values for the coordinates $h_1$ of the points $h\in H$ are therefore
\begin{equation*}
0,\frac{\mu_1}{\mu},\frac{2\mu_1}{\mu},\ldots,\frac{\mu-\mu_1}{\mu},
\end{equation*}
and every $l\mu_1/\mu$, $l\in\{0,\ldots,\mu/\mu_1-1\}$, is the $w_1$-coordinate of exactly $\mu_1$ points of $H$. Again, there are similar results for the other two coordinates $h_2$ and $h_3$. We conclude Theorem~\ref{algfp}(iv).

\subsection{More divisibility relations}
\begin{notation}\label{notatiegammalambda}
For the remaining of this section, we will use the following notations:
\begin{equation*}
\gamma=\gcd(\mu_1,\mu_2)\qquad\text{and}\qquad\lambda=\lcm(\mu_1,\mu_2)=\frac{\mu_1\mu_2}{\gamma}.
\end{equation*}
We will denote as well $\mu_1'=\mu_1/\gamma$ and $\mu_2'=\mu_2/\gamma$.
\end{notation}

Recall that
\begin{equation*}
\mu=\mu_1\mu_2\phi_3=\mu_1\mu_3\phi_2=\mu_2\mu_3\phi_1.
\end{equation*}
It follows that $\mu_1\mu_3,\mu_2\mu_3\mid\mu_1\mu_2\phi_3$. Hence $\mu_3\mid\mu_1\phi_3,\mu_2\phi_3$ and thus $\mu_3\mid\gamma\phi_3$.

We already know that the subgroup $H_1+H_2$ of $H$ is isomorphic to the direct sum $H_1\oplus H_2$ of $H_1$ and $H_2$ and therefore contains $\mu_1\mu_2$ elements. We can write down the $\mu_1\mu_2$ points of $H_1+H_2$ explicitly.

The $\mu_1$ points of $H_1$ are given by
\begin{equation*}
\left\{\frac{i\xi_1'}{\mu_1}\right\}w_2+\frac{i}{\mu_1}w_3;\qquad i=0,\ldots,\mu_1-1;
\end{equation*}
for some $\xi_1'\in\vereen$ with $\gcd(\xi_1',\mu_1)=1$. We prefer this representation (with $\xi_1'$ instead of $\xi_1$) of the points of $H_1$ to the one on p.~\pageref{puntenvanHeen} (Eq.~\eqref{puntenvanHeen}), because this one is more convenient for what follows.

In the same way, we can list the points of $H_2$ as
\begin{equation*}
\left\{\frac{j\xi_2'}{\mu_2}\right\}w_1+\frac{j}{\mu_2}w_3;\qquad j=0,\ldots,\mu_2-1;
\end{equation*}
for some uniquely determined $\xi_2'\in\vertwee$ relatively prime to $\mu_2$. Consequently, $H_1+H_2$ consists of the following $\mu_1\mu_2$ points:
\begin{multline}\label{puntenvanHeenplusHtwee}
\left\{\frac{j\xi_2'}{\mu_2}\right\}w_1+\left\{\frac{i\xi_1'}{\mu_1}\right\}w_2+\left\{\frac{j\mu_1+i\mu_2}{\mu_1\mu_2}\right\}w_3;\\i=0,\ldots,\mu_1-1;\quad j=0,\ldots,\mu_2-1.
\end{multline}

Let us take a look at the $w_3$-coordinates of the above points. We see that for each $h\in H_1+H_2$, the $w_3$-coordinate $h_3$ is a multiple of $\gamma/\mu_1\mu_2$. Indeed, for all $i,j$ it holds that $\gamma=\gcd(\mu_1,\mu_2)\mid j\mu_1+i\mu_2$. Moreover, $\gamma/\mu_1\mu_2$ and all of its multiples in $[0,1)$ are the $w_3$-coordinate of some point in $H_1+H_2$. Indeed, if we write $\gamma$ as $\gamma=\alpha_1\mu_1+\alpha_2\mu_2$ with $\alpha_1,\alpha_2\in\Z$, we have that
\begin{equation*}
\frac{\gamma}{\mu_1\mu_2}=\left\{\frac{j\mu_1+i\mu_2}{\mu_1\mu_2}\right\}
\end{equation*}
for $j=\{\alpha_1\}_{\mu_2}$ and $i=\{\alpha_2\}_{\mu_1}$. It follows that
\begin{align*}
&\{h_3\mid h\in H_1+H_2\}\\
&\qquad\qquad=\left\{\left\{\frac{j\mu_1+i\mu_2}{\mu_1\mu_2}\right\}\;\middle\vert\;i\in\vereen,\ j\in\vertwee\right\}\\
&\qquad\qquad=\left\{0,\frac{\gamma}{\mu_1\mu_2},\frac{2\gamma}{\mu_1\mu_2},\ldots,\frac{\mu_1\mu_2-\gamma}{\mu_1\mu_2}\right\}\\
&\qquad\qquad=\left\{0,\frac{1}{\lambda},\frac{2}{\lambda},\ldots,\frac{\lambda-1}{\lambda}\right\}.
\end{align*}

As we know, every multiple of $\mu_3/\mu$ in $[0,1)$ is the $w_3$-coordinate of some point in $H$. Choose $h^{\ast}\in H$ with $h^{\ast}_3=\mu_3/\mu$, and consider the coset $h^{\ast}+(H_1+H_2)$ of $H_1+H_2$ in the quotient group $H/(H_1+H_2)$. Since
\begin{equation*}
\left\lvert\frac{H}{H_1+H_2}\right\rvert=\frac{\abs{H}}{\abs{H_1+H_2}}=\frac{\mu}{\mu_1\mu_2}=\phi_3,
\end{equation*}
it holds that
\begin{equation*}
\bigl\{\phi_3\bigl(h^{\ast}+(H_1+H_2)\bigr)\bigr\}=\{\phi_3h^{\ast}\}+(H_1+H_2)=H_1+H_2,
\end{equation*}
and thus $\{\phi_3h^{\ast}\}\in H_1+H_2$.

The $w_3$-coordinate
\begin{equation*}
\{\phi_3h^{\ast}_3\}=\left\{\frac{\phi_3\mu_3}{\mu}\right\}=\left\{\frac{\mu_3}{\mu_1\mu_2}\right\}
\end{equation*}
of $\{\phi_3h^{\ast}\}$ therefore must equal
\begin{equation*}
\left\{\frac{\mu_3}{\mu_1\mu_2}\right\}=\left\{\frac{j\mu_1+i\mu_2}{\mu_1\mu_2}\right\}
\end{equation*}
for some $i\in\vereen$ and some $j\in\vertwee$. It follows that $\mu_3$ is a \Z-linear combination of $\mu_1$ and $\mu_2$; hence $\gamma\mid\mu_3$.

Next, we will count the number of points in $(H_1+H_2)\cap H_3$. Based on the explicit description \eqref{puntenvanHeenplusHtwee} of the $\mu_1\mu_2$ points of $H_1+H_2$, we have to examine for which $(i,j)\in\vereen\times\vertwee$ it holds that
\begin{equation*}
h_3=\left\{\frac{j\mu_1+i\mu_2}{\mu_1\mu_2}\right\}=0.
\end{equation*}

Since
\begin{equation*}
0\leqslant\frac{j\mu_1+i\mu_2}{\mu_1\mu_2}<2,
\end{equation*}
we have that $h_3=0$ if and only if $i=j=0$ or
\begin{equation}\label{eqvoorwaarde}
j\mu_1+i\mu_2=\mu_1\mu_2.
\end{equation}
For this last equality to hold, it is necessary that $\mu_2\mid j\mu_1$ and $\mu_1\mid i\mu_2$, which is equivalent to\footnote{Cfr.\ Notation~\ref{notatiegammalambda}.} $\lambda\mid j\mu_1,i\mu_2$, and even to $\mu_2'\mid j$ and $\mu_1'\mid i$. In other words, Equality~\ref{eqvoorwaarde} implies that
\begin{align*}
&j=\frac{g\mu_2}{\gamma}\qquad\text{and}\qquad i=\frac{g'\mu_1}{\gamma}\\
\intertext{for certain $g,g'\in\{0,\ldots,\gamma-1\}$ and is therefore equivalent to}
&j=\frac{g\mu_2}{\gamma}\qquad\text{and}\qquad i=\frac{(\gamma-g)\mu_1}{\gamma}
\end{align*}
for some $g\in\{1,\ldots,\gamma-1\}$.

We can conclude that $(H_1+H_2)\cap H_3$ contains precisely $\gamma=\gcd(\mu_1,\mu_2)$ points, and they are given by
\begin{equation*}
h=\left\{\frac{g\xi_2'}{\gamma}\right\}w_1+\left\{\frac{(\gamma-g)\xi_1'}{\gamma}\right\}w_2;\qquad g=0,\ldots,\gamma-1.
\end{equation*}

Because $\xi_2'+\mu_2\Z$ and $\xi_1'+\mu_1\Z$ generate $\Z/\mu_2\Z$ and $\Z/\mu_1\Z$, respectively, and $\gamma\mid\mu_1,\mu_2$, it follows that $\xi_1'+\gamma\Z$ and $\xi_2'+\gamma\Z$ are both generators of $\Z/\gamma\Z$. (Moreover, the map $g\mapsto\{\gamma-g\}_{\gamma}$ is a permutation of $\{0,\ldots,\gamma-1\}$.) The coordinates
\begin{equation*}
h_1(g)=\left\{\frac{g\xi_2'}{\gamma}\right\}\qquad\text{and}\qquad h_2(g)=\left\{\frac{(\gamma-g)\xi_1'}{\gamma}\right\}
\end{equation*}
therefore both run through all the elements of
\begin{equation*}
\left\{0,\frac{1}{\gamma},\frac{2}{\gamma},\ldots,\frac{\gamma-1}{\gamma}\right\}
\end{equation*}
when $g$ runs through $\{0,\ldots,\gamma-1\}$. (Hence the maps $g\mapsto h_1(g)$ and $g\mapsto h_2(g)$ from $\{0,\ldots,\gamma-1\}$ to $\{0,1/\gamma,\ldots,(\gamma-1)/\gamma\}$ are both bijections.)

This leads to the existence of a unique $\xi_{\gamma}\in\{0,\ldots,\gamma-1\}$, coprime to $\gamma$, such that the elements of $(H_1+H_2)\cap H_3$ can be represented as
\begin{equation}\label{xigamma}
\frac{g}{\gamma}w_1+\left\{\frac{g\xi_{\gamma}}{\gamma}\right\}w_2;\qquad g=0,\ldots,\gamma-1.
\end{equation}
Hence $(H_1+H_2)\cap H_3$ is the cyclic subgroup of $H$ generated by $(1/\gamma)w_1+(\xi_{\gamma}/\gamma)w_2$.

\begin{remark}
The number $\xi_{\gamma}$ appearing in \eqref{xigamma} is determined by the equality
\begin{equation*}
\xi_{\gamma}+\gamma\Z=-(\xi_2'+\gamma\Z)^{-1}(\xi_1'+\gamma\Z)
\end{equation*}
in the ring $\Z/\gamma\Z,+,\cdot$. Furthermore, we have that $\xi_{\gamma}=\{\xi_3\}_{\gamma}$, with $\xi_3$ the unique element of \verdrie\ such that
\begin{equation*}
\frac{1}{\mu_3}w_1+\frac{\xi_3}{\mu_3}w_2\in H_3.
\end{equation*}
\end{remark}

\begin{remark}
From $\gamma=\gcd(\mu_1,\mu_2)\mid\mu_3$, we see that in fact
\begin{equation*}
\gamma=\gcd(\mu_1,\mu_2,\mu_3)=\gcd(\abs{d_{ij}})_{i,j=1,2,3},
\end{equation*}
and by symmetry, $\gamma=\gcd(\mu_1,\mu_2)=\gcd(\mu_1,\mu_3)=\gcd(\mu_2,\mu_3)$.
\end{remark}

The quotientgroup
\begin{equation}\label{quotientgroep}
\frac{H_1+H_2}{(H_1+H_2)\cap H_3}
\end{equation}
of order $\mu_1\mu_2/\gamma=\lambda$ partitions $H_1+H_2$ based on the $w_3$-coordinates of its points, and
\begin{equation*}
\{h_3+\Z\mid h\in H_1+H_2\},+
\end{equation*}
therefore is the unique subgroup of $\R/\Z,+$ of order $\lambda$. The set of $w_3$-coordinates of points of $H_1+H_2$ is thus $\{0,1/\lambda,\ldots,(\lambda-1)/\lambda\}$ (we already knew that), and since every coset of $(H_1+H_2)\cap H_3$ counts $\gamma$ points, every $l/\lambda$, $l\in\{0,\ldots,\lambda-1\},$ is the $w_3$-coordinate of precisely $\gamma$ points in $H_1+H_2$. This ends the proof of Theorem~\ref{algfp}(v).

\subsection{Explicit description of the points of $H$}
We shall write the $\mu_1$ points of $H_1$ and the $\mu_2$ points of $H_2$ as
\begin{alignat*}{3}
h(i,0,0)&=\frac{i}{\mu_1}w_2&&+\left\{\frac{i\xi_1}{\mu_1}\right\}w_3;&\qquad&i=0,\ldots,\mu_1-1;\\\intertext{and}
h(0,j,0)&=\frac{j}{\mu_2}w_1&&+\left\{\frac{j\xi_2}{\mu_2}\right\}w_3;&&j=0,\ldots,\mu_2-1;
\end{alignat*}
respectively. The $\mu_1\mu_2$ points of $H_1+H_2$ are then given by
\begin{multline*}
h(i,j,0)=\frac{j}{\mu_2}w_1+\frac{i}{\mu_1}w_2+\left\{\frac{i\xi_1\mu_2+j\xi_2\mu_1}{\mu_1\mu_2}\right\}w_3;\\i=0,\ldots,\mu_1-1;\quad j=0,\ldots,\mu_2-1.
\end{multline*}
The $w_3$-coordinate $h_3(i,j,0)$ of $h(i,j,0)$ can also be written as
\begin{equation*}
h_3(i,j,0)=\left\{\frac{i\xi_1\mu_2+j\xi_2\mu_1}{\mu_1\mu_2}\right\}=\frac{l(i,j,0)\gamma}{\mu_1\mu_2}=\frac{l(i,j,0)}{\lambda},
\end{equation*}
with
\begin{multline*}
l(i,j,0)=\frac{\{i\xi_1\mu_2+j\xi_2\mu_1\}_{\mu_1\mu_2}}{\gamma}=\left\{\frac{i\xi_1\mu_2+j\xi_2\mu_1}{\gamma}\right\}_{\lambda}\\
=\{i\xi_1\mu_2'+j\xi_2\mu_1'\}_{\lambda}\in\{0,\ldots,\lambda-1\}
\end{multline*}
for all $i,j$. This results in
\begin{equation*}
h(i,j,0)=\frac{j}{\mu_2}w_1+\frac{i}{\mu_1}w_2+\frac{l(i,j,0)}{\lambda}w_3;\qquad i=0,\ldots,\mu_1-1;\quad j=0,\ldots,\mu_2-1;
\end{equation*}
whereby $l(i,j,0)$ runs exactly $\gamma$ times through all the elements of $\{0,\ldots,\lambda-1\}$ when $i$ and $j$ run through \vereen\ and \vertwee, respectively.

Because $H$ is the disjoint union of the $\phi_3$ cosets of $H_1+H_2$ in $H$, we can describe all elements of $H$ by choosing representatives $h(0,0,k)$; $k=0,\ldots,\phi_3-1$; one for each coset, and then view $H$ as the set
\begin{equation*}
H=H_1+H_2+\{h(0,0,1),\ldots,h(0,0,\phi_3-1)\}.
\end{equation*}

We know that every $h\in H$ has a $w_1$-coordinate $h_1$ of the form $h_1=t\mu_1/\mu$ for some $t\in\{0,\ldots,\mu/\mu_1-1\}$, i.e., of the form
\begin{equation*}
h_1=\frac{j\phi_3+k}{\mu_2\phi_3}
\end{equation*}
for some $j\in\vertwee$ and some $k\in\{0,\ldots,\phi_3-1\}$,
and that every such number $(j\phi_3+k)/\mu_2\phi_3$ is the $w_1$-coordinate of precisely $\mu_1$ points of $H$. In this way, we can associate to each $h\in H$ a number $k=\{\mu_2\phi_3h_1\}_{\phi_3}\in\{0,\ldots,\phi_3-1\}$, and we see that $H/(H_1+H_2)$ is the partition of $H$ based on these values of $k$. The analogous result holds for the $w_2$-coordinates of the points of $H$.

The $w_3$-coordinate $h_3$ of every point $h\in H$ has the form $h_3=t\mu_3/\mu$ for some $t\in\{0,\ldots,\mu/\mu_3-1\}$, and every $t\mu_3/\mu$ is the $w_3$-coordinate of precisely $\mu_3$ points of $H$. Since $\gamma\mid\mu_3\mid\gamma\phi_3$ and we actually have $\gamma=\gcd(\mu_1,\mu_2,\mu_3)$, we can put $\mu_3=\gamma\mu_3'$ and $\phi_3=\mu_3'\phi_3'$ with $\mu_3',\phi_3'\in\Zplusnul$.

We can now write every $w_3$-coordinate $h_3$ as
\begin{equation*}
h_3=\frac{t\mu_3}{\mu}=\frac{t\gamma\mu_3'}{\mu_1\mu_2\phi_3}=\frac{t}{\lambda\phi_3'}
\end{equation*}
for some $t\in\{0,\ldots,\lambda\phi_3'-1\}$, and thus as
\begin{equation*}
h_3=\frac{t}{\lambda\phi_3'}=\frac{l\phi_3'+k'}{\lambda\phi_3'}
\end{equation*}
for some $l\in\{0,\ldots,\lambda-1\}$ and some $k'\in\{0,\ldots,\phi_3'-1\}$.

The $k'=\{\lambda\phi_3'h_3\}_{\phi_3'}$, associated in this way to every $h\in H$, is constant on the cosets of $H_1+H_2$, but points in different cosets may have the same value for $k'$. In fact each of the $\phi_3'$ possible values for $k'$ is adopted in precisely $\mu_3'$ cosets of $H_1+H_2$. (This agrees with the fact that every $(l\phi_3'+k')/\lambda\phi_3'$ appears precisely $\mu_3=\gamma\mu_3'$ times as the $w_3$-coordinate of a point of $H$, considered that each $l/\lambda$ is the $w_3$-coordinate of precisely $\gamma$ points in $H_1+H_2$.)

We can now choose representatives for the elements of $H/(H_1+H_2)$. First, choose a point $h^{\ast}\in H$ with $w_1$-coordinate $h^{\ast}_1=1/\mu_2\phi_3$. The $w_2$-coordinate of this point equals
\begin{equation*}
h^{\ast}_2=\frac{i_0\phi_3+\eta}{\mu_1\phi_3}
\end{equation*}
for some $i_0\in\vereen$ and some $\eta\in\{0,\ldots,\phi_3-1\}$. All $\mu_1$ points of $H$ with $w_1$-coordinate $1/\mu_2\phi_3$ are given by $\{h^{\ast}+h\}$, $h\in H_1$, and their $w_2$-coordinates by
\begin{equation*}
\left\{\frac{(i_0+i)\phi_3+\eta}{\mu_1\phi_3}\right\};\qquad i=0,\ldots,\mu_1-1;
\end{equation*}
this is, after reordering, by
\begin{equation*}
\frac{i\phi_3+\eta}{\mu_1\phi_3};\qquad i=0,\ldots,\mu_1-1.
\end{equation*}
It follows that there exists a unique point $h(0,0,1)\in H$ of the form
\begin{equation*}
h(0,0,1)=\frac{1}{\mu_2\phi_3}w_1+\frac{\eta}{\mu_1\phi_3}w_2+\frac{l_0\phi_3'+\eta'}{\lambda\phi_3'}w_3
\end{equation*}
with $\eta\in\{0,\ldots,\phi_3-1\}$, $l_0\in\{0,\ldots,\lambda-1\}$, and $\eta'\in\{0,\ldots,\phi_3'-1\}$. We will choose this point $h(0,0,1)$ as the representative of its coset $h(0,0,1)+(H_1+H_2)$.

The $\phi_3$ multiples
\begin{multline*}
\{kh(0,0,1)\}=\frac{k}{\mu_2\phi_3}w_1+\left\{\frac{k\eta}{\mu_1\phi_3}\right\}w_2+\left\{\frac{kl_0\phi_3'+k\eta'}{\lambda\phi_3'}\right\}w_3;\\
k=0,\ldots,\phi_3-1;
\end{multline*}
of $h(0,0,1)$ run through all cosets of $H_1+H_2$, and therefore, $h(0,0,1)+(H_1+H_2)$ is a generator of the cyclic group $H/(H_1+H_2)$.

We can choose the elements $\{kh(0,0,1)\}$; $k=0,\ldots,\phi_3-1$; as representatives of their respective cosets, but we can also choose, for each $k$, as a representative for $\{kh(0,0,1)\}+(H_1+H_2)$, the unique element $h(0,0,k)\in\{kh(0,0,1)\}+H_1$ for which $h_2(0,0,k)<1/\mu_1$. We take the last option. This is,
\begin{equation*}
h(0,0,k)=\left\{kh(0,0,1)-\frac{i(k)}{\mu_1}w_2-\frac{i(k)\xi_1}{\mu_1}w_3\right\},
\end{equation*}
with
\begin{equation}\label{defik}
i(k)=\left\{\frac{k\eta-\{k\eta\}_{\phi_3}}{\phi_3}\right\}_{\mu_1}=\left\{\left\lfloor\frac{k\eta}{\phi_3}\right\rfloor\right\}_{\mu_1}\in\vereen
\end{equation}
for $k=0,\ldots,\phi_3-1$, resulting in the following set of representatives for the elements of $H/(H_1+H_2)$:
\begin{equation*}
h(0,0,k)=\frac{k}{\mu_2\phi_3}w_1+\frac{\{k\eta\}_{\phi_3}}{\mu_1\phi_3}w_2+\frac{l(k)\phi_3'+\{k\eta'\}_{\phi_3'}}{\lambda\phi_3'}w_3;\qquad k=0,\ldots,\phi_3-1;
\end{equation*}
with for every $k$,
\begin{equation}\label{deflk}
l(k)=\left\{kl_0-i(k)\xi_1\mu_2'+\left\lfloor\frac{k\eta'}{\phi_3'}\right\rfloor\right\}_{\lambda}\in\{0,\ldots,\lambda-1\},
\end{equation}
$\mu_2'=\mu_2/\gamma$, and $i(k)$ as in \eqref{defik}.

When $k$ runs through $\{0,\ldots,\phi_3-1\}$, the coset $h(0,0,k)+(H_1+H_2)$ runs through all elements of $H/(H_1+H_2)$. This means that $\{k\eta\}_{\phi_3}$ runs through $\{0,\ldots,\phi_3-1\}$ once, while $\{k\eta'\}_{\phi_3'}$ runs through $\{0,\ldots,\phi_3'-1\}$ precisely $\mu_3'$ times. It follows that $\eta+\phi_3\Z$ and $\eta'+\phi_3'\Z$ are generators of $\Z/\phi_3\Z$ and $\Z/\phi_3'\Z$, respectively, and therefore $\gcd(\eta,\phi_3)=\gcd(\eta',\phi_3')=1$.

We can now list all the points of $H$. We start with an overview. The $\mu_1$ points of $H_1$ are
\begin{alignat*}{3}
h(i,0,0)&=\frac{i}{\mu_1}w_2&&+\left\{\frac{i\xi_1}{\mu_1}\right\}w_3;&\qquad&i=0,\ldots,\mu_1-1;\\\intertext{while the $\mu_2$ points of $H_2$ are given by}
h(0,j,0)&=\frac{j}{\mu_2}w_1&&+\left\{\frac{j\xi_2}{\mu_2}\right\}w_3;&&j=0,\ldots,\mu_2-1.
\end{alignat*}
This gives the following $\mu_1\mu_2$ points for $H_1+H_2$:
\begin{multline*}
h(i,j,0)=\{h(i,0,0)+h(0,j,0)\}=\frac{j}{\mu_2}w_1+\frac{i}{\mu_1}w_2+\frac{l(i,j)}{\lambda}w_3;\\i=0,\ldots,\mu_1-1;\quad j=0,\ldots,\mu_2-1;
\end{multline*}
with for all $i,j$,
\begin{equation*}
l(i,j)=\{i\xi_1\mu_2'+j\xi_2\mu_1'\}_{\lambda}\in\{0,\ldots,\lambda-1\}.
\end{equation*}
As representatives for the $\phi_3$ cosets of $H_1+H_2$, we chose
\begin{equation*}
h(0,0,k)=\frac{k}{\mu_2\phi_3}w_1+\frac{\{k\eta\}_{\phi_3}}{\mu_1\phi_3}w_2+\frac{l(k)\phi_3'+\{k\eta'\}_{\phi_3'}}{\lambda\phi_3'}w_3;\qquad k=0,\ldots,\phi_3-1;
\end{equation*}
with $l(k)$ as in \eqref{deflk}.

Consequently, the $\mu=\mu_1\mu_2\phi_3$ points of $H$ are given by
\begin{multline*}
\begin{aligned}
h(i,j,k)&=\{h(i,0,0)+h(0,j,0)+h(0,0,k)\}\\
&=\frac{j\phi_3+k}{\mu_2\phi_3}w_1+\frac{i\phi_3+\{k\eta\}_{\phi_3}}{\mu_1\phi_3}w_2+\frac{l(i,j,k)\phi_3'+\{k\eta'\}_{\phi_3'}}{\lambda\phi_3'}w_3;
\end{aligned}\\
i=0,\ldots,\mu_1-1;\quad j=0,\ldots,\mu_2-1;\quad k=0,\ldots,\phi_3-1;
\end{multline*}
with for all $i,j,k$,
\begin{align*}
l(i,j,k)&=\{l(i,j)+l(k)\}_{\lambda}\\
&=\left\{(i-i(k))\xi_1\mu_2'+j\xi_2\mu_1'+kl_0+\left\lfloor\frac{k\eta'}{\phi_3'}\right\rfloor\right\}_{\lambda}\in\{0,\ldots,\lambda-1\}\qquad\text{and}\\
i(k)&=\left\{\left\lfloor\frac{k\eta}{\phi_3}\right\rfloor\right\}_{\mu_1}\in\vereen,
\end{align*}
and where
\begin{align*}
\xi_1&\in\vereen,&\eta&\in\{0,\ldots,\phi_3-1\},&l_0&\in\{0,\ldots,\lambda-1\}\\
\xi_2&\in\vertwee,&\eta'&\in\{0,\ldots,\phi_3'-1\},&&
\end{align*}
are uniquely determined by
\begin{equation*}
\frac{1}{\mu_1}w_2+\frac{\xi_1}{\mu_1}w_3,\quad\frac{1}{\mu_2}w_1+\frac{\xi_2}{\mu_2}w_3,\quad
\frac{1}{\mu_2\phi_3}w_1+\frac{\eta}{\mu_1\phi_3}w_2+\frac{l_0\phi_3'+\eta'}{\lambda\phi_3'}w_3\quad\in H.
\end{equation*}

We repeat that when $k$ runs through $\{0,\ldots,\phi_3-1\}$, the numbers $\{k\eta\}_{\phi_3}$ and $\{k\eta'\}_{\phi_3'}$ run through $\{0,\ldots,\phi_3-1\}$ and $\{0,\ldots,\phi_3'-1\}$ once and $\mu_3'$ times, respectively, while for fixed $k$, we have that $l(i,j,k)$ runs $\gamma$ times through $\{0,\ldots,\lambda-1\}$ when $i$ and $j$ run through \vereen\ and \vertwee, respectively. This concludes Theorem~\ref{algfp}(vi).

\subsection{Determination of the numbers $\xi_1,\xi_2,\xi_3,\eta,\eta',l_0$ from the coordinates of $w_1,w_2,w_3$}
\subsubsection{The numbers $\xi_1,\xi_2,\xi_3$}\label{pardetxieencnul}
We will give the explanation for $\xi_1$. Recall that we introduced $\xi_1$ as the unique element of \vereen\ for which the $\mu_1$ points of $H_1=\Z^3\cap\lozenge(w_2,w_3)$ are given by
\begin{equation*}
\frac{i}{\mu_1}w_2+\left\{\frac{i\xi_1}{\mu_1}\right\}w_3;\qquad i=0,\ldots,\mu_1-1.
\end{equation*}
How can we find $\xi_1$ from the coordinates of $w_2(a_2,b_2,c_2)$ and $w_3(a_3,b_3,c_3)$?

Consider the vector
\begin{align*}
-a_3w_2+a_2w_3&=(-a_3a_2+a_2a_3,-a_3b_2+a_2b_3,-a_3c_2+a_2c_3)\\
&=(0,d_{13},d_{12})\in\Z^3.
\end{align*}
Since $\mu_1=\gcd(\abs{d_{11}},\abs{d_{12}},\abs{d_{13}})$ divides every coordinate\footnote{with respect to the standard basis of $\R^3$} of $-a_3w_2+a_2w_3$, it holds that
\begin{equation*}
\frac{1}{\mu_1}(-a_3w_2+a_2w_3)=\frac{-a_3}{\mu_1}w_2+\frac{a_2}{\mu_1}w_3\in\Z^3.
\end{equation*}
On the other hand, we have
\begin{equation*}
\frac{1}{\mu_1}w_2+\frac{\xi_1}{\mu_1}w_3\in\Z^3.
\end{equation*}
It follows that also
\begin{equation*}
\frac{-a_3}{\mu_1}w_2+\frac{-a_3\xi_1}{\mu_1}w_3\in\Z^3\qquad\text{and}\qquad\frac{a_2+a_3\xi_1}{\mu_1}w_3\in\Z^3.
\end{equation*}
Since $w_3$ is primitive, we obtain
\begin{equation*}
a_3\xi_1\equiv-a_2\mod\mu_1.
\end{equation*}

Analogously, we find that
\begin{equation*}
b_3\xi_1\equiv-b_2\mod\mu_1\qquad\text{and}\qquad c_3\xi_1\equiv-c_2\mod\mu_1.
\end{equation*}
Consequently, $\xi_1$ is a solution of the following system of linear congruences:
\begin{equation}\label{stelseleenfpld}
\left\{
\begin{alignedat}{2}
a_3x&\equiv-a_2&&\mod\mu_1,\\
b_3x&\equiv-b_2&&\mod\mu_1,\\
c_3x&\equiv-c_2&&\mod\mu_1.
\end{alignedat}
\right.
\end{equation}

The first linear congruence has a solution if and only if $\gcd(a_3,\mu_1)\mid a_2$. We show that this is indeed the case. Put
\begin{equation*}
\gamma_a=\gcd(a_3,\mu_1)=\gcd(a_3,\abs{d_{11}},\abs{d_{12}},\abs{d_{13}}).
\end{equation*}
It then follows from $\gamma_a\mid a_3,d_{12},d_{13}$ that $\gamma_a\mid a_2a_3,a_2b_3,a_2c_3$. Hence
\begin{equation*}
\gamma_a\mid a_2\gcd(a_3,b_3,c_3)=a_2.
\end{equation*}
Analogously we have $\gcd(a_2,\mu_1)\mid a_3$, and thus we can write
\begin{equation*}
\gamma_a=\gcd(a_2,\mu_1)=\gcd(a_3,\mu_1).
\end{equation*}

In the same way the other two congruences have solutions, and we may put
\begin{alignat*}{2}
\gamma_b&=\gcd(b_2,\mu_1)&&=\gcd(b_3,\mu_1)\qquad\text{and}\\
\gamma_c&=\gcd(c_2,\mu_1)&&=\gcd(c_3,\mu_1).
\end{alignat*}
The system \eqref{stelseleenfpld} is then equivalent to
\begin{equation}\label{stelseltweefpld}
\left\{
\begin{alignedat}{3}
x&\equiv-a_2'&&\{a_3'\}_{\mu_1^{(a)}}^{-1}&&\mod\mu_1^{(a)},\\
x&\equiv-b_2'&&\{b_3'\}_{\mu_1^{(b)}}^{-1}&&\mod\mu_1^{(b)},\\
x&\equiv-c_2'&&\{c_3'\}_{\mu_1^{(c)}}^{-1}&&\mod\mu_1^{(c)},
\end{alignedat}
\right.
\end{equation}
with
\begin{equation*}
a_2'=a_2/\gamma_a,\qquad a_3'=a_3/\gamma_a,\qquad \mu_1^{(a)}=\mu_1/\gamma_a,
\end{equation*}
and where $\{a_3'\}_{\mu_1^{(a)}}^{-1}$ denotes the unique element of $\{0,\ldots,\mu_1^{(a)}-1\}$ such that
\begin{equation*}
a_3'\{a_3'\}_{\mu_1^{(a)}}^{-1}\equiv1\mod\mu_1^{(a)}.
\end{equation*}
(Analogously for the numbers appearing in the other two congruences.)

Since the moduli $\mu_1^{(a)},\mu_1^{(b)},\mu_1^{(c)}$ are generally not pairwise coprime, according to the Generalized Chinese Remainder Theorem, the system \eqref{stelseltweefpld} has a solvability condition in the form of
\begin{equation}\label{solvcondsteltwee}
a_2'\{a_3'\}_{\mu_1^{(a)}}^{-1}\equiv b_2'\{b_3'\}_{\mu_1^{(b)}}^{-1}\mod\gcd(\mu_1^{(a)},\mu_1^{(b)}),
\end{equation}
together with the analogous conditions for the other two combinations of two out of three congruences. Of course we know that the system is solvable since $\xi_1$ is a solution, but for the sake of completeness, let us verify Condition~\eqref{solvcondsteltwee} in a direct way.

Because $a_3'$ and $b_3'$ are units modulo $\mu_1^{(a)}$ and $\mu_1^{(b)}$, respectively, they are both units modulo $\gcd(\mu_1^{(a)},\mu_1^{(b)})$. Furthermore, we have that
\begin{equation}\label{viersterrekesfp}
a_3'\{a_3'\}_{\mu_1^{(a)}}^{-1}\equiv b_3'\{b_3'\}_{\mu_1^{(b)}}^{-1}\equiv1\mod\gcd\bigl(\mu_1^{(a)},\mu_1^{(b)}\bigr).
\end{equation}
If we multiply both sides of \eqref{solvcondsteltwee} with the unit $a_3'b_3'$ and apply \eqref{viersterrekesfp}, we find that Condition~\eqref{solvcondsteltwee} is equivalent to
\begin{equation*}
a_2'b_3'\equiv a_3'b_2'\mod\gcd(\mu_1^{(a)},\mu_1^{(b)}),
\end{equation*}
and---after multiplying both sides and the modulus with $\gamma_a\gamma_b$---even to
\begin{multline*}
\gcd(a_2,a_3,b_2,b_3)\mu_1=\gcd(\gamma_a,\gamma_b)\mu_1\\
=\gamma_a\gamma_b\gcd(\mu_1^{(a)},\mu_1^{(b)})\mid\gamma_a\gamma_b(a_2'b_3'-a_3'b_2')=d_{13}.
\end{multline*}

Of course we have that $\mu_1\mid d_{13}$. It is therefore sufficient to show that for every prime $p$ with $p\mid\gcd(a_2,a_3,b_2,b_3)$, it holds that
\begin{equation*}
\ord_pd_{13}\geqslant\ord_p\gcd(a_2,a_3,b_2,b_3)+\ord_p\mu_1.
\end{equation*}
Let $p$ be such a prime. Since $p\mid a_2,b_2$ and $w_2$ is primitive, it certainly holds that $p\nmid c_2$. It now follows from $a_2d_{11}-b_2d_{12}+c_2d_{13}=0$ that
\begin{align*}
\ord_pd_{13}&=\ord_pc_2d_{13}\\
&=\ord_p(-a_2d_{11}+b_2d_{12})\\
&\geqslant\min\{\ord_pa_2+\ord_pd_{11},\ord_pb_2+\ord_pd_{12}\}\\
&\geqslant\min\{\ord_pa_2,\ord_pa_3,\ord_pb_2,\ord_pb_3\}\\
&\qquad\qquad\qquad\qquad\qquad\quad+\min\{\ord_pd_{11},\ord_pd_{12},\ord_pd_{13}\}\\
&=\ord_p\gcd(a_2,a_3,b_2,b_3)+\ord_p\mu_1.
\end{align*}

The system is thus indeed solvable and the Generalized Chinese Remainder Theorem asserts that its solution is unique modulo
\begin{equation*}
\lcm(\mu_1^{(a)},\mu_1^{(b)},\mu_1^{(c)})=\mu_1.
\end{equation*}
We can thus find $\xi_1$ as the unique solution in \vereen\ of the system \eqref{stelseleenfpld}.

\subsubsection{Determination of $\eta,\eta',$ and $l_0$}
Recall that we introduced the numbers $\eta,\eta',l_0$ as the unique $\eta\in\{0,\ldots,\phi_3-1\}$, $l_0\in\{0,\ldots,\lambda-1\}$, and $\eta'\in\{0,\ldots,\phi_3'-1\}$ such that
\begin{equation}\label{tweedriehoekjes}
\frac{1}{\mu_2\phi_3}w_1+\frac{\eta}{\mu_1\phi_3}w_2+\frac{l_0\phi_3'+\eta'}{\lambda\phi_3'}w_3\in H.
\end{equation}
Recall as well that $(\adj M)M=dI$, with $d=\det M$; i.e.,
\begin{equation}\label{adjgelijkheidfpls}
\begin{pmatrix}
d_{11}&-d_{21}&d_{31}\\
-d_{12}&d_{22}&-d_{32}\\
d_{13}&-d_{23}&d_{33}
\end{pmatrix}
\begin{pmatrix}
a_1&b_1&c_1\\
a_2&b_2&c_2\\
a_3&b_3&c_3
\end{pmatrix}=
\begin{pmatrix}
d&0&0\\
0&d&0\\
0&0&d
\end{pmatrix}.
\end{equation}
Let $j\in\{1,2,3\}$. Since $\mu=\abs{d}$ divides $d$, it follows from \eqref{adjgelijkheidfpls} that
\begin{equation*}
h(j)=\frac{d_{1j}}{\mu}w_1-\frac{d_{2j}}{\mu}w_2+\frac{d_{3j}}{\mu}w_3\in\Z^3.
\end{equation*}
Recall also that $\mu=\mu_1\mu_2\phi_3$ and that
\begin{equation*}
\mu_i=\gcd(\abs{d_{i1}},\abs{d_{i2}},\abs{d_{i3}});\qquad i=1,2,3.
\end{equation*}
If we now put $d_{ij}'=d_{ij}/\mu_i$; $i=1,2,3$; we obtain
\begin{align*}
h(j)&=\frac{d_{1j}'}{\mu_2\phi_3}w_1-\frac{d_{2j}'}{\mu_1\phi_3}w_2+\frac{d_{3j}'\mu_3}{\mu_1\mu_2\phi_3}w_3\\
&=\frac{d_{1j}'}{\mu_2\phi_3}w_1-\frac{d_{2j}'}{\mu_1\phi_3}w_2+\frac{d_{3j}'}{\lambda\phi_3'}w_3\in\Z^3.
\end{align*}

On the other hand, we also know the point
\begin{equation*}
h'(j)=\frac{d_{1j}'}{\mu_2\phi_3}w_1+\frac{d_{1j}'\eta}{\mu_1\phi_3}w_2+\frac{d_{1j}'(l_0\phi_3'+\eta')}{\lambda\phi_3'}w_3\in\Z^3
\end{equation*}
with the same $w_1$-coordinate as $h(j)$. After reduction of its coordinates\footnote{with respect to the basis $(w_1,w_2,w_3)$} modulo one, $h(j)-h'(j)$ thus belongs to $H_1$, and since the coordinates\footnotemark[\value{footnote}] of the elements of $H_1$ belong to $(1/\mu_1)\Z$, we have that
\begin{alignat*}{2}
\frac{d_{1j}'\eta}{\mu_1\phi_3}&\equiv-\frac{d_{2j}'}{\mu_1\phi_3}&&\mod\frac{1}{\mu_1}\qquad\text{and}\\
\frac{d_{1j}'(l_0\phi_3'+\eta')}{\lambda\phi_3'}&\equiv\frac{d_{3j}'}{\lambda\phi_3'}&&\mod\frac{1}{\mu_1},
\end{alignat*}
or, equivalently, that
\begin{alignat*}{2}
d_{1j}'\eta&\equiv -d_{2j}'&&\mod\phi_3\qquad\text{and}\\
d_{1j}'\eta'&\equiv d_{3j}'&&\mod\mu_2'\phi_3'.\\\intertext{(Recall that $\lambda=\lcm(\mu_1,\mu_2)=\mu_1\mu_2/\gamma=\mu_1\mu_2'$.) A fortiori, it thus holds that}
d_{1j}'\eta'&\equiv d_{3j}'&&\mod\phi_3'.
\end{alignat*}

We have just showed that $\eta$ and $\eta'$ are solutions of the respective systems of linear congruences
\begin{equation}\label{tweestelseltjes}
\left\{
\begin{alignedat}{2}
d_{11}'x&\equiv -d_{21}'&&\mod\phi_3,\\
d_{12}'x&\equiv -d_{22}'&&\mod\phi_3,\\
d_{13}'x&\equiv -d_{23}'&&\mod\phi_3;
\end{alignedat}
\right.
\qquad\text{and}\qquad
\left\{
\begin{alignedat}{2}
d_{11}'x&\equiv d_{31}'&&\mod\phi_3',\\
d_{12}'x&\equiv d_{32}'&&\mod\phi_3',\\
d_{13}'x&\equiv d_{33}'&&\mod\phi_3'.
\end{alignedat}
\right.
\end{equation}
Moreover, it turns out that $\eta$ and $\eta'$ are the unique solutions of these systems in $\{0,\ldots,\phi_3-1\}$ and $\{0,\ldots,\phi_3'-1\}$, respectively. This gives us a method to determine $\eta$ and $\eta'$ from the coordinates of $w_1,w_2,w_3$. We will study the first system of \eqref{tweestelseltjes} in more detail, for the second system analogous conclusions will be true.

Let us verify the solvability conditions of the first system. The first linear congruence has solutions if and only if $\gcd(d_{11}',\phi_3)\mid d_{21}'$, i.e., if and only if
\begin{equation*}
\gcd(\mu_2d_{11},\mu)\mid\mu_1d_{21}.
\end{equation*}
Put $\Gamma_1=\gcd(\mu_2d_{11},\mu)$. Then we have $\Gamma_1\mid\mu$ and $\Gamma_1\mid d_{11}d_{2j}$ for every $j\in\{1,2,3\}$, and it is sufficient to prove that $\Gamma_1\mid d_{1j}d_{21}$ for every $j$.

We already know that $\Gamma_1\mid d_{11}d_{21}$. Furthermore, from $\adj(\adj M)=dM$, it follows that
\begin{alignat}{3}
d_{11}d_{22}&-d_{12}d_{21}&&=\bigl(\adj(\adj M)\bigr)_{33}&&=dc_3\qquad\text{and}\notag\\
d_{11}d_{23}&-d_{13}d_{21}&&=\bigl(\adj(\adj M)\bigr)_{32}&&=db_3.\label{vgltweeergfpls}
\end{alignat}
We find thus that $\Gamma_1\mid\mu\mid dc_3=d_{11}d_{22}-d_{12}d_{21}$, and together with $\Gamma_1\mid d_{11}d_{22}$, this implies that $\Gamma_1\mid d_{12}d_{21}$. Analogously, it follows from \eqref{vgltweeergfpls} that $\Gamma_1\mid d_{13}d_{21}$. The first linear congruence therefore has solutions, and the same thing holds for the other two congruences.

The first system of \eqref{tweestelseltjes} is now solvable if and only if for all $j_1,j_2\in\{1,2,3\}$, it holds that
\begin{equation}\label{allemaaljekes}
d_{1j_1}'d_{2j_2}'\equiv d_{1j_2}'d_{2j_1}'\mod\gcd\left(\frac{\phi_3}{\gamma_{j_1}},\frac{\phi_3}{\gamma_{j_2}}\right),
\end{equation}
with $\gamma_j=\gcd(d_{1j}',\phi_3)$ for all $j$. Let us verify this for $j_1=1$ and $j_2=2$. For these values of $j_1$ and $j_2$, Condition~\eqref{allemaaljekes} is equivalent to
\begin{equation*}
d_{11}d_{22}\equiv d_{12}d_{21}\mod\gcd\left(\frac{\mu}{\gamma_1},\frac{\mu}{\gamma_2}\right),
\end{equation*}
which follows from $\mu\mid dc_3=d_{11}d_{22}-d_{12}d_{21}$.

The (Generalized) Chinese Remainder Theorem now states that the system has a unique solution modulo
\begin{equation*}
\lcm\left(\frac{\phi_3}{\gamma_1},\frac{\phi_3}{\gamma_2},\frac{\phi_3}{\gamma_3}\right)=\frac{\phi_3}{\gcd(d_{11}',d_{12}',d_{13}',\phi_3)}=\phi_3.
\end{equation*}

An alternative way to find $\eta$ and $\eta'$, and a way to find $l_0$ is as follows. We know that
\begin{equation*}
h(j)=\frac{d_{1j}'}{\mu_2\phi_3}w_1-\frac{d_{2j}'}{\mu_1\phi_3}w_2+\frac{d_{3j}'}{\lambda\phi_3'}w_3\in\Z^3;\qquad j=1,2,3;
\end{equation*}
and that $\gcd(d_{11}',d_{12}',d_{13}')=1$. Find $\lambda_j\in\Z$; $j=1,2,3$; such that $\sum_j\lambda_jd_{1j}'=1$, and consider the point
\begin{equation*}
\left\{\sum\nolimits_j\lambda_jh(j)\right\}=\frac{1}{\mu_2\phi_3}w_1+\left\{\frac{-\sum_j\lambda_jd_{2j}'}{\mu_1\phi_3}\right\}w_2+\left\{\frac{\sum_j\lambda_jd_{3j}'}{\lambda\phi_3'}\right\}w_3\in H.
\end{equation*}
Substract from $\bigl\{\sum_j\lambda_jh(j)\bigr\}$ the point
\begin{equation*}
\frac{i}{\mu_1}w_2+\left\{\frac{i\xi_1}{\mu_1}\right\}w_3\in H_1,
\end{equation*}
with
\begin{equation*}
i=\left\{\left\lfloor\frac{-\sum_j\lambda_jd_{2j}'}{\phi_3}\right\rfloor\right\}_{\mu_1},
\end{equation*}
and find the point
\begin{equation*}
\frac{1}{\mu_2\phi_3}w_1+\frac{\bigl\{-\sum_j\lambda_jd_{2j}'\bigr\}_{\phi_3}}{\mu_1\phi_3}w_2+\left\{\frac{\sum_j\lambda_jd_{3j}'-i\xi_1\mu_2'\phi_3'}{\lambda\phi_3'}\right\}w_3\in H.
\end{equation*}
Because of the uniqueness in $H$ of a point of the form \eqref{tweedriehoekjes}, we find that
\begin{align*}
\eta&=\left\{-\sum\nolimits_j\lambda_jd_{2j}'\right\}_{\phi_3},\\
\eta'&=\left\{\sum\nolimits_j\lambda_jd_{3j}'\right\}_{\phi_3'},\qquad\text{and}\\
l_0&=\left\{\left\lfloor\frac{\sum_j\lambda_jd_{3j}'}{\phi_3'}\right\rfloor-i\xi_1\mu_2'\right\}_{\lambda}.
\end{align*}

\section{Case~I: exactly one facet contributes to $s_0$ and this facet is a $B_1$-simplex}\label{secgeval1art3}
\subsection{Figure and notations}
Without loss of generality, we may assume that the $B_1$-simplex $\tau_0$ contributing to $s_0$ is as drawn in Figure~\ref{figcase1}.

%
%
\begin{figure}
\psset{unit=.03462099125\textwidth}
\centering
\subfigure[$B_1$-simplex $\tau_0$, its subfaces, and its neighbor facets $\tau_1,\tau_2,$ and $\tau_3$]{
\begin{pspicture}(-7.58,-4.95)(6.14,7)
{\footnotesize
\pstThreeDCoor[xMin=0,yMin=0,zMin=0,xMax=10,yMax=8,zMax=7.24,linecolor=black,linewidth=.7pt]
{
\psset{linecolor=black,linewidth=.3pt,linestyle=dashed,subticks=1}
\pstThreeDPlaneGrid[planeGrid=xy](0,0)(2,3)
\pstThreeDPlaneGrid[planeGrid=xy](0,0)(8,4)
\pstThreeDPlaneGrid[planeGrid=xy](0,0)(3,6)
\pstThreeDPlaneGrid[planeGrid=xz](0,0)(2,1)
\pstThreeDPlaneGrid[planeGrid=yz](0,0)(3,1)
\pstThreeDPlaneGrid[planeGrid=xy,planeGridOffset=1](0,0)(2,3)
\pstThreeDPlaneGrid[planeGrid=xz,planeGridOffset=3](0,0)(2,1)
\pstThreeDPlaneGrid[planeGrid=yz,planeGridOffset=2](0,0)(3,1)
}
\pstThreeDPut[pOrigin=c](4.33,4.33,0.33){\psframebox*[framesep=0pt,framearc=0.3]{\phantom{$\tau_0$}}}
{
\psset{dotstyle=none,dotscale=1,drawCoor=false}
\psset{linecolor=black,linewidth=1pt,linejoin=1}
\psset{fillcolor=lightgray,opacity=.6,fillstyle=solid}
\pstThreeDTriangle(8,4,0)(3,6,0)(2,3,1)
}
\pstThreeDPut[pOrigin=t](8,4,0){$A(x_A,y_A,0)$}
\pstThreeDPut[pOrigin=lt](3,6,0){$B(x_B,y_B,0)$}
\pstThreeDPut[pOrigin=lb](2,3,1){\psframebox*[framesep=-.3pt,framearc=1]{$C(x_C,y_C,1)$}}
\pstThreeDPut[pOrigin=c](4.33,4.33,0.33){$\tau_0$}
\pstThreeDPut[pOrigin=lt](5.8,5.7,0){$\tau_3$}
\pstThreeDPut[pOrigin=rb](5,3.2,0.7){$\tau_2$}
\pstThreeDPut[pOrigin=lb](2.3,4.75,0.6){\psframebox*[framesep=0pt,framearc=0.3]{$\tau_1$}}
}
\end{pspicture}
}\hfill\subfigure[Relevant cones associated to relevant faces of~\Gf]{
\psset{unit=.03125\textwidth}
\begin{pspicture}(-7.6,-3.8)(7.6,9.6)
{\footnotesize
\pstThreeDCoor[xMin=0,yMin=0,zMin=0,xMax=10,yMax=10,zMax=10,nameZ={},linecolor=gray,linewidth=.7pt]
{
\psset{linecolor=gray,linewidth=.3pt,linejoin=1,linestyle=dashed,fillcolor=lightgray,fillstyle=none}
\pstThreeDLine(10,0,0)(0,10,0)\pstThreeDLine(0,10,0)(0,0,10)\pstThreeDLine(0,0,10)(10,0,0)
}
{
\psset{linecolor=black,linewidth=.7pt,linejoin=1,fillcolor=lightgray,fillstyle=none}
\pstThreeDLine(0,0,0)(1.13,2.15,6.72)
\pstThreeDLine(0,0,0)(4.67,.888,4.44)
\pstThreeDLine(0,0,0)(.267,4.97,4.78)
\pstThreeDLine(0,0,0)(0,0,10)
}
{
\psset{linecolor=darkgray,linewidth=.8pt,linejoin=1,arrows=->,arrowscale=1,fillcolor=lightgray,fillstyle=none}
\pstThreeDLine(0,0,0)(.678,1.29,4.03)
\pstThreeDLine(0,0,0)(3.74,.710,3.55)
\pstThreeDLine(0,0,0)(.134,2.48,2.39)
\pstThreeDLine(0,0,0)(0,0,2)
}
{
\psset{linecolor=white,linewidth=1.7pt,linejoin=1,fillcolor=lightgray,fillstyle=none}
\pstThreeDLine(2.900,1.5190,5.580)(1.484,2.0238,6.492)
\pstThreeDLine(2.9088,2.5208,4.576)(2.0282,3.3372,4.644)
}
{
\psset{linecolor=black,linewidth=.7pt,linejoin=1,fillcolor=lightgray,fillstyle=none}
\pstThreeDLine(0,0,10)(1.13,2.15,6.72)
\pstThreeDLine(.267,4.97,4.78)(1.13,2.15,6.72)(4.67,.888,4.44)
}
{
\psset{linecolor=black,linewidth=.7pt,linejoin=1,linestyle=dashed,fillcolor=lightgray,fillstyle=none}
\pstThreeDLine(0,0,10)(4.67,.888,4.44)
\pstThreeDLine(0,0,10)(.267,4.97,4.78)
\pstThreeDLine(4.67,.888,4.44)(.267,4.97,4.78)
}
{
\psset{labelsep=2pt}
\uput[-12](.2,1.3){\darkgray\scriptsize$v_0$}
\uput[-50](.8,.55){\darkgray\scriptsize$v_2$}
}
{
\psset{labelsep=1.5pt}
\uput[230](-1.05,.75){\darkgray\scriptsize$v_1$}
\uput[180](0,.8){\darkgray\scriptsize$v_3$}
}
{
\psset{labelsep=4pt}
\uput[0](3.3,2.3){\psframebox*[framesep=0.3pt,framearc=0]{$\Delta_{\tau_2}$}}
}
{
\psset{labelsep=2.6pt}
\uput[180](-2.63,1.9){\psframebox*[framesep=0.3pt,framearc=0]{$\Delta_{\tau_1}$}}
\uput[90](0,8.65){$z,\Delta_{\tau_3}$}
}
\rput(1.35,4.95){$\delta_A$}
\rput(-.34,4.95){\psframebox*[framesep=0.3pt,framearc=0]{$\delta_B$}}
\rput(.575,3.13){\psframebox*[framesep=0.3pt,framearc=0]{$\delta_C$}}
\pstThreeDNode(2.90,1.52,5.60){BC}
\pstThreeDNode(1.13,2.15,6.72){dt0}
\pstThreeDNode(.565,1.08,8.35){AB}
\pstThreeDNode(.700,3.56,5.75){AC}
\rput[Bl](-7.95,9.075){\rnode{BClabel}{$\Delta_{[BC]}$}}
\rput[Bl](-4.11,9.075){\rnode{dt0label}{$\Delta_{\tau_0}$}}
\rput[Bl](2.43,9.075){\rnode{ABlabel}{$\Delta_{[AB]}$}}
\rput[Br](7.98,9.075){\rnode{AClabel}{$\Delta_{[AC]}$}}
\ncline[linewidth=.15pt,nodesepB=2pt,nodesepA=1pt]{->}{BClabel}{BC}
\ncline[linewidth=.15pt,nodesepB=2.5pt,nodesepA=2pt]{->}{dt0label}{dt0}
\ncline[linewidth=.15pt,nodesepB=2pt,nodesepA=1pt]{->}{ABlabel}{AB}
\ncline[linewidth=.15pt,nodesepB=2pt,nodesepA=2pt]{->}{AClabel}{AC}
}
\end{pspicture}
}
\caption{Case I: the only facet contributing to $s_0$ is the $B_1$-simplex $\tau_0$}
\label{figcase1}
\end{figure}
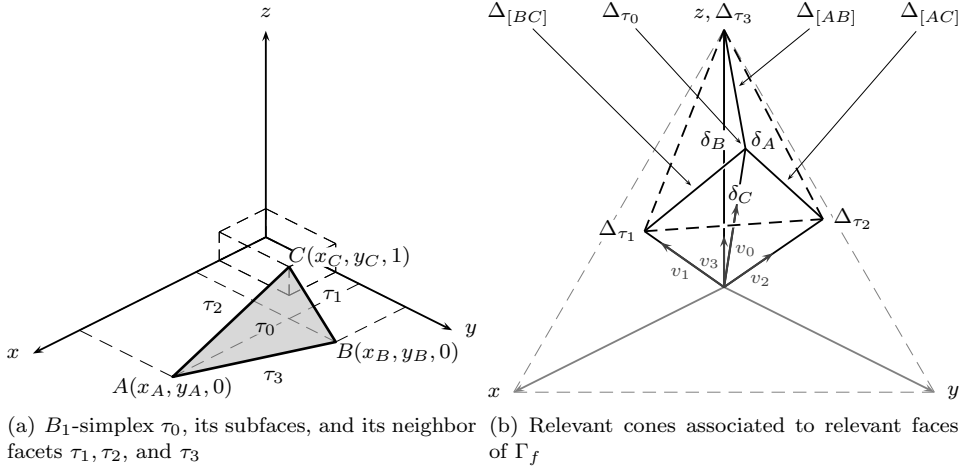
%
%

Let us fix notations. We shall denote the vertices of $\tau_0$ and their co\-or\-di\-nates by
\begin{equation*}
A(x_A,y_A,0),\quad B(x_B,y_B,0),\quad\text{and}\quad C(x_C,y_C,1).
\end{equation*}
The neighbor facets of $\tau_0$ will be denoted $\tau_1,\tau_2,\tau_3$, as indicated in Figure~\ref{figcase1}, and the unique primitive vectors perpendicular to them will be denoted by
\begin{equation*}
v_0(a_0,b_0,c_0),\quad v_1(a_1,b_1,c_1),\quad v_2(a_2,b_2,c_2),\quad v_3(0,0,1),
\end{equation*}
respectively. Consequently, the affine supports of the considered facets should have equations of the form
\begin{alignat*}{4}
\aff(\tau_0)&\leftrightarrow a_0x&&+b_0y&&+c_0&&z=m_0,\\
\aff(\tau_1)&\leftrightarrow a_1x&&+b_1y&&+c_1&&z=m_1,\\
\aff(\tau_2)&\leftrightarrow a_2x&&+b_2y&&+c_2&&z=m_2,\\
\aff(\tau_3)&\leftrightarrow &&&&&&z=0,
\end{alignat*}
and we associate to them the following numerical data:
\begin{align*}
m_0=m(v_0)&=a_0x_C+b_0y_C+c_0,&\sigma_0=\sigma(v_0)&=a_0+b_0+c_0,\\
m_1=m(v_1)&=a_1x_C+b_1y_C+c_1,&\sigma_1=\sigma(v_1)&=a_1+b_1+c_1,\\
m_2=m(v_2)&=a_2x_C+b_2y_C+c_2,&\sigma_2=\sigma(v_2)&=a_2+b_2+c_2,\\
m_3=m(v_3)&=0,&\sigma_3=\sigma(v_3)&=1.
\end{align*}
We assume that $\tau_0$ (and only $\tau_0$) contributes to the candidate pole $s_0$. With the notations above this is, we assume that $p^{\sigma_0+m_0s_0}=1$, or equivalently, that
\begin{equation*}
\Re(s_0)=-\frac{\sigma_0}{m_0}=-\frac{a_0+b_0+c_0}{a_0x_C+b_0y_C+c_0}\qquad\text{and}\qquad\Im(s_0)=\frac{2n\pi}{m_0\log p}
\end{equation*}
for some $n\in\Z$.

In this section we will consider the following simplicial cones:
\begin{align*}
\dA&=\cone(v_0,v_2,v_3),&\DAB&=\cone(v_0,v_3),&\Dtnul&=\cone(v_0).\\
\dB&=\cone(v_0,v_1,v_3),&\DAC&=\cone(v_0,v_2),&&\\
\dC&=\cone(v_0,v_1,v_2),&\DBC&=\cone(v_0,v_1),&&
\end{align*}
The $\Delta_{\tau}$ are the simplicial cones associated to the faces $\tau$ in the usual way. The cones $\Delta_A,\Delta_B,\Delta_C$, associated to the vertices of $\tau_0$, are generally not simplicial. Later in this section we will consider simplicial subdivisions (without creating new rays) of $\Delta_A,\Delta_B,\Delta_C$ that contain the respective simplicial cones $\dA,\dB,\dC$.

Lastly, we fix notations for the vectors along the edges of $\tau_0$:
\begin{alignat*}{8}
&\overrightarrow{AC}&&(x_C&&-x_A&&,y_C&&-y_A&&,1&&)&&=(\aA,\bA,1),\\
&\overrightarrow{BC}&&(x_C&&-x_B&&,y_C&&-y_B&&,1&&)&&=(\aB,\bB,1),\\
&\overrightarrow{AB}&&(x_B&&-x_A&&,y_B&&-y_A&&,0&&)&&=(\aA-\aB,\bA-\bB,0).
\end{alignat*}
The first two vectors are primitive; the last one is generally not. We put
\begin{equation*}
\fAB=\gcd(x_B-x_A,y_B-y_A)=\gcd(\aA-\aB,\bA-\bB).
\end{equation*}

\subsection{Some relations between the variables}\label{srelbettvarcaseen}
Expressing that $\overrightarrow{AC}\perp v_0,v_2$ and $\overrightarrow{BC}\perp v_0,v_1$, we obtain
\begin{alignat*}{4}
\begin{pmatrix}
c_0\\c_2
\end{pmatrix}
&=-\aA&&
\begin{pmatrix}
a_0\\a_2
\end{pmatrix}
&&-\bA&&
\begin{pmatrix}
b_0\\b_2
\end{pmatrix}\qquad\text{and}\\
\begin{pmatrix}
c_0\\c_1
\end{pmatrix}
&=-\aB&&
\begin{pmatrix}
a_0\\a_1
\end{pmatrix}
&&-\bB&&
\begin{pmatrix}
b_0\\b_1
\end{pmatrix}.
\end{alignat*}
These relations imply that
\begin{equation*}
\gcd(a_i,b_i,c_i)=\gcd(a_i,b_i)=1;\qquad i=0,\ldots,2.
\end{equation*}
Another consequence is that
\begin{equation*}
\begin{vmatrix}
a_0&c_0\\a_2&c_2
\end{vmatrix}=
\begin{vmatrix}
a_0&-\aA a_0-\bA b_0\\a_2&-\aA a_2-\bA b_2
\end{vmatrix}
=-\bA
\begin{vmatrix}
a_0&b_0\\a_2&b_2
\end{vmatrix},
\end{equation*}
and analogously,
\begin{equation*}
\begin{vmatrix}
b_0&c_0\\b_2&c_2
\end{vmatrix}
=\aA
\begin{vmatrix}
a_0&b_0\\a_2&b_2
\end{vmatrix},\quad
\begin{vmatrix}
a_0&c_0\\a_1&c_1
\end{vmatrix}
=-\bB
\begin{vmatrix}
a_0&b_0\\a_1&b_1
\end{vmatrix},\quad
\begin{vmatrix}
b_0&c_0\\b_1&c_1
\end{vmatrix}
=\aB
\begin{vmatrix}
a_0&b_0\\a_1&b_1
\end{vmatrix}.
\end{equation*}

It will turn out to be convenient (and sometimes necessary) to know the signs of certain determinants. Considering the orientations of the corresponding coordinate systems, one can show that
\begin{equation*}
\begin{vmatrix}
a_0&b_0\\a_2&b_2
\end{vmatrix}>0,\quad
\begin{vmatrix}
a_0&b_0\\a_1&b_1
\end{vmatrix}<0,\quad\Psi=
\begin{vmatrix}
a_1&b_1\\a_2&b_2
\end{vmatrix}>0,\quad\text{ and }\quad
\begin{vmatrix}
a_0&b_0&c_0\\a_1&b_1&c_1\\a_2&b_2&c_2
\end{vmatrix}>0.
\end{equation*}

\subsection{Igusa's local zeta function}
As $f$ is non-degenerated over \Fp\ with respect to the compact faces of its Newton polyhedron \Gf, by Theorem~\ref{formdenhoor} the local Igusa zeta function \Zof\ of $f$ is the meromorphic complex function
\begin{equation}\label{deflvilzfvmdcas1}
\Zof=\sum_{\substack{\tau\mathrm{\ compact}\\\mathrm{face\ of\ }\Gf}}L_{\tau}S(\Dtu),
\end{equation}
with
\begin{gather*}
L_{\tau}:s\mapsto L_{\tau}(s)=\left(\frac{p-1}{p}\right)^3-\frac{N_{\tau}}{p^2}\frac{p^s-1}{p^{s+1}-1},\\
N_{\tau}=\#\left\{(x,y,z)\in(\Fpcross)^3\;\middle\vert\;\fbart(x,y,z)=0\right\},
\end{gather*}
and
\begin{align}
S(\Dtu):s\mapsto S(\Dtu)(s)&=\sum_{k\in\Z^3\cap\Delta_{\tau}}p^{-\sigma(k)-m(k)s}\notag\\
&=\sum_{i\in I}\frac{\Sigma(\delta_i)(s)}{\prod_{j\in J_i}(p^{\sigma(w_j)+m(w_j)s}-1)}.\label{deflvilzfvmdbiscas1}
\end{align}
Here $\{\delta_i\}_{i\in I}$ denotes a simplicial decomposition without introducing new rays of the cone $\Delta_{\tau}$ associated to $\tau$. The simplicial cone $\delta_i$ is supposed to be strictly positively spanned by the linearly independent primitive vectors $w_j$, $j\in J_i$, in $\Zplusn\setminus\{0\}$, and $\Sigma(\delta_i)$ is the function
\begin{equation*}
\Sigma(\delta_i):s\mapsto \Sigma(\delta_i)(s)=\sum_hp^{\sigma(h)+m(h)s},
\end{equation*}
where $h$ runs through the elements of the set
\begin{equation*}
H(w_j)_{j\in J_i}=\Z^3\cap\lozenge(w_j)_{j\in J_i},
\end{equation*}
with
\begin{equation*}
\lozenge(w_j)_{j\in J_i}=\left\{\sum\nolimits_{j\in J_i}h_jw_j\;\middle\vert\;h_j\in[0,1)\text{ for all }j\in J_i\right\}
\end{equation*}
the fundamental parallelepiped spanned by the vectors $w_j$, $j\in J_i$.

\subsection{The candidate pole $s_0$ and its residue}
We want to prove that $s_0$ is not a pole of \Zof. Since $s_0$ is a candidate pole of expected order one (and therefore is either no pole or a pole of order one), it is enough to prove that the coefficient $a_{-1}$ in the Laurent series
\begin{equation*}
\Zof(s)=\sum_{k=-1}^{\infty}a_k(s-s_0)^k
\end{equation*}
of \Zof\ centered at $s_0$, equals zero. This coefficient, also called the residue of \Zof\ in $s_0$, is given by
\begin{equation*}
a_{-1}=\Res(\Zof,s_0)=\lim_{s\to s_0}(s-s_0)\Zof(s).
\end{equation*}

Equivalently, we will prove in the rest of this section that $R_1=0$, with
\begin{equation*}
R_1=\lim_{s\to s_0}\left(p^{\sigma_0+m_0s}-1\right)\Zof(s)=(\log p)m_0a_{-1}.
\end{equation*}

\subsection{Terms contributing to $R_1$}
We will next calculate $R_1$ based on Formula~\eqref{deflvilzfvmdcas1} for \Zof.

The only (compact) faces of \Gf\ that contribute to the candidate pole $s_0$ are the subfaces $A,B,C,[AB],[AC],[BC],\tau_0$ of the single facet having $s_0$ as an associated candidate pole. It are only the terms of \eqref{deflvilzfvmdcas1} corresponding to these faces that should be taken into account in the calculation of $R_1$:
\begin{equation*}
R_1=\lim_{s\to s_0}\left(p^{\sigma_0+m_0s}-1\right)\sum_{\substack{\tau=\tau_0,A,B,C,\\ [AB],[AC],[BC]}}L_{\tau}(s)S(\Dtu)(s).
\end{equation*}

A second simplification is the following. First, note that vertex $A$ is contained in facets $\tau_0,\tau_2,\tau_3$, but can still be contained in other facets. Hence $\Delta_A$ is---in general---not simplicial and the same thing holds for the other vertices $B,C$ and their associated cones. Consequently, to handle $S_A,S_B,$ and $S_C$, we need to consider simplicial decompositions of $\Delta_A,\Delta_B,$ and $\Delta_C$, and we will choose ones that contain the simplicial cones $\delta_A,\delta_B,$ and $\delta_C$, respectively. Terms of \eqref{deflvilzfvmdbiscas1} associated to cones, other than $\delta_A,\delta_B,\delta_C,$ in these decompositions, do not have a pole in $s_0$ and hence do not contribute to $R_1$.

Let us write down the seven contributions to the \lq residue\rq\ $R_1$ explicitly. We obtain
\begin{multline*}
R_1=L_A(s_0)\frac{\Sigma(\delta_A)(s_0)}{\Ftwee(p-1)}+L_B(s_0)\frac{\Sigma(\delta_B)(s_0)}{\Feen(p-1)}\\
+L_C(s_0)\frac{\Sigma(\delta_C)(s_0)}{\Feen\Ftwee}+L_{[AB]}(s_0)\frac{\Sigma(\Delta_{[AB]})(s_0)}{p-1}\\
+L_{[AC]}(s_0)\frac{\Sigma(\Delta_{[AC]})(s_0)}{p^{\sigma_2+m_2s_0}-1}+L_{[BC]}(s_0)\frac{\Sigma(\Delta_{[BC]})(s_0)}{p^{\sigma_1+m_1s_0}-1}+L_{\tau_0}(s_0)\Sigma(\Delta_{\tau_0})(s_0).
\end{multline*}

\subsection{The numbers $N_{\tau}$}
Let us fix notations for the coefficients of $f$. We put
\begin{equation*}
f(x,y,z)=\sum_{\omega=(\omega_1,\omega_2,\omega_3)\in\N^3}a_{\omega}x^{\omega_1}y^{\omega_2}z^{\omega_3}\in\Zp[x,y,z].
\end{equation*}
For $a\in\Zp$, we denote by $\overline{a}=a+p\Zp\in\Fp$ its reduction modulo $p\Zp$. Recall that for every face $\tau$ of \Gf, we have
\begin{equation*}
\ft(x,y,z)=\sum_{\omega\in\Z^3\cap\tau}a_{\omega}x^{\omega_1}y^{\omega_2}z^{\omega_3}\quad\text{ and }\quad\fbart(x,y,z)=\sum_{\omega\in\Z^3\cap\tau}\overline{a_{\omega}}x^{\omega_1}y^{\omega_2}z^{\omega_3}.
\end{equation*}
Because the polynomial $f$ is non-degenerated over \Fp\ with respect to all the compact faces of its Newton polyhedron (and thus especially with respect to the vertices $A,B,C$), we have that none of the numbers $\overline{a_A},\overline{a_B},\overline{a_C}$ equals zero.

Hence the numbers $N_{\tau}$ in the formula for \Zof\ are as follows. For the vertices of $\tau_0$ we find
\begin{equation*}
N_A=\#\left\{(x,y,z)\in(\Fpcross)^3\;\middle\vert\;\overline{a_A}x^{x_A}y^{y_A}=0\right\}=0,
\end{equation*}
and analogously, $N_B=N_C=0$. About the number $N_{[AB]}$ we don't know so much, except that
\begin{align*}
N_{[AB]}&=\#\bigl\{(x,y,z)\in(\Fpcross)^3\;\big\vert\;\overline{f_{[AB]}}(x,y)=\overline{a_A}x^{x_A}y^{y_A}+\cdots+\overline{a_B}x^{x_B}y^{y_B}=0\bigr\}\\
&=(p-1)N,
\end{align*}
with
\begin{equation*}
N=\#\left\{(x,y)\in(\Fpcross)^2\;\middle\vert\;\overline{f_{[AB]}}(x,y)=0\right\}.
\end{equation*}
For the other edges we find
\begin{equation*}
N_{[AC]}=\#\bigl\{(x,y,z)\in(\Fpcross)^3\;\big\vert\;\overline{a_A}x^{x_A}y^{y_A}+\overline{a_C}x^{x_C}y^{y_C}z=0\bigr\}=(p-1)^2,
\end{equation*}
and analogously, $N_{[BC]}=(p-1)^2$. Finally, for $\tau_0$ we obtain
\begin{equation*}
N_{\tau_0}=\#\bigl\{(x,y,z)\in(\Fpcross)^3\;\big\vert\;\overline{f_{[AB]}}(x,y)+\overline{a_C}x^{x_C}y^{y_C}z=0\bigr\}=(p-1)^2-N.
\end{equation*}

\subsection{The factors $L_{\tau}(s_0)$}
The above formulas for the $N_{\tau}$ give rise to the following expressions for the $L_{\tau}(s_0)$:
\begin{gather*}
L_A(s_0)=L_B(s_0)=L_C(s_0)=\left(\frac{p-1}{p}\right)^3,\\
L_{[AB]}(s_0)=\left(\frac{p-1}{p}\right)^3-\frac{(p-1)N}{p^2}\frac{p^{s_0}-1}{p^{s_0+1}-1},\\
L_{[AC]}(s_0)=L_{[BC]}(s_0)=\left(\frac{p-1}{p}\right)^3-\left(\frac{p-1}{p}\right)^2\frac{p^{s_0}-1}{p^{s_0+1}-1},\\
\text{and}\qquad L_{\tau_0}(s_0)=\left(\frac{p-1}{p}\right)^3-\frac{(p-1)^2-N}{p^2}\frac{p^{s_0}-1}{p^{s_0+1}-1}.
\end{gather*}

\subsection{Multiplicities of the relevant simplicial cones}
We use Proposition~\ref{multipliciteit} to calculate the multiplicities of the relevant simplicial cones (and their corresponding fundamental parallelepipeds), thereby exploiting the relations we obtained in Subsection~\ref{srelbettvarcaseen}. That way we find\footnote{Cfr.\ Notation~\ref{notatiematrices}.}
\begin{alignat*}{4}
\mu_A&=\mult\delta_A&&=\#H(v_0,v_2,v_3)&&=
\begin{Vmatrix}
a_0&b_0&c_0\\a_2&b_2&c_2\\0&0&1
\end{Vmatrix}=&&
\begin{vmatrix}
a_0&b_0\\a_2&b_2
\end{vmatrix}>0,\\
\mu_B&=\mult\delta_B&&=\#H(v_0,v_1,v_3)&&=
\begin{Vmatrix}
a_0&b_0&c_0\\a_1&b_1&c_1\\0&0&1
\end{Vmatrix}=-&&
\begin{vmatrix}
a_0&b_0\\a_1&b_1
\end{vmatrix}>0,\\
\mu_C&=\mult\delta_C&&=\#H(v_0,v_1,v_2)&&=
\begin{Vmatrix}
a_0&b_0&c_0\\a_1&b_1&c_1\\a_2&b_2&c_2
\end{Vmatrix}=&&
\begin{vmatrix}
a_0&b_0&c_0\\a_1&b_1&c_1\\a_2&b_2&c_2
\end{vmatrix}>0
\end{alignat*}
for the maximal dimensional simplicial cones, while for the two-dimensional cones we obtain
\begin{alignat*}{2}
\mult\Delta_{[AB]}&=\#H(v_0,v_3)&&=\gcd\left(
\begin{Vmatrix}
a_0&b_0\\0&0
\end{Vmatrix},
\begin{Vmatrix}
a_0&c_0\\0&1
\end{Vmatrix},
\begin{Vmatrix}
b_0&c_0\\0&1
\end{Vmatrix}
\right)\\
&&&=\gcd(0,a_0,b_0)=1,\\
\mult\Delta_{[AC]}&=\#H(v_0,v_2)&&=\gcd\left(
\begin{Vmatrix}
a_0&b_0\\a_2&b_2
\end{Vmatrix},
\begin{Vmatrix}
a_0&c_0\\a_2&c_2
\end{Vmatrix},
\begin{Vmatrix}
b_0&c_0\\b_2&b_2
\end{Vmatrix}
\right)\\
&&&=\gcd\left(
\begin{vmatrix}
a_0&b_0\\a_2&b_2
\end{vmatrix},\abs{\bA}
\begin{vmatrix}
a_0&b_0\\a_2&b_2
\end{vmatrix},\abs{\aA}
\begin{vmatrix}
a_0&b_0\\a_2&b_2
\end{vmatrix}
\right)\\&&&=
\begin{vmatrix}
a_0&b_0\\a_2&b_2
\end{vmatrix}=\mu_A,\qquad\text{and}\\
\mult\Delta_{[BC]}&=\#H(v_0,v_1)&&=\gcd\left(
\begin{Vmatrix}
a_0&b_0\\a_1&b_1
\end{Vmatrix},
\begin{Vmatrix}
a_0&c_0\\a_1&c_1
\end{Vmatrix},
\begin{Vmatrix}
b_0&c_0\\b_1&b_1
\end{Vmatrix}
\right)\\
&&&=\gcd\left(-
\begin{vmatrix}
a_0&b_0\\a_1&b_1
\end{vmatrix},-\abs{\bB}
\begin{vmatrix}
a_0&b_0\\a_1&b_1
\end{vmatrix},-\abs{\aB}
\begin{vmatrix}
a_0&b_0\\a_1&b_1
\end{vmatrix}
\right)\\&&&=-
\begin{vmatrix}
a_0&b_0\\a_1&b_1
\end{vmatrix}=\mu_B.
\end{alignat*}
For the one-dimensional cone $\delta_{\tau_0}$, finally, we have of course that
\begin{equation*}
\mult\Delta_{\tau_0}=\#H(v_0)=\gcd(a_0,b_0,c_0)=1.
\end{equation*}

\subsection{The sums $\Sigma(\cdot)(s_0)$}
We found above that the multiplicities of $\Delta_{[AB]}$ and $\Delta_{\tau_0}$ both equal one; i.e., their corresponding fundamental parallelepipeds contain only one integral point which must be the origin: $H(v_0,v_3)=H(v_0)=\{(0,0,0)\}$. Hence
\begin{equation*}
\Sigma(\Delta_{[AB]})(s_0)=\Sigma(\Delta_{\tau_0})(s_0)=\sum_{h\in\{(0,0,0)\}}p^{\sigma(h)+m(h)s_0}=1.
\end{equation*}

Furthermore we saw that the multiplicities of $\Delta_A$ and $\Delta_{[AC]}$ are equal:
\begin{equation*}
\mu_A=\#H(v_0,v_2,v_3)=\#H(v_0,v_2).
\end{equation*}
The inclusion $H(v_0,v_2,v_3)\supseteq H(v_0,v_2)$ thus implies equality:
\begin{equation*}
H_A=H(v_0,v_2,v_3)=H(v_0,v_2),
\end{equation*}
and therefore,
\begin{equation*}
\Sigma_A=\Sigma(\delta_A)(s_0)=\Sigma(\Delta_{[AC]})(s_0)=\sum_{h\in H_A}p^{\sigma(h)+m(h)s_0}.
\end{equation*}
Analogously we have
\begin{gather*}
H_B=H(v_0,v_1,v_3)=H(v_0,v_1)\qquad\text{and}\\
\Sigma_B=\Sigma(\delta_B)(s_0)=\Sigma(\Delta_{[BC]})(s_0)=\sum_{h\in H_B}p^{\sigma(h)+m(h)s_0}.
\end{gather*}
Consistently, we shall also denote
\begin{gather*}
H_C=H(v_0,v_1,v_2)\qquad\text{and}\\
\Sigma_C=\Sigma(\delta_C)(s_0)=\sum_{h\in H_C}p^{\sigma(h)+m(h)s_0}.
\end{gather*}

Note that, since $\overline{\Delta_{[AC]}},\overline{\Delta_{[BC]}},\overline{\delta_C}\subseteq\overline{\Delta_C}$, we have that\footnote{In this text, by the dot product $w_1\cdot w_2$ of two complex vectors $w_1(a_1,b_1,c_1),\allowbreak w_2(a_2,b_2,c_2)\in\C^3$, we mean $w_1\cdot w_2=a_1a_2+b_1b_2+c_1c_2$.}
\begin{equation*}
m(h)=C\cdot h\qquad\text{for all}\qquad h\in H_A\cup H_B\cup H_C\subseteq\overline{\Delta_C}.
\end{equation*}
Hence, if we denote by $w$ the vector
\begin{equation*}
w=(1,1,1)+s_0(x_C,y_C,1)\in\C^3,
\end{equation*}
it holds that
\begin{equation*}
\Sigma_V=\sum_{h\in H_V}p^{w\cdot h};\qquad V=A,B,C.
\end{equation*}

\subsection{A new formula for $R_1$}
If we denote
\begin{equation*}
F_1=p^{w\cdot v_1}-1=p^{\sigma_1+m_1s_0}-1\qquad\text{and}\qquad F_2=p^{w\cdot v_2}-1=p^{\sigma_2+m_2s_0}-1,
\end{equation*}
the results above on the numbers $N_{\tau}$ and the multiplicities of the cones lead to
\begin{equation*}
\begin{multlined}[.98\textwidth]
R_1=\left(\frac{p-1}{p}\right)^3\left[\frac{\Sigma_A}{F_2(p-1)}+\frac{\Sigma_B}{F_1(p-1)}+\frac{\Sigma_C}{F_1F_2}+\frac{1}{p-1}+\frac{\Sigma_A}{F_2}+\frac{\Sigma_B}{F_1}+1\right]\\
-\left(\frac{p-1}{p}\right)^2\frac{p^{s_0}-1}{p^{s_0+1}-1}\left[\frac{N}{(p-1)^2}+\frac{\Sigma_A}{F_2}+\frac{\Sigma_B}{F_1}+\frac{(p-1)^2-N}{(p-1)^2}\right].
\end{multlined}
\end{equation*}
If we put $R_1'=(p/(p-1))^3R_1$, this formula can be simplified to
\begin{equation}\label{formreenaccentfincaseen}
R_1'=\frac{1}{1-p^{-s_0-1}}\left(\frac{\Sigma_A}{F_2}+\frac{\Sigma_B}{F_1}+1\right)+\frac{\Sigma_C}{F_1F_2}.
\end{equation}
Note that the number $N$ disappears from the equation. In what follows we shall prove that $R_1'=0$.

\subsection{Formulas for $\Sigma_A$ and $\Sigma_B$}
As in Section~\ref{fundpar}, we will consider the set
\begin{equation*}
H_C=H(v_0,v_1,v_2)=\Z^3\cap\lozenge(v_0,v_1,v_2)
\end{equation*}
as an additive group, endowed with addition modulo the lattice
\begin{equation*}
\Lambda(v_0,v_1,v_2)=\Z v_0+\Z v_1+\Z v_2.
\end{equation*}
In this way, $H_A=\Z^3\cap\lozenge(v_0,v_2)$ and $H_B=\Z^3\cap\lozenge(v_0,v_1)$ become subgroups of $H_C$ that correspond to the subgroups $H_1$ and $H_2$ of $H$ in Section~\ref{fundpar}.

From the description of the elements of these groups there, we know that there exist numbers $\xi_A\in\verA$ and $\xi_B\in\verB$ with $\gcd(\xi_A,\mu_A)=\gcd(\xi_B,\mu_B)=1$, such that the $\mu_A$ points of $H_A$ are precisely
\begin{alignat}{2}
\left\{\frac{i\xi_A}{\mu_A}\right\}v_0&+\frac{i}{\mu_A}v_2;&\qquad&i=0,\ldots,\mu_A-1;\label{stereenmuA}\\\intertext{while the $\mu_B$ points of $H_B$ are given by}
\left\{\frac{j\xi_B}{\mu_B}\right\}v_0&+\frac{j}{\mu_B}v_1;&&j=0,\ldots,\mu_B-1.\label{stertweemuB}
\end{alignat}
Recall that $\xi_A$ and $\xi_B$ are, as elements of \verA\ and \verB, respectively, uniquely determined by
\begin{equation}\label{defxiaxibcaseeen}
\xi_Av_0+v_2\in\mu_A\Z^3\qquad\text{and}\qquad\xi_Bv_0+v_1\in\mu_B\Z^3.
\end{equation}

These descriptions allow us to find \lq closed\rq\ formulas for $\Sigma_A$ and $\Sigma_B$. We know that
\begin{equation*}
\Sigma_A=\Sigma(\delta_A)(s_0)=\Sigma(\Delta_{[AC]})(s_0)=\sum_{h\in H_A}p^{\sigma(h)+m(h)s_0}=\sum_{h\in H_A}p^{w\cdot h},
\end{equation*}
with $w=(1,1,1)+s_0(x_C,y_C,1)$. Note that since $s_0$ is a candidate pole associated to $\tau_0$, we have that $p^{w\cdot v_0}=p^{\sigma_0+m_0s_0}=1$. Hence $p^{a(w\cdot v_0)}=p^{\{a\}(w\cdot v_0)}$ for every real number $a$. So if we write $h$ as $h=h_0v_0+h_2v_2$, we obtain
\begin{equation}\label{formsigmaAfincaseen}
\begin{multlined}[.8\textwidth]
\Sigma_A=\sum_{h\in H_A}p^{h_0(w\cdot v_0)+h_2(w\cdot v_2)}
=\sum_{i=0}^{\mu_A-1}\Bigl(p^{\frac{\xi_A(w\cdot v_0)+w\cdot v_2}{\mu_A}}\Bigr)^i\\
=\frac{p^{w\cdot v_2}-1}{p^{\frac{\xi_A(w\cdot v_0)+w\cdot v_2}{\mu_A}}-1}
=\frac{F_2}{p^{\frac{\xi_A(w\cdot v_0)+w\cdot v_2}{\mu_A}}-1}.
\end{multlined}
\end{equation}

Completely analogously we find
\begin{equation}\label{formsigmaBfincaseen}
\Sigma_B=\frac{F_1}{p^{\frac{\xi_B(w\cdot v_0)+w\cdot v_1}{\mu_B}}-1}.
\end{equation}

\subsection{A formula for $\mu_C=\mult\delta_C$}\label{formulaformuCcaseeen}
We know from Section~\ref{fundpar} that $\mu_A\mu_B\mid\mu_C$. We will give a useful interpretation of the quotient $\mu_C/\mu_A\mu_B$. We have the following:
\begin{align*}
\mu_C&=
\begin{vmatrix}
a_0&b_0&c_0\\a_1&b_1&c_1\\a_2&b_2&c_2
\end{vmatrix}\\
&=-a_1
\begin{vmatrix}
b_0&c_0\\b_2&c_2
\end{vmatrix}
+b_1
\begin{vmatrix}
a_0&c_0\\a_2&c_2
\end{vmatrix}
-c_1
\begin{vmatrix}
a_0&b_0\\a_2&b_2
\end{vmatrix}.\\\intertext{Using the relations from Subsection~\ref{srelbettvarcaseen}, we continue:}
\mu_C&=-a_1\aA
\begin{vmatrix}
a_0&b_0\\a_2&b_2
\end{vmatrix}
-b_1\bA
\begin{vmatrix}
a_0&b_0\\a_2&b_2
\end{vmatrix}
-c_1
\begin{vmatrix}
a_0&b_0\\a_2&b_2
\end{vmatrix}\\
&=-\mu_A(a_1\aA+b_1\bA+c_1)\\
&=-\mu_A\left(v_1\cdot\overrightarrow{AC}\right),\\\intertext{and since $v_1\perp\overrightarrow{BC}$, we obtain}
\mu_C&=-\mu_A\left(v_1\cdot\overrightarrow{AC}-v_1\cdot\overrightarrow{BC}\right)\\
&=-\mu_A\left(v_1\cdot\overrightarrow{AB}\right).
\end{align*}

Because the vector $\overrightarrow{AB}$ lies in the $xy$-plane and is perpendicular to $v_0$ and its coordinates have greatest common divisor $\fAB$ and we assume that $x_A>x_B$, it must hold that
\begin{equation*}
\overrightarrow{AB}=\fAB(-b_0,a_0,0).
\end{equation*}
Hence
\begin{equation*}
\mu_C=-\mu_A\fAB(a_0b_1-a_1b_0)=\mu_A\mu_B\fAB.
\end{equation*}

Next, we will use this formula in describing the points of $H_C$.

\subsection{Description of the points of $H_C$}\label{descpointsHCgevaleen}
It follows from \eqref{stereenmuA} and \eqref{stertweemuB} that the $\mu_A\mu_B$ points of the subgroup $H_A+H_B\cong H_A\oplus H_B$ of $H_C$ are
\begin{equation*}
\left\{\frac{i\xi_A\mu_B+j\xi_B\mu_A}{\mu_A\mu_B}\right\}v_0+\frac{j}{\mu_B}v_1+\frac{i}{\mu_A}v_2;\quad\ \ \ i=0,\ldots,\mu_A-1;\quad j=0,\ldots,\mu_B-1.
\end{equation*}

We know that the $v_2$-coordinates $h_2$ of the points $h\in H_C$ belong to the set
\begin{equation*}
\left\{0,\frac{1}{\mu_A\fAB},\frac{2}{\mu_A\fAB},\ldots,\frac{\mu_A\fAB-1}{\mu_A\fAB}\right\},
\end{equation*}
and that every $l/\mu_A\fAB$ in this set occurs $\mu_B$ times as the $v_2$-coordinate of a point in $H_C$, while every $h\in H_A+H_B$ has a $v_2$-coordinate of the form $i/\mu_A$ with $i\in\verA$, and every such $i/\mu_A$ is the $v_2$-coordinate of exactly $\mu_B$ points in $H_A+H_B$. (Analogously for the $v_1$-coordinates.)

In order to describe all the points of $H_C$ in a way as we did in Section~\ref{fundpar} for the points of $H$, we need to find a set of representatives for the elements of $H_C/(H_A+H_B)$. The $\fAB$ cosets of $H_A+H_B$ are characterised by constant $\{h_1\}_{1/\mu_B}$ and constant $\{h_2\}_{1/\mu_A}$, which can each take indeed \fAB\ possible values.

From the discussion in Section~\ref{fundpar}, we know there exists a unique point $h^{\ast}\in H_C$ with $v_2$-coordinate $h_2^{\ast}=1/\mu_A\fAB$ and $v_1$-coordinate $h_1^{\ast}=\eta/\mu_B\fAB<1/\mu_B$, and that the $\fAB$ multiples $\{kh^{\ast}\}$; $k=0,\ldots,\fAB-1$; of $h^{\ast}$ in $H_C$ make good representatives for the cosets of $H_A+H_B$. We will now try to find $h^{\ast}$.

If we denote by $M$ the matrix
\begin{equation*}
M=
\begin{pmatrix}
a_0&b_0&c_0\\a_1&b_1&c_1\\a_2&b_2&c_2
\end{pmatrix}
\end{equation*}
with $\det M=\mu_C$, it follows from $(\adj M)M=(\det M)I=\mu_CI$ that
\begin{equation*}
\begin{vmatrix}
a_1&b_1\\a_2&b_2
\end{vmatrix}
v_0-
\begin{vmatrix}
a_0&b_0\\a_2&b_2
\end{vmatrix}
v_1+
\begin{vmatrix}
a_0&b_0\\a_1&b_1
\end{vmatrix}
v_2=\Psi v_0-\mu_A v_1-\mu_B v_2=(0,0,\mu_C),
\end{equation*}
and hence
\begin{equation*}
h^{\ast}=\left\{\frac{-\Psi}{\mu_C}\right\}v_0+\frac{1}{\mu_B\fAB}v_1+\frac{1}{\mu_A\fAB}v_2\in H_C
\end{equation*}
is the point we are looking for.

So, considering all possible sums (in the group $H_C$) of one of the $\mu_A\mu_B$ points
\begin{equation*}
\left\{\frac{i\xi_A\mu_B+j\xi_B\mu_A}{\mu_A\mu_B}\right\}v_0+\frac{j}{\mu_B}v_1+\frac{i}{\mu_A}v_2;\quad\ \ \ i=0,\ldots,\mu_A-1;\quad j=0,\ldots,\mu_B-1;
\end{equation*}
of $H_A+H_B$ and one of the \fAB\ chosen representatives
\begin{equation*}
\{kh^{\ast}\}=\left\{\frac{-k\Psi}{\mu_C}\right\}v_0+\frac{k}{\mu_B\fAB}v_1+\frac{k}{\mu_A\fAB}v_2;\qquad k=0,\ldots,\fAB-1;
\end{equation*}
for the cosets of $H_A+H_B$ in $H_C$, we find the $\mu_C=\mu_A\mu_B\fAB$ points of $H_C$ as
\begin{multline*}
\left\{\frac{i\xi_A\mu_B\fAB+j\xi_B\mu_A\fAB-k\Psi}{\mu_C}\right\}v_0+\frac{j\fAB+k}{\mu_B\fAB}v_1+\frac{i\fAB+k}{\mu_A\fAB}v_2;\\
i=0,\ldots,\mu_A-1;\quad j=0,\ldots,\mu_B-1;\quad k=0,\ldots,\fAB-1.
\end{multline*}

Using the above description of the points of $H_C$, we will next derive a formula for $\Sigma_C$.

\subsection{A formula for $\Sigma_C$}
Recall that
\begin{equation*}
\Sigma_C=\Sigma(\delta_C)(s_0)=\sum_{h\in H_C}p^{\sigma(h)+m(h)s_0}=\sum_{h\in H_C}p^{w\cdot h},
\end{equation*}
with $w=(1,1,1)+s_0(x_C,y_C,1)$. If we write $h=h_0v_0+h_1v_1+h_2v_2$ and remember that $p^{w\cdot v_0}=1$ and $\mu_C=\mu_A\mu_B\fAB$, we find
\begin{align*}
\Sigma_C&=\sum_{h\in H_C}p^{h_0(w\cdot v_0)+h_1(w\cdot v_1)+h_2(w\cdot v_2)}\\
&=\sum_{i=0}^{\mu_A-1}\sum_{j=0}^{\mu_B-1}\sum_{k=0}^{\fAB-1}p^{\frac{i\xi_A\mu_B\fAB+j\xi_B\mu_A\fAB-k\Psi}{\mu_C}(w\cdot v_0)+\frac{j\fAB+k}{\mu_B\fAB}(w\cdot v_1)+\frac{i\fAB+k}{\mu_A\fAB}(w\cdot v_2)}\\
&=\sum_i\Bigl(p^{\frac{\xi_A(w\cdot v_0)+w\cdot v_2}{\mu_A}}\Bigr)^i\sum_j\Bigl(p^{\frac{\xi_B(w\cdot v_0)+w\cdot v_1}{\mu_B}}\Bigr)^j\sum_k\Bigl(p^{\frac{-\Psi(w\cdot v_0)+\mu_A(w\cdot v_1)+\mu_B(w\cdot v_2)}{\mu_C}}\Bigr)^k\\
&=\frac{F_2}{p^{\frac{\xi_A(w\cdot v_0)+w\cdot v_2}{\mu_A}}-1}\;\frac{F_1}{p^{\frac{\xi_B(w\cdot v_0)+w\cdot v_1}{\mu_B}}-1}\;\frac{p^{\frac{-\Psi(w\cdot v_0)+\mu_A(w\cdot v_1)+\mu_B(w\cdot v_2)}{\mu_A\mu_B}}-1}{p^{\frac{-\Psi(w\cdot v_0)+\mu_A(w\cdot v_1)+\mu_B(w\cdot v_2)}{\mu_C}}-1}.
\end{align*}

We already observed in Subsection~\ref{descpointsHCgevaleen} that if we put
\begin{equation*}
M=
\begin{pmatrix}
a_0&b_0&c_0\\a_1&b_1&c_1\\a_2&b_2&c_2
\end{pmatrix},
\end{equation*}
the identity $(\adj M)M=(\det M)I=\mu_CI$ implies that
\begin{equation}\label{identiteitreccaseeen}
\begin{vmatrix}
a_1&b_1\\a_2&b_2
\end{vmatrix}
v_0-
\begin{vmatrix}
a_0&b_0\\a_2&b_2
\end{vmatrix}
v_1+
\begin{vmatrix}
a_0&b_0\\a_1&b_1
\end{vmatrix}
v_2=\Psi v_0-\mu_A v_1-\mu_B v_2=(0,0,\mu_C).
\end{equation}
Making the dot product with $w=(1,1,1)+s_0(x_C,y_C,1)$ on all sides of the equation yields
\begin{equation*}
-\Psi(w\cdot v_0)+\mu_A(w\cdot v_1)+\mu_B(w\cdot v_2)=\mu_C(-s_0-1).
\end{equation*}

Hence we find
\begin{equation}\label{formsigmaCfincaseen}
\Sigma_C=\frac{F_1F_2}{p^{-s_0-1}-1}\frac{p^{\frac{-\Psi(w\cdot v_0)+\mu_A(w\cdot v_1)+\mu_B(w\cdot v_2)}{\mu_A\mu_B}}-1}{\Bigl(p^{\frac{\xi_A(w\cdot v_0)+w\cdot v_2}{\mu_A}}-1\Bigr)\Bigl(p^{\frac{\xi_B(w\cdot v_0)+w\cdot v_1}{\mu_B}}-1\Bigr)}.
\end{equation}

\subsection{Proof of $R_1'=0$}
Bringing together Equations (\ref{formreenaccentfincaseen}, \ref{formsigmaAfincaseen}, \ref{formsigmaBfincaseen}, \ref{formsigmaCfincaseen}) for $R_1',\allowbreak\Sigma_A,\allowbreak\Sigma_B,$ and $\Sigma_C$, we obtain that
\begin{align*}
R_1'&=\frac{1}{1-p^{-s_0-1}}\left(\frac{\Sigma_A}{F_2}+\frac{\Sigma_B}{F_1}+1\right)+\frac{\Sigma_C}{F_1F_2}\\
&=\frac{1}{1-p^{-s_0-1}}\left(\frac{1}{p^{\frac{\xi_A(w\cdot v_0)+w\cdot v_2}{\mu_A}}-1}+\frac{1}{p^{\frac{\xi_B(w\cdot v_0)+w\cdot v_1}{\mu_B}}-1}+1\right)\\*
&\qquad\qquad\qquad\qquad+\frac{1}{p^{-s_0-1}-1}\frac{p^{\frac{-\Psi(w\cdot v_0)+\mu_A(w\cdot v_1)+\mu_B(w\cdot v_2)}{\mu_A\mu_B}}-1}{\Bigl(p^{\frac{\xi_A(w\cdot v_0)+w\cdot v_2}{\mu_A}}-1\Bigr)\Bigl(p^{\frac{\xi_B(w\cdot v_0)+w\cdot v_1}{\mu_B}}-1\Bigr)}\\
&=\frac{1}{p^{-s_0-1}-1}\frac{p^{\frac{(\xi_A\mu_B+\xi_B\mu_A)(w\cdot v_0)+\mu_A(w\cdot v_1)+\mu_B(w\cdot v_2)}{\mu_A\mu_B}}-1}{\Bigl(p^{\frac{\xi_A(w\cdot v_0)+w\cdot v_2}{\mu_A}}-1\Bigr)\Bigl(p^{\frac{\xi_B(w\cdot v_0)+w\cdot v_1}{\mu_B}}-1\Bigr)}\\*
&\qquad\qquad\qquad\qquad-\frac{1}{1-p^{-s_0-1}}\frac{p^{\frac{-\Psi(w\cdot v_0)+\mu_A(w\cdot v_1)+\mu_B(w\cdot v_2)}{\mu_A\mu_B}}-1}{\Bigl(p^{\frac{\xi_A(w\cdot v_0)+w\cdot v_2}{\mu_A}}-1\Bigr)\Bigl(p^{\frac{\xi_B(w\cdot v_0)+w\cdot v_1}{\mu_B}}-1\Bigr)}.
\end{align*}
Hence it is sufficient to prove that
\begin{equation*}
p^{\frac{(\xi_A\mu_B+\xi_B\mu_A)(w\cdot v_0)}{\mu_A\mu_B}}=p^{-\frac{\Psi(w\cdot v_0)}{\mu_A\mu_B}},
\end{equation*}
or, as $p^{w\cdot v_0}=1$, equivalently, that
\begin{equation}\label{eqprovecaseeen}
\frac{\xi_A\mu_B+\xi_B\mu_A+\Psi}{\mu_A\mu_B}\in\Z.
\end{equation}

Well, it follows from \eqref{defxiaxibcaseeen} and \eqref{identiteitreccaseeen} that
\begin{multline*}
(\xi_A\mu_B+\xi_B\mu_A+\Psi)v_0\\=\mu_B(\xi_Av_0+v_2)+\mu_A(\xi_Bv_0+v_1)+(\Psi v_0-\mu_A v_1-\mu_B v_2)\in\mu_A\mu_B\Z^3.
\end{multline*}
The primitivity of $v_0$ now implies \eqref{eqprovecaseeen}, concluding Case~I.

\section{Case~II: exactly one facet contributes to $s_0$ and this facet is a non-compact $B_1$-facet}
\subsection{Figure and notations}
We shall assume that the one facet $\tau_0$ contributing to $s_0$ is non-compact for the variable $x$, and $B_1$ with respect to the variable $z$. We denote by $A(x_A,y_A,0)$ the vertex of $\tau_0$ in the $xy$-plane and by $B(x_B,y_B,1)$ the vertex in the plane $\{z=1\}$. The situation is sketched in Figure~\ref{figcase2}.

%
%
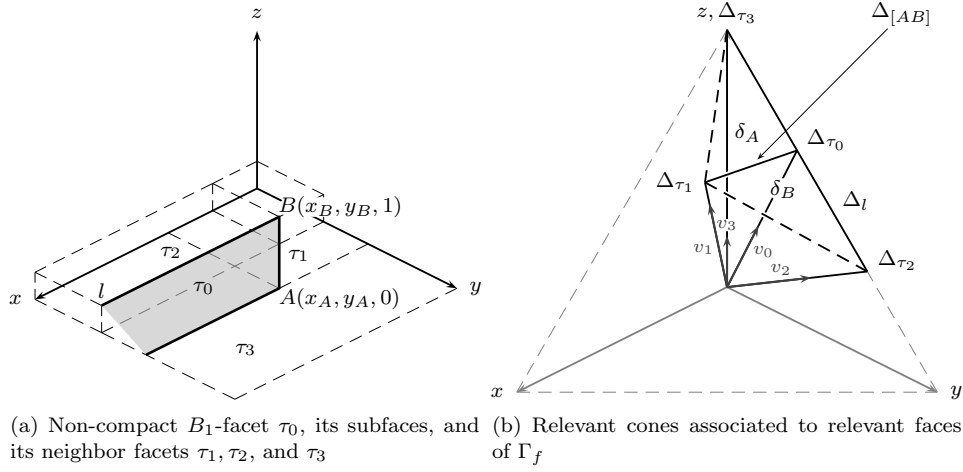
\begin{figure}
\psset{unit=.03298611111\textwidth}
\centering
\subfigure[Non-compact $B_1$-facet $\tau_0$, its subfaces, and its neighbor facets $\tau_1,\tau_2,$ and $\tau_3$]{
\begin{pspicture}(-7.58,-6.75)(6.82,5.8)
{\footnotesize
\pstThreeDCoor[xMin=0,yMin=0,zMin=0,xMax=10,yMax=9,zMax=5.8,linecolor=black,linewidth=.7pt]
{
\psset{linecolor=black,linewidth=.3pt,linestyle=dashed,subticks=1}
\pstThreeDLine(2,0,0)(2,3,0)\pstThreeDLine(2,3,0)(0,3,0)
\pstThreeDLine(10,0,0)(10,3,0)\pstThreeDLine(10,3,0)(2,3,0)
\pstThreeDLine(4,0,0)(4,5,0)\pstThreeDLine(4,5,0)(0,5,0)
\pstThreeDLine(10,3,0)(10,9,0)\pstThreeDLine(10,9,0)(0,9,0)
\pstThreeDLine(2,0,1)(2,3,1)\pstThreeDLine(2,3,1)(0,3,1)
\pstThreeDLine(10,0,1)(10,3,1)\pstThreeDLine(0,0,1)(2,0,1)
\pstThreeDLine(2,0,1)(10,0,1)\pstThreeDLine(0,0,1)(0,3,1)
\pstThreeDLine(2,0,0)(2,0,1)\pstThreeDLine(10,0,0)(10,0,1)
\pstThreeDLine(0,3,0)(0,3,1)\pstThreeDLine(10,3,0)(10,3,1)
}
\pstThreeDPut[pOrigin=c](6.3,3.9,0.5){\psframebox*[framesep=0.8pt,framearc=0.3]{\phantom{$\tau_0$}}}
{
\psset{dotstyle=none,dotscale=1,drawCoor=false}
\psset{linecolor=black,linewidth=1pt,linejoin=1}
\psset{fillcolor=lightgray,opacity=.6,fillstyle=solid}
\pstThreeDLine(10,5,0)(4,5,0)(2,3,1)(10,3,1)
}
\pstThreeDPut[pOrigin=tl](4,5,0){$A(x_A,y_A,0)$}
\pstThreeDPut[pOrigin=bl](2,3,1){\psframebox*[framesep=-.3pt,framearc=1]{$B(x_B,y_B,1)$}}
\pstThreeDPut[pOrigin=c](6.3,3.9,0.5){$\tau_0$}
\pstThreeDPut[pOrigin=c](2.62,4.51,0.45){$\tau_1$}
\pstThreeDPut[pOrigin=c](7.5,7,0){$\tau_3$}
\pstThreeDPut[pOrigin=rb](6.2,2.8,1.1){$\tau_2$}
\pstThreeDPut[pOrigin=b](10,3,1.27){$l$}
}
\end{pspicture}
}\hfill\subfigure[Relevant cones associated to relevant faces of~\Gf]{
\psset{unit=.03125\textwidth}
\begin{pspicture}(-7.6,-3.8)(7.6,9.6)
{\footnotesize
\pstThreeDCoor[xMin=0,yMin=0,zMin=0,xMax=10,yMax=10,zMax=10,nameZ={},linecolor=gray,linewidth=.7pt]
{
\psset{linecolor=gray,linewidth=.3pt,linejoin=1,linestyle=dashed,fillcolor=lightgray,fillstyle=none}
\pstThreeDLine(10,0,0)(0,10,0)\pstThreeDLine(0,10,0)(0,0,10)\pstThreeDLine(0,0,10)(10,0,0)
}
{
\psset{linecolor=black,linewidth=.7pt,linejoin=1,fillcolor=lightgray,fillstyle=none}
\pstThreeDLine(0,0,0)(0,0,10)
}
{
\psset{labelsep=2pt}
\uput[90](0,1.7){\psframebox*[framesep=0.3pt,framearc=1]{\darkgray\scriptsize$v_3$}}
}
{
\psset{linecolor=black,linewidth=.7pt,linejoin=1,fillcolor=lightgray,fillstyle=none}
\pstThreeDLine(0,0,0)(0,3.33,6.67)
\pstThreeDLine(0,0,0)(2.63,1.58,5.79)
\pstThreeDLine(0,0,0)(0,6.67,3.33)
}
{
\psset{linecolor=darkgray,linewidth=.8pt,linejoin=1,arrows=->,arrowscale=1,fillcolor=lightgray,fillstyle=none}
\pstThreeDLine(0,0,0)(0,1.5,3)
\pstThreeDLine(0,0,0)(2.10,1.27,4.62)
\pstThreeDLine(0,0,0)(0,4,2)
\pstThreeDLine(0,0,0)(0,0,2)
}
{
\psset{linecolor=white,linewidth=2pt,linejoin=1,fillcolor=lightgray,fillstyle=none}
\pstThreeDLine(2.24,1.84,5.92)(1.32,2.45,6.24)
\pstThreeDLine(2.37,2.09,5.54)(1.05,4.63,4.32)
}
{
\psset{linecolor=black,linewidth=.7pt,linejoin=1,fillcolor=lightgray,fillstyle=none}
\pstThreeDLine(0,0,10)(0,3.33,6.67)
\pstThreeDLine(0,6.67,3.33)(0,3.33,6.67)(2.63,1.58,5.79)
}
{
\psset{linecolor=black,linewidth=.7pt,linejoin=1,linestyle=dashed,fillcolor=lightgray,fillstyle=none}
\pstThreeDLine(0,0,10)(2.63,1.58,5.79)
\pstThreeDLine(2.63,1.58,5.79)(0,6.67,3.33)
}
{
\psset{labelsep=2pt}
\uput[-30](.7,1.4){\darkgray\scriptsize$v_0$}
\uput[105](1.96,.224){\darkgray\scriptsize$v_2$}
}
{
\psset{labelsep=1.5pt}
\uput[202](-.3,1.4){\darkgray\scriptsize$v_1$}
}
{
\psset{labelsep=3.8pt}
\uput[30](2.35,4.58){$\Delta_{\tau_0}$}
\uput[180](-.73,3.53){$\Delta_{\tau_1}$}
\uput[30](3.54,2.55){$\Delta_l$}
\uput[30](4.73,.52){$\Delta_{\tau_2}$}
}
\rput(.7,5.2){\psframebox*[framesep=0.3pt,framearc=1]{\footnotesize$\delta_A$}}
\rput(1.9,3.3){\psframebox*[framesep=0.7pt,framearc=1]{\footnotesize$\delta_B$}}
\pstThreeDNode(1.32,2.45,6.24){AB}
\rput[B](0,9.075){$z,\Delta_{\tau_3}$}
\rput[Br](6.91,9.075){\rnode{ABlabel}{$\Delta_{[AB]}$}}
\ncline[linewidth=.3pt,nodesepB=3.5pt,nodesepA=1pt]{->}{ABlabel}{AB}
}
\end{pspicture}
}
\caption{Case II: the only facet contributing to $s_0$ is the non-compact $B_1$-facet~$\tau_0$}
\label{figcase2}
\end{figure}
%
%

If we denote by $\overrightarrow{AB}(x_B-x_A,y_B-y_A,1)=(\alpha,\beta,1)$ the vector along the edge $[AB]$, then the unique primitive vector $v_0\in\Zplus^3$ perpendicular to $\tau_0$ equals $v_0(0,1,-\beta)$, and an equation for the affine hull of $\tau_0$ is given by
\begin{equation*}
\aff(\tau_0)\leftrightarrow y-\beta z=y_A.
\end{equation*}
Note that since $\tau_0$ is $B_1$, we must have $\beta<0$ and hence $y_B<y_A$. The numerical data associated to $\tau_0$ are therefore $(m(v_0),\sigma(v_0))=(y_A,1-\beta)$, and thus we assume
\begin{equation*}
s_0=\frac{\beta-1}{y_A}+\frac{2n\pi i}{y_A\log p}\qquad\text{for some $n\in\Z$.}
\end{equation*}

We denote by $\tau_1$ the facet of \Gf\ that has the edge $[AB]$ in common with $\tau_0$, by $\tau_2$ the non-compact facet of \Gf\ sharing with $\tau_0$ a half-line with endpoint $B$, and finally, by $\tau_3$ the facet lying in the $xy$-plane. Primitive vectors in $\Zplus^3$ perpendicular to $\tau_1,\tau_2,\tau_3$ will be denoted by
\begin{equation*}
v_1(a_1,b_1,c_1),\quad v_2(0,b_2,c_2),\quad v_3(0,0,1),
\end{equation*}
respectively, and equations for the affine supports of these facets are denoted
\begin{alignat*}{3}
\aff(\tau_1)&\leftrightarrow a_1&x+b_1y&+c_1&z&=m_1,\\
\aff(\tau_2)&\leftrightarrow    &  b_2y&+c_2&z&=m_2,\\
\aff(\tau_3)&\leftrightarrow    &      &    &z&=0,
\end{alignat*}
for certain $m_1,m_2\in\Zplus$. If we put $\sigma_1=a_1+b_1+c_1$ and $\sigma_2=b_2+c_2$, then the numerical data for $\tau_1,\tau_2,\tau_3$ are $(m_1,\sigma_1),(m_2,\sigma_2),$ and $(0,1)$, respectively.

\subsection{The candidate pole $s_0$ and the contributions to its residue}
The aim of this section is to prove that $s_0$ is not a pole of \Zof; i.e., we want to demonstrate that
\begin{equation*}
R_1=\lim_{s\to s_0}\left(p^{1-\beta+y_As}-1\right)\Zof(s)=0.
\end{equation*}
Since we work with the local version of Igusa's $p$-adic zeta function, we only consider the compact faces of \Gf\ in the formula for $\Zof(s)$ for non-degenerated $f$. Of course, in order to find an expression for $R_1$, we only need to account those compact faces that contribute to $s_0$, i.e., the compact subfaces $A,B,$ and $[AB]$ of $\tau_0$:
\begin{equation*}
R_1=\lim_{s\to s_0}\left(p^{1-\beta+y_As}-1\right)\sum_{\tau=A,B,[AB]}L_{\tau}(s)S(\Dtu)(s).
\end{equation*}

As in Case~I, we note that vertices $A$ and $B$ may be contained in facets other than $\tau_i$; $i=0,\ldots,3$; and subsequently their associated cones $\Delta_A$ and $\Delta_B$ may be not simplicial. Therefore, instead of $\Delta_A$ and $\Delta_B$, we shall consider the simplicial cones
\begin{equation*}
\dA=\cone(v_0,v_1,v_3)\qquad\text{and}\qquad\dB=\cone(v_0,v_1,v_2)
\end{equation*}
as members of simplicial decompositions of $\Delta_A$ and $\Delta_B$, respectively. It follows as before that of all cones in these decompositions, only $\dA$ and $\dB$ are relevant in the calculation of $R_1$:
\begin{multline*}
R_1=L_A(s_0)\frac{\Sigma(\delta_A)(s_0)}{\Feen(p-1)}\\
+L_B(s_0)\frac{\Sigma(\delta_B)(s_0)}{\Feen\Ftwee}+L_{[AB]}(s_0)\frac{\Sigma(\Delta_{[AB]})(s_0)}{p^{\sigma_1+m_1s_0}-1}.
\end{multline*}

\subsection{The factors $L_{\tau}(s_0)$, the sums $\Sigma(\cdot)(s_0)$ and a new formula for $R_1$}
As in Case~I we find easily that $N_A=N_B=0$ and $N_{[AB]}=(p-1)^2$. Hence the factors $L_{\tau}(s_0)$ are as follows:
\begin{gather*}
L_A(s_0)=L_B(s_0)=\left(\frac{p-1}{p}\right)^3\qquad\text{and}\\
L_{[AB]}(s_0)=\left(\frac{p-1}{p}\right)^3-\left(\frac{p-1}{p}\right)^2\frac{p^{s_0}-1}{p^{s_0+1}-1}.
\end{gather*}

Let us look at the multiplicities of $\dA,\dB,$ and $\Delta_{[AB]}$. For $\mult\delta_A$ we find
\begin{equation*}
\mu_A=\mult\delta_A=\#H(v_0,v_1,v_3)=
\begin{Vmatrix}
0&1&-\beta\\a_1&b_1&c_1\\0&0&1
\end{Vmatrix}=a_1>0.
\end{equation*}
Although this non-compact edge does not appear in the formula for $R_1$, we also mention the multiplicity $\mu_l$ of the cone $\Delta_l$ associated to the half-line $l=\tau_0\cap\tau_2$:
\begin{equation*}
\mu_l=\mult\Delta_l=\#H(v_0(0,1,-\beta),v_2(0,b_2,c_2))=
\begin{Vmatrix}
1&-\beta\\b_2&c_2
\end{Vmatrix}.
\end{equation*}
Since the coordinate system $(v_0,v_2)$ for the $yz$-plane has the opposite o\-ri\-en\-ta\-tion of the coordinate system $(e_y(0,1,0),e_z(0,0,1))$ we work in, we have that
\begin{equation*}
\mu_l=
\begin{Vmatrix}
1&-\beta\\b_2&c_2
\end{Vmatrix}=-
\begin{vmatrix}
1&-\beta\\b_2&c_2
\end{vmatrix}=-\beta b_2-c_2>0.
\end{equation*}
We see now that
\begin{multline*}
\mu_B=\mult\delta_B=\#H(v_0,v_1,v_2)\\=
\begin{Vmatrix}
0&1&-\beta\\a_1&b_1&c_1\\0&b_2&c_2
\end{Vmatrix}=a_1
\begin{Vmatrix}
1&-\beta\\b_2&c_2
\end{Vmatrix}=a_1(-\beta b_2-c_2)=\mu_A\mu_l.
\end{multline*}
Finally, for $\mult\Delta_{[AB]}$ we obtain
\begin{multline*}
\mult\Delta_{[AB]}=\#H(v_0,v_1)=\gcd\left(
\begin{Vmatrix}
0&1\\a_1&b_1
\end{Vmatrix},
\begin{Vmatrix}
0&-\beta\\a_1&c_1
\end{Vmatrix},
\begin{Vmatrix}
1&-\beta\\b_1&c_1
\end{Vmatrix}
\right)\\
=\gcd(a_1,-\beta a_1,\abs{\beta b_1+c_1})=\gcd(a_1,-\beta a_1,\abs{\alpha}a_1)=a_1=\mu_A.
\end{multline*}
In the third to last equality we used that $\beta b_1+c_1=-\alpha a_1$, which follows from the fact that
$\overrightarrow{AB}(\alpha,\beta,1)\perp v_1(a_1,b_1,c_1)$.

Since $H(v_0,v_1,v_3)\supseteq H(v_0,v_1)$ and $\mu_A=\#H(v_0,v_1,v_3)=\#H(v_0,v_1)$, we have that
\begin{equation*}
H_A=H(v_0,v_1,v_3)=H(v_0,v_1),
\end{equation*}
and therefore,
\begin{equation*}
\Sigma_A=\Sigma(\delta_A)(s_0)=\Sigma(\Delta_{[AB]})(s_0)=\sum_{h\in H_A}p^{\sigma(h)+m(h)s_0}=\sum_{h\in H_A}p^{w\cdot h},
\end{equation*}
with $w=(1,1,1)+s_0(x_B,y_B,1)\in\C^3$. Furthermore we denote
\begin{gather*}
H_l=H(v_0,v_2),\qquad H_B=H(v_0,v_1,v_2),\\
\Sigma_B=\Sigma(\delta_B)(s_0)=\sum_{h\in H_B}p^{\sigma(h)+m(h)s_0}=\sum_{h\in H_B}p^{w\cdot h},\\
F_1=p^{w\cdot v_1}-1=p^{\sigma_1+m_1s_0}-1,\qquad\text{and}\qquad F_2=p^{w\cdot v_2}-1=p^{\sigma_2+m_2s_0}-1.
\end{gather*}

The considerations above result in the following concrete formula for $R_1$:
\begin{equation*}
R_1=\left(\frac{p-1}{p}\right)^3\left[\frac{\Sigma_A}{F_1(p-1)}+\frac{\Sigma_B}{F_1F_2}+\frac{\Sigma_A}{F_1}\right]
-\left(\frac{p-1}{p}\right)^2\frac{p^{s_0}-1}{p^{s_0+1}-1}\frac{\Sigma_A}{F_1}.
\end{equation*}
With $R_1'=(p/(p-1))^3R_1$, this can be simplified to
\begin{equation}\label{formreenaccentcasetwee}
R_1'=\frac{1}{1-p^{-s_0-1}}\frac{\Sigma_A}{F_1}+\frac{\Sigma_B}{F_1F_2}.
\end{equation}

Next, we will prove that $R_1'=0$.

\subsection{Proof of $R_1'=0$}
First, note that
\begin{equation*}
-b_2v_0+v_2=-b_2(0,1,-\beta)+(0,b_2,c_2)=-(0,0,-\beta b_2-c_2)=-(0,0,\mu_l)
\end{equation*}
yields
\begin{equation}\label{vectidcasetwee}
\frac{-b_2}{\mu_l}v_0+\frac{1}{\mu_l}v_2=(0,0,-1)\in\Z^3\quad\text{ and }\quad
p^{\frac{-b_2(w\cdot v_0)+w\cdot v_2}{\mu_l}}=p^{-s_0-1},
\end{equation}
with $w=(1,1,1)+s_0(x_B,y_B,1)$.

Let us, as before, consider
\begin{equation*}
H_B=H(v_0,v_1,v_2)=\Z^3\cap\lozenge(v_0,v_1,v_2)
\end{equation*}
as a group, endowed with addition modulo $\Z v_0+\Z v_1+\Z v_2$. Then, by \eqref{vectidcasetwee} and Theorem~\ref{algfp}, there exists a $\xi_A\in\verA$ such that the elements of the subgroups $H_A=H(v_0,v_1)$ and $H_l=H(v_0,v_2)$ of $H_B$ are given by
\begin{align*}
\left\{\frac{i\xi_A}{\mu_A}\right\}v_0+\frac{i}{\mu_A}v_1;&\qquad i=0,\ldots,\mu_A-1;\\\shortintertext{and}
\left\{\frac{-jb_2}{\mu_l}\right\}v_0+\frac{j}{\mu_l}v_2;&\qquad j=0,\ldots,\mu_l-1;
\end{align*}
respectively.

Furthermore, we found above that in this special case
\begin{equation*}
\#H_B=\mu_B=\mu_A\mu_l=\#H_A\#H_l.
\end{equation*}
Hence $H_A\cap H_l=\{(0,0,0)\}$ implies that $H_B=H_A+H_l\cong H_A\oplus H_l$ and its elements are the following:
\begin{equation*}
\left\{\frac{i\xi_A\mu_l-jb_2\mu_A}{\mu_A\mu_l}\right\}v_0+\frac{i}{\mu_A}v_1+\frac{j}{\mu_l}v_2;\quad\ \ \ i=0,\ldots,\mu_A-1;\quad j=0,\ldots,\mu_l-1.
\end{equation*}

We can now easily calculate $\Sigma_A$ and $\Sigma_B$. If, for $h\in H_B$, we denote by $(h_0,h_1,h_2)$ the coordinates of $h$ with respect to the basis $(v_0,v_1,v_2)$ and keep in mind that $p^{w\cdot v_0}=1$, we obtain
\begin{equation}\label{formsigmaacasetwee}
\begin{aligned}
\Sigma_A&=\sum_{h\in H_A}p^{w\cdot h}=\sum_hp^{h_0(w\cdot v_0)+h_1(w\cdot v_1)}\\
&=\sum_{i=0}^{\mu_A-1}\Bigl(p^{\frac{\xi_A(w\cdot v_0)+w\cdot v_1}{\mu_A}}\Bigr)^i
=\frac{p^{w\cdot v_1}-1}{p^{\frac{\xi_A(w\cdot v_0)+w\cdot v_1}{\mu_A}}-1}
=\frac{F_1}{p^{\frac{\xi_A(w\cdot v_0)+w\cdot v_1}{\mu_A}}-1},
\end{aligned}
\end{equation}
while $\Sigma_B$ is given by
\begin{align}
\Sigma_B&=\sum_{h\in H_B}p^{w\cdot h}\notag\\
&=\sum_hp^{h_0(w\cdot v_0)+h_1(w\cdot v_1)+h_2(w\cdot v_2)}\notag\\
&=\sum_{i=0}^{\mu_A-1}\sum_{j=0}^{\mu_l-1}p^{\frac{i\xi_A\mu_l-jb_2\mu_A}{\mu_A\mu_l}(w\cdot v_0)+\frac{i}{\mu_A}(w\cdot v_1)+\frac{j}{\mu_l}(w\cdot v_2)}\notag\\
&=\sum_i\Bigl(p^{\frac{\xi_A(w\cdot v_0)+w\cdot v_1}{\mu_A}}\Bigr)^i\sum_j\Bigl(p^{\frac{-b_2(w\cdot v_0)+w\cdot v_2}{\mu_l}}\Bigr)^j\notag\\
&=\frac{F_1}{p^{\frac{\xi_A(w\cdot v_0)+w\cdot v_1}{\mu_A}}-1}\;\frac{F_2}{p^{\frac{-b_2(w\cdot v_0)+w\cdot v_2}{\mu_l}}-1}\notag\\
&=\frac{1}{p^{-s_0-1}-1}\;\frac{F_1F_2}{p^{\frac{\xi_A(w\cdot v_0)+w\cdot v_1}{\mu_A}}-1},\label{formsigmabcasetwee}
\end{align}
where we used \eqref{vectidcasetwee} in the last step.

By Equations~(\ref{formreenaccentcasetwee}, \ref{formsigmaacasetwee}, \ref{formsigmabcasetwee}) we have $R_1'=0$. This concludes Case~II.

%
%
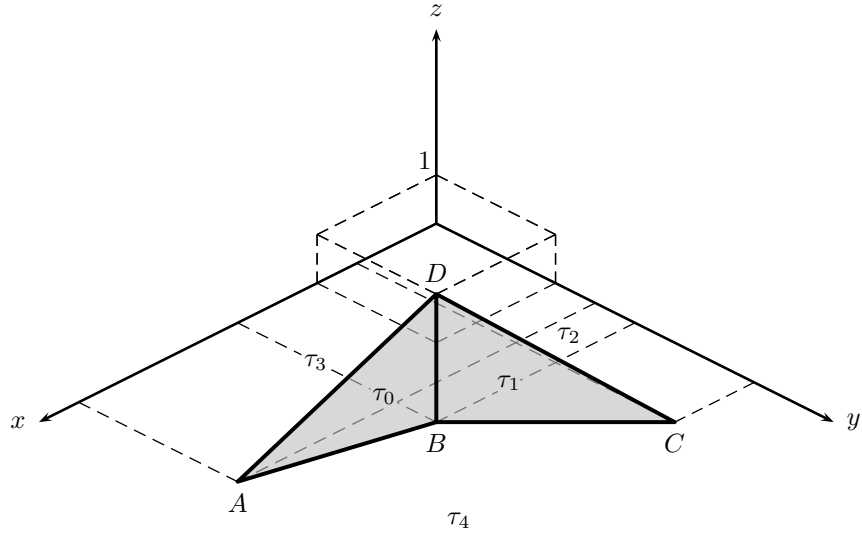
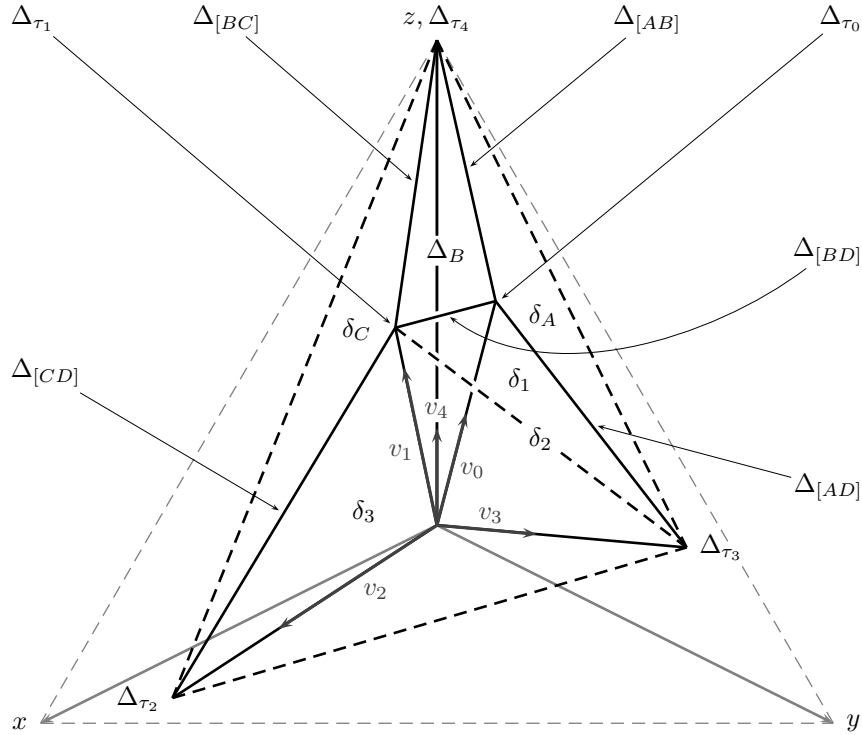
\begin{figure}
\centering
\psset{unit=.05905773059\textwidth}
\subfigure[$B_1$-simplices $\tau_0$ and $\tau_1$, their subfaces and neighbor facets $\tau_2,\tau_3,$ and $\tau_4$]{
\begin{pspicture}(-7.54,-5.7)(7.53,4.16)
\pstThreeDCoor[xMin=0,yMin=0,zMin=0,xMax=10,yMax=10,zMax=4,labelsep=5pt,linecolor=black,linewidth=1.0pt]
{
\psset{linecolor=black,linewidth=.5pt,linestyle=dashed,subticks=1}
\pstThreeDPlaneGrid[planeGrid=xy](0,0)(9,4)
\pstThreeDPlaneGrid[planeGrid=xy](0,0)(5,5)
\pstThreeDPlaneGrid[planeGrid=xy](0,0)(2,8)
\pstThreeDPlaneGrid[planeGrid=xz](0,0)(3,1)
\pstThreeDPlaneGrid[planeGrid=yz](0,0)(3,1)
\pstThreeDPlaneGrid[planeGrid=xy,planeGridOffset=1](0,0)(3,3)
\pstThreeDPlaneGrid[planeGrid=xz,planeGridOffset=3](0,0)(3,1)
\pstThreeDPlaneGrid[planeGrid=yz,planeGridOffset=3](0,0)(3,1)
}
\pstThreeDPut[pOrigin=c](5.01,3.69,0){\psframebox*[framesep=.6pt,framearc=0]{\phantom{$\tau_0$}}}
\pstThreeDPut[pOrigin=c](3.54,5.33,0.33){\psframebox*[framesep=.7pt,framearc=1]{\phantom{$\tau_1$}}}
{
\psset{dotstyle=none,dotscale=1,drawCoor=false}
\psset{linecolor=black,linewidth=1.5pt,linejoin=1}
\psset{fillcolor=lightgray,opacity=.6,fillstyle=solid}
\pstThreeDLine(3,3,1)(9,4,0)(5,5,0)(3,3,1)(2,8,0)(5,5,0)
}
\pstThreeDPut[pOrigin=t](9,4,-0.24){$A$}
\pstThreeDPut[pOrigin=t](5,5,-0.25){$B$}
\pstThreeDPut[pOrigin=t](2,8,-0.24){$C$}
\pstThreeDPut[pOrigin=b](3,3,1.25){$D$}
\pstThreeDPut[pOrigin=c](5,3.69,0){$\tau_0$}
\pstThreeDPut[pOrigin=c](5,1.97,0){\psframebox*[framesep=1pt,framearc=0]{$\tau_3$}}
\pstThreeDPut[pOrigin=c](3.5,5.33,0.33){$\tau_1$}
\pstThreeDPut[pOrigin=c](1.55,4.89,0.33){$\tau_2$}
\pstThreeDPut[pOrigin=l](7.45,7.5,0){$\ \tau_4$}
\pstThreeDPut[pOrigin=br](0.06,-0.06,1.12){$1$}
\end{pspicture}
}
\\[+3.8ex]
\subfigure[Relevant cones associated to relevant faces of~\Gf]{
\psset{unit=.05890138981\textwidth}
\begin{pspicture}(-7.57,-3.85)(7.54,9.33)
\pstThreeDCoor[xMin=0,yMin=0,zMin=0,xMax=10,yMax=10,zMax=10,nameZ={},labelsep=5pt,linecolor=gray,linewidth=1.0pt]
{
\psset{linecolor=gray,linewidth=.5pt,linejoin=1,linestyle=dashed,fillcolor=lightgray,fillstyle=none}
\pstThreeDLine(10,0,0)(0,10,0)\pstThreeDLine(0,10,0)(0,0,10)\pstThreeDLine(0,0,10)(10,0,0)
}
{
\psset{linecolor=black,linewidth=1.0pt,linejoin=1,fillcolor=lightgray,fillstyle=none}
\pstThreeDLine(0,0,0)(0,0,10)
}
{
\psset{labelsep=3pt}
\uput[90](0,1.7){\psframebox*[framesep=1.5pt,framearc=0]{\darkgray$v_4$}}
}
{
\psset{linecolor=black,linewidth=1.0pt,linejoin=1,fillcolor=lightgray,fillstyle=none}
\pstThreeDLine(0,0,0)(1.17,2.65,6.18)
\pstThreeDLine(0,0,0)(2.63,1.58,5.79)
\pstThreeDLine(0,0,0)(8.15,1.48,.370)
\pstThreeDLine(0,0,0)(.573,6.87,2.58)
}
{
\psset{linecolor=darkgray,linewidth=1.2pt,linejoin=1,arrows=->,arrowscale=1,fillcolor=lightgray,fillstyle=none}
\pstThreeDLine(0,0,0)(.585,1.32,3.09)
\pstThreeDLine(0,0,0)(2.10,1.27,4.62)
\pstThreeDLine(0,0,0)(4.89,.888,.222)
\pstThreeDLine(0,0,0)(.229,2.75,1.03)
\pstThreeDLine(0,0,0)(0,0,2)
}
{
\psset{linecolor=white,linewidth=4pt,linejoin=1,fillcolor=lightgray,fillstyle=none}
\pstThreeDLine(2.19,1.90,5.90)(1.75,2.22,6.03)
\pstThreeDLine(2.43,2.11,5.47)(1.19,5.28,3.55)
}
{
\psset{linecolor=black,linewidth=1.0pt,linejoin=1,fillcolor=lightgray,fillstyle=none}
\pstThreeDLine(0,0,10)(1.17,2.65,6.18)
\pstThreeDLine(0,0,10)(2.63,1.58,5.79)
\pstThreeDLine(.573,6.87,2.58)(1.17,2.65,6.18)(2.63,1.58,5.79)(8.15,1.48,.370)
}
{
\psset{linecolor=black,linewidth=1.0pt,linejoin=1,linestyle=dashed,fillcolor=lightgray,fillstyle=none}
\pstThreeDLine(0,0,10)(8.15,1.48,.370)
\pstThreeDLine(0,0,10)(.573,6.87,2.58)
\pstThreeDLine(8.15,1.48,.370)(.573,6.87,2.58)
\pstThreeDLine(2.63,1.58,5.79)(.573,6.87,2.58)
}
{
\psset{labelsep=3pt}
\uput[-22](.275,1.1){\darkgray$v_0$}
\uput[-60](-1.3,-.9){\darkgray$v_2$}
\uput[80](.9,-0.1){\darkgray$v_3$}
}
{
\psset{labelsep=2.5pt}
\uput[202](-.3,1.4){\darkgray$v_1$}
}
{
\psset{labelsep=5pt}
\psdots[dotsize=10pt,linecolor=white](5.3,-.54)
\uput[0](4.45,-.4){$\Delta_{\tau_3}$}
}
{
\psset{labelsep=5pt}
\uput[180](-4.7,-3.1){$\Delta_{\tau_2}$}
}
\rput(1.9,3.75){$\delta_A$}
\rput(.17,4.9){\psframebox*[framesep=1.5pt,framearc=0]{$\Delta_B$}}
\rput(-1.43,3.5){$\delta_C$}
\rput(1.48,2.6){$\delta_1$}
\rput(1.85,1.55){\psframebox*[framesep=1.3pt,framearc=0]{$\delta_2$}}
\rput(-1.3,.25){$\delta_3$}
\pstThreeDNode(1.17,2.7,6.18){dt0}
\pstThreeDNode(2.63,1.58,5.79){dt1}
\pstThreeDNode(.585,1.32,8.10){AB}
\pstThreeDNode(.870,4.76,4.38){AD}
\pstThreeDNode(1.32,.790,7.90){BC}
\pstThreeDNode(1.90,2.12,6.00){BD}
\pstThreeDNode(5.40,1.53,3.08){CD}
\rput[Bl](-7.6,2.65){\rnode{CDlabel}{$\Delta_{[CD]}$}}
\rput[Bl](-7.6,8.94){\rnode{dt1label}{$\Delta_{\tau_1}$}}
\rput[Bl](-4.36,8.94){\rnode{BClabel}{$\Delta_{[BC]}$}}
\rput[B](0,8.94){$z,\Delta_{\tau_4}$}
\rput[Br](4.36,8.94){\rnode{ABlabel}{$\Delta_{[AB]}$}}
\rput[Br](7.6,8.94){\rnode{dt0label}{$\Delta_{\tau_0}$}}
\rput[Br](7.6,4.8){\rnode{BDlabel}{$\Delta_{[BD]}$}}
\rput[Br](7.6,.6){\rnode{ADlabel}{$\Delta_{[AD]}$}}
\ncline[linewidth=.3pt,nodesepB=2pt,nodesepA=1pt]{->}{dt0label}{dt0}
\ncline[linewidth=.3pt,nodesepB=2.5pt,nodesepA=2pt]{->}{dt1label}{dt1}
\ncline[linewidth=.3pt,nodesepB=2pt,nodesepA=1.5pt]{->}{ABlabel}{AB}
\ncline[linewidth=.3pt,nodesepB=2.5pt,nodesepA=1pt]{->}{ADlabel}{AD}
\ncline[linewidth=.3pt,nodesepB=2pt,nodesepA=1pt]{->}{BClabel}{BC}
\nccurve[linewidth=.3pt,nodesepB=2pt,nodesepA=1pt,angleA=217,angleB=-37]{->}{BDlabel}{BD}
\ncline[linewidth=.3pt,nodesepB=2pt,nodesepA=2.5pt]{->}{CDlabel}{CD}
\end{pspicture}
}
\caption{Case III: the only facets contributing to $s_0$ are the $B_1$-simplices $\tau_0$ and $\tau_1$}
\label{figcase3}
\end{figure}
%
%

\section{Case~III: exactly two facets of \Gf\ contribute to $s_0$, and these two facets are both $B_1$-simplices with respect to a same variable and have an edge in common}
\subsection{Figure and notations}
Without loss of generality, we may assume that the $B_1$-simplices $\tau_0$ and $\tau_1$ contributing to $s_0$ are as drawn in Figure~\ref{figcase3}.

Let us fix notations. We denote, as indicated in Figure~\ref{figcase3}, the vertices of $\tau_0$ and $\tau_1$ and their coordinates by
\begin{equation*}
A(x_A,y_A,0),\quad B(x_B,y_B,0),\quad C(x_C,y_C,0),\quad\text{and}\quad D(x_D,y_D,1).
\end{equation*}
We denote the neighbor facets of $\tau_0$ and $\tau_1$ by $\tau_2,\tau_3,\tau_4$. The unique primitive vectors perpendicular to $\tau_i$; $i=0,\ldots,4$; will be denoted by
\begin{equation*}
v_0(a_0,b_0,c_0),\quad v_1(a_1,b_1,c_1),\quad v_2(a_2,b_2,c_2),\quad v_3(a_3,b_3,c_3),\quad v_4(0,0,1),
\end{equation*}
respectively. In this way the affine supports of these facets have equations of the form
\begin{alignat*}{2}
\aff(\tau_i)&\leftrightarrow a_ix+b_iy+c_i&&z=m_i;\qquad i=0,\ldots,3;\\
\aff(\tau_4)&\leftrightarrow &&z=0;
\end{alignat*}
and we associate to them the numerical data
\begin{alignat*}{2}
(m_i,\sigma_i)&=(m(v_i),\sigma(v_i))&&=(a_ix_D+b_iy_D+c_i,a_i+b_i+c_i);\qquad i=0,\ldots,3;\\
(m_4,\sigma_4)&=(m(v_4),\sigma(v_4))&&=(0,1).
\end{alignat*}

We assume that $\tau_0$ and $\tau_1$ both contribute to the candidate pole $s_0$. With the present notations, this is, we assume that $p^{\sigma_0+m_0s_0}=p^{\sigma_1+m_1s_0}=1$, or equivalently,
\begin{align*}
\Re(s_0)&=-\frac{\sigma_0}{m_0}=-\frac{\sigma_1}{m_1}=-\frac{a_0+b_0+c_0}{a_0x_D+b_0y_D+c_0}=-\frac{a_1+b_1+c_1}{a_1x_D+b_1y_D+c_1}\qquad\text{and}\\
\Im(s_0)&=\frac{2n\pi}{\gcd(m_0,m_1)\log p}\qquad\text{for some $n\in\Z$.}
\end{align*}

%
%
\begin{figure}
\centering
\psset{unit=.03458213256\textwidth}
\begin{pspicture}(-8.25,-3.7)(9.1,9.7)
{\footnotesize
{
\psset{linecolor=gray,linewidth=.3pt,linejoin=1,linestyle=dashed,fillcolor=lightgray,fillstyle=none}
\pstThreeDLine(10,0,0)(0,10,0)\pstThreeDLine(0,10,0)(0,0,10)\pstThreeDLine(0,0,10)(10,0,0)
}
{
\psset{linecolor=black,linewidth=.7pt,linejoin=1,fillcolor=lightgray,fillstyle=none}
\pstThreeDLine(0,0,10)(1.17,2.65,6.18)
\pstThreeDLine(0,0,10)(2.63,1.58,5.79)
\pstThreeDLine(.573,6.87,2.58)(1.17,2.65,6.18)(2.63,1.58,5.79)(8.15,1.48,.370)
}
{
\psset{linecolor=black,linewidth=.7pt,linejoin=1,linestyle=dashed,fillcolor=lightgray,fillstyle=none}
\pstThreeDLine(0,0,10)(8.15,1.48,.370)
\pstThreeDLine(0,0,10)(.573,6.87,2.58)
\pstThreeDLine(8.15,1.48,.370)(.573,6.87,2.58)
\pstThreeDLine(2.63,1.58,5.79)(.573,6.87,2.58)
}
{
\psset{labelsep=3.8pt}
\uput[0](4.45,-.4){\psframebox*[framesep=0.4pt,framearc=0]{$\Delta_{\tau_3}$}}
}
{
\psset{labelsep=3.0pt}
\rput(-6.1,-3.0){\psframebox*[framesep=0.6pt,framearc=1]{\phantom{$\Delta$}}}
\uput[180](-4.7,-3.1){$\Delta_{\tau_2}$}
}
\rput(1.94,3.72){\footnotesize$\delta_A$}
\rput(.17,4.9){\psframebox*[framesep=0.3pt,framearc=1]{\footnotesize$\Delta_B$}}
\rput(-1.43,3.5){\footnotesize$\delta_C$}
\rput(1.48,2.6){\footnotesize$\delta_1$}
\rput(1.85,1.55){\psframebox*[framesep=0.3pt,framearc=.3]{\footnotesize$\delta_2$}}
\rput(-.3,.4){\footnotesize$\delta_3$}
\rput[br](6.55,-3.35){\footnotesize\gray$\{x+y+z=1\}\cap\Rplus^3$}
\pstThreeDNode(1.17,2.7,6.18){dt0}
\pstThreeDNode(2.63,1.58,5.79){dt1}
\pstThreeDNode(.585,1.32,8.10){AB}
\pstThreeDNode(.870,4.76,4.38){AD}
\pstThreeDNode(1.32,.790,7.90){BC}
\pstThreeDNode(1.90,2.12,6.00){BD}
\pstThreeDNode(5.40,1.53,3.08){CD}
\rput[Bl](-8.3,3.075){\rnode{CDlabel}{$\Delta_{[CD]}$}}
\rput[Bl](-8.3,9.075){\rnode{dt1label}{$\Delta_{\tau_1}$}}
\rput[Bl](-4.85,9.075){\rnode{BClabel}{$\Delta_{[BC]}$}}
\rput[B](0,9.075){$\Delta_{\tau_4}$}
\rput[Bl](2.85,9.075){\rnode{ABlabel}{$\Delta_{[AB]}$}}
\rput[Bl](7.1,9.075){\rnode{dt0label}{$\Delta_{\tau_0}$}}
\rput[Bl](7.1,6.075){\rnode{BDlabel}{$\Delta_{[BD]}$}}
\rput[Bl](7.1,3.075){\rnode{ADlabel}{$\Delta_{[AD]}$}}
\ncline[linewidth=.3pt,nodesepB=2pt,nodesepA=1pt]{->}{dt0label}{dt0}
\ncline[linewidth=.3pt,nodesepB=2.5pt,nodesepA=2pt]{->}{dt1label}{dt1}
\ncline[linewidth=.3pt,nodesepB=2pt,nodesepA=1.5pt]{->}{ABlabel}{AB}
\ncline[linewidth=.3pt,nodesepB=2.5pt,nodesepA=1pt]{->}{ADlabel}{AD}
\ncline[linewidth=.3pt,nodesepB=2pt,nodesepA=1pt]{->}{BClabel}{BC}
\nccurve[linewidth=.3pt,nodesepB=2pt,nodesepA=1pt,angleA=217,angleB=-37]{->}{BDlabel}{BD}
\ncline[linewidth=.3pt,nodesepB=2pt,nodesepA=2.5pt]{->}{CDlabel}{CD}
}
\end{pspicture}
\caption{Sketch of the intersection of the contributing cones with the plane $\{x+y+z=1\}$}
\label{figcase3detail}
\end{figure}
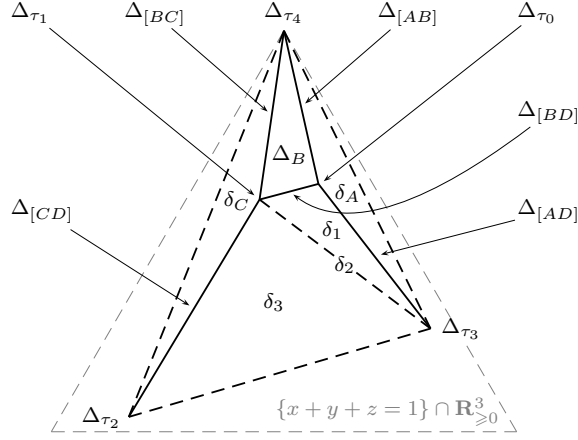
%
%

Throughout this section we will consider the following thirteen simplicial cones:
\begin{align*}
\dA&=\cone(v_0,v_3,v_4),&\DAB&=\cone(v_0,v_4),&\Dtnul&=\cone(v_0),\\
\DB&=\cone(v_0,v_1,v_4),&\DBC&=\cone(v_1,v_4),&\Dteen&=\cone(v_1).\\
\dC&=\cone(v_1,v_2,v_4),&\DAD&=\cone(v_0,v_3),&&\\
\delta_1&=\cone(v_0,v_1,v_3),&\DBD&=\cone(v_0,v_1),&&\\
\delta_2&=\cone(v_1,v_3),&\DCD&=\cone(v_1,v_2),&&\\
\delta_3&=\cone(v_1,v_2,v_3),&&&&
\end{align*}
The $\Delta_{\tau}$ listed above are the simplicial cones associated to the faces $\tau$. The cones $\Delta_A,\Delta_C,\Delta_D$, associated to the respective vertices $A,C,D$, are generally not simplicial. Later in this section we will consider simplicial subdivisions (without creating new rays) of $\Delta_A,\Delta_C,$ and $\Delta_D$ that include $\{\dA\},\{\dC\},$ and $\{\delta_1,\delta_2,\delta_3\}$, respectively (cfr.\ Figure~\ref{figcase3detail}).

Finally, let us fix notations for the vectors along the edges of $\tau_0$ and $\tau_1$:
\begin{alignat*}{8}
&\overrightarrow{AD}&&(x_D&&-x_A&&,y_D&&-y_A&&,1&&)&&=(\aA,\bA,1),\\
&\overrightarrow{BD}&&(x_D&&-x_B&&,y_D&&-y_B&&,1&&)&&=(\aB,\bB,1),\\
&\overrightarrow{CD}&&(x_D&&-x_C&&,y_D&&-y_C&&,1&&)&&=(\aC,\bC,1),\\
&\overrightarrow{AB}&&(x_B&&-x_A&&,y_B&&-y_A&&,0&&)&&=(\aA-\aB,\bA-\bB,0),\\
&\overrightarrow{BC}&&(x_C&&-x_B&&,y_C&&-y_B&&,0&&)&&=(\aB-\aC,\bB-\bC,0).
\end{alignat*}
The first three vectors are primitive; the last two are generally not. We put
\begin{equation*}
\fAB=\gcd(x_B-x_A,y_B-y_A)\qquad\text{and}\qquad\fBC=\gcd(x_C-x_B,y_C-y_B).
\end{equation*}

\subsection{Some relations between the variables}\label{srelbettvarcasdrie}
In the same way as in Case~I we obtain that
\begin{alignat*}{4}
\begin{pmatrix}
c_0\\c_3
\end{pmatrix}
&=-\aA&&
\begin{pmatrix}
a_0\\a_3
\end{pmatrix}
&&-\bA&&
\begin{pmatrix}
b_0\\b_3
\end{pmatrix},\\
\begin{pmatrix}
c_0\\c_1
\end{pmatrix}
&=-\aB&&
\begin{pmatrix}
a_0\\a_1
\end{pmatrix}
&&-\bB&&
\begin{pmatrix}
b_0\\b_1
\end{pmatrix},\\
\begin{pmatrix}
c_1\\c_2
\end{pmatrix}
&=-\aC&&
\begin{pmatrix}
a_1\\a_2
\end{pmatrix}
&&-\bC&&
\begin{pmatrix}
b_1\\b_2
\end{pmatrix}.
\end{alignat*}
A first consequence is that
\begin{equation*}
\gcd(a_i,b_i,c_i)=\gcd(a_i,b_i)=1;\qquad i=0,\ldots,3.
\end{equation*}
As a second consequence, we have
\begin{alignat*}{4}
\begin{vmatrix}
a_0&c_0\\a_3&c_3
\end{vmatrix}
&=-\bA&&
\begin{vmatrix}
a_0&b_0\\a_3&b_3
\end{vmatrix},
&\qquad\qquad\quad
\begin{vmatrix}
b_0&c_0\\b_3&c_3
\end{vmatrix}
&=\aA&&
\begin{vmatrix}
a_0&b_0\\a_3&b_3
\end{vmatrix},
\\
\begin{vmatrix}
a_0&c_0\\a_1&c_1
\end{vmatrix}
&=-\bB&&
\begin{vmatrix}
a_0&b_0\\a_1&b_1
\end{vmatrix},
&
\begin{vmatrix}
b_0&c_0\\b_1&c_1
\end{vmatrix}
&=\aB&&
\begin{vmatrix}
a_0&b_0\\a_1&b_1
\end{vmatrix},
\\
\begin{vmatrix}
a_1&c_1\\a_2&c_2
\end{vmatrix}
&=-\bC&&
\begin{vmatrix}
a_1&b_1\\a_2&b_2
\end{vmatrix},
&
\begin{vmatrix}
b_1&c_1\\b_2&c_2
\end{vmatrix}
&=\aC&&
\begin{vmatrix}
a_1&b_1\\a_2&b_2
\end{vmatrix}.
\end{alignat*}

In the calculations that will follow it is often convenient (or necessary) to know the signs of certain determinants. Coordinate system orientation considerations show that
\begin{align}
\begin{vmatrix}
a_0&b_0\\a_3&b_3
\end{vmatrix}
&>0,&
\begin{vmatrix}
a_0&b_0\\a_1&b_1
\end{vmatrix}
&<0,&
\begin{vmatrix}
a_1&b_1\\a_2&b_2
\end{vmatrix}
&<0,\notag\\
\Psi=
\begin{vmatrix}
a_1&b_1\\a_3&b_3
\end{vmatrix}
&>0,&-\Omega=
\begin{vmatrix}
a_0&b_0\\a_2&b_2
\end{vmatrix}
&<0,&\Theta=
\begin{vmatrix}
a_2&b_2\\a_3&b_3
\end{vmatrix}
&>0,\label{defPOT}\\
\begin{vmatrix}
a_0&b_0&c_0\\a_1&b_1&c_1\\a_3&b_3&c_3
\end{vmatrix}
&>0,&
\begin{vmatrix}
a_0&b_0&c_0\\a_1&b_1&c_1\\a_2&b_2&c_2
\end{vmatrix}
&>0,&
\begin{vmatrix}
a_1&b_1&c_1\\a_2&b_2&c_2\\a_3&b_3&c_3
\end{vmatrix}
&>0.\notag
\end{align}

\subsection{Igusa's local zeta function}
Since $f$ is non-degenerated over \Fp\ with respect to all the compact faces of its Newton polyhedron \Gf, by Theorem~\ref{formdenhoor} the local Igusa zeta function \Zof\ of $f$ is given by
\begin{equation}\label{deflvilzfvmd}
\Zof=\sum_{\substack{\tau\mathrm{\ compact}\\\mathrm{face\ of\ }\Gf}}L_{\tau}S(\Dtu),
\end{equation}
with
\begin{gather*}
L_{\tau}:s\mapsto L_{\tau}(s)=\left(\frac{p-1}{p}\right)^3-\frac{N_{\tau}}{p^2}\frac{p^s-1}{p^{s+1}-1},\\
N_{\tau}=\#\left\{(x,y,z)\in(\Fpcross)^3\;\middle\vert\;\fbart(x,y,z)=0\right\},
\end{gather*}
and
\begin{align}
S(\Dtu):s\mapsto S(\Dtu)(s)&=\sum_{k\in\Z^3\cap\Delta_{\tau}}p^{-\sigma(k)-m(k)s}\notag\\
&=\sum_{i\in I}\frac{\Sigma(\delta_i)(s)}{\prod_{j\in J_i}(p^{\sigma(w_j)+m(w_j)s}-1)}.\label{deflvilzfvmdbis}
\end{align}
Here $\{\delta_i\}_{i\in I}$ denotes a simplicial decomposition without introducing new rays of the cone $\Delta_{\tau}$ associated to $\tau$. The simplicial cone $\delta_i$ is supposed to be strictly positively spanned by the linearly independent primitive vectors $w_j$, $j\in J_i$, in $\Zplusn\setminus\{0\}$, and $\Sigma(\delta_i)$ is the function
\begin{equation*}
\Sigma(\delta_i):s\mapsto \Sigma(\delta_i)(s)=\sum_hp^{\sigma(h)+m(h)s},
\end{equation*}
where $h$ runs through the elements of the set
\begin{equation*}
H(w_j)_{j\in J_i}=\Z^3\cap\lozenge(w_j)_{j\in J_i},
\end{equation*}
with
\begin{equation*}
\lozenge(w_j)_{j\in J_i}=\left\{\sum\nolimits_{j\in J_i}h_jw_j\;\middle\vert\;h_j\in[0,1)\text{ for all }j\in J_i\right\}
\end{equation*}
the fundamental parallelepiped spanned by the vectors $w_j$, $j\in J_i$.

\subsection{The candidate pole $s_0$ and its residues}
We want to prove that $s_0$ is not a pole of \Zof. Since $s_0$ is a candidate pole of expected order two (and therefore is a pole of actual order at most two), it is enough to prove that the coefficients $a_{-2}$ and $a_{-1}$ in the Laurent series
\begin{equation*}
\Zof(s)=\sum_{k=-2}^{\infty}a_k(s-s_0)^k
\end{equation*}
of \Zof\ centered at $s_0$, both equal zero. These coefficients are given by
\begin{align*}
a_{-2}&=\lim_{s\to s_0}(s-s_0)^2\Zof(s)\qquad\text{and}\\
a_{-1}=\Res(\Zof,s_0)&=\lim_{s\to s_0}\frac{d}{ds}\left[(s-s_0)^2\Zof(s)\right].
\end{align*}

Alternatively (and consequently), it is sufficient to show that
\begin{align*}
R_2&=\lim_{s\to s_0}\left(p^{\sigma_0+m_0s}-1\right)\left(p^{\sigma_1+m_1s}-1\right)\Zof(s)\\
&=(\log p)^2m_0m_1a_{-2}\\\shortintertext{and}
R_1&=\lim_{s\to s_0}\frac{d}{ds}\left[\left(p^{\sigma_0+m_0s}-1\right)\left(p^{\sigma_1+m_1s}-1\right)\Zof(s)\right]\\
&=(\log p)^2m_0m_1a_{-1}+\frac{1}{2}(\log p)^3m_0m_1(m_0+m_1)a_{-2}
\end{align*}
both vanish. We will in the rest of this section prove that $R_2=R_1=0$.

\subsection{Terms contributing to $R_2$ and $R_1$}
We intend to calculate $R_2$ and $R_1$ based on Formula~\eqref{deflvilzfvmd} for \Zof.

Precisely $11$ compact faces of \Gf\ contribute to the candidate pole $s_0$. These are the subfaces $A,B,C,D,[AB],[BC],[AD],[BD],[CD],\tau_0,$ and $\tau_1$ of the two compact facets $\tau_0$ and $\tau_1$ that have $s_0$ as an associated candidate pole. It are only the terms of \eqref{deflvilzfvmd} associated to these faces that should be taken into account in the calculation of $R_1$. The other terms do not have $s_0$ as a pole and therefore do not contribute to the limit $R_1$.

Vertex $B$ is only contained in the facets $\tau_0,\tau_1,$ and $\tau_4$; hence its associated cone $\Delta_B$ is simplicial. The cones associated to the other vertices $A,C,$ and $D$ are generally not simplicial. For dealing with $S_A$ and $S_C$, we will, just as in Case~I, consider simplicial decompositions of $\Delta_A$ and $\Delta_C$ that contain $\delta_A$ and $\delta_C$, respectively. Terms of \eqref{deflvilzfvmdbis} associated to other cones than $\delta_A$ and $\delta_C$ in these decompositions do not have a pole in $s_0$, hence do not contribute to $R_1$. Vertex $D$ is contained in at least four facets. We shall consider a decomposition of $\Delta_D$ into simplicial cones among which $\delta_1,\delta_2,$ and $\delta_3$. Only the terms associated to these three cones should be taken into account when calculating $R_1$. This makes a total of $13$ terms contributing to $R_1$ (three coming from $D$ and one for every other contributing face).

The limit $R_2$ counts fewer contributions: the only terms of \eqref{deflvilzfvmd} and \eqref{deflvilzfvmdbis} that need to be considered are the ones that have a double pole in $s_0$. These are the terms associated to $B,[BD],$ and $\delta_1$. All other terms have at most a single pole in $s_0$ and do not contribute to $R_2$.

Let us write down these contributions explicitly. For $R_2$ we obtain
\begin{equation*}
R_2=L_B(s_0)\frac{\Sigma(\Delta_B)(s_0)}{p-1}+L_D(s_0)\frac{\Sigma(\delta_1)(s_0)}{p^{\sigma_3+m_3s_0}-1}+L_{[BD]}(s_0)\Sigma(\Delta_{[BD]})(s_0).
\end{equation*}
The thirteen terms making up $R_1$ are
\begin{multline*}
R_1=\dds{L_A(s)\frac{\Feens\Sigma(\delta_A)(s)}{\Fdries(p-1)}}
+\dds{L_B(s)\frac{\Sigma(\Delta_B)(s)}{p-1}}\\
+\dds{L_C(s)\frac{\Fnuls\Sigma(\delta_C)(s)}{\Ftwees(p-1)}}
+\dds{L_D(s)\frac{\Sigma(\delta_1)(s)}{p^{\sigma_3+m_3s}-1}}\\
+\dds{L_D(s)\frac{\Fnuls\Sigma(\delta_2)(s)}{p^{\sigma_3+m_3s}-1}}\\
+\dds{L_D(s)\frac{\Fnuls\Sigma(\delta_3)(s)}{\Ftwees\Fdries}}\\
+\dds{L_{[AB]}(s)\frac{\Feens\Sigma(\Delta_{[AB]})(s)}{p-1}}\\
+\dds{L_{[BC]}(s)\frac{\Fnuls\Sigma(\Delta_{[BC]})(s)}{p-1}}\\
+\dds{L_{[AD]}(s)\frac{\Feens\Sigma(\Delta_{[AD]})(s)}{p^{\sigma_3+m_3s}-1}}\\
+\dds{L_{[BD]}(s)\Sigma(\Delta_{[BD]})(s)}\\
+\dds{L_{[CD]}(s)\frac{\Fnuls\Sigma(\Delta_{[CD]})(s)}{p^{\sigma_2+m_2s}-1}}\\
+\dds{L_{\tau_0}(s)\Feens\Sigma(\Delta_{\tau_0})(s)}\\
+\dds{L_{\tau_1}(s)\Fnuls\Sigma(\Delta_{\tau_1})(s)}.
\end{multline*}
After simplification, $R_1$ is given by
\begin{multline*}
R_1=L_A(s_0)\frac{m_1(\log p)\Sigma(\delta_A)(s_0)}{\Fdrie(p-1)}
+L_B'(s_0)\frac{\Sigma(\Delta_B)(s_0)}{p-1}
+L_B(s_0)\frac{\Sigma(\Delta_B)'(s_0)}{p-1}\\
+L_C(s_0)\frac{m_0(\log p)\Sigma(\delta_C)(s_0)}{\Ftwee(p-1)}
+L_D'(s_0)\frac{\Sigma(\delta_1)(s_0)}{p^{\sigma_3+m_3s_0}-1}
+L_D(s_0)\frac{\Sigma(\delta_1)'(s_0)}{p^{\sigma_3+m_3s_0}-1}\\
-L_D(s_0)\frac{m_3(\log p)p^{\sigma_3+m_3s_0}\Sigma(\delta_1)(s_0)}{\Fdrie^2}
+L_D(s_0)\frac{m_0(\log p)\Sigma(\delta_2)(s_0)}{p^{\sigma_3+m_3s_0}-1}\\
+L_D(s_0)\frac{m_0(\log p)\Sigma(\delta_3)(s_0)}{\Ftwee\Fdrie}\\
+L_{[AB]}(s_0)\frac{m_1(\log p)\Sigma(\Delta_{[AB]})(s_0)}{p-1}
+L_{[BC]}(s_0)\frac{m_0(\log p)\Sigma(\Delta_{[BC]})(s_0)}{p-1}\\
+L_{[AD]}(s_0)\frac{m_1(\log p)\Sigma(\Delta_{[AD]})(s_0)}{p^{\sigma_3+m_3s_0}-1}
+L_{[BD]}'(s_0)\Sigma(\Delta_{[BD]})(s_0)\\
+L_{[BD]}(s_0)\Sigma(\Delta_{[BD]})'(s_0)
+L_{[CD]}(s_0)\frac{m_0(\log p)\Sigma(\Delta_{[CD]})(s_0)}{p^{\sigma_2+m_2s_0}-1}\\
+L_{\tau_0}(s_0)m_1(\log p)\Sigma(\Delta_{\tau_0})(s_0)
+L_{\tau_1}(s_0)m_0(\log p)\Sigma(\Delta_{\tau_1})(s_0).
\end{multline*}

\subsection{The numbers $N_{\tau}$}
Analogously to Case~I we obtain
\begin{gather*}
N_A=N_B=N_C=N_D=0,\\
N_{[AB]}=(p-1)N_0,\qquad N_{[BC]}=(p-1)N_1,\\
N_{[AD]}=N_{[BD]}=N_{[CD]}=(p-1)^2,\\
N_{\tau_0}=(p-1)^2-N_0,\qquad N_{\tau_1}=(p-1)^2-N_1,
\end{gather*}
with
\begin{align*}
N_0&=\#\left\{(x,y)\in(\Fpcross)^2\;\middle\vert\;\overline{f_{[AB]}}(x,y)=0\right\},\\
N_1&=\#\left\{(x,y)\in(\Fpcross)^2\;\middle\vert\;\overline{f_{[BC]}}(x,y)=0\right\}.
\end{align*}

\subsection{The factors $L_{\tau}(s_0)$ and $L_{\tau}'(s_0)$}
For the $L_{\tau}(s_0)$ we obtain
\begin{gather*}
L_A(s_0)=L_B(s_0)=L_C(s_0)=L_D(s_0)=\left(\frac{p-1}{p}\right)^3,\\
\begin{alignedat}{3}
L_{[AB]}(s_0)&=\left(\frac{p-1}{p}\right)^3&&-\frac{(p-1)N_0}{p^2}&&\frac{p^{s_0}-1}{p^{s_0+1}-1},\\
L_{[BC]}(s_0)&=\left(\frac{p-1}{p}\right)^3&&-\frac{(p-1)N_1}{p^2}&&\frac{p^{s_0}-1}{p^{s_0+1}-1},
\end{alignedat}\\
L_{[AD]}(s_0)=L_{[BD]}(s_0)=L_{[CD]}(s_0)=\left(\frac{p-1}{p}\right)^3-\left(\frac{p-1}{p}\right)^2\frac{p^{s_0}-1}{p^{s_0+1}-1},\\
\begin{alignedat}{3}
L_{\tau_0}(s_0)&=\left(\frac{p-1}{p}\right)^3&&-\frac{(p-1)^2-N_0}{p^2}&&\frac{p^{s_0}-1}{p^{s_0+1}-1},\\
L_{\tau_1}(s_0)&=\left(\frac{p-1}{p}\right)^3&&-\frac{(p-1)^2-N_1}{p^2}&&\frac{p^{s_0}-1}{p^{s_0+1}-1},
\end{alignedat}
\end{gather*}
while the $L_{\tau}'(s_0)$ are given by
\begin{gather*}
L_B'(s_0)=L_D'(s_0)=0,\\
L_{[BD]}'(s_0)=-(\log p)\left(\frac{p-1}{p}\right)^3\frac{p^{s_0+1}}{\bigl(p^{s_0+1}-1\bigr)^2}.
\end{gather*}

\subsection{Multiplicities of the relevant simplicial cones}
Based on Proposition~\ref{multipliciteit} and the relations obtained in Subsection~\ref{srelbettvarcasdrie}, we have, analogously to Case~I, that
\begin{gather*}
\mult\Delta_{[AB]}=\mult\Delta_{[BC]}=\mult\Delta_{\tau_0}=\mult\Delta_{\tau_1}=1,\\
\begin{alignedat}{6}
\mu_A=\mult\delta_A&=\#H(v_0,v_3,v_4)&&=\mult\Delta_{[AD]}&&=\#H(v_0,v_3)&&=&&
\begin{vmatrix}
a_0&b_0\\a_3&b_3
\end{vmatrix}
&&>0,\\
\mu_B=\mult\Delta_B&=\#H(v_0,v_1,v_4)&&=\mult\Delta_{[BD]}&&=\#H(v_0,v_1)&&=-&&
\begin{vmatrix}
a_0&b_0\\a_1&b_1
\end{vmatrix}
&&>0,\\
\mu_C=\mult\delta_C&=\#H(v_1,v_2,v_4)&&=\mult\Delta_{[CD]}&&=\#H(v_1,v_2)&&=-&&
\begin{vmatrix}
a_1&b_1\\a_2&b_2
\end{vmatrix}
&&>0,
\end{alignedat}\\
\mu_2=\mult\delta_2=\#H(v_1,v_3)=\gcd\left(\Psi,
\begin{Vmatrix}
a_1&c_1\\a_3&c_3
\end{Vmatrix},
\begin{Vmatrix}
b_1&c_1\\b_3&c_3
\end{Vmatrix}
\right)>0,
\end{gather*}
with $\Psi>0$ as in \eqref{defPOT}.

Although we did not choose $\delta_1'=\cone(v_0,v_1,v_2)$ to be part of a simplicial decomposition of $\Delta_D$, we will consider its multiplicity as well. As in Case~I, we then find that
\begin{alignat*}{5}
\mu_1&=\mult\delta_1&&=\#H(v_0,v_1,v_3)&&=
\begin{vmatrix}
a_0&b_0&c_0\\a_1&b_1&c_1\\a_3&b_3&c_3
\end{vmatrix}
&&=\mu_A\mu_B\fAB&&>0\qquad\text{and}\\
\mu_1'&=\mult\delta_1'&&=\#H(v_0,v_1,v_2)&&=
\begin{vmatrix}
a_0&b_0&c_0\\a_1&b_1&c_1\\a_2&b_2&c_2
\end{vmatrix}
&&=\mu_B\mu_C\fBC&&>0.
\end{alignat*}
Finally, we will derive a more useful formula for
\begin{equation*}
\mu_3=\mult\delta_3=\#H(v_1,v_2,v_3)=
\begin{vmatrix}
a_1&b_1&c_1\\a_2&b_2&c_2\\a_3&b_3&c_3
\end{vmatrix}>0,
\end{equation*}
similar to the ones for $\mu_1$ and $\mu_1'$, in Subsection~\ref{multmudrieformule}.

\subsection{The sums $\Sigma(\cdot)(s_0)$ and $\Sigma(\cdot)'(s_0)$}
Since the corresponding multiplicities equal one, we find that
\begin{equation*}
\Sigma(\Delta_{[AB]})(s_0)=\Sigma(\Delta_{[BC]})(s_0)=\Sigma(\Delta_{\tau_0})(s_0)=\Sigma(\Delta_{\tau_1})(s_0)=1.
\end{equation*}
From the overview of the multiplicities, it is also clear\footnote{See Case~I for more details.} that we may put
\begin{gather*}
\begin{alignedat}{2}
H_A&=H(v_0,v_3,v_4)&&=H(v_0,v_3),\\
H_B&=H(v_0,v_1,v_4)&&=H(v_0,v_1),\\
H_C&=H(v_1,v_2,v_4)&&=H(v_1,v_2),
\end{alignedat}\\
H_1=H(v_0,v_1,v_3),\qquad H_2=H(v_1,v_3),\qquad H_3=H(v_1,v_2,v_3).
\end{gather*}
It follows that
\begin{gather*}
\begin{alignedat}{3}
\Sigma_A&=\Sigma(\delta_A)(s_0)&&=\Sigma(\Delta_{[AD]})(s_0)&&=\sum\nolimits_{h\in H_A}p^{\sigma(h)+m(h)s_0};\\
\Sigma_B&=\Sigma(\Delta_B)(s_0)&&=\Sigma(\Delta_{[BD]})(s_0)&&=\sum\nolimits_{h\in H_B}p^{\sigma(h)+m(h)s_0};\\
\Sigma_C&=\Sigma(\delta_C)(s_0)&&=\Sigma(\Delta_{[CD]})(s_0)&&=\sum\nolimits_{h\in H_C}p^{\sigma(h)+m(h)s_0};
\end{alignedat}\\
\Sigma_i=\Sigma(\delta_i)(s_0)=\sum\nolimits_{h\in H_i}p^{\sigma(h)+m(h)s_0};\qquad i=1,2,3;\\
\Sigma_B'=\Sigma(\Delta_B)'(s_0)=\Sigma(\Delta_{[BD]})'(s_0)=\dds{\sum\nolimits_{h\in H_B}p^{\sigma(h)+m(h)s}};\\
\Sigma_1'=\Sigma(\delta_1)'(s_0)=\dds{\sum\nolimits_{h\in H_1}p^{\sigma(h)+m(h)s}}.
\end{gather*}

Let us for the rest of this section denote by $w$ the vector
\begin{equation*}
w=(1,1,1)+s_0(x_D,y_D,1)\in\C^3.
\end{equation*}
Then since $\overline{\Delta_D}$ contains all points of $H_V$; $V=A,B,C,1,2,3$; we have moreover that
\begin{alignat*}{2}
\Sigma_V&=\sum\nolimits_{h\in H_V}p^{w\cdot h};&\qquad V&=A,B,C,1,2,3;\\
\text{and}\qquad\Sigma_W'&=(\log p)\sum\nolimits_{h\in H_W}m(h)p^{w\cdot h};&W&=B,1.
\end{alignat*}

\subsection{Simplified formulas for $R_2$ and $R_1$}
Let us put
\begin{equation*}
F_2=p^{w\cdot v_2}-1=p^{\sigma_2+m_2s_0}-1\qquad\text{and}\qquad F_3=p^{w\cdot v_3}-1=p^{\sigma_3+m_3s_0}-1.
\end{equation*}
Then, exploiting the information above on the numbers $N_{\tau}$ and the multiplicities of the cones, we obtain the following new formulas for $R_2$ and $R_1$:
\begin{gather}
R_2=\left(\frac{p-1}{p}\right)^3\left(\frac{\Sigma_B}{1-p^{-s_0-1}}+\frac{\Sigma_1}{F_3}\right),\label{followingnewformulaRtweecasedrie}\\
\begin{multlined}[.85\textwidth]
R_1=(\log p)\left(\frac{p-1}{p}\right)^3\cdot\\
\Biggl[\frac{1}{1-p^{-s_0-1}}\left(\frac{m_1\Sigma_A}{F_3}+\frac{\Sigma_B'}{\log p}-\frac{\Sigma_B}{p^{s_0+1}-1}+\frac{m_0\Sigma_C}{F_2}+m_0+m_1\right)\\
+\frac{\Sigma_1'}{(\log p)F_3}-\frac{m_3(F_3+1)\Sigma_1}{F_3^2}+\frac{m_0\Sigma_2}{F_3}+\frac{m_0\Sigma_3}{F_2F_3}\Biggr].
\end{multlined}\label{followingnewformulaReencasedrie}
\end{gather}
Note that the \lq unknown\rq\ numbers $N_0$ and $N_1$ disappear from the equation.

\subsection{Vector identities}
We will quite often use the following identities:
\begin{alignat}{6}
\begin{vmatrix}
a_1&b_1\\a_3&b_3
\end{vmatrix}
v_0&-
\begin{vmatrix}
a_0&b_0\\a_3&b_3
\end{vmatrix}
v_1&&+
\begin{vmatrix}
a_0&b_0\\a_1&b_1
\end{vmatrix}
v_3&&=&\;\Psi v_0&\;-\;&\mu_A v_1&\;-\;&\mu_B v_3&=(0,0,\mu_1),\label{vi1c3}\\
\begin{vmatrix}
a_1&b_1\\a_2&b_2
\end{vmatrix}
v_0&-
\begin{vmatrix}
a_0&b_0\\a_2&b_2
\end{vmatrix}
v_1&&+
\begin{vmatrix}
a_0&b_0\\a_1&b_1
\end{vmatrix}
v_2&&=&\;-\mu_C v_0&\;+\;&\Omega v_1&\;-\;&\mu_B v_2&=(0,0,\mu_1'),\label{vi2c3}\\
\begin{vmatrix}
a_2&b_2\\a_3&b_3
\end{vmatrix}
v_1&-
\begin{vmatrix}
a_1&b_1\\a_3&b_3
\end{vmatrix}
v_2&&+
\begin{vmatrix}
a_1&b_1\\a_2&b_2
\end{vmatrix}
v_3&&=&\;\Theta v_1&\;-\;&\Psi v_2&\;-\;&\mu_C v_3&=(0,0,\mu_3).\label{vi3c3}
\end{alignat}
Hereby $\Psi,\Omega,\Theta>0$ are as introduced in \eqref{defPOT}. As also mentioned in Case~I, these equations simply express the equalities of the last rows of the identical matrices $(\adj M)M$ and $(\det M)I$ for $M$ the respective matrices
\begin{equation*}
\begin{pmatrix}
a_0&b_0&c_0\\a_1&b_1&c_1\\a_3&b_3&c_3
\end{pmatrix},\qquad
\begin{pmatrix}
a_0&b_0&c_0\\a_1&b_1&c_1\\a_2&b_2&c_2
\end{pmatrix},\qquad\text{and}\qquad
\begin{pmatrix}
a_1&b_1&c_1\\a_2&b_2&c_2\\a_3&b_3&c_3
\end{pmatrix}
\end{equation*}
with respective determinants $\mu_1,\mu_1',$ and $\mu_3$.

Useful consequences of (\ref{vi1c3}--\ref{vi3c3}) arise from making the dot product with $w=(1,1,1)+s_0(x_D,y_D,1)$ on all sides of the equations:
\begin{alignat}{3}
-\Psi(w\cdot v_0)&\;+\;&\mu_A(w\cdot v_1)&\;+\;&\mu_B(w\cdot v_3)&=\mu_1(-s_0-1),\label{dpi1c3}\\
\mu_C(w\cdot v_0)&\;-\;&\Omega(w\cdot v_1)&\;+\;&\mu_B(w\cdot v_2)&=\mu_1'(-s_0-1),\label{dpi2c3}\\
-\Theta(w\cdot v_1)&\;+\;&\Psi(w\cdot v_2)&\;+\;&\mu_C(w\cdot v_3)&=\mu_3(-s_0-1).\label{dpi3c3}
\end{alignat}

\subsection{Points of $H_A,H_B,H_C,H_2,H_1$ and additional relations}\label{pointsofandaddrel}
Based on the discussion on integral points in fundamental parallelepipeds in Section~\ref{fundpar}, we can state that the points of $H_A,H_B,H_C,$ and $H_2$ are given by
\begin{alignat}{2}
\left\{\frac{i\xi_A}{\mu_A}\right\}v_0&+\frac{i}{\mu_A}v_3;&\qquad&i=0,\ldots,\mu_A-1;\label{pofAcasedrie}\\\shortintertext{by}
\left\{\frac{j\xi_B}{\mu_B}\right\}v_0&+\frac{j}{\mu_B}v_1;&&j=0,\ldots,\mu_B-1;\label{pofBcasedrie}\\\shortintertext{by}
\left\{\frac{i\xi_C}{\mu_C}\right\}v_1&+\frac{i}{\mu_C}v_2;&&i=0,\ldots,\mu_C-1;\label{pofCcasedrie}\\\shortintertext{and by}
\left\{\frac{j\xi_2}{\mu_2}\right\}v_1&+\frac{j}{\mu_2}v_3;&&j=0,\ldots,\mu_2-1;\label{pof2casedrie}
\end{alignat}
respectively. Here $\xi_A$ denotes the unique element $\xi_A\in\verA$ such that $\xi_Av_0+v_3$ belongs to $\mu_A\Z^3$. It follows that $\xi_A$ is coprime to $\mu_A$. (Analogously for $\xi_B,\xi_C,$ and $\xi_2$.)

In exactly the same way as we did in Case~I for the points of $H_C$ (cfr.\ Sub\-sec\-tion~\ref{descpointsHCgevaleen}), we obtain that the $\mu_1=\mu_A\mu_B\fAB$ points of $H_1=H(v_0,v_1,v_3)$ are precisely
\begin{multline}\label{pof1casedrie}
\left\{\frac{i\xi_A\mu_B\fAB+j\xi_B\mu_A\fAB-k\Psi}{\mu_1}\right\}v_0+\frac{j\fAB+k}{\mu_B\fAB}v_1+\frac{i\fAB+k}{\mu_A\fAB}v_3;\\
i=0,\ldots,\mu_A-1;\quad j=0,\ldots,\mu_B-1;\quad k=0,\ldots,\fAB-1.
\end{multline}

On the other hand, we also know from Section~\ref{fundpar} that $\mu_2\mid\mu_1$ and that when $h$ runs through the elements of $H_1$, its $v_0$-coordinate $h_0$ runs precisely $\mu_2$ times through the numbers
\begin{equation*}
\frac{l\mu_2}{\mu_1};\qquad l=0,\ldots,\frac{\mu_1}{\mu_2}-1.
\end{equation*}
This implies that\footnote{First, note that in the left-hand side of the equation, the inner curly brackets denote the reduction of the argument modulo $\mu_1$ (cfr.\ Notation~\ref{notatiemodulo}), while the outer curly brackets serve as set delimiters. Secondly, recall that the maps $i\mapsto\{i\xi_A\}_{\mu_A}$ and $j\mapsto\{j\xi_B\}_{\mu_B}$ are permutations of \verA\ and \verB, respectively, so that after reordering the elements of the set, we can indeed omit the $\xi_A$ and $\xi_B$ from the equation.}
\begin{equation*}
\bigl\{\{i\mu_B\fAB+j\mu_A\fAB-k\Psi\}_{\mu_1}\bigr\}_{i,j,k=0}^{\mu_A-1,\mu_B-1,\fAB-1}=\bigl\{l\mu_2\bigr\}_{l=0}^{\mu_1/\mu_2-1}.
\end{equation*}
From this equality of sets, we easily conclude that
\begin{equation*}
\mu_2\mid\mu_A\fAB,\mu_B\fAB,\Psi,
\end{equation*}
what we already knew\footnote{The fact that $\mu_2\mid\mu_A\fAB,\mu_B\fAB$ was shown in several ways in the proof of Theorem~\ref{algfp}(v), while it follows from Proposition~\ref{multipliciteit} that $\mu_2=\gcd\left(\Psi,\bigl\lVert\begin{smallmatrix}a_1&c_1\\a_3&c_3\end{smallmatrix}\bigr\rVert,\bigl\lVert\begin{smallmatrix}b_1&c_1\\b_3&c_3\end{smallmatrix}\bigr\rVert\right)$.}, but also that $\mu_2$ can be written as a linear combination with integer coefficients of $\mu_A\fAB,\mu_B\fAB,$ and $\Psi$. Hence
\begin{equation*}
\mu_2=\gcd(\mu_A\fAB,\mu_B\fAB,\Psi).
\end{equation*}

Recall from \eqref{vi1c3} that
\begin{equation*}
\Psi v_0-\mu_A v_1-\mu_B v_3=(0,0,\mu_1).
\end{equation*}
If we put $\gamma=\gcd(\mu_A\fAB,\Psi)$, we have $\gamma\mid\mu_1$, and therefore it follows that
\begin{equation*}
\mu_B\fAB v_3=\Psi\fAB v_0-\mu_A\fAB v_1-(0,0,\fAB\mu_1)\in\gamma\Z^3.
\end{equation*}
The primitivity of $v_3$ now implies that $\gamma\mid\mu_B\fAB$, and thus we obtain that
\begin{align*}
\mu_2&=\gcd(\mu_A\fAB,\mu_B\fAB,\Psi)\\
&=\gcd(\mu_A\fAB,\Psi)\\
&=\gcd(\mu_B\fAB,\Psi),
\end{align*}
the last equality due to the symmetry of the argument above.

Finally, let us denote
\begin{equation}\label{consnotc3}
\begin{gathered}
\frac{\mu_A\fAB}{\mu_2}=\frac{\mu_1}{\mu_B\mu_2}=\fBtwee\in\Zplusnul,\qquad
\frac{\mu_B\fAB}{\mu_2}=\frac{\mu_1}{\mu_A\mu_2}=\fAtwee\in\Zplusnul,\\
\text{and}\qquad\frac{\Psi}{\mu_2}=\psi\in\Zplusnul,
\end{gathered}
\end{equation}
resulting in
\begin{equation*}
\mu_1=\mu_A\mu_B\fAB=\mu_A\mu_2\fAtwee=\mu_B\mu_2\fBtwee.
\end{equation*}

\subsection{Investigation of the $\Sigma_{\bullet}$ and the $\Sigma_{\bullet}'$, except for $\Sigma_1'$, $\Sigma_3$}
\subsubsection{The sum $\Sigma_B$}\label{sssSigmaBc3}
Because in this case $\tau_0$ and $\tau_1$ both contribute to $s_0$, and therefore $p^{w\cdot v_0}=p^{w\cdot v_1}=1$, the term $\Sigma_B$ plays a special role. By \eqref{pofBcasedrie} and the fact that $p^{a(w\cdot v_0)}=p^{\{a\}(w\cdot v_0)}$ for every $a\in\R$, we have
\begin{equation*}
\Sigma_B=\sum_{h\in H_B}p^{w\cdot h}=\sum_{j=0}^{\mu_B-1}\Bigl(p^{\frac{\xi_B(w\cdot v_0)+w\cdot v_1}{\mu_B}}\Bigr)^j=\sum_{j=0}^{\mu_B-1}\Bigl(p^{w\cdot h^{\ast}}\Bigr)^j,
\end{equation*}
with
\begin{equation*}
h^{\ast}=\frac{\xi_B}{\mu_B}v_0+\frac{1}{\mu_B}v_1\in\Z^3,
\end{equation*}
a generating element of $H_B$ (if $\mu_B>1$).

Unlike, e.g., $p^{(\xi_A(w\cdot v_0)+w\cdot v_3)/\mu_A}$, appearing in Formula~\eqref{formSigAalgc3} for $\Sigma_A$ below, the $\mu_B$th root of unity $p^{w\cdot h^{\ast}}=p^{(\xi_B(w\cdot v_0)+w\cdot v_1)/\mu_B}$ may equal one, but may as well differ from one. We need to distinguish between these two cases. As
\begin{equation*}
s_0=-\frac{\sigma_0}{m_0}+\frac{2n\pi i}{\gcd(m_0,m_1)\log p}=-\frac{\sigma_1}{m_1}+\frac{2n\pi i}{\gcd(m_0,m_1)\log p}
\end{equation*}
for some $n\in\Z$ and hence
\begin{equation*}
p^{w\cdot h^{\ast}}=p^{\sigma(h^{\ast})+m(h^{\ast})s_0}=\exp\frac{2nm(h^{\ast})\pi i}{\gcd(m_0,m_1)},
\end{equation*}
we see that $p^{w\cdot h^{\ast}}=1$ if and only if
\begin{equation*}
n^{\ast}=\frac{\gcd(m_0,m_1)}{\gcd(m_0,m_1,m(h^{\ast}))}\mid n.
\end{equation*}

In this way we find
\begin{equation}\label{formSigBc3}
\Sigma_B=\sum_{j=0}^{\mu_B-1}\Bigl(p^{w\cdot h^{\ast}}\Bigr)^j=
\begin{cases}
{\displaystyle\sum\nolimits_j1=\mu_B,}&\text{if $n^{\ast}\mid n$;}\\[+2ex]
{\displaystyle\frac{\left(p^{w\cdot h^{\ast}}\right)^{\mu_B}-1}{p^{w\cdot h^{\ast}}-1}=0,}&\text{otherwise.}
\end{cases}
\end{equation}

Let us next look at $\Sigma_B'$.

\subsubsection{The sum $\Sigma_B'$}\label{sssSigmaBaccentc3}
As we know, the $\mu_B$ points of $H_B$ are given by
\begin{equation*}
\left\{\frac{j\xi_B}{\mu_B}\right\}v_0+\frac{j}{\mu_B}v_1;\qquad j=0,\ldots,\mu_B-1;
\end{equation*}
but if $\xi_B'$ denotes the unique element $\xi_B'\in\verB$ such that $\xi_B\xi_B'\equiv1\mod\mu_B$, they are as well given by
\begin{equation*}
\frac{j}{\mu_B}v_0+\left\{\frac{j\xi_B'}{\mu_B}\right\}v_1;\qquad j=0,\ldots,\mu_B-1.
\end{equation*}

Recall that we introduced $\Sigma_B'$ as
\begin{align*}
\Sigma_B'&=\Sigma(\DB)'(s_0)=\Sigma(\DBD)'(s_0)\\
&=\dds{\sum_{h\in H_B}p^{\sigma(h)+m(h)s}}\\
&=(\log p)\sum_{h\in H_B}m(h)p^{\sigma(h)+m(h)s_0}\\
&=(\log p)\sum_hm(h)p^{w\cdot h}.
\end{align*}
Hence if we write $h=h_0v_0+h_1v_1$ for $h\in H_B=H(v_0,v_1)$, we find
\begin{align*}
\frac{\Sigma_B'}{\log p}&=m_0\sum_{h\in H_B}h_0p^{h_0(w\cdot v_0)+h_1(w\cdot v_1)}
+m_1\sum_{h\in H_B}h_1p^{h_0(w\cdot v_0)+h_1(w\cdot v_1)}\\
&=\frac{m_0}{\mu_B}\sum_{j=0}^{\mu_B-1}j\Bigl(p^{\frac{w\cdot v_0+\xi_B'(w\cdot v_1)}{\mu_B}}\Bigr)^j
+\frac{m_1}{\mu_B}\sum_{j=0}^{\mu_B-1}j\Bigl(p^{\frac{\xi_B(w\cdot v_0)+w\cdot v_1}{\mu_B}}\Bigr)^j.
\end{align*}

As
\begin{equation}\label{teentottander}
p^{\frac{w\cdot v_0+\xi_B'(w\cdot v_1)}{\mu_B}}=\Bigl(p^{\frac{\xi_B(w\cdot v_0)+w\cdot v_1}{\mu_B}}\Bigr)^{\xi_B'}\quad\text{and}\quad\ p^{\frac{\xi_B(w\cdot v_0)+w\cdot v_1}{\mu_B}}=\Bigl(p^{\frac{w\cdot v_0+\xi_B'(w\cdot v_1)}{\mu_B}}\Bigr)^{\xi_B},
\end{equation}
the numbers $p^{(\xi_B(w\cdot v_0)+w\cdot v_1)/\mu_B}$ and $p^{(w\cdot v_0+\xi_B'(w\cdot v_1))/\mu_B}$ are either both one (if $n^{\ast}\mid n$) or both different from one (if $n^{\ast}\nmid n$). We obtain
\begin{equation*}
\frac{\Sigma_B'}{\log p}=
\begin{cases}
{\displaystyle\frac{m_0+m_1}{\mu_B}\sum_{j=0}^{\mu_B-1}j=\frac{(m_0+m_1)(\mu_B-1)}{2},}&\text{if $n^{\ast}\mid n$;}\\[+2.5ex]
{\displaystyle\frac{m_0}{p^{\frac{w\cdot v_0+\xi_B'(w\cdot v_1)}{\mu_B}}-1}+
\frac{m_1}{p^{\frac{\xi_B(w\cdot v_0)+w\cdot v_1}{\mu_B}}-1},}&\text{otherwise.}
\end{cases}
\end{equation*}

\subsubsection{The sums $\Sigma_A$, $\Sigma_C$, and $\Sigma_2$}\label{sssSASCS2c3}
From (\ref{pofAcasedrie}, \ref{pofCcasedrie}, \ref{pof2casedrie}) and the fact that $p^{w\cdot v_0}=p^{w\cdot v_1}=1$, we obtain in the same way as in Case~I that
\begin{alignat}{4}
\Sigma_A&=\sum_{h\in H_A}p^{w\cdot h}
&&=\sum_{i=0}^{\mu_A-1}&&\Bigl(p^{\frac{\xi_A(w\cdot v_0)+w\cdot v_3}{\mu_A}}\Bigr)^i
&&=\frac{F_3}{p^{\frac{\xi_A(w\cdot v_0)+w\cdot v_3}{\mu_A}}-1},\label{formSigAalgc3}\\
\Sigma_C&=\sum_{h\in H_C}p^{w\cdot h}
&&=\sum_{i=0}^{\mu_C-1}&&\Bigl(p^{\frac{\xi_C(w\cdot v_1)+w\cdot v_2}{\mu_C}}\Bigr)^i
&&=\frac{F_2}{p^{\frac{\xi_C(w\cdot v_1)+w\cdot v_2}{\mu_C}}-1},\qquad\text{and}\notag\\
\Sigma_2&=\sum_{h\in H_2}p^{w\cdot h}
&&=\sum_{j=0}^{\mu_2-1}&&\Bigl(p^{\frac{\xi_2(w\cdot v_1)+w\cdot v_3}{\mu_2}}\Bigr)^j
&&=\frac{F_3}{p^{\frac{\xi_2(w\cdot v_1)+w\cdot v_3}{\mu_2}}-1}.\notag
\end{alignat}

In Case~I we also observed that $\xi_Av_0+v_3\in\mu_A\Z^3$, $\xi_Bv_0+v_1\in\mu_B\Z^3$, and \eqref{vi1c3} give rise to $(\xi_A\mu_B+\xi_B\mu_A+\Psi)v_0\in\mu_A\mu_B\Z^3$ and hence to
\begin{equation}\label{inZetidABc3}
\frac{\xi_A\mu_B+\xi_B\mu_A+\Psi}{\mu_A\mu_B}\in\Z.
\end{equation}
Using \eqref{dpi1c3} it follows that
\begin{multline}\label{sigasigbidcasedrie}
p^{\frac{\xi_A(w\cdot v_0)+w\cdot v_3}{\mu_A}}p^{\frac{\xi_B(w\cdot v_0)+w\cdot v_1}{\mu_B}}
=p^{\frac{(\xi_A\mu_B+\xi_B\mu_A)(w\cdot v_0)+\mu_A(w\cdot v_1)+\mu_B(w\cdot v_3)}{\mu_A\mu_B}}\\
=p^{\frac{-\Psi(w\cdot v_0)+\mu_A(w\cdot v_1)+\mu_B(w\cdot v_3)}{\mu_A\mu_B}}=p^{\fAB(-s_0-1)}.
\end{multline}
Analogously, $v_0+\xi_B'v_1\in\mu_B\Z^3$, $\xi_Cv_1+v_2\in\mu_C\Z^3$, \eqref{vi2c3}, and \eqref{dpi2c3} yield
\begin{multline}\label{sigcsigbidcasedrie}
p^{\frac{w\cdot v_0+\xi_B'(w\cdot v_1)}{\mu_B}}p^{\frac{\xi_C(w\cdot v_1)+w\cdot v_2}{\mu_C}}
=p^{\frac{\mu_C(w\cdot v_0)+(\xi_B'\mu_C+\xi_C\mu_B)(w\cdot v_1)+\mu_B(w\cdot v_2)}{\mu_B\mu_C}}\\
=p^{\frac{\mu_C(w\cdot v_0)-\Omega(w\cdot v_1)+\mu_B(w\cdot v_2)}{\mu_B\mu_C}}=p^{\fBC(-s_0-1)},
\end{multline}
while $v_0+\xi_B'v_1\in\mu_B\Z^3$, $\xi_2v_1+v_3\in\mu_2\Z^3$, \eqref{vi1c3}, \eqref{dpi1c3}, and \eqref{consnotc3} lead to
\begin{multline}\label{sig2sigbidcasedrie}
\Bigl(p^{\frac{w\cdot v_0+\xi_B'(w\cdot v_1)}{\mu_B}}\Bigr)^{-\psi}p^{\frac{\xi_2(w\cdot v_1)+w\cdot v_3}{\mu_2}}
=p^{\frac{-\Psi(w\cdot v_0)+(-\xi_B'\Psi+\xi_2\mu_B)(w\cdot v_1)+\mu_B(w\cdot v_3)}{\mu_B\mu_2}}\\
=p^{\frac{-\Psi(w\cdot v_0)+\mu_A(w\cdot v_1)+\mu_B(w\cdot v_3)}{\mu_B\mu_2}}=p^{\fBtwee(-s_0-1)}.
\end{multline}

Consequently, if $n^{\ast}\mid n$ and hence $p^{(\xi_B(w\cdot v_0)+w\cdot v_1)/\mu_B}=p^{(w\cdot v_0+\xi_B'(w\cdot v_1))/\mu_B}=1$, one has that
\begin{equation}\label{fSASCS2ifnsdnc3}
\begin{gathered}
\Sigma_A=\frac{F_3}{p^{\fAB(-s_0-1)}-1},\qquad\quad\ \qquad
\Sigma_C=\frac{F_2}{p^{\fBC(-s_0-1)}-1},\\
\text{and}\qquad\Sigma_2=\frac{F_3}{p^{\fBtwee(-s_0-1)}-1}.\qquad\text{\phantom{and}}
\end{gathered}
\end{equation}

\subsubsection{The sum $\Sigma_1$}\label{parberSig1c3}
If for $h\in H_1=H(v_0,v_1,v_3)$, we denote by $(h_0,h_1,h_3)$ the coordinates of $h$ with respect to the basis $(v_0,v_1,v_3)$, then by \eqref{pof1casedrie} and $p^{w\cdot v_0}=1$ we have that
\begin{align*}
\Sigma_1&=\sum_{h\in H_1}p^{w\cdot h}\\
&=\sum_hp^{h_0(w\cdot v_0)+h_1(w\cdot v_1)+h_3(w\cdot v_3)}\\
&=\sum_{i=0}^{\mu_A-1}\sum_{j=0}^{\mu_B-1}\sum_{k=0}^{\fAB-1}p^{\frac{i\xi_A\mu_B\fAB+j\xi_B\mu_A\fAB-k\Psi}{\mu_1}(w\cdot v_0)+\frac{j\fAB+k}{\mu_B\fAB}(w\cdot v_1)+\frac{i\fAB+k}{\mu_A\fAB}(w\cdot v_3)}\\
&=\sum_i\Bigl(p^{\frac{\xi_A(w\cdot v_0)+w\cdot v_3}{\mu_A}}\Bigr)^i\sum_j\Bigl(p^{\frac{\xi_B(w\cdot v_0)+w\cdot v_1}{\mu_B}}\Bigr)^j\sum_k\Bigl(p^{\frac{-\Psi(w\cdot v_0)+\mu_A(w\cdot v_1)+\mu_B(w\cdot v_3)}{\mu_1}}\Bigr)^k\\
&=\Sigma_A\Sigma_B\,\frac{p^{\fAB(-s_0-1)}-1}{p^{-s_0-1}-1},
\end{align*}
where in the last step we again used \eqref{dpi1c3}. It now follows from \eqref{formSigBc3} and \eqref{fSASCS2ifnsdnc3} that
\begin{equation*}
\Sigma_1=
\begin{cases}
{\displaystyle\frac{\mu_BF_3}{p^{-s_0-1}-1},}&\text{if $n^{\ast}\mid n$;}\\[+2ex]
{\displaystyle0,}&\text{otherwise.}
\end{cases}
\end{equation*}

\subsection{Proof of $R_2=0$ and a new formula for $R_1$}
If we fill in the formulas for $\Sigma_B$ and $\Sigma_1$ in Formula~\eqref{followingnewformulaRtweecasedrie} for $R_2$, we obtain
\begin{align*}
R_2&=\left(\frac{p-1}{p}\right)^3\left(\frac{\Sigma_B}{1-p^{-s_0-1}}+\frac{\Sigma_1}{F_3}\right)\\
&=\left(\frac{p-1}{p}\right)^3\left(\frac{\mu_B}{1-p^{-s_0-1}}+\frac{\mu_B}{p^{-s_0-1}-1}\right)\\
&=0
\end{align*}
in the case that $n^{\ast}\mid n$ and clearly the same result in the other case as well.

Let us check how much progress we made on $R_1$. First of all, denote by $R_1'$ the third factor in Formula~\eqref{followingnewformulaReencasedrie} for $R_1$; i.e., put
\begin{equation}
R_1=(\log p)\left(\frac{p-1}{p}\right)^3R_1'.
\end{equation}
Obviously, we want to prove that $R_1'=0$. Secondly, let us from now on denote $p^{-s_0-1}$ by $q$.

If we then fill in the formulas for $\Sigma_A,\Sigma_B,\Sigma_C,\Sigma_1,\Sigma_2,$ and $\Sigma_B'$ obtained above in the formula for $R_1'$, we find that in the case $n^{\ast}\mid n$, the \lq residue\rq\ $R_1'$ equals
\begin{multline}\label{lastformulaforReenaccent}
R_1'=\frac{1}{1-q}\left(\frac{m_0}{1-q^{-\fBC}}+\frac{m_1}{1-q^{-\fAB}}+\frac{\mu_B}{1-q^{-1}}+\frac{m_3\mu_B(F_3+1)}{F_3}\right.\\+\left.\frac{(m_0+m_1)(\mu_B-1)}{2}\right)+\frac{\Sigma_1'}{(\log p)F_3}+\frac{m_0}{q^{\fBtwee}-1}+\frac{m_0\Sigma_3}{F_2F_3}.
\end{multline}
In the complementary case we have
\begin{align}
R_1'&=\frac{m_1}{1-q}\Biggl(\frac{1}{p^{\frac{\xi_A(w\cdot v_0)+w\cdot v_3}{\mu_A}}-1}+\frac{1}{p^{\frac{\xi_B(w\cdot v_0)+w\cdot v_1}{\mu_B}}-1}+1\Biggr)\notag\\
&\qquad\qquad\quad\ \ +\frac{m_0}{1-q}\Biggl(\frac{1}{p^{\frac{w\cdot v_0+\xi_B'(w\cdot v_1)}{\mu_B}}-1}+\frac{1}{p^{\frac{\xi_C(w\cdot v_1)+w\cdot v_2}{\mu_C}}-1}+1\Biggr)\notag\\
&\qquad\qquad\qquad\qquad\qquad\quad\,+\frac{\Sigma_1'}{(\log p)F_3}+\frac{m_0}{p^{\frac{\xi_2(w\cdot v_1)+w\cdot v_3}{\mu_2}}-1}+\frac{m_0\Sigma_3}{F_2F_3},\notag\\
\shortintertext{and by (\ref{sigasigbidcasedrie}--\ref{sigcsigbidcasedrie}) this is,}
R_1'&=\frac{m_1}{1-q}\frac{q^{\fAB}-1}{\Bigl(p^{\frac{\xi_A(w\cdot v_0)+w\cdot v_3}{\mu_A}}-1\Bigr)\Bigl(p^{\frac{\xi_B(w\cdot v_0)+w\cdot v_1}{\mu_B}}-1\Bigr)}\notag\\
&\qquad\qquad\quad\ \ +\frac{m_0}{1-q}\frac{q^{\fBC}-1}{\Bigl(p^{\frac{w\cdot v_0+\xi_B'(w\cdot v_1)}{\mu_B}}-1\Bigr)\Bigl(p^{\frac{\xi_C(w\cdot v_1)+w\cdot v_2}{\mu_C}}-1\Bigr)}\label{lastformulaforReenaccentdn}\\
&\qquad\qquad\qquad\qquad\qquad\quad\,+\frac{\Sigma_1'}{(\log p)F_3}+\frac{m_0}{p^{\frac{\xi_2(w\cdot v_1)+w\cdot v_3}{\mu_2}}-1}+\frac{m_0\Sigma_3}{F_2F_3}.\notag
\end{align}

\subsection{Study of $\Sigma_1'$}\label{studysigmaeenacc3}
The term $\Sigma_1'$ was defined as
\begin{align*}
\Sigma_1'&=\Sigma(\delta_1)'(s_0)\\
&=\dds{\sum_{h\in H_1}p^{\sigma(h)+m(h)s}}\\
&=(\log p)\sum_hm(h)p^{\sigma(h)+m(h)s_0}\\
&=(\log p)\sum_hm(h)p^{w\cdot h}.
\end{align*}
Writing $h=h_0v_0+h_1v_1+h_3v_3$ for $h\in H_1=H(v_0,v_1,v_3)$, we have that
\begin{align}
\frac{\Sigma_1'}{\log p}&=\sum_{h\in H_1}(h_0m_0+h_1m_1+h_3m_3)p^{w\cdot h}\notag\\
&=m_0\Sigma_1^{(0)}+m_1\Sigma_1^{(1)}+m_3\Sigma_1^{(3)},\label{formuleSigmaeenaccent}\\\shortintertext{with}
\Sigma_1^{(i)}&=\sum_{h\in H_1}h_ip^{w\cdot h};\qquad i=0,1,3.\notag
\end{align}

We will now calculate $\Sigma_1^{(1)},\Sigma_1^{(3)},$ and $\Sigma_1^{(0)}$.

\subsubsection{The sum $\Sigma_1^{(1)}$}
With the notation $q=p^{-s_0-1}$, Identity~\eqref{dpi1c3} yields
\begin{equation*}
p^{\frac{-\Psi(w\cdot v_0)+\mu_A(w\cdot v_1)+\mu_B(w\cdot v_3)}{\mu_1}}=p^{-s_0-1}=q.
\end{equation*}
Hence, based on the description \eqref{pof1casedrie} of the points of $H_1$ and proceeding as in Paragraph~\ref{parberSig1c3}, we obtain
\begin{align}
\Sigma_1^{(1)}&=\sum_{h\in H_1}h_1p^{w\cdot h}\notag\\
&=\sum_{i=0}^{\mu_A-1}\sum_{j=0}^{\mu_B-1}\sum_{k=0}^{\fAB-1}\frac{j\fAB+k}{\mu_B\fAB}\Bigl(p^{\frac{\xi_A(w\cdot v_0)+w\cdot v_3}{\mu_A}}\Bigr)^i\Bigl(p^{\frac{\xi_B(w\cdot v_0)+w\cdot v_1}{\mu_B}}\Bigr)^jq^k\notag\\
&=\frac{\Sigma_A}{\mu_B}\Biggl[\sum_jj\Bigl(p^{\frac{\xi_B(w\cdot v_0)+w\cdot v_1}{\mu_B}}\Bigr)^j\sum_kq^k+\frac{1}{\fAB}\sum_j\Bigl(p^{\frac{\xi_B(w\cdot v_0)+w\cdot v_1}{\mu_B}}\Bigr)^j\sum_kkq^k\Biggr].\label{termofS11c3}
\end{align}

For $n^{\ast}\mid n$, by Formula~\eqref{fSASCS2ifnsdnc3} for $\Sigma_A$, we then find
\begin{align}
\Sigma_1^{(1)}&=\frac{F_3}{\mu_B(q^{\fAB}-1)}\biggl(\frac{\mu_B(\mu_B-1)}{2}\frac{q^{\fAB}-1}{q-1}\notag\\
&\qquad\qquad\qquad\quad\ \ +\frac{1}{\fAB}\,\mu_B\,\frac{q^{\fAB}(\fAB q-\fAB-q)+q}{(q-1)^2}\biggr)\notag\\
&=\frac{F_3}{1-q}\left(-\frac{\mu_B-1}{2}-\frac{1}{1-q^{-\fAB}}+\frac{1}{\fAB(1-q^{-1})}\right),\label{forSig11deelt}\\\intertext{while in the complementary case, we conclude}
\Sigma_1^{(1)}&=\frac{F_3}{\mu_B\Bigl(p^{\frac{\xi_A(w\cdot v_0)+w\cdot v_3}{\mu_A}}-1\Bigr)}\frac{\mu_B}{p^{\frac{\xi_B(w\cdot v_0)+w\cdot v_1}{\mu_B}}-1}\frac{q^{\fAB}-1}{q-1}\notag\\
&=\frac{F_3}{q-1}\frac{q^{\fAB}-1}{\Bigl(p^{\frac{\xi_A(w\cdot v_0)+w\cdot v_3}{\mu_A}}-1\Bigr)\Bigl(p^{\frac{\xi_B(w\cdot v_0)+w\cdot v_1}{\mu_B}}-1\Bigr)}.\label{forSig11deeltnt}
\end{align}
Note that if $n^{\ast}\nmid n$, the second term of \eqref{termofS11c3} vanishes, as the sum over $j$ equals zero in this case.

\subsubsection{The sum $\Sigma_1^{(3)}$}
Similarly, $\Sigma_1^{(3)}$ is given by
\begin{align*}
\Sigma_1^{(3)}&=\sum_{h\in H_1}h_3p^{w\cdot h}\\
&=\sum_{i=0}^{\mu_A-1}\sum_{j=0}^{\mu_B-1}\sum_{k=0}^{\fAB-1}\frac{i\fAB+k}{\mu_A\fAB}\Bigl(p^{\frac{\xi_A(w\cdot v_0)+w\cdot v_3}{\mu_A}}\Bigr)^i\Bigl(p^{\frac{\xi_B(w\cdot v_0)+w\cdot v_1}{\mu_B}}\Bigr)^jq^k.
\end{align*}

If $n^{\ast}\nmid n$, the sum over $j$ vanishes and $\Sigma_1^{(3)}=0$. In the other case, one has
\begin{equation*}
p^{\frac{\xi_B(w\cdot v_0)+w\cdot v_1}{\mu_B}}=1\quad\ \ \text{and subsequently}\quad\ \ p^{\frac{\xi_A(w\cdot v_0)+w\cdot v_3}{\mu_A}}=p^{\fAB(-s_0-1)}=q^{\fAB},
\end{equation*}
leading to
\begin{align}
\Sigma_1^{(3)}&=\frac{\mu_B}{\mu_A\fAB}\sum_{i=0}^{\mu_A-1}\sum_{k=0}^{\fAB-1}(i\fAB+k)q^{i\fAB+k}\notag\\
&=\frac{\mu_B}{\mu_A\fAB}\sum_{l=0}^{\mu_A\fAB-1}lq^l\notag\\
&=\frac{\mu_B}{\mu_A\fAB}\frac{q^{\mu_A\fAB}(\mu_A\fAB q-\mu_A\fAB-q)+q}{(q-1)^2}\notag\\
&=\frac{\mu_B}{1-q}\left(-(F_3+1)+\frac{F_3}{\mu_A\fAB(1-q^{-1})}\right).\label{formSig13deelt}
\end{align}
Note that $q^{\mu_A\fAB}=p^{\xi_A(w\cdot v_0)+w\cdot v_3}=p^{w\cdot v_3}=F_3+1$ in this case.

Let us now look at $\Sigma_1^{(0)}$.

\subsubsection{The sum $\Sigma_1^{(0)}$}
Still based on \eqref{pof1casedrie}, this time we ought to consider the following sum:
\begin{align*}
\Sigma_1^{(0)}&=\sum_{h\in H_1}h_0p^{w\cdot h}\\*
&=\sum_{i=0}^{\mu_A-1}\sum_{j=0}^{\mu_B-1}\sum_{k=0}^{\fAB-1}\left\{\frac{i\xi_A\mu_B\fAB+j\xi_B\mu_A\fAB-k\Psi}{\mu_1}\right\}\cdot\\*
&\quad\qquad\qquad\qquad\qquad\qquad\qquad\Bigl(p^{\frac{\xi_A(w\cdot v_0)+w\cdot v_3}{\mu_A}}\Bigr)^i
\Bigl(p^{\frac{\xi_B(w\cdot v_0)+w\cdot v_1}{\mu_B}}\Bigr)^jq^k.
\end{align*}
If we put
\begin{equation}\label{defjnulc3}
j_0=\left\lfloor\frac{i\xi_A\mu_B\fAB-k\Psi}{\mu_A\fAB}\right\rfloor,
\end{equation}
we can write this sum as
\begin{align}
\Sigma_1^{(0)}&=\frac{1}{\mu_B}\sum_{i=0}^{\mu_A-1}\sum_{k=0}^{\fAB-1}\Bigl(p^{\frac{\xi_A(w\cdot v_0)+w\cdot v_3}{\mu_A}}\Bigr)^iq^kS(i,k),\qquad\text{with}\label{sumoverjref}\\
S(i,k)&=\sum_{j=0}^{\mu_B-1}\left(\left\{\frac{i\xi_A\mu_B\fAB-k\Psi}{\mu_A\fAB}\right\}+\{j_0+j\xi_B\}_{\mu_B}\right)\Bigl(p^{\frac{\xi_B(w\cdot v_0)+w\cdot v_1}{\mu_B}}\Bigr)^j.\notag
\end{align}

Since $\gcd(\xi_B,\mu_B)=1$, the map
\begin{align*}
&\verB\to\verB:j\mapsto\{j_0+j\xi_B\}_{\mu_B}\\
\intertext{is a permutation, and with $\xi_B'$ as before the unique element $\xi_B'\in\verB$ such that $\xi_B\xi_B'\equiv1\bmod\mu_B$, the inverse permutation is given by}
&\verB\to\verB:j\mapsto\{(j-j_0)\xi_B'\}_{\mu_B}.
\end{align*}
Therefore, after reordering the terms, the sum $S(i,k)$ can be written as
\begin{align}
S(i,k)&=\sum_{j=0}^{\mu_B-1}\left(\left\{\frac{i\xi_A\mu_B\fAB-k\Psi}{\mu_A\fAB}\right\}+j\right)\Bigl(p^{\frac{\xi_B(w\cdot v_0)+w\cdot v_1}{\mu_B}}\Bigr)^{\{(j-j_0)\xi_B'\}_{\mu_B}}\label{sumoverjex1}\\
&=\Bigl(p^{\frac{w\cdot v_0+\xi_B'(w\cdot v_1)}{\mu_B}}\Bigr)^{-j_0}\sum_j\left(\left\{\frac{i\xi_A\mu_B\fAB-k\Psi}{\mu_A\fAB}\right\}+j\right)\Bigl(p^{\frac{w\cdot v_0+\xi_B'(w\cdot v_1)}{\mu_B}}\Bigr)^j.\label{sumoverjex2}
\end{align}
Indeed, as $p^{(\xi_B(w\cdot v_0)+w\cdot v_1)/\mu_B}$ is a $\mu_B$th root of unity, one may omit the curly brackets $\{\cdot\}_{\mu_B}$ in the exponent in \eqref{sumoverjex1}. Expression~\eqref{sumoverjex2} is then obtained by \eqref{teentottander} and the fact that $j_0$ is independent of $j$.

It is now a good time to make a case distinction between $n^{\ast}\mid n$ and $n^{\ast}\nmid n$.

\paragraph{Case $n^{\ast}\mid n$}
Since in this case
\begin{equation*}
p^{\frac{\xi_A(w\cdot v_0)+w\cdot v_3}{\mu_A}}=p^{\fAB(-s_0-1)}=q^{\fAB}\qquad\text{and}\qquad p^{\frac{w\cdot v_0+\xi_B'(w\cdot v_1)}{\mu_B}}=1,
\end{equation*}
Equations~\eqref{sumoverjref} and \eqref{sumoverjex2} give rise to
\begin{align}
\Sigma_1^{(0)}&=\frac{1}{\mu_B}\sum_{i=0}^{\mu_A-1}\sum_{k=0}^{\fAB-1}q^{i\fAB+k}
\sum_{j=0}^{\mu_B-1}\left(\left\{\frac{i\xi_A\mu_B\fAB-k\Psi}{\mu_A\fAB}\right\}+j\right)\notag\\
&=\frac{\mu_B-1}{2}\sum_{l=0}^{\mu_A\fAB-1}q^l+\sum_i\sum_k\left\{\frac{i\xi_A\mu_B\fAB-k\Psi}{\mu_A\fAB}\right\}q^{i\fAB+k}\label{defTtwee}\\
&=-\frac{\mu_B-1}{2}\frac{F_3}{1-q}+T_2.\label{formuleSigmaeennul}
\end{align}

The second term of \eqref{defTtwee}, which we temporarily denote by $T_2$, can be further simplified as follows. Either directly from $\xi_Av_0+v_3\in\mu_A\Z^3$ and \eqref{vi1c3}, or as a corollary of \eqref{inZetidABc3}, we have that
\begin{equation*}
\frac{\xi_A\mu_B+\Psi}{\mu_A}\in\Z.
\end{equation*}
This makes that we can replace $\xi_A\mu_B$ by $-\Psi$ in the second term $T_2$ of \eqref{defTtwee}:
\begin{align*}
T_2&=\sum_{i=0}^{\mu_A-1}\sum_{k=0}^{\fAB-1}\left\{\frac{i\xi_A\mu_B\fAB-k\Psi}{\mu_A\fAB}\right\}q^{i\fAB+k}\\
&=\sum_i\sum_k\left\{\frac{i(-\Psi)\fAB-k\Psi}{\mu_A\fAB}\right\}q^{i\fAB+k}\\
&=\sum_i\sum_k\left\{\frac{-(i\fAB+k)\Psi}{\mu_A\fAB}\right\}q^{i\fAB+k}\\
&=\sum_{l=0}^{\mu_A\fAB-1}\left\{\frac{-l\Psi}{\mu_A\fAB}\right\}q^l.
\end{align*}

In Subsection~\ref{pointsofandaddrel} we showed that $\gcd(\mu_A\fAB,\Psi)=\mu_2$. Therefore, if we recall that $\Psi=\psi\mu_2$ and $\mu_A\fAB=\mu_2\fBtwee$, we can write the fraction $-\Psi/\mu_A\fAB$ in lowest terms and continue:
\begin{align*}
T_2&=\sum_{l=0}^{\mu_A\fAB-1}\left\{\frac{-l\psi}{\fBtwee}\right\}q^l\\
&=\sum_{\iota=0}^{\mu_2-1}\sum_{\kappa=0}^{\fBtwee-1}\left\{\frac{-(\iota\fBtwee+\kappa)\psi}{\fBtwee}\right\}q^{\iota\fBtwee+\kappa}\\
&=\sum_{\iota}\left(q^{\fBtwee}\right)^{\iota}\sum_{\kappa}\left\{\frac{-\kappa\psi}{\fBtwee}\right\}q^{\kappa}\\
&=\frac{F_3}{q^{\fBtwee}-1}\sum_{\kappa}\left\{\frac{-\kappa\barpsi}{\fBtwee}\right\}q^{\kappa},
\end{align*}
where $\barpsi=\{\psi\}_{\fBtwee}$ denotes the reduction of $\psi$ modulo $\fBtwee$. Obviously, we have $\barpsi\in\{0,\ldots,\fBtwee-1\}$ and still $\gcd(\barpsi,\fBtwee)=1$. Note also that $\barpsi=0$ if and only if $\fBtwee=1$, and that if $\fBtwee=1$, then $T_2=0$. In what follows, we study $T_2$ under the assumption that $\fBtwee>1$.

For any real number $a$, one has that $\{-a\}=0$ if $a\in\Z$, and $\{-a\}=1-\{a\}$ otherwise. Since \barpsi\ and \fBtwee\ are coprime, the only $\kappa\in\{0,\ldots,\fBtwee-1\}$ for which
\begin{equation*}
\frac{\kappa\barpsi}{\fBtwee}\in\Z,
\end{equation*}
is $\kappa=0$. Consequently,
\begin{equation*}
T_2=\frac{F_3}{q^{\fBtwee}-1}\left(\sum_{\kappa=0}^{\fBtwee-1}\left(1-\left\{\frac{\kappa\barpsi}{\fBtwee}\right\}\right)q^{\kappa}-1\right),
\end{equation*}
and since $\{a\}=a-\lfloor a\rfloor$ for any $a\in\R$, we obtain that $T_2$ equals
\begin{align*}
&\,\frac{F_3}{q^{\fBtwee}-1}\left(\sum_{\kappa}q^{\kappa}-\frac{\barpsi}{\fBtwee}\sum_{\kappa}\kappa q^{\kappa}+\sum_{\kappa}\left\lfloor\frac{\kappa\barpsi}{\fBtwee}\right\rfloor q^{\kappa}-1\right)\\
&=\frac{F_3}{q^{\fBtwee}-1}\left(\frac{q^{\fBtwee}-1}{q-1}-\frac{\barpsi}{\fBtwee}\frac{q^{\fBtwee}(\fBtwee(q-1)-q)+q}{(q-1)^2}+\sum_{\kappa}\left\lfloor\frac{\kappa\barpsi}{\fBtwee}\right\rfloor q^{\kappa}-1\right)\\
&=\frac{F_3}{1-q}\left(\frac{\barpsi}{1-q^{-\fBtwee}}-\frac{\barpsi}{\fBtwee(1-q^{-1})}-1\right)+\frac{F_3}{q^{\fBtwee}-1}\left(\sum_{\kappa}\left\lfloor\frac{\kappa\barpsi}{\fBtwee}\right\rfloor q^{\kappa}-1\right).
\end{align*}

As we assume that $\fBtwee>1$, we have $\barpsi\in\{1,\ldots,\fBtwee-1\}$. Hence the finite sequence
\begin{equation}\label{finseqkappa}
\left(\left\lfloor\frac{\kappa\barpsi}{\fBtwee}\right\rfloor\right)_{\kappa=0}^{\fBtwee-1}
\end{equation}
of non-negative integers ascends from $0$ to $\barpsi-1$ with steps of zero or one. If we denote
\begin{equation*}
\kappa_{\rho}=\min\left\{\kappa\in\Zplus\;\middle\vert\;\left\lfloor\frac{\kappa\barpsi}{\fBtwee}\right\rfloor=\rho\right\}=\left\lceil\frac{\rho\fBtwee}{\barpsi}\right\rceil;\qquad\rho=0,\ldots,\barpsi;
\end{equation*}
then
\begin{equation*}
0=\kappa_0<\kappa_1<\cdots<\kappa_{\barpsi-1}<\kappa_{\barpsi}=\fBtwee,
\end{equation*}
and $\kappa_1,\ldots,\kappa_{\barpsi-1}$ are the indices where a \lq jump\rq\ in the sequence \eqref{finseqkappa} takes place.

Let us express
\begin{equation*}
\sum_{\kappa}\left\lfloor\frac{\kappa\barpsi}{\fBtwee}\right\rfloor q^{\kappa}
\end{equation*}
in terms of these numbers. We have
\begin{align*}
\left(\sum_{\kappa=0}^{\fBtwee-1}\left\lfloor\frac{\kappa\barpsi}{\fBtwee}\right\rfloor q^{\kappa}\right)(1-q)&=\sum_{\kappa=1}^{\fBtwee-1}\left(\left\lfloor\frac{\kappa\barpsi}{\fBtwee}\right\rfloor-\left\lfloor\frac{(\kappa-1)\barpsi}{\fBtwee}\right\rfloor\right)q^{\kappa}-(\barpsi-1)q^{\fBtwee}\\
&=\sum_{\rho=1}^{\barpsi-1}q^{\kappa_{\rho}}-(\barpsi-1)q^{\fBtwee}\\
&=\sum_{\rho=1}^{\barpsi}q^{\kappa_{\rho}}-\barpsi q^{\fBtwee},
\end{align*}
and therefore,
\begin{equation}\label{eindformuleTtwee}
T_2=\frac{F_3}{1-q}\Biggl(\frac{1}{q^{\fBtwee}-1}\sum_{\rho=1}^{\barpsi}q^{\kappa_{\rho}}-\frac{\barpsi}{\fBtwee(1-q^{-1})}-1\Biggr)-\frac{F_3}{q^{\fBtwee}-1}.
\end{equation}
If we agree that an empty sum equals zero, then the above formula stays valid for $\fBtwee=1$.

\paragraph{Case $n^{\ast}\nmid n$}
With Equations~\eqref{sumoverjref} and \eqref{sumoverjex2} as a starting point, we now calculate $\Sigma_1^{(0)}$ in the complementary case. First of all, as $p^{(w\cdot v_0+\xi_B'(w\cdot v_1))/\mu_B}$ is now a $\mu_B$th root of unity different from one, one has that
\begin{equation*}
\sum_{j=0}^{\mu_B-1}\Bigl(p^{\frac{w\cdot v_0+\xi_B'(w\cdot v_1)}{\mu_B}}\Bigr)^j=0,
\end{equation*}
and Expression~\eqref{sumoverjex2} simplifies to
\begin{align*}
S(i,k)&=\Bigl(p^{\frac{w\cdot v_0+\xi_B'(w\cdot v_1)}{\mu_B}}\Bigr)^{-j_0}\sum_{j=0}^{\mu_B-1}j\Bigl(p^{\frac{w\cdot v_0+\xi_B'(w\cdot v_1)}{\mu_B}}\Bigr)^j\\
&=\frac{\mu_B}{p^{\frac{w\cdot v_0+\xi_B'(w\cdot v_1)}{\mu_B}}-1}\Bigl(p^{\frac{w\cdot v_0+\xi_B'(w\cdot v_1)}{\mu_B}}\Bigr)^{-j_0}.
\end{align*}

Next, let us recall from \eqref{defjnulc3} and \eqref{inZetidABc3} that
\begin{equation*}
j_0=\left\lfloor\frac{i\xi_A\mu_B\fAB-k\Psi}{\mu_A\fAB}\right\rfloor\qquad\text{and}\qquad\frac{\xi_A\mu_B+\xi_B\mu_A+\Psi}{\mu_A}\in\mu_B\Z.
\end{equation*}
We observe that
\begin{equation*}
-j_0\equiv-\left\lfloor\frac{i(-\xi_B\mu_A-\Psi)\fAB-k\Psi}{\mu_A\fAB}\right\rfloor=i\xi_B-\left\lfloor\frac{-(i\fAB+k)\Psi}{\mu_A\fAB}\right\rfloor\mod\mu_B,
\end{equation*}
which gives rise to
\begin{equation}\label{sumoverjex3}
S(i,k)=\frac{\mu_B}{p^{\frac{w\cdot v_0+\xi_B'(w\cdot v_1)}{\mu_B}}-1}\Bigl(p^{\frac{\xi_B(w\cdot v_0)+w\cdot v_1}{\mu_B}}\Bigr)^i\Bigl(p^{\frac{w\cdot v_0+\xi_B'(w\cdot v_1)}{\mu_B}}\Bigr)^{-\left\lfloor\frac{-(i\fAB+k)\Psi}{\mu_A\fAB}\right\rfloor}.
\end{equation}

Finally, in Paragraph~\ref{sssSASCS2c3} we obtained
\begin{equation*}
p^{\frac{\xi_A(w\cdot v_0)+w\cdot v_3}{\mu_A}}p^{\frac{\xi_B(w\cdot v_0)+w\cdot v_1}{\mu_B}}=p^{\fAB(-s_0-1)}=q^{\fAB};
\end{equation*}
using this identity, Formulas~\eqref{sumoverjref} and \eqref{sumoverjex3} for $\Sigma_1^{(0)}$ and $S(i,k)$ eventually yield
\begin{equation*}
\Sigma_1^{(0)}=\frac{1}{p^{\frac{w\cdot v_0+\xi_B'(w\cdot v_1)}{\mu_B}}-1}\sum_{i=0}^{\mu_A-1}\sum_{k=0}^{\fAB-1}q^{i\fAB+k}\Bigl(p^{\frac{w\cdot v_0+\xi_B'(w\cdot v_1)}{\mu_B}}\Bigr)^{-\left\lfloor\frac{-(i\fAB+k)\Psi}{\mu_A\fAB}\right\rfloor}.
\end{equation*}

Proceeding as in Case $n^{\ast}\mid n$, we write the double sum $DS$ in the expression above as
\begin{align*}
DS&=\sum_{l=0}^{\mu_A\fAB-1}q^l\Bigl(p^{\frac{w\cdot v_0+\xi_B'(w\cdot v_1)}{\mu_B}}\Bigr)^{-\left\lfloor\frac{-l\psi}{\fBtwee}\right\rfloor}\\
&=\sum_{\iota=0}^{\mu_2-1}\sum_{\kappa=0}^{\fBtwee-1}q^{\iota\fBtwee+\kappa}\Bigl(p^{\frac{w\cdot v_0+\xi_B'(w\cdot v_1)}{\mu_B}}\Bigr)^{-\left\lfloor\frac{-(\iota\fBtwee+\kappa)\psi}{\fBtwee}\right\rfloor}\\
&=\sum_{\iota}\left[q^{\fBtwee}\Bigl(p^{\frac{w\cdot v_0+\xi_B'(w\cdot v_1)}{\mu_B}}\Bigr)^{\psi}\right]^{\iota}\sum_{\kappa}q^{\kappa}\Bigl(p^{\frac{w\cdot v_0+\xi_B'(w\cdot v_1)}{\mu_B}}\Bigr)^{-\left\lfloor\frac{-\kappa\psi}{\fBtwee}\right\rfloor}.\\
\intertext{By \eqref{sig2sigbidcasedrie} and the fact that $\kappa\psi/\fBtwee\notin\Z$ for $\kappa\in\{1,\ldots,\fBtwee-1\}$,\footnotemark\ we then have}
DS&=\sum_{\iota}\Bigl(p^{\frac{\xi_2(w\cdot v_1)+w\cdot v_3}{\mu_2}}\Bigr)^{\iota}
\Biggl[\sum_{\kappa}q^{\kappa}\Bigl(p^{\frac{w\cdot v_0+\xi_B'(w\cdot v_1)}{\mu_B}}\Bigr)^{\left\lfloor\frac{\kappa\psi}{\fBtwee}\right\rfloor+1}-p^{\frac{w\cdot v_0+\xi_B'(w\cdot v_1)}{\mu_B}}+1\Biggr],
\end{align*}
\footnotetext{Recall that $\psi$ and $\fBtwee$ are coprime.}
and hence we conclude
\begin{align}
\Sigma_1^{(0)}&=\frac{p^{\frac{w\cdot v_0+\xi_B'(w\cdot v_1)}{\mu_B}}\,F_3}{\Bigl(p^{\frac{w\cdot v_0+\xi_B'(w\cdot v_1)}{\mu_B}}-1\Bigr)\Bigl(p^{\frac{\xi_2(w\cdot v_1)+w\cdot v_3}{\mu_2}}-1\Bigr)}\sum_{\kappa=0}^{\fBtwee-1}q^{\kappa}\Bigl(p^{\frac{w\cdot v_0+\xi_B'(w\cdot v_1)}{\mu_B}}\Bigr)^{\left\lfloor\frac{\kappa\psi}{\fBtwee}\right\rfloor}\notag\\
&\quad\,-\frac{F_3}{p^{\frac{\xi_2(w\cdot v_1)+w\cdot v_3}{\mu_2}}-1}.\label{formSig10dnc3}
\end{align}

\subsubsection{A formula for $\Sigma_1'$}
Bringing together Equations (\ref{formuleSigmaeenaccent}, \ref{forSig11deelt}, \ref{formSig13deelt}, \ref{formuleSigmaeennul}, and \ref{eindformuleTtwee}) for $\Sigma_1'$, $\Sigma_1^{(0)}$, $\Sigma_1^{(1)}$, $\Sigma_1^{(3)}$, and $T_2$, we find the following formula in case that $n^{\ast}\mid n$:
\begin{align}
\frac{\Sigma_1'}{(\log p)F_3}&=\frac{m_0\Sigma_1^{(0)}+m_1\Sigma_1^{(1)}+m_3\Sigma_1^{(3)}}{F_3}\notag\\
&=\frac{m_0}{1-q}\Biggl(-\frac{\mu_B-1}{2}+\frac{1}{q^{\fBtwee}-1}\sum_{\rho=1}^{\barpsi}q^{\kappa_{\rho}}-\frac{\barpsi}{\fBtwee(1-q^{-1})}-1\Biggr)\notag\\
&\quad-\frac{m_0}{q^{\fBtwee}-1}+\frac{m_1}{1-q}\left(-\frac{\mu_B-1}{2}-\frac{1}{1-q^{-\fAB}}+\frac{1}{\fAB(1-q^{-1})}\right)\notag\\
&\quad+\frac{m_3\mu_B}{1-q}\left(-\frac{F_3+1}{F_3}+\frac{1}{\mu_A\fAB(1-q^{-1})}\right)\notag\\
&=\frac{1}{1-q}\Biggl(\frac{m_0}{q^{\fBtwee}-1}\sum_{\rho=1}^{\barpsi}q^{\kappa_{\rho}}-\frac{(m_0+m_1)(\mu_B-1)}{2}-\frac{m_1}{1-q^{-\fAB}}\label{eindformulesigmaeenaccent}\\*
&\quad\,+\frac{m_1\mu_A+m_3\mu_B-m_0\barpsi\mu_2}{\mu_A\fAB(1-q^{-1})}-\frac{m_3\mu_B(F_3+1)}{F_3}-m_0\Biggr)-\frac{m_0}{q^{\fBtwee}-1};\notag
\end{align}
if on the contrary $n^{\ast}\nmid n$, then Equations (\ref{formuleSigmaeenaccent}, \ref{forSig11deeltnt}, and \ref{formSig10dnc3}) yield
\begin{gather}
\frac{\Sigma_1'}{(\log p)F_3}=\frac{m_0\Sigma_1^{(0)}+m_1\Sigma_1^{(1)}+m_3\Sigma_1^{(3)}}{F_3}\notag\\
\begin{aligned}
&=\frac{m_0p^{\frac{w\cdot v_0+\xi_B'(w\cdot v_1)}{\mu_B}}}{\Bigl(p^{\frac{w\cdot v_0+\xi_B'(w\cdot v_1)}{\mu_B}}-1\Bigr)\Bigl(p^{\frac{\xi_2(w\cdot v_1)+w\cdot v_3}{\mu_2}}-1\Bigr)}\sum_{\kappa=0}^{\fBtwee-1}q^{\kappa}\Bigl(p^{\frac{w\cdot v_0+\xi_B'(w\cdot v_1)}{\mu_B}}\Bigr)^{\left\lfloor\frac{\kappa\psi}{\fBtwee}\right\rfloor}\\
&\quad\,-\frac{m_0}{p^{\frac{\xi_2(w\cdot v_1)+w\cdot v_3}{\mu_2}}-1}+\frac{m_1}{q-1}\frac{q^{\fAB}-1}{\Bigl(p^{\frac{\xi_A(w\cdot v_0)+w\cdot v_3}{\mu_A}}-1\Bigr)\Bigl(p^{\frac{\xi_B(w\cdot v_0)+w\cdot v_1}{\mu_B}}-1\Bigr)}.
\end{aligned}\label{eindformulesigmaeenaccentdn}
\end{gather}

\subsection{An easier formula for the residue $R_1$}
\subsubsection{Case $n^{\ast}\mid n$}
If we fill in Formula~\eqref{eindformulesigmaeenaccent} in Equation~\eqref{lastformulaforReenaccent} for $R_1'$, the latter simplifies to
\begin{equation*}
\frac{1}{1-q}\Biggl(\frac{m_0}{q^{\fBtwee}-1}\sum_{\rho=1}^{\barpsi}q^{\kappa_{\rho}}+\frac{m_0}{q^{\fBC}-1}+\frac{m_1\mu_A+m_3\mu_B+\mu_1-m_0\barpsi\mu_2}{\mu_A\fAB(1-q^{-1})}\Biggr)+\frac{m_0\Sigma_3}{F_2F_3}.
\end{equation*}

There is a very convenient interpretation of $m_1\mu_A+m_3\mu_B+\mu_1$ appearing in the equation above. Recall from \eqref{vi1c3} that
\begin{equation*}
\Psi v_0-\mu_A v_1-\mu_B v_3=(0,0,\mu_1).
\end{equation*}
Making the dot product with $D=(x_D,y_D,1)$ on both sides yields
\begin{equation*}
m_0\Psi-m_1\mu_A-m_3\mu_B=\Psi(D\cdot v_0)-\mu_A(D\cdot v_1)-\mu_B(D\cdot v_3)=D\cdot(0,0,\mu_1)=\mu_1,
\end{equation*}
and hence
\begin{equation*}
m_1\mu_A+m_3\mu_B+\mu_1=m_0\Psi=m_0\psi\mu_2.
\end{equation*}
It follows that
\begin{equation*}
\frac{m_1\mu_A+m_3\mu_B+\mu_1-m_0\barpsi\mu_2}{\mu_A\fAB(1-q^{-1})}=\frac{m_0(\psi-\barpsi)}{\fBtwee(1-q^{-1})}=\frac{m_0t}{1-q^{-1}},
\end{equation*}
with
\begin{equation*}
t=\frac{\psi-\barpsi}{\fBtwee}=\frac{\psi-\{\psi\}_{\fBtwee}}{\fBtwee}=\left\lfloor\frac{\psi}{\fBtwee}\right\rfloor
\end{equation*}
the quotient of Euclidean division of $\psi$ by \fBtwee. Note that if $\fBtwee=1$, then $t=\psi$.

If we now put $R_1''=R_1'/m_0$, it remains to prove that
\begin{equation}\label{laatsteformulevrReen}
R_1''=\frac{\Sigma_3}{F_2F_3}+\frac{1}{1-q}\Biggl(\frac{1}{q^{\fBtwee}-1}\sum_{\rho=1}^{\barpsi}q^{\kappa_{\rho}}+\frac{1}{q^{\fBC}-1}+\frac{t}{1-q^{-1}}\Biggr)
\end{equation}
vanishes.

\subsubsection{Case $n^{\ast}\nmid n$}
In this case, according to \eqref{lastformulaforReenaccentdn} and \eqref{eindformulesigmaeenaccentdn}, it now comes to proving that
\begin{multline}\label{laatsteformulevrReendn}
R_1''=\frac{R_1'}{m_0}=\frac{\Sigma_3}{F_2F_3}-\frac{q^{\fBC}-1}{(q-1)\Bigl(p^{\frac{w\cdot v_0+\xi_B'(w\cdot v_1)}{\mu_B}}-1\Bigr)\Bigl(p^{\frac{\xi_C(w\cdot v_1)+w\cdot v_2}{\mu_C}}-1\Bigr)}\\
+\frac{p^{\frac{w\cdot v_0+\xi_B'(w\cdot v_1)}{\mu_B}}}{\Bigl(p^{\frac{w\cdot v_0+\xi_B'(w\cdot v_1)}{\mu_B}}-1\Bigr)\Bigl(p^{\frac{\xi_2(w\cdot v_1)+w\cdot v_3}{\mu_2}}-1\Bigr)}\sum_{\kappa=0}^{\fBtwee-1}q^{\kappa}\Bigl(p^{\frac{w\cdot v_0+\xi_B'(w\cdot v_1)}{\mu_B}}\Bigr)^{\left\lfloor\frac{\kappa\psi}{\fBtwee}\right\rfloor}
\end{multline}
equals zero.

\subsection{Investigation of $\Sigma_3$}\label{multmudrieformule}
First we try to find a useful formula for $\mu_3$.

\subsubsection{Multiplicity $\mu_3$ of $\delta_3$}
From our study in Section~\ref{fundpar}, we remember that
\begin{equation*}
\mu_C\mu_2=\#H(v_1,v_2)\#H(v_1,v_3)\mid\mu_3=\#H(v_1,v_2,v_3).
\end{equation*}
We look for more information on the factor $\fCtwee=\mu_3/\mu_C\mu_2$. Let us proceed in the same way as when interpreting $\mu_C/\mu_A\mu_B=\fAB$ in Case~I (cfr.\ Subsection~\ref{formulaformuCcaseeen}).

One has
\begin{align*}
\mu_3&=
\begin{vmatrix}
a_1&b_1&c_1\\a_2&b_2&c_2\\a_3&b_3&c_3
\end{vmatrix}\\
&=a_3
\begin{vmatrix}
b_1&c_1\\b_2&c_2
\end{vmatrix}
-b_3
\begin{vmatrix}
a_1&c_1\\a_2&c_2
\end{vmatrix}
+c_3
\begin{vmatrix}
a_1&b_1\\a_2&b_2
\end{vmatrix}\\
&=
\begin{vmatrix}
a_1&b_1\\a_2&b_2
\end{vmatrix}
(a_3\alpha_C+b_3\beta_C+c_3)\\
&=-\mu_C\left(v_3\cdot\overrightarrow{CD}\right),\\\intertext{and since $v_3$ and $\overrightarrow{AD}$ are perpendicular, we can continue:}
\mu_3&=\mu_C\left(v_3\cdot\overrightarrow{AB}+v_3\cdot\overrightarrow{BC}\right).
\end{align*}
The fact that $\overrightarrow{AB}\perp v_0$ and $\overrightarrow{BC}\perp v_1$ implies that
\begin{equation*}
\overrightarrow{AB}=\fAB(-b_0,a_0,0)\qquad\text{and}\qquad\overrightarrow{BC}=\fBC(-b_1,a_1,0).
\end{equation*}
Hence
\begin{align*}
\mu_3&=\mu_C\bigl(\fAB(a_0b_3-a_3b_0)+\fBC(a_1b_3-a_3b_1)\bigr)\\
&=\mu_C(\mu_A\fAB+\fBC\Psi)\\
&=\mu_C\mu_2\fCtwee,
\end{align*}
with
\begin{equation}\label{deffCtwee}
\fCtwee=\fBtwee+\fBC\psi.
\end{equation}
Note that \eqref{deffCtwee} and the coprimality of $\psi$ and \fBtwee\ imply $\psi\in\{0,\ldots,\fCtwee-1\}$ and $\gcd(\psi,\fCtwee)=1$.

Next, we try to list all the $\mu_3=\mu_C\mu_2\fCtwee$ points of $H_3$.

\subsubsection{Points of $H_3=H(v_1,v_2,v_3)$}\label{ssspoH3c3}
We proceed in the same way as in Case~I for the points of $H_C$. As we know, the points of $H_C=H(v_1,v_2)$ and $H_2=H(v_1,v_3)$ can be presented as
\begin{alignat*}{3}
h(i,0,0)&=\left\{\frac{i\xi_C}{\mu_C}\right\}v_1&&+\frac{i}{\mu_C}v_2;&\qquad&i=0,\ldots,\mu_C-1;\\\intertext{and}
h(0,j,0)&=\left\{\frac{j\xi_2}{\mu_2}\right\}v_1&&+\frac{j}{\mu_2}v_3;&&j=0,\ldots,\mu_2-1;
\end{alignat*}
respectively.

To generate a complete list of points of $H_3$, it is now sufficient to find a set of representatives for the \fCtwee\ cosets of the subgroup $H_C+H_2$ of $H_3$. Recall that the cosets of $H_C+H_2$ can be described as\footnote{Here $h_3$ denotes the $v_3$-coordinate of $h$. We can as well, and completely similarly, describe these cosets in terms of the $v_2$-coordinate, but the choice for $h_3$ is more convenient in this case.}
\begin{equation*}
\mathcal{C}_k=\left\{h\in H_3\;\middle\vert\;\{h_3\}_{\frac{1}{\mu_2}}=\frac{k\mu_C}{\mu_3}=\frac{k}{\mu_2\fCtwee}\right\};\qquad k=0,\ldots,\fCtwee-1.
\end{equation*}
We will follow the approach of Section~\ref{fundpar} and select for each coset $\mathcal{C}_k$, as a representative, the unique element $h(0,0,k)\in\mathcal{C}_k$ with $v_3$-coordinate $h_3(0,0,k)=k/\mu_2\fCtwee$ and $v_2$-coordinate $h_2(0,0,k)<1/\mu_C$. We find $h(0,0,1)$ as follows.

Recall from \eqref{vi3c3} that
\begin{equation*}
\Theta v_1-\Psi v_2-\mu_C v_3=(0,0,\mu_3),
\end{equation*}
with $\Theta=\begin{vsmallmatrix}a_2&b_2\\a_3&b_3\end{vsmallmatrix}>0$. It follows that
\begin{equation*}
h(0,0,1)=\left\{\frac{-\Theta}{\mu_3}\right\}v_1+\frac{\psi}{\mu_C\fCtwee}v_2+\frac{1}{\mu_2\fCtwee}v_3=\{(0,0,-1)\}\in\mathcal{C}_1
\end{equation*}
is the representative for $\mathcal{C}_1$ we are looking for. Indeed, it follows from Equation~\eqref{deffCtwee} that $h_2(0,0,1)=\psi/\mu_C\fCtwee$ is not only reduced modulo $1$, it is also already reduced modulo $1/\mu_C$.

It is now natural to find all representatives $h(0,0,k)$ by considering the \fCtwee\ multiples $\{kh(0,0,1)\}$; $k=0,\ldots,\fCtwee-1$; of $h(0,0,1)$ in the group $H_3$, and adding to each multiple $\{kh(0,0,1)\}$ the unique element of $H_C$ such that the $v_2$-coordinate $h_2(0,0,k)$ of the sum $h(0,0,k)$ is reduced modulo $1/\mu_C$:
\begin{multline*}
h(0,0,k)=\left\{\frac{-k\Theta-\lfloor k\psi/\fCtwee\rfloor\xi_C\mu_2\fCtwee}{\mu_3}\right\}v_1+\frac{\{k\psi\}_{\fCtwee}}{\mu_C\fCtwee}v_2+\frac{k}{\mu_2\fCtwee}v_3\in\mathcal{C}_k;\\k=0,\ldots,\fCtwee-1.
\end{multline*}
Note that since $\psi$ and \fCtwee\ are coprime, $\{k\psi\}_{\fCtwee}$ runs, as expected, through the numbers $0,\ldots,\fCtwee-1$ when $k$ does so.

All this leads to the following member list of $H_3$:
\begin{multline*}
\begin{aligned}
h(i,j,k)&=\{h(i,0,0)+h(0,j,0)+h(0,0,k)\}\\
&=\left\{\frac{(i-\lfloor k\psi/\fCtwee\rfloor)\xi_C\mu_2\fCtwee+j\xi_2\mu_C\fCtwee-k\Theta}{\mu_3}\right\}v_1\\
&\qquad\qquad\qquad\qquad\qquad\qquad\,+\frac{i\fCtwee+\{k\psi\}_{\fCtwee}}{\mu_C\fCtwee}v_2+\frac{j\fCtwee+k}{\mu_2\fCtwee}v_3;
\end{aligned}\\
i=0,\ldots,\mu_C-1;\quad j=0,\ldots,\mu_2-1;\quad k=0,\ldots,\fCtwee-1.
\end{multline*}

Finally, we will try to calculate $\Sigma_3$ based on the above description of $H_3$'s points.

\subsubsection{Calculation of $\Sigma_3$}\label{ssscalcsigmadriecasedrie}
Writing $h$ as $h=h_1v_1+h_2v_2+h_3v_3$ for $h\in H_3=H(v_1,v_2,v_3)$ and noting that $p^{w\cdot v_1}=1$, we find
\begin{align}
\Sigma_3&=\sum_{h\in H_3}p^{w\cdot h}\notag\\
&=\sum_{i=0}^{\mu_C-1}\sum_{j=0}^{\mu_2-1}\sum_{k=0}^{\fCtwee-1}\notag\\*
&\qquad\ p^{\frac{(i-\lfloor k\psi/\fCtwee\rfloor)\xi_C\mu_2\fCtwee+j\xi_2\mu_C\fCtwee-k\Theta}{\mu_3}(w\cdot v_1)+\frac{i\fCtwee+\{k\psi\}_{\fCtwee}}{\mu_C\fCtwee}(w\cdot v_2)+\frac{j\fCtwee+k}{\mu_2\fCtwee}(w\cdot v_3)}\notag\\
&=\sum_i\Bigl(p^{\frac{\xi_C(w\cdot v_1)+w\cdot v_2}{\mu_C}}\Bigr)^i\sum_j\Bigl(p^{\frac{\xi_2(w\cdot v_1)+w\cdot v_3}{\mu_2}}\Bigr)^j\notag\\*
&\qquad\qquad\qquad\qquad\;\;\,\sum_k\Bigl(p^{\frac{-\Theta(w\cdot v_1)+\Psi(w\cdot v_2)+\mu_C(w\cdot v_3)}{\mu_3}}\Bigr)^k\Bigl(p^{\frac{\xi_C(w\cdot v_1)+w\cdot v_2}{\mu_C}}\Bigr)^{-\left\lfloor\frac{k\psi}{\fCtwee}\right\rfloor}\notag\\
&=\frac{F_2F_3}{\Bigl(p^{\frac{\xi_C(w\cdot v_1)+w\cdot v_2}{\mu_C}}-1\Bigr)\Bigl(p^{\frac{\xi_2(w\cdot v_1)+w\cdot v_3}{\mu_2}}-1\Bigr)}\sum_{k=0}^{\fCtwee-1}q^k\Bigl(p^{\frac{\xi_C(w\cdot v_1)+w\cdot v_2}{\mu_C}}\Bigr)^{-\left\lfloor\frac{k\psi}{\fCtwee}\right\rfloor},\label{formuleSigmadriedn}
\end{align}
where in the last step we used Identity~\eqref{dpi3c3}.

In the special case that $n^{\ast}\mid n$, based on (\ref{sigcsigbidcasedrie}--\ref{fSASCS2ifnsdnc3}), we obtain the slightly simpler formula
\begin{equation}\label{formuleSigmadrie}
\Sigma_3=\frac{F_2F_3}{(q^{\fBC}-1)(q^{\fBtwee}-1)}\sum_{k=0}^{\fCtwee-1}q^{k-\fBC\left\lfloor\frac{k\psi}{\fCtwee}\right\rfloor}.
\end{equation}

\subsection{Proof that the residue $R_1$ equals zero}\label{ssfinalsscasedrie}
\subsubsection{Case $n^{\ast}\mid n$}
According to Formula~\eqref{laatsteformulevrReen} for $R_1''$ and Formula~\eqref{formuleSigmadrie} for $\Sigma_3$, it now suffices to prove that
\begin{equation}\label{finalcheckcas3}
\sum_{k=0}^{\fCtwee-1}q^{k-\fBC\left\lfloor\frac{k\psi}{\fCtwee}\right\rfloor}=\frac{q^{\fBC}-1}{q-1}\sum_{\rho=1}^{\barpsi}q^{\kappa_{\rho}}+\frac{q^{\fBtwee}-1}{q-1}\left(tq\frac{q^{\fBC}-1}{q-1}+1\right).
\end{equation}
Let us do this now.

First of all, if $\fBtwee=1$, then by \eqref{deffCtwee} we have $\fCtwee=\fBC\psi+1$, and
\begin{align*}
\sum_{k=0}^{\fCtwee-1}q^{k-\fBC\left\lfloor\frac{k\psi}{\fCtwee}\right\rfloor}&=1+\sum_{k=1}^{\psi\fBC}q^{k-\fBC\left\lfloor\frac{k\psi}{\fCtwee}\right\rfloor}\\
&=1+\sum_{r=0}^{\psi-1}\sum_{l=1}^{\fBC}q^{(r\fBC+l)-\fBC\left\lfloor\frac{(r\fBC+l)\psi}{\fBC\psi+1}\right\rfloor}\\
&=1+\sum_r\sum_lq^l\\
&=\psi q\frac{q^{\fBC}-1}{q-1}+1,
\end{align*}
which agrees with \eqref{finalcheckcas3} for $\fBtwee=1$.

In what follows, we shall assume that $\fBtwee>1$ and thus that $\barpsi>0$. Since $\psi\in\{1,\ldots,\fCtwee-1\}$, the finite sequence
\begin{equation}\label{finseqkkk}
\left(\left\lfloor\frac{k\psi}{\fCtwee}\right\rfloor\right)_{k=0}^{\fCtwee-1}
\end{equation}
of non-negative integers ascends from $0$ to $\psi-1$ with steps of zero or one. Let us denote
\begin{equation*}
k_r=\min\left\{k\in\Zplus\;\middle\vert\;\left\lfloor\frac{k\psi}{\fCtwee}\right\rfloor=r\right\}=\left\lceil\frac{r\fCtwee}{\psi}\right\rceil;\qquad r=0,\ldots,\psi.
\end{equation*}
Then
\begin{equation*}
0=k_0<k_1<\cdots<k_{\psi-1}<k_{\psi}=\fCtwee,
\end{equation*}
and obviously,
\begin{equation*}
\sum_{k=0}^{\fCtwee-1}q^{k-\fBC\left\lfloor\frac{k\psi}{\fCtwee}\right\rfloor}=\sum_{r=0}^{\psi-1}\sum_{k=k_r}^{k_{r+1}-1}q^{k-\fBC r}.
\end{equation*}

We recall that
\begin{equation}\label{efferecallen}
\fCtwee=\fBtwee+\fBC\psi\qquad\text{and}\qquad\psi=t\fBtwee+\barpsi,
\end{equation}
with $t\in\Zplus$ and $\barpsi=\{\psi\}_{\fBtwee}\in\{1,\ldots,\fBtwee-1\}$. Remember also that for $\rho\in\{0,\ldots,\barpsi\}$, the number $\kappa_{\rho}$ denotes the smallest integer satisfying $\kappa_{\rho}\barpsi\geqslant\rho\fBtwee$.

Let us first verify \eqref{finalcheckcas3} for $t=0$. In this case we have that $\psi=\barpsi$, and hence
\begin{equation*}
k_{\rho}=\left\lceil\frac{\rho\fCtwee}{\barpsi}\right\rceil=\left\lceil\frac{\rho(\fBtwee+\fBC\barpsi)}{\barpsi}\right\rceil=\fBC\rho+\left\lceil\frac{\rho\fBtwee}{\barpsi}\right\rceil=\fBC\rho+\kappa_{\rho},
\end{equation*}
for all $\rho\in\{0,\ldots,\barpsi\}$. It follows that
\begin{align*}
\sum_{k=0}^{\fCtwee-1}q^{k-\fBC\left\lfloor\frac{k\psi}{\fCtwee}\right\rfloor}&=\sum_{\rho=0}^{\barpsi-1}\sum_{k=k_{\rho}}^{k_{\rho+1}-1}q^{k-\fBC\rho}\\
&=\sum_{\rho=0}^{\barpsi-1}\sum_{k=\fBC\rho+\kappa_{\rho}}^{\fBC(\rho+1)+\kappa_{\rho+1}-1}q^{k-\fBC\rho}\\
&=\sum_{\rho=0}^{\barpsi-1}\sum_{\kappa=\kappa_{\rho}}^{\fBC+\kappa_{\rho+1}-1}q^{\kappa}\\
&=\sum_{\rho=0}^{\barpsi-1}\Biggl(\sum_{\kappa=\kappa_{\rho}}^{\kappa_{\rho+1}-1}q^{\kappa}+q^{\kappa_{\rho+1}}\sum_{l=0}^{\fBC-1}q^l\Biggr)\\
&=\sum_{\kappa=0}^{\fBtwee-1}q^{\kappa}+\frac{q^{\fBC}-1}{q-1}\sum_{\rho=0}^{\barpsi-1}q^{\kappa_{\rho+1}}\\
&=\frac{q^{\fBC}-1}{q-1}\sum_{\rho=1}^{\barpsi}q^{\kappa_{\rho}}+\frac{q^{\fBtwee}-1}{q-1},
\end{align*}
which agrees with \eqref{finalcheckcas3} for $t=0$.

Let us from now on assume that $t>0$. In the lemma below we express $k_r$ and $k_{r+1}-1$ explicitly as a function of $r$ after writing $r$ in a special form, but first we introduce the following notation.

\begin{notation}[Iverson's convention]\label{iverson} Cfr.\ \cite{knuth92}.
For any proposition $P$ we shall denote by
\begin{equation*}
[P]=
\begin{cases}
1,&\text{if $P$ is true;}\\
0,&\text{otherwise;}
\end{cases}
\end{equation*}
the truth value of $P$.
\end{notation}

\begin{lemma}\label{lemmaspecialevorm}
Assume that $t>0$. Then the map
\begin{equation}\label{mapspecialevorm}
\begin{multlined}[.89\textwidth]
\bigl\{(\rho,\kappa,\lambda)\in\Zplus^3\mid\\
0\leqslant\rho\leqslant\barpsi-1,\ \,\kappa_{\rho}\leqslant\kappa\leqslant\kappa_{\rho+1}-1,\ \,0\leqslant\lambda\leqslant t-[\kappa<\kappa_{\rho+1}-1]\bigr\}\\
\to\{0,\ldots,\psi-1\}:(\rho,\kappa,\lambda)\mapsto r=\kappa t+\rho+\lambda
\end{multlined}
\end{equation}
is bijective, and for $r=\kappa t+\rho+\lambda\in\{0,\ldots,\psi-1\}$ written in this way, we have that
\begin{align*}
k_r&=\fBC r+\kappa+[\lambda>0]\qquad\text{and}\\
k_{r+1}-1&=\fBC(r+1)+\kappa.
\end{align*}
\end{lemma}

We will prove this lemma shortly. If we accept it now, we obtain\footnote{Note that $r=\kappa t+\rho+\lambda$ in the second line.}
\begin{align*}
\sum_{k=0}^{\fCtwee-1}q^{k-\fBC\left\lfloor\frac{k\psi}{\fCtwee}\right\rfloor}&=\sum_{r=0}^{\psi-1}\sum_{k=k_r}^{k_{r+1}-1}q^{k-\fBC r}\\
&=\sum_{\rho=0}^{\barpsi-1}\sum_{\kappa=\kappa_{\rho}}^{\kappa_{\rho+1}-1}\sum_{\lambda=0}^{t-[\kappa<\kappa_{\rho+1}-1]}\sum_{k=\fBC r+\kappa+[\lambda>0]}^{\fBC(r+1)+\kappa}q^{k-\fBC r}\\
&=\sum_{\rho=0}^{\barpsi-1}\sum_{\kappa=\kappa_{\rho}}^{\kappa_{\rho+1}-1}q^{\kappa}\sum_{\lambda=0}^{t-[\kappa<\kappa_{\rho+1}-1]}\sum_{l=[\lambda>0]}^{\fBC}q^l\\
&=\sum_{\rho=0}^{\barpsi-1}\sum_{\kappa=\kappa_{\rho}}^{\kappa_{\rho+1}-1}q^{\kappa}\sum_{\lambda=0}^{t-[\kappa<\kappa_{\rho+1}-1]}\left(q\frac{q^{\fBC}-1}{q-1}+[\lambda=0]\right)\\
&=\sum_{\rho=0}^{\barpsi-1}\sum_{\kappa=\kappa_{\rho}}^{\kappa_{\rho+1}-1}q^{\kappa}\left(\bigl(t+[\kappa=\kappa_{\rho+1}-1]\bigr)q\frac{q^{\fBC}-1}{q-1}+1\right)\\
&=\sum_{\rho=0}^{\barpsi-1}\Biggl[\left(tq\frac{q^{\fBC}-1}{q-1}+1\right)\sum_{\kappa=\kappa_{\rho}}^{\kappa_{\rho+1}-1}q^{\kappa}+q^{\kappa_{\rho+1}}\frac{q^{\fBC}-1}{q-1}\Biggr]\\
&=\left(tq\frac{q^{\fBC}-1}{q-1}+1\right)\sum_{\kappa=0}^{\fBtwee-1}q^{\kappa}+\frac{q^{\fBC}-1}{q-1}\sum_{\rho=0}^{\barpsi-1}q^{\kappa_{\rho+1}}\\
&=\frac{q^{\fBC}-1}{q-1}\sum_{\rho=1}^{\barpsi}q^{\kappa_{\rho}}+\frac{q^{\fBtwee}-1}{q-1}\left(tq\frac{q^{\fBC}-1}{q-1}+1\right),
\end{align*}
which agrees with \eqref{finalcheckcas3}.

We conclude the proof of $R_1=0$ in Case $n^{\ast}\mid n$ by verifying Lemma~\ref{lemmaspecialevorm}. Since for all $(\rho,\kappa,\lambda)$ in the domain, we have that
\begin{equation*}
0\leqslant\kappa t+\rho+\lambda\leqslant(\kappa_{\barpsi}-1)t+(\barpsi-1)+t=t\fBtwee+\barpsi-1=\psi-1,
\end{equation*}
the map \eqref{mapspecialevorm} is well-defined.

We check that the map is onto. Let $r\in\{0,\ldots,\psi-1\}$. Because the finite sequence
\begin{equation*}
\bigl(\kappa_{\rho}t+\rho\bigr)_{\rho=0}^{\barpsi}
\end{equation*}
of non-negative integers strictly ascends from $\kappa_0t+0=0$ to $\kappa_{\barpsi}t+\barpsi=\psi$, there exists a (unique) $\rho\in\{0,\ldots,\barpsi-1\}$ such that
\begin{equation*}
\kappa_{\rho}t+\rho\leqslant r<\kappa_{\rho+1}t+(\rho+1).
\end{equation*}
If $r=\kappa_{\rho+1}t+\rho$, we can write $r$ as $r=(\kappa_{\rho+1}-1)t+\rho+t$, and $r$ is the image of $(\rho,\kappa_{\rho+1}-1,t)$ under the map \eqref{mapspecialevorm}. Otherwise we have that
\begin{equation*}
\kappa_{\rho}t\leqslant r-\rho<\kappa_{\rho+1}t,
\end{equation*}
and we can write $r-\rho$ (in a unique way) as $r-\rho=\kappa t+\lambda$ with $\kappa,\lambda\in\Z;$ $\kappa_{\rho}\leqslant\kappa\leqslant\kappa_{\rho+1}-1$; and $0\leqslant\lambda\leqslant t-1$. In this case $r=\kappa t+\rho+\lambda$ is the image of $(\rho,\kappa,\lambda)$ under the map \eqref{mapspecialevorm}. This proves surjectivity.

The uniqueness of the representation $r=\kappa t+\rho+\lambda$ can either be checked directly, or by verifying that the cardinality of the domain,
\begin{align*}
\sum_{\rho=0}^{\barpsi-1}\sum_{\kappa=\kappa_{\rho}}^{\kappa_{\rho+1}-1}\sum_{\lambda=0}^{ t-[\kappa<\kappa_{\rho+1}-1]}1
&=\sum_{\rho=0}^{\barpsi-1}\sum_{\kappa=\kappa_{\rho}}^{\kappa_{\rho+1}-1}\bigl(t+[\kappa=\kappa_{\rho+1}-1]\bigr)\\
&=\sum_{\kappa=0}^{\fBtwee-1}t+\sum_{\rho=0}^{\barpsi-1}1\\
&=t\fBtwee+\barpsi\\
&=\psi,
\end{align*}
indeed equals the cardinality of the codomain $\{0,\ldots,\psi-1\}$.

Let $r=\kappa t+\rho+\lambda\in\{0,\ldots,\psi-1\}$, written in the appropriate way. We prove the expression for $k_r$ stated in the lemma. On the one hand, because $\kappa\geqslant\kappa_{\rho}$, it holds that $\kappa\barpsi\geqslant\kappa_{\rho}\barpsi\geqslant\rho\fBtwee$, and since $\lambda\leqslant t$, we have
\begin{equation*}
(\rho+\lambda)\fBtwee\leqslant\kappa\barpsi+[\lambda>0](t\fBtwee+\barpsi).
\end{equation*}
On the other hand, since we assume $t>0$, it follows from $\kappa\leqslant\kappa_{\rho+1}-1$ that
\begin{equation*}
\kappa\barpsi\leqslant(\kappa_{\rho+1}-1)\barpsi<(\rho+1)\fBtwee\leqslant(\rho+\lambda)\fBtwee+[\lambda=0](t\fBtwee+\barpsi).
\end{equation*}

Hence
\begin{equation*}
\kappa\barpsi-[\lambda=0](t\fBtwee+\barpsi)<(\rho+\lambda)\fBtwee\leqslant\kappa\barpsi+[\lambda>0](t\fBtwee+\barpsi).
\end{equation*}
Adding $\kappa t\fBtwee$ in all sides of the equation, we get
\begin{equation*}
(\kappa-[\lambda=0])(t\fBtwee+\barpsi)<(\kappa t+\rho+\lambda)\fBtwee\leqslant(\kappa+[\lambda>0])(t\fBtwee+\barpsi).
\end{equation*}
If we apply \eqref{efferecallen} and the representation of $r$, we obtain
\begin{equation*}
(\kappa-[\lambda=0])\psi<r\fBtwee\leqslant(\kappa+[\lambda>0])\psi,
\end{equation*}
and after adding $\fBC r\psi$, we have
\begin{equation*}
(\fBC r+\kappa-[\lambda=0])\psi<r(\fBtwee+\fBC\psi)\leqslant(\fBC r+\kappa+[\lambda>0])\psi.
\end{equation*}
Using Formula~\eqref{efferecallen} for $\fCtwee$, we eventually obtain
\begin{equation*}
(\fBC r+\kappa+[\lambda>0]-1)\psi<r\fCtwee\leqslant(\fBC r+\kappa+[\lambda>0])\psi,
\end{equation*}
which proves that
\begin{equation}\label{formulekaer}
k_r=\fBC r+\kappa+[\lambda>0].
\end{equation}

Finally, let us verify the expression for $k_{r+1}-1$. If $r=\psi-1$, then
\begin{gather*}
r=(\kappa_{\barpsi}-1)t+(\barpsi-1)+t\qquad\text{and}\\
\begin{multlined}[.95\textwidth]
k_{r+1}-1=k_{\psi}-1=\fCtwee-1=\fBtwee+\fBC\psi-1\\
=\fBC\psi+(\kappa_{\barpsi}-1)=\fBC(r+1)+\kappa.
\end{multlined}
\end{gather*}
Otherwise $r+1\leqslant\psi-1$ and we can use \eqref{formulekaer} to find $k_{r+1}$. First suppose that $\lambda<t-[\kappa<\kappa_{\rho+1}-1]$. Then we have
\begin{gather*}
r+1=\kappa t+\rho+(\lambda+1)\qquad\text{and}\\
k_{r+1}-1=\fBC(r+1)+\kappa+[\lambda+1>0]-1=\fBC(r+1)+\kappa.
\end{gather*}
If on the contrary $\lambda=t-[\kappa<\kappa_{\rho+1}-1]$, we have
\begin{gather*}
r+1=(\kappa+1)t+(\rho+[\kappa=\kappa_{\rho+1}-1])+0\qquad\text{and again}\\
k_{r+1}-1=\fBC(r+1)+(\kappa+1)+[0>0]-1=\fBC(r+1)+\kappa.
\end{gather*}

This ends the proof of the lemma and concludes Case $n^{\ast}\mid n$.

\subsubsection{Case $n^{\ast}\nmid n$}
By Equations~\eqref{laatsteformulevrReendn} and \eqref{formuleSigmadriedn} for $R_1''$ and $\Sigma_3$, proving $R_1''=0$ boils down to verifying that
\begin{multline*}
\Bigl(p^{\frac{w\cdot v_0+\xi_B'(w\cdot v_1)}{\mu_B}}-1\Bigr)\sum_{k=0}^{\fCtwee-1}q^k\Bigl(p^{\frac{\xi_C(w\cdot v_1)+w\cdot v_2}{\mu_C}}\Bigr)^{-\left\lfloor\frac{k\psi}{\fCtwee}\right\rfloor}\\
+p^{\frac{w\cdot v_0+\xi_B'(w\cdot v_1)}{\mu_B}}\Bigl(p^{\frac{\xi_C(w\cdot v_1)+w\cdot v_2}{\mu_C}}-1\Bigr)\sum_{\kappa=0}^{\fBtwee-1}q^{\kappa}\Bigl(p^{\frac{w\cdot v_0+\xi_B'(w\cdot v_1)}{\mu_B}}\Bigr)^{\left\lfloor\frac{\kappa\psi}{\fBtwee}\right\rfloor}\\
=\Bigl(p^{\frac{\xi_2(w\cdot v_1)+w\cdot v_3}{\mu_2}}-1\Bigr)\frac{q^{\fBC}-1}{q-1}.
\end{multline*}
Expressing everything in terms of
\begin{equation*}
q\qquad\text{and}\qquad\beta=p^{\frac{w\cdot v_0+\xi_B'(w\cdot v_1)}{\mu_B}}
\end{equation*}
by means of Identities~\eqref{sigcsigbidcasedrie} and \eqref{sig2sigbidcasedrie}, the above statement is equivalent to
\begin{multline}\label{finalcheckcas3dn}
(q-1)(\beta-1)\sum_{k=0}^{\fCtwee-1}q^{k-\fBC\left\lfloor\frac{k\psi}{\fCtwee}\right\rfloor}\beta^{\left\lfloor\frac{k\psi}{\fCtwee}\right\rfloor}+(q-1)(q^{\fBC}-\beta)\sum_{\kappa=0}^{\fBtwee-1}q^{\kappa}\beta^{\left\lfloor\frac{\kappa\psi}{\fBtwee}\right\rfloor}\\
=(q^{\fBC}-1)(q^{\fBtwee}\beta^{\psi}-1).
\end{multline}
This equality in fact turns out to be a polynomial identity in the variables $q$ and $\beta$, as we will show now.

Both sequences, $\bigl(\lfloor k\psi/\fCtwee\rfloor\bigr)_{k=0}^{\fCtwee}$ and $\bigl(\lfloor\kappa\psi/\fBtwee\rfloor\bigr)_{\kappa=0}^{\fBtwee}$, ascend from $0$ to $\psi$. As $\psi<\fCtwee$ and $\psi$ may be strictly greater than $\fBtwee$, the first sequence adopts all values in $\{0,\ldots,\psi\}$, but the second one may not. We put
\begin{alignat*}{3}
k_r&=\min\left\{k\in\Zplus\;\middle\vert\;\left\lfloor\frac{k\psi}{\fCtwee}\right\rfloor=r\right\}&&=\left\lceil\frac{r\fCtwee}{\psi}\right\rceil&\qquad&\text{and}\\
\kappa_r&=\min\left\{\kappa\in\Zplus\;\middle\vert\;\left\lfloor\frac{\kappa\psi}{\fBtwee}\right\rfloor\geqslant r\right\}&&=\left\lceil\frac{r\fBtwee}{\psi}\right\rceil;&\qquad&r=0,\ldots,\psi.
\end{alignat*}
The numbers $k_r$ are the same as in Case $n^{\ast}\mid n$, while the $\kappa_r$ are defined differently; note that the sequence $(\kappa_r)_r$ is still ascending, but no longer necessarily strictly. We have
\begin{alignat*}{6}
0&=k_0&&<k_1&&<\cdots&&<k_{\psi-1}&&<k_{\psi}&&=\fCtwee\qquad\text{and}\\
0&=\kappa_0&&<\kappa_1&&\leqslant\cdots&&\leqslant\kappa_{\psi-1}&&\leqslant\kappa_{\psi}&&=\fBtwee;
\end{alignat*}
furthermore, there is the following relation between the numbers $k_r$ and $\kappa_r$:
\begin{equation*}
k_r=\left\lceil\frac{r\fCtwee}{\psi}\right\rceil=\left\lceil\frac{r(\fBtwee+\fBC\psi)}{\psi}\right\rceil=\fBC r+\left\lceil\frac{r\fBtwee}{\psi}\right\rceil=\fBC r+\kappa_r,
\end{equation*}
for all $r\in\{0,\ldots,\psi\}$.

Next, we use these data in rewriting both sums appearing in \eqref{finalcheckcas3dn}. If we adopt the convention that empty sums equal zero, then the first sum is given by
\begin{align}
\sum_{k=0}^{\fCtwee-1}q^{k-\fBC\left\lfloor\frac{k\psi}{\fCtwee}\right\rfloor}\beta^{\left\lfloor\frac{k\psi}{\fCtwee}\right\rfloor}
&=\sum_{r=0}^{\psi-1}\sum_{k=k_r}^{k_{r+1}-1}q^{k-\fBC r}\beta^r\notag\\
&=\sum_r\beta^r\sum_{k=\fBC r+\kappa_r}^{\fBC(r+1)+\kappa_{r+1}-1}q^{k-\fBC r}\notag\\
&=\sum_r\beta^r\sum_{\kappa=\kappa_r}^{\fBC+\kappa_{r+1}-1}q^{\kappa}\notag\\
&\overset{(\star)}{=}\sum_r\beta^r\Biggl(\sum_{\kappa=\kappa_r}^{\kappa_{r+1}-1}q^{\kappa}+q^{\kappa_{r+1}}\sum_{l=0}^{\fBC-1}q^l\Biggr)\notag\\
&=\sum_r\beta^r\sum_{\kappa}q^{\kappa}+\frac{q^{\fBC}-1}{q-1}\sum_r\beta^rq^{\kappa_{r+1}}.\label{firstsumc3}
\end{align}
Note that Equality~$(\star)$ holds even if $\kappa_r=\kappa_{r+1}$ for some $r$. With Notation~\ref{iverson}, the second sum can be written as
\begin{equation}\label{secondsumc3}
\sum_{\kappa=0}^{\fBtwee-1}q^{\kappa}\beta^{\left\lfloor\frac{\kappa\psi}{\fBtwee}\right\rfloor}=\sum_{r=0}^{\psi-1}\beta^r\sum_{\kappa}\left[\left\lfloor\frac{\kappa\psi}{\fBtwee}\right\rfloor=r\right]q^{\kappa}=\sum_r\beta^r\sum_{\kappa=\kappa_r}^{\kappa_{r+1}-1}q^{\kappa}.
\end{equation}
Indeed, if there is no $\kappa\in\{0,\ldots,\fBtwee-1\}$ such that $\lfloor\kappa\psi/\fBtwee\rfloor=r$, then $\kappa_r=\kappa_{r+1}$ and $\sum_{\kappa=\kappa_r}^{\kappa_{r+1}-1}q^{\kappa}=0$, otherwise $\kappa_r<\kappa_{r+1}$ and $\kappa_r,\ldots,\kappa_{r+1}-1$ are precisely the indices $\kappa$ satisfying $\lfloor\kappa\psi/\fBtwee\rfloor=r$.

From (\ref{firstsumc3}--\ref{secondsumc3}), it now follows that
\begin{align*}
&(q-1)(\beta-1)\sum_{k=0}^{\fCtwee-1}q^{k-\fBC\left\lfloor\frac{k\psi}{\fCtwee}\right\rfloor}\beta^{\left\lfloor\frac{k\psi}{\fCtwee}\right\rfloor}+(q-1)(q^{\fBC}-\beta)\sum_{\kappa=0}^{\fBtwee-1}q^{\kappa}\beta^{\left\lfloor\frac{\kappa\psi}{\fBtwee}\right\rfloor}\\
&\quad=(q^{\fBC}-1)\left[(q-1)\sum_{r=0}^{\psi-1}\beta^r\sum_{\kappa=\kappa_r}^{\kappa_{r+1}-1}q^{\kappa}+(\beta-1)\sum_{r=0}^{\psi-1}\beta^rq^{\kappa_{r+1}}\right]\\
&\quad=(q^{\fBC}-1)\left[\sum_r\beta^rq^{\kappa_{r+1}}-\sum_r\beta^rq^{\kappa_r}+\sum_r\beta^{r+1}q^{\kappa_{r+1}}-\sum_r\beta^rq^{\kappa_{r+1}}\right]\\
&\quad=(q^{\fBC}-1)(q^{\fBtwee}\beta^{\psi}-1).
\end{align*}
Having verified that the calculations above make sense even if $\kappa_r=\kappa_{r+1}$ for some $r$, we achieve \eqref{finalcheckcas3dn} and therefore conclude Case~$n^{\ast}\nmid n$. This ends the proof of the main theorem in Case~III.

\section{Case~IV: exactly two facets of \Gf\ contribute to $s_0$, and these two facets are both non-compact $B_1$-facets with respect to a same variable and have an edge in common}
\subsection{Figure and notations}
Let us assume that the two facets $\tau_0$ and $\tau_1$ contributing to $s_0$ are both $B_1$-facets with respect to the variable $z$. Note that $\tau_0$ and $\tau_1$ cannot be non-compact for the same variable unless they coincide. Therefore, we may assume that $\tau_0$ is non-compact for $x$, while $\tau_1$ is non-compact for $y$, and that $\tau_0$ and $\tau_1$ share their unique compact edge $[AB]$. Here $A(x_A,y_A,0)$ and $B(x_B,y_B,1)\in\Zplus^3$ denote the common vertices of $\tau_0$ and $\tau_1$ in the $xy$-plane and at \lq height\rq\ one, respectively. The situation is shown in Figure~\ref{figcase4}.

%
%
\begin{figure}
\psset{unit=.02844311377\textwidth}
\centering
\subfigure[Non-compact $B_1$-facets $\tau_0$ and $\tau_1$, their subfaces and neighbor facets $\tau_2,\tau_3,$ and $\tau_4$]{
\begin{pspicture}(-9.1,-7.8)(7.6,6.8)
{\footnotesize
\pstThreeDCoor[xMin=0,yMin=0,zMin=0,xMax=12,yMax=10,zMax=6.75,linecolor=black,linewidth=.7pt]
{
\psset{linecolor=black,linewidth=.3pt,linestyle=dashed,subticks=1}
\pstThreeDLine(4,0,0)(4,3,0)\pstThreeDLine(4,3,0)(0,3,0)
\pstThreeDLine(12,0,0)(12,3,0)\pstThreeDLine(12,3,0)(4,3,0)
\pstThreeDLine(0,10,0)(4,10,0)\pstThreeDLine(4,10,0)(4,3,0)
\pstThreeDLine(7,0,0)(7,5,0)\pstThreeDLine(7,5,0)(0,5,0)
\pstThreeDLine(12,3,0)(12,10,0)\pstThreeDLine(12,10,0)(4,10,0)
\pstThreeDLine(4,0,1)(4,3,1)\pstThreeDLine(4,3,1)(0,3,1)
\pstThreeDLine(12,0,1)(12,3,1)\pstThreeDLine(0,10,1)(4,10,1)
\pstThreeDLine(0,0,1)(4,0,1)\pstThreeDLine(4,0,1)(12,0,1)
\pstThreeDLine(0,0,1)(0,3,1)\pstThreeDLine(0,3,1)(0,10,1)
\pstThreeDLine(4,0,0)(4,0,1)\pstThreeDLine(12,0,0)(12,0,1)
\pstThreeDLine(0,3,0)(0,3,1)\pstThreeDLine(0,10,0)(0,10,1)
\pstThreeDLine(12,3,0)(12,3,1)\pstThreeDLine(4,10,0)(4,10,1)
}
\pstThreeDPut[pOrigin=c](8.69,4,0.5){\psframebox*[framesep=1pt,framearc=0.3]{\phantom{$\tau_0$}}}
{
\psset{dotstyle=none,dotscale=1,drawCoor=false}
\psset{linecolor=black,linewidth=1pt,linejoin=1}
\psset{fillcolor=lightgray,opacity=.6,fillstyle=solid}
\pstThreeDLine(12,5,0)(7,5,0)(4,3,1)(12,3,1)
\pstThreeDLine(7,10,0)(7,5,0)(4,3,1)(4,10,1)
}
\pstThreeDPut[pOrigin=t](7,5,-0.28){$A$}
\pstThreeDPut[pOrigin=b](4,3,1.28){$B$}
\pstThreeDPut[pOrigin=c](8.69,4,0.5){$\tau_0$}
\pstThreeDPut[pOrigin=c](5.5,7,0.5){$\tau_1$}
\pstThreeDPut[pOrigin=c](8,2.7,1.5){\psframebox*[framesep=0.4pt,framearc=0.3]{$\tau_3$}}
\pstThreeDPut[pOrigin=c](3.7,6.85,1.5){\psframebox*[framesep=0.2pt,framearc=0.5]{$\tau_2$}}
\pstThreeDPut[pOrigin=c](10,8,0){$\tau_4$}
\pstThreeDPut[pOrigin=b](12,3,1.28){$l_x$}
\pstThreeDPut[pOrigin=b](4,10,1.25){$l_y$}
\pstThreeDPut[pOrigin=br](0.06,-0.06,1.12){$1$}
}
\end{pspicture}
}\hfill\subfigure[Relevant cones associated to relevant faces of~\Gf]{
\psset{unit=.03125\textwidth}
\begin{pspicture}(-7.6,-3.8)(7.6,9.6)
{\footnotesize
\pstThreeDCoor[xMin=0,yMin=0,zMin=0,xMax=10,yMax=10,zMax=10,nameZ={},linecolor=gray,linewidth=.7pt]
{
\psset{linecolor=gray,linewidth=.3pt,linejoin=1,linestyle=dashed,fillcolor=lightgray,fillstyle=none}
\pstThreeDLine(10,0,0)(0,10,0)\pstThreeDLine(0,10,0)(0,0,10)\pstThreeDLine(0,0,10)(10,0,0)
}
{
\psset{linecolor=black,linewidth=.7pt,linejoin=1,fillcolor=lightgray,fillstyle=none}
\pstThreeDLine(0,0,0)(0,0,10)
}
{
\psset{labelsep=2pt}
\uput[90](0,1.7){\psframebox*[framesep=0.5pt,framearc=1]{\darkgray\scriptsize$v_4$}}
}
{
\psset{linecolor=black,linewidth=.7pt,linejoin=1,fillcolor=lightgray,fillstyle=none}
\pstThreeDLine(0,0,0)(0,3.33,6.67)
\pstThreeDLine(0,0,0)(2.5,0,7.5)
\pstThreeDLine(0,0,0)(9,0,1)
\pstThreeDLine(0,0,0)(0,8,2)
}
{
\psset{linecolor=darkgray,linewidth=.8pt,linejoin=1,arrows=->,arrowscale=1,fillcolor=lightgray,fillstyle=none}
\pstThreeDLine(0,0,0)(0,1.5,3)
\pstThreeDLine(0,0,0)(1.4,0,4.2)
\pstThreeDLine(0,0,0)(6.3,0,.7)
\pstThreeDLine(0,0,0)(0,4,1)
\pstThreeDLine(0,0,0)(0,0,2)
}
{
\psset{linecolor=white,linewidth=2pt,linejoin=1,fillcolor=lightgray,fillstyle=none}
\pstThreeDLine(1.75,.999,7.25)(.75,2.33,6.92)
\pstThreeDLine(2.12,1.20,6.68)(1.00,4.8,4.2)
}
{
\psset{linecolor=black,linewidth=.7pt,linejoin=1,fillcolor=lightgray,fillstyle=none}
\pstThreeDLine(0,0,10)(0,3.33,6.67)
\pstThreeDLine(0,0,10)(2.5,0,7.5)
\pstThreeDLine(0,8,2)(0,3.33,6.67)(2.5,0,7.5)(9,0,1)
}
{
\psset{linecolor=black,linewidth=.7pt,linejoin=1,linestyle=dashed,fillcolor=lightgray,fillstyle=none}
\pstThreeDLine(9,0,1)(0,8,2)
\pstThreeDLine(2.5,0,7.5)(0,8,2)
}
{
\psset{labelsep=2pt}
\uput[-35](.51,.994){\darkgray\scriptsize$v_0$}
\uput[75](1.41,-.275){\darkgray\scriptsize$v_3$}
}
{
\psset{labelsep=2.5pt}
\uput[106](-2.22,-.808){\darkgray\scriptsize$v_2$}
}
{
\psset{labelsep=1.4pt}
\uput[210](-.493,1.57){\darkgray\scriptsize$v_1$}
}
{
\psset{labelsep=3.8pt}
\uput[30](2.35,4.58){$\Delta_{\tau_0}$}
\uput[30](4,1.74){$\Delta_{l_x}$}
\uput[30](5.65,-1.1){$\Delta_{\tau_3}$}
\uput[150](-1.76,5.61){$\Delta_{\tau_1}$}
\uput[150](-4.06,1.65){$\Delta_{l_y}$}
\uput[150](-6.34,-2.31){$\Delta_{\tau_2}$}
}
\rput(0.15,6.3){\psframebox*[framesep=0.3pt,framearc=1]{\footnotesize$\Delta_A$}}
\rput(1.9,3.6){\psframebox*[framesep=0.7pt,framearc=1]{\footnotesize$\delta_1$}}
\rput(1.94,2.25){\psframebox*[framesep=0.3pt,framearc=.3]{\footnotesize$\delta_2$}}
\rput(-1.25,.3){\footnotesize$\delta_3$}
\pstThreeDNode(1.25,1.66,7.09){AB}
\rput[B](0,9.075){$z,\Delta_{\tau_4}$}
\rput[Br](8.01,9.075){\rnode{ABlabel}{$\Delta_{[AB]}$}}
\ncline[linewidth=.3pt,nodesepB=2pt,nodesepA=0.5pt]{->}{ABlabel}{AB}
}
\end{pspicture}
}
\caption{Case IV: the only facets contributing to $s_0$ are the non-compact $B_1$-facets $\tau_0$ and $\tau_1$}
\label{figcase4}
\end{figure}
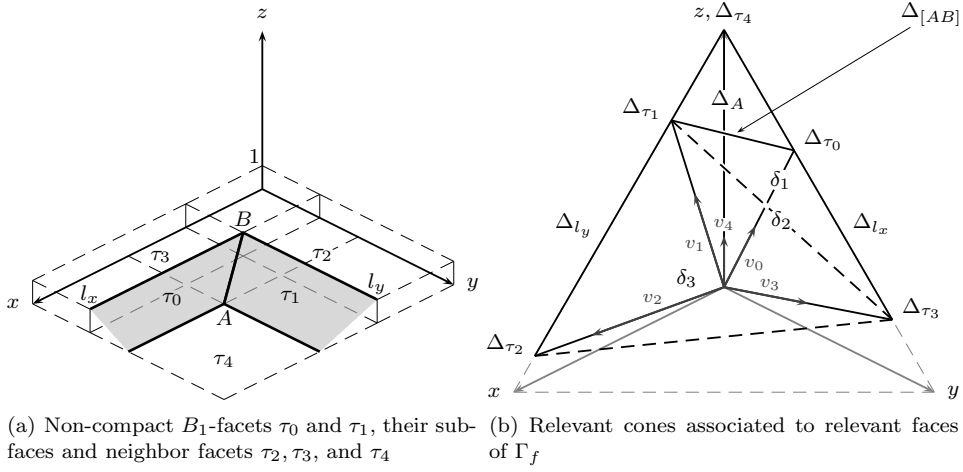
%
%

If we put $\overrightarrow{AB}(x_B-x_A,y_B-y_A,1)=(\alpha,\beta,1)$ as usual, then $v_0(0,1,-\beta)$ and $v_1(1,0,-\alpha)$ are the unique primitive vectors in $\Zplus^3$ perpendicular to $\tau_0$ and $\tau_1$, respectively, while equations for the affine hulls of $\tau_0$ and $\tau_1$ are provided by
\begin{equation*}
\aff(\tau_0)\leftrightarrow y-\beta z=y_A\qquad\text{and}\qquad\aff(\tau_1)\leftrightarrow x-\alpha z=x_A.
\end{equation*}
Necessarily, we have that $\alpha,\beta<0$; i.e., $x_B<x_A$ and $y_B<y_A$. Given that the numerical data associated to $\tau_0$ and $\tau_1$ are $(m(v_0),\sigma(v_0))=(y_A,1-\beta)$ and $(m(v_1),\sigma(v_1))=(x_A,1-\alpha)$, respectively, we assume
\begin{equation*}
\Re(s_0)=\frac{\beta-1}{y_A}=\frac{\alpha-1}{x_A}\quad\ \ \text{and}\ \ \quad\Im(s_0)=\frac{2n\pi}{\gcd(x_A,y_A)\log p}\quad\ \ \text{for some $n\in\Z$.}
\end{equation*}

As indicated in Figure~\ref{figcase4}, we denote by $\tau_2$ and $\tau_3$ the non-compact facets of \Gf\ sharing with $\tau_1$ and $\tau_0$, respectively, a half-line with endpoint $B$, and by $\tau_4$ the facet lying in the $xy$-plane. Primitive vectors in $\Zplus^3$ perpendicular to $\tau_2,\tau_3,\tau_4$ will be denoted
\begin{equation*}
v_2(a_2,0,c_2),\quad v_3(0,b_3,c_3),\quad v_4(0,0,1),
\end{equation*}
respectively, and equations for the affine supports of these facets are given by
\begin{alignat*}{3}
\aff(\tau_2)&\leftrightarrow&\ a_2x&+c_2&z&=m_2,\\
\aff(\tau_3)&\leftrightarrow&\ b_3y&+c_3&z&=m_3,\\
\aff(\tau_4)&\leftrightarrow&\     &    &z&=0
\end{alignat*}
for certain $m_2,m_3\in\Zplus$. Finally, the numerical data for $\tau_2,\tau_3,$ and $\tau_4$ are $(m_2,\sigma_2)$, $(m_3,\sigma_3),$ and $(0,1)$, respectively, with $\sigma_2=a_2+c_2$ and $\sigma_3=b_3+c_3$.

\subsection{The candidate pole $s_0$ and the contributions to its residues}
Again we want to prove that $s_0$ is not a pole of \Zof. Since $s_0$ has expected order two as a candidate pole of \Zof, in order to do this, we will, as in Case~III, show that
\begin{align*}
R_2&=\lim_{s\to s_0}\left(p^{1-\beta+y_As}-1\right)\left(p^{1-\alpha+x_As}-1\right)\Zof(s)\qquad\text{and}\\
R_1&=\lim_{s\to s_0}\frac{d}{ds}\left[\left(p^{1-\beta+y_As}-1\right)\left(p^{1-\alpha+x_As}-1\right)\Zof(s)\right]
\end{align*}
both equal zero.

In this case the (compact) faces contributing to $s_0$ are $A,B,$ and $[AB]$; i.e., we may in the above expressions for $R_2$ and $R_1$ replace $\Zof(s)$ by
\begin{equation*}
\sum_{\tau=A,B,[AB]}L_{\tau}(s)S(\Dtu)(s).
\end{equation*}
Vertex $A$ is exclusively contained in the facets $\tau_0,\tau_1,$ and $\tau_4$; its associated cone $\Delta_A$ is therefore simplicial. Vertex $B$, on the other hand, is contained in at least the facets $\tau_0,\tau_1,\tau_2,$ and $\tau_3$; hence $\Delta_B$ is certainly not simplicial. However, if we consider the cones $\delta_1,\delta_2,\delta_3$ defined below as members of a simplicial subdivision of $\Delta_B$, the relevant contributions to $s_0$ come from the simplicial cones
\begin{align*}
\DA&=\cone(v_0,v_1,v_4),&\delta_1&=\cone(v_0,v_1,v_3),&\DAB&=\cone(v_0,v_1).\\
   &                    &\delta_2&=\cone(v_1,v_3),    &    &                \\
   &                    &\delta_3&=\cone(v_1,v_2,v_3),&    &
\end{align*}
This way we find, similarly to Case~III, that $R_2$ and $R_1$ are explicitly given by
\begin{gather*}
\begin{aligned}
R_2&=L_A(s_0)\frac{\Sigma(\Delta_A)(s_0)}{p-1}
+L_B(s_0)\frac{\Sigma(\delta_1)(s_0)}{p^{\sigma_3+m_3s_0}-1}
+L_{[AB]}(s_0)\Sigma(\Delta_{[AB]})(s_0),\\
R_1&=L_A'(s_0)\frac{\Sigma(\Delta_A)(s_0)}{p-1}
+L_A(s_0)\frac{\Sigma(\Delta_A)'(s_0)}{p-1}
+L_B'(s_0)\frac{\Sigma(\delta_1)(s_0)}{p^{\sigma_3+m_3s_0}-1}
\end{aligned}\\
+L_B(s_0)\frac{\Sigma(\delta_1)'(s_0)}{p^{\sigma_3+m_3s_0}-1}
-L_B(s_0)\frac{m_3(\log p)p^{\sigma_3+m_3s_0}\Sigma(\delta_1)(s_0)}{\Fdrie^2}\\
+L_B(s_0)\frac{y_A(\log p)\Sigma(\delta_2)(s_0)}{p^{\sigma_3+m_3s_0}-1}
+L_B(s_0)\frac{y_A(\log p)\Sigma(\delta_3)(s_0)}{\Ftwee\Fdrie}\\
\hspace{.365\textwidth}+L_{[AB]}'(s_0)\Sigma(\Delta_{[AB]})(s_0)
+L_{[AB]}(s_0)\Sigma(\Delta_{[AB]})'(s_0).
\end{gather*}

\subsection{Towards simplified formulas for $R_2$ and $R_1$}
\subsubsection{The factors $L_{\tau}(s_0)$ and $L_{\tau}'(s_0)$}
Since $N_A=N_B=0$ and $N_{[AB]}=(p-1)^2$, we obtain
\begin{gather*}
L_A(s_0)=L_B(s_0)=\left(\frac{p-1}{p}\right)^3,\qquad L_A'(s_0)=L_B'(s_0)=0,\\
\begin{aligned}
L_{[AB]}(s_0)&=\left(\frac{p-1}{p}\right)^3-\left(\frac{p-1}{p}\right)^2\frac{p^{s_0}-1}{p^{s_0+1}-1},\qquad\text{and}\\
L_{[AB]}'(s_0)&=-(\log p)\left(\frac{p-1}{p}\right)^3\frac{p^{s_0+1}}{\bigl(p^{s_0+1}-1\bigr)^2}.
\end{aligned}
\end{gather*}

\subsubsection{Cone multiplicities}
We calculate the multiplicities of the five contributing simplicial cones, as well as the multiplicities $\mu_x$ and $\mu_y$ of the cones associated to the non-compact edges $l_x=\tau_0\cap\tau_3$ and $l_y=\tau_1\cap\tau_2$ (see Figure~\ref{figcase4}):
\begin{gather*}
\mult\Delta_A=\#H(v_0,v_1,v_4)=
\begin{Vmatrix}
0&1&-\beta\\1&0&-\alpha\\0&0&1
\end{Vmatrix}=1,\\
\begin{alignedat}{6}
\mu_x&=\mult\Delta_{l_x}&&=\#H(v_0,v_3)&&=
\begin{Vmatrix}
1&-\beta\\b_3&c_3
\end{Vmatrix}&&=-
\begin{vmatrix}
1&-\beta\\b_3&c_3
\end{vmatrix}&&=-\beta b_3-c_3&&>0,\\
\mu_y&=\mult\Delta_{l_y}&&=\#H(v_1,v_2)&&=
\begin{Vmatrix}
1&-\alpha\\a_2&c_2
\end{Vmatrix}&&=-
\begin{vmatrix}
1&-\alpha\\a_2&c_2
\end{vmatrix}&&=-\alpha a_2-c_2&&>0,
\end{alignedat}\\
\begin{alignedat}{2}
\mult\delta_1&=\#H(v_0,v_1,v_3)&&=
\begin{Vmatrix}
0&1&-\beta\\1&0&-\alpha\\0&b_3&c_3
\end{Vmatrix}=
\begin{Vmatrix}
1&-\beta\\b_3&c_3
\end{Vmatrix}=\mu_x,\\
\mu_3=\mult\delta_3&=\#H(v_1,v_2,v_3)&&=
\begin{Vmatrix}
1&0&-\alpha\\a_2&0&c_2\\0&b_3&c_3
\end{Vmatrix}=b_3
\begin{Vmatrix}
1&-\alpha\\a_2&c_2
\end{Vmatrix}=b_3\mu_y,
\end{alignedat}\\
\begin{alignedat}{2}
\mult\delta_2&=\#H(v_1,v_3)&&=\gcd(b_3,c_3,-\alpha b_3)=1,\\
\mult\Delta_{[AB]}&=\#H(v_0,v_1)&&=\gcd(1,-\beta,-\alpha)=1.
\end{alignedat}
\end{gather*}

\subsubsection{The sums $\Sigma(\cdot)(s_0)$ and $\Sigma(\cdot)'(s_0)$}
Because the corresponding multiplicities are one, we have that
\begin{gather*}
\Sigma(\Delta_A)(s_0)=\Sigma(\delta_2)(s_0)=\Sigma(\Delta_{[AB]})(s_0)=1\qquad\text{and}\\
\Sigma(\Delta_A)'(s_0)=\Sigma(\Delta_{[AB]})'(s_0)=0.
\end{gather*}

Furthermore, since $H(v_0,v_3)\subseteq H(v_0,v_1,v_3)$ and
\begin{equation*}
\mu_x=\#H(v_0,v_3)=\#H(v_0,v_1,v_3),
\end{equation*}
we may put
\begin{gather*}
H_x=H(v_0,v_3)=H(v_0,v_1,v_3),\\
\Sigma_x=\Sigma(\delta_1)(s_0)=\sum\nolimits_{h\in H_x}p^{\sigma(h)+m(h)s_0}=\sum\nolimits_{h\in H_x}p^{w\cdot h},\qquad\text{and}\\
\Sigma_x'=\Sigma(\delta_1)'(s_0)=\dds{\sum\nolimits_{h\in H_x}p^{\sigma(h)+m(h)s}}=(\log p)\sum\nolimits_{h\in H_x}m(h)p^{w\cdot h},
\end{gather*}
with $w=(1,1,1)+s_0(x_B,y_B,1)\in\C^3$.

Finally we denote
\begin{gather*}
H_y=H(v_1,v_2),\qquad H_3=H(v_1,v_2,v_3),\qquad\text{and}\\
\Sigma_3=\Sigma(\delta_3)(s_0)=\sum\nolimits_{h\in H_3}p^{\sigma(h)+m(h)s_0}=\sum\nolimits_{h\in H_3}p^{w\cdot h}.
\end{gather*}

\subsubsection{New formulas for the residues}
If we put
\begin{gather*}
R_2=\left(\frac{p-1}{p}\right)^3R_2',\qquad R_1=(\log p)\left(\frac{p-1}{p}\right)^3R_1',\\
F_2=p^{\sigma_2+m_2s_0}-1,\qquad F_3=p^{\sigma_3+m_3s_0}-1,\qquad\text{and}\qquad q=p^{-s_0-1},
\end{gather*}
the observations above yield
\begin{gather}
R_2'=\frac{1}{1-q}+\frac{\Sigma_x}{F_3}\qquad\text{and}\label{formR2accasevier}\\
R_1'=-\frac{q}{(1-q)^2}+\frac{\Sigma_x'}{(\log p)F_3}-\frac{m_3(F_3+1)\Sigma_x}{F_3^2}+\frac{y_A}{F_3}+\frac{y_A\Sigma_3}{F_2F_3}.\label{formR1accasevier}
\end{gather}

We shall prove that $R_2'=R_1'=0$.

\subsection{Some vector identities and their consequences}
Given the coordinates of $v_i$; $i=0,\ldots,3$; one easily checks that\footnote{As in the previous cases, the first two identities arise from $(\adj M)M=(\det M)I$ for $M=\begin{psmallmatrix}1&-\beta\\b_3&c_3\end{psmallmatrix},\begin{psmallmatrix}1&-\alpha\\a_2&c_2\end{psmallmatrix}$ with $\det M=\mu_x,\mu_y$, respectively, while the third one follows immediately from the other two.}
\begin{align}
b_3v_0-v_3&=(0,0,\mu_x),\label{id1c4}\\
a_2v_1-v_2&=(0,0,\mu_y),\qquad\text{and}\label{id2c4}\\
-\mu_3v_0+a_2\mu_xv_1-\mu_xv_2+\mu_yv_3&=(0,0,0).\label{id3c4}
\end{align}

Considering dot products with $w=(1,1,1)+s_0(x_B,y_B,1)$, it follows from \eqref{id1c4} and \eqref{id2c4} that
\begin{equation}\label{pwvdriemuxc4}
\frac{-b_3(w\cdot v_0)+w\cdot v_3}{\mu_x}=\frac{-a_2(w\cdot v_1)+w\cdot v_2}{\mu_y}=-s_0-1,
\end{equation}
whereas making the dot product with $B(x_B,y_B,1)$ on both sides of \eqref{id1c4} yields
\begin{equation}\label{interpbdrieyammdriec4}
y_Ab_3-m_3=\mu_x.
\end{equation}

Other consequences of (\ref{id1c4}--\ref{id3c4}) include
\begin{align}
\frac{-b_3}{\mu_x}v_0+\frac{1}{\mu_x}v_3&=(0,0,-1)\in\Z^3,\label{bp1c4}\\
\frac{-a_2}{\mu_y}v_1+\frac{1}{\mu_y}v_2&=(0,0,-1)\in\Z^3,\qquad\text{and}\label{bp2c4}\\
\frac{a_2\mu_x}{\mu_3}v_1+\frac{-\mu_x}{\mu_3}v_2+\frac{1}{b_3}v_3&=v_0\in\Z^3.\label{bp3c4}
\end{align}

\subsection{Points of $H_x,H_y,$ and $H_3$}
It follows from \eqref{bp1c4} that the $\mu_x$ points of $H_x$ are given by
\begin{alignat}{2}
\left\{\frac{-jb_3}{\mu_x}\right\}v_0&+\frac{j}{\mu_x}v_3;&\qquad&j=0,\ldots,\mu_x-1;\label{pointshxc4}\\
\intertext{while it follows from \eqref{bp2c4} that the $\mu_y$ points of $H_y=H(v_1,v_2)$ are}
\left\{\frac{-ia_2}{\mu_y}\right\}v_1&+\frac{i}{\mu_y}v_2;&&i=0,\ldots,\mu_y-1.\notag
\end{alignat}
Note that $b_3$ and $a_2$ are, as expected, coprime to $\mu_x$ and $\mu_y$, respectively.\footnote{This follows from $\mu_x=-\beta b_3-c_3$, $\mu_y=-\alpha a_2-c_2$, and the primitivity of $v_2$ and $v_3$.}

If we consider $H_3=H(v_1,v_2,v_3)$ in the usual way as an additive group with subgroup\footnote{Recall that $H(v_1,v_3)$ is the trivial subgroup of $H_3$.} $H_y=H_y+H(v_1,v_3)$ of index $b_3$, then we see from \eqref{bp2c4} and \eqref{bp3c4} that the points
\begin{equation*}
\left\{\frac{-a_2\{-k\mu_x\}_{b_3}}{\mu_3}\right\}v_1+\frac{\{-k\mu_x\}_{b_3}}{\mu_3}v_2+\frac{k}{b_3}v_3\in H_3;\qquad k=0,\ldots,b_3-1;
\end{equation*}
can serve as representatives for the $b_3$ cosets of $H_y$ in $H_3$. Hence a complete list of the $\mu_3=b_3\mu_y$ points of $H_3$ is provided by
\begin{multline}\label{pointsh3c4}
\left\{\frac{-a_2(ib_3+\{-k\mu_x\}_{b_3})}{\mu_3}\right\}v_1+\frac{ib_3+\{-k\mu_x\}_{b_3}}{\mu_3}v_2+\frac{k}{b_3}v_3;\\
i=0,\ldots,\mu_y-1;\quad k=0,\ldots,b_3-1.
\end{multline}

These descriptions should allow us to find expressions for $\Sigma_x,\Sigma_x',$ and $\Sigma_3$ in the next subsection.

\subsection{Formulas for $\Sigma_x,\Sigma_x',$ and $\Sigma_3$}
If for $h\in H_x=H(v_0,v_3)$, we denote by $(h_0,h_3)$ the coordinates of $h$ with respect to the basis $(v_0,v_3)$, then by \eqref{pwvdriemuxc4}, \eqref{pointshxc4}, and $p^{w\cdot v_0}=1$, we have
\begin{equation}\label{formsigmaxcasevier}
\Sigma_x=\sum_{h\in H_x}p^{w\cdot h}=\sum_{j=0}^{\mu_x-1}\Bigl(p^{\frac{-b_3(w\cdot v_0)+w\cdot v_3}{\mu_x}}\Bigr)^j=\sum_jq^j=\frac{F_3}{q-1},
\end{equation}
whereas
\begin{align}
\frac{\Sigma_x'}{\log p}&=\sum_{h\in H_x}m(h)p^{w\cdot h}\notag\\
&=\sum_h(h_0y_A+h_3m_3)p^{w\cdot h}\notag\\
&=\sum_{j=0}^{\mu_x-1}\left(y_A\left\{\frac{-jb_3}{\mu_x}\right\}
+m_3\frac{j}{\mu_x}\right)\Bigl(p^{\frac{-b_3(w\cdot v_0)+w\cdot v_3}{\mu_x}}\Bigr)^j\notag\\
&=y_A\sum_j\left\{\frac{-jb_3}{\mu_x}\right\}q^j+\frac{m_3}{\mu_x}\sum_jjq^j.\label{formsigmaxaccaseviertemp}
\end{align}

If $\mu_x=1$, then clearly $\Sigma_x'=0$. Let us find an expression for $\Sigma_x'$ in the complementary case. So from now on assume that $\mu_x>1$. Write $b_3$ as $b_3=t\mu_x+\overline{b_3}$ with $t\in\Zplus$ and $\overline{b_3}=\{b_3\}_{\mu_x}\in\{1,\ldots,\mu_x-1\}$; note that by the coprimality of $b_3$ and $\mu_x$, we have $\gcd(\overline{b_3},\mu_x)=1$ and hence $\overline{b_3}\neq0$. Furthermore, put
\begin{equation}\label{defjkbarcase4}
j_{\overline{k}}=\min\left\{j\in\Zplus\;\middle\vert\;\left\lfloor\frac{j\overline{b_3}}{\mu_x}\right\rfloor=\overline{k}\right\}=\left\lceil\frac{\overline{k}\mu_x}{\overline{b_3}}\right\rceil;\qquad \overline{k}=0,\ldots,\overline{b_3};
\end{equation}
yielding
\begin{equation*}
0=j_0<j_1<\cdots<j_{\overline{b_3}-1}<j_{\overline{b_3}}=\mu_x.
\end{equation*}

Then, proceeding as in Case~III (Subsection~\ref{studysigmaeenacc3}), we write \eqref{formsigmaxaccaseviertemp} as
\begin{align*}
&\,\frac{\Sigma_x'}{\log p}\\
&=y_A\sum_{j=0}^{\mu_x-1}\left\{\frac{-j\overline{b_3}}{\mu_x}\right\}q^j+\frac{m_3[q^{\mu_x}(\mu_x q-\mu_x-q)+q]}{\mu_x(q-1)^2}\\
&=\frac{y_A}{1-q}\Biggl(\sum_{\overline{k}=0}^{\overline{b_3}-1}q^{j_{\overline{k}}}-\frac{\overline{b_3}F_3}{\mu_x(1-q^{-1})}\Biggr)-y_A
+\frac{m_3}{1-q}\left(-(F_3+1)+\frac{F_3}{\mu_x(1-q^{-1})}\right)\\
&=\frac{F_3}{1-q}\Biggl(\frac{y_A}{F_3}\sum_{\overline{k}}q^{j_{\overline{k}}}-\frac{y_A\overline{b_3}-m_3}{\mu_x(1-q^{-1})}
-\frac{m_3(F_3+1)}{F_3}\Biggr)-y_A.
\end{align*}
Finally, if we use that $y_Ab_3-m_3=\mu_x$ (cfr.\ \eqref{interpbdrieyammdriec4}) and $b_3=t\mu_x+\overline{b_3}$, we obtain
\begin{equation}\label{formsigmaxaccasevier}
\frac{\Sigma_x'}{(\log p)F_3}=\frac{1}{1-q}\Biggl(\frac{y_A}{F_3}\sum_{\overline{k}=0}^{\overline{b_3}-1}q^{j_{\overline{k}}}+\frac{y_At-1}{1-q^{-1}}
-\frac{m_3(F_3+1)}{F_3}\Biggr)-\frac{y_A}{F_3}\qquad\text{(if $\mu_x>1$).}
\end{equation}

Let us now calculate $\Sigma_3$. Using (\ref{pwvdriemuxc4}, \ref{pointsh3c4}) and $p^{w\cdot v_0}=p^{w\cdot v_1}=1$, we find
\begin{align*}
\Sigma_3&=\sum_{h\in H_3}p^{w\cdot h}\\
&=\sum_hp^{h_1(w\cdot v_1)+h_2(w\cdot v_2)+h_3(w\cdot v_3)}\\
&=\sum_{i=0}^{\mu_y-1}\sum_{k=0}^{b_3-1}p^{\frac{-a_2\left(ib_3+\{-k\mu_x\}_{b_3}\right)}{\mu_3}(w\cdot v_1)+\frac{ib_3+\{-k\mu_x\}_{b_3}}{\mu_3}(w\cdot v_2)+\frac{k}{b_3}(w\cdot v_3)}\\
&=\sum_i\Bigl(p^{\frac{-a_2(w\cdot v_1)+w\cdot v_2}{\mu_y}}\Bigr)^i\sum_kp^{\frac{-a_2(w\cdot v_1)+w\cdot v_2}{\mu_y}\left\{\frac{-k\mu_x}{b_3}\right\}+\frac{-b_3(w\cdot v_0)+w\cdot v_3}{\mu_x}\frac{k\mu_x}{b_3}}\\
&=\sum_iq^i\sum_kp^{(-s_0-1)\left(\left\{\frac{-k\mu_x}{b_3}\right\}+\frac{k\mu_x}{b_3}\right)}.
\end{align*}
Since $\mu_x$ and $b_3$ are coprime, one has $k\mu_x/b_3\notin\Z$ and $\{-k\mu_x/b_3\}=1-\{k\mu_x/b_3\}$ for $k$ not a multiple of $b_3$. Hence
\begin{equation*}
\left\{-\frac{k\mu_x}{b_3}\right\}+\frac{k\mu_x}{b_3}=
\begin{cases}
0,&\text{if $k=0$};\\
1+\left\lfloor\frac{k\mu_x}{b_3}\right\rfloor\in\Z,&\text{if $k\in\{1,\ldots,b_3-1\}$};
\end{cases}
\end{equation*}
and we obtain
\begin{equation}\label{formuleSigmadriecase4}
\Sigma_3=\frac{F_2}{q-1}\Biggl(1+q\sum_{k=1}^{b_3-1}q^{\left\lfloor\frac{k\mu_x}{b_3}\right\rfloor}\Biggr),
\end{equation}
with the understanding that the empty sum equals zero in case $b_3=1$.

\subsection{Proof of $R_2'=R_1'=0$}
As it follows immediately from \eqref{formR2accasevier} and \eqref{formsigmaxcasevier} that $R_2'=0$, we can further focus on $R_1'$. Let us first assume that $\mu_x=1$. In this case, we found that $\Sigma_x'=0$, while it follows from \eqref{formuleSigmadriecase4} that
\begin{equation*}
\Sigma_3=\frac{F_2}{q-1}(1+(b_3-1)q);
\end{equation*}
furthermore, note that $F_3=q-1$ and hence $\Sigma_x=1$, while $y_Ab_3-m_3=1$ by \eqref{interpbdrieyammdriec4}. With these observations, \eqref{formR1accasevier} easily yields $R_1'=0$.

From now on, suppose that $\mu_x>1$ and thus that $\overline{b_3}>0$. If we then fill in (\ref{formsigmaxcasevier}, \ref{formsigmaxaccasevier}, \ref{formuleSigmadriecase4}) in \eqref{formR1accasevier}, one sees that proving $R_1'=0$ eventually boils down to proving that
\begin{equation}\label{finalcheckcasevier}
\sum_{k=1}^{b_3-1}q^{\left\lfloor\frac{k\mu_x}{b_3}\right\rfloor}=\sum_{\overline{k}=1}^{\overline{b_3}-1}q^{j_{\overline{k}}\,-1}+t\frac{F_3}{q-1},
\end{equation}
whereby the sum over $\overline{k}$ is again understood to be zero if $\overline{b_3}=1$. Let us do this now.

Recall that $b_3=t\mu_x+\overline{b_3}$ with $t\in\Zplus$ and $\overline{b_3}\in\{1,\ldots,\mu_x-1\}$. So if $t=0$, we have $b_3=\overline{b_3}$, and by \eqref{defjkbarcase4} and the coprimality of $\overline{b_3}$ and $\mu_x$, we then find
\begin{equation*}
\sum_{k=1}^{b_3-1}q^{\left\lfloor\frac{k\mu_x}{b_3}\right\rfloor}=\sum_{\overline{k}=1}^{\overline{b_3}-1}q^{\left\lfloor\frac{\overline{k}\mu_x}{\overline{b_3}}\right\rfloor}=\sum_{\overline{k}}q^{\left\lceil\frac{\overline{k}\mu_x}{\overline{b_3}}\right\rceil-1}=\sum_{\overline{k}}q^{j_{\overline{k}}\,-1},
\end{equation*}
which agrees with \eqref{finalcheckcasevier} for $t=0$.

In what follows, we assume that $t>0$ and hence that $b_3>\mu_x$. Define the numbers
\begin{equation*}
k_j=\min\left\{k\in\Zplus\;\middle\vert\;\left\lfloor\frac{k\mu_x}{b_3}\right\rfloor=j\right\}=\left\lceil\frac{jb_3}{\mu_x}\right\rceil;\qquad j=0,\ldots,\mu_x;
\end{equation*}
and note that
\begin{equation*}
0=k_0<k_1<\cdots<k_{\mu_x-1}<k_{\mu_x}=b_3.
\end{equation*}
This gives rise to
\begin{equation*}
\sum_{k=1}^{b_3-1}q^{\left\lfloor\frac{k\mu_x}{b_3}\right\rfloor}=\sum_{j=0}^{\mu_x-1}\sum_{k=k_j}^{k_{j+1}-1}q^j-1=\sum_{\overline{k}=0}^{\overline{b_3}-1}\sum_{j=j_{\overline{k}}}^{j_{\overline{k}+1}-1}(k_{j+1}-k_j)q^j-1.
\end{equation*}

Finally, observe that
\begin{equation*}
k_j=\left\lceil\frac{jb_3}{\mu_x}\right\rceil=\left\lceil\frac{j(t\mu_x+\overline{b_3})}{\mu_x}\right\rceil=jt+\left\lceil\frac{j\overline{b_3}}{\mu_x}\right\rceil=jt+\left\lfloor\frac{j\overline{b_3}}{\mu_x}\right\rfloor+1-[j=0]-[j=\mu_x]
\end{equation*}
for $j\in\{0,\ldots,\mu_x\}$; hence for $0\leqslant\overline{k}\leqslant\overline{b_3}-1$ and $j_{\overline{k}}\leqslant j\leqslant j_{\overline{k}+1}-1$, we have
\begin{align*}
&\;k_{j+1}-k_j\\
&=\bigl((j+1)t+(\overline{k}+[j+1=j_{\overline{k}+1}])+1-[j+1=\mu_x]\bigr)-\bigl(jt+\overline{k}+1-[j=0]\bigr)\\
&=t+[j=j_{\overline{k}+1}-1]+[j=0]-[j=\mu_x-1].
\end{align*}
Therefore,
\begin{align*}
\sum_{k=1}^{b_3-1}q^{\left\lfloor\frac{k\mu_x}{b_3}\right\rfloor}&=\sum_{\overline{k}=0}^{\overline{b_3}-1}\sum_{j=j_{\overline{k}}}^{j_{\overline{k}+1}-1}(k_{j+1}-k_j)q^j-1\\
&=\sum_{\overline{k}}\sum_j\bigl(t+[j=j_{\overline{k}+1}-1]+[j=0]-[j=\mu_x-1]\bigr)q^j-1\\
&=t\sum_{j=0}^{\mu_x-1}q^j+\sum_{\overline{k}=0}^{\overline{b_3}-1}q^{j_{\overline{k}+1}-1}+q^0-q^{\mu_x-1}-1\\
&=\sum_{\overline{k}=1}^{\overline{b_3}-1}q^{j_{\overline{k}}\,-1}+t\frac{F_3}{q-1},
\end{align*}
which agrees with \eqref{finalcheckcasevier}. This concludes Case~IV.

\section{Case~V: exactly two facets of \Gf\ contribute to $s_0$; one of them is a non-compact $B_1$-facet, the other one a $B_1$-simplex; these facets are $B_1$ with respect to a same variable and have an edge in common}
\subsection{Figure and notations}
We assume that the two facets $\tau_0$ and $\tau_1$ contributing to $s_0$ are both $B_1$-facets with respect to the variable $z$. Let $\tau_0$ be non-compact, say for the variable $x$, and let $\tau_1$ be a $B_1$-simplex. Facet $\tau_0$ shares its unique compact edge $[AC]$ with $\tau_1$. We denote the vertices of $\tau_0$ and $\tau_1$ and their coordinates by
\begin{equation*}
A(x_A,y_A,0),\quad B(x_B,y_B,0),\quad C(x_C,y_C,1)
\end{equation*}
and the neighbor facets of $\tau_0$ and $\tau_1$ by $\tau_2,\tau_3,\tau_4$ as indicated in Figure~\ref{figcase5}.

%
%
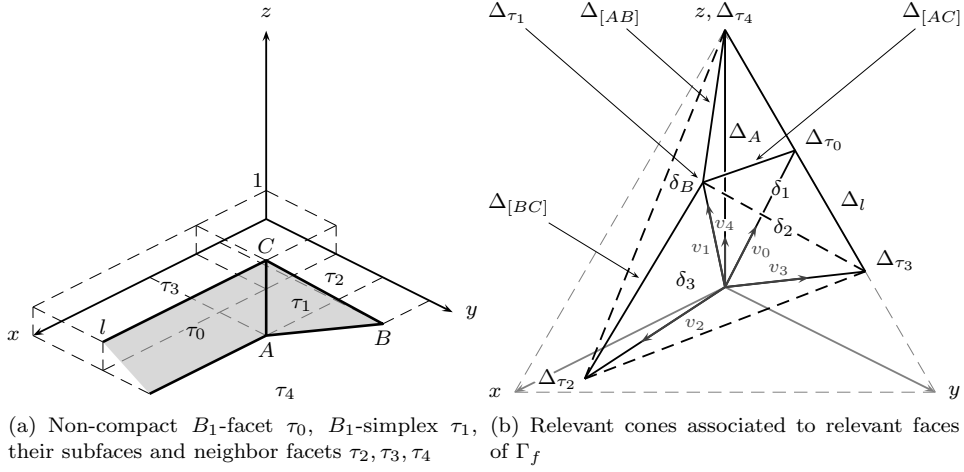
\begin{figure}
\psset{unit=.03464624362\textwidth}
\centering
\subfigure[Non-compact $B_1$-facet $\tau_0$, $B_1$-simplex $\tau_1$, their subfaces and neighbor facets $\tau_2,\tau_3,\tau_4$]{
\begin{pspicture}(-7.56,-5.5)(6.15,6.5)
{\footnotesize
\pstThreeDCoor[xMin=0,yMin=0,zMin=0,xMax=10,yMax=8,zMax=6.6,linecolor=black,linewidth=.7pt]
{
\psset{linecolor=black,linewidth=.3pt,linestyle=dashed,subticks=1}
\pstThreeDLine(3,0,0)(3,3,0)\pstThreeDLine(3,3,0)(0,3,0)
\pstThreeDLine(10,0,0)(10,3,0)\pstThreeDLine(10,3,0)(3,3,0)
\pstThreeDLine(0,7,0)(2,7,0)\pstThreeDLine(2,7,0)(2,0,0)
\pstThreeDLine(5,0,0)(5,5,0)\pstThreeDLine(5,5,0)(0,5,0)
\pstThreeDLine(10,3,0)(10,5,0)\pstThreeDLine(3,0,1)(3,3,1)
\pstThreeDLine(3,3,1)(0,3,1)\pstThreeDLine(10,0,1)(10,3,1)
\pstThreeDLine(0,0,1)(3,0,1)\pstThreeDLine(3,0,1)(10,0,1)
\pstThreeDLine(0,0,1)(0,3,1)\pstThreeDLine(3,0,0)(3,0,1)
\pstThreeDLine(10,0,0)(10,0,1)\pstThreeDLine(0,3,0)(0,3,1)
\pstThreeDLine(10,3,0)(10,3,1)
}
\pstThreeDPut[pOrigin=c](7,4,0.5){\psframebox*[framesep=0.8pt,framearc=0.3]{\phantom{$\tau_0$}}}
\pstThreeDPut[pOrigin=c](3.33,4.85,0.33){\psframebox*[framesep=0.3pt,framearc=0.3]{\phantom{$\tau_1$}}}
{
\psset{dotstyle=none,dotscale=1,drawCoor=false}
\psset{linecolor=black,linewidth=1pt,linejoin=1}
\psset{fillcolor=lightgray,opacity=.6,fillstyle=solid}
\pstThreeDLine(3,3,1)(2,7,0)(5,5,0)\pstThreeDLine(10,5,0)(5,5,0)(3,3,1)(10,3,1)
}
\pstThreeDPut[pOrigin=t](5,5,-0.25){$A$}
\pstThreeDPut[pOrigin=t](2,7,-0.24){$B$}
\pstThreeDPut[pOrigin=b](3,3,1.25){\psframebox*[framesep=.3pt,framearc=1]{$C$}}
\pstThreeDPut[pOrigin=c](7,4,0.5){$\tau_0$}
\pstThreeDPut[pOrigin=c](3.33,4.85,0.33){$\tau_1$}
\pstThreeDPut[pOrigin=lb](2.65,5.1,0.8){\psframebox*[framesep=0.5pt,framearc=0.3]{$\tau_2$}}
\pstThreeDPut[pOrigin=c](6.85,2.75,1.5){\psframebox*[framesep=0.4pt,framearc=0.3]{$\tau_3$}}
\pstThreeDPut[pOrigin=l](7.45,7.5,0){$\ \tau_4$}
\pstThreeDPut[pOrigin=b](10,3,1.27){$l$}
\pstThreeDPut[pOrigin=br](0.06,-0.06,1.12){$1$}
}
\end{pspicture}
}\hfill\subfigure[Relevant cones associated to relevant faces of~\Gf]{
\psset{unit=.03125\textwidth}
\begin{pspicture}(-7.6,-3.8)(7.6,9.6)
{\footnotesize
\pstThreeDCoor[xMin=0,yMin=0,zMin=0,xMax=10,yMax=10,zMax=10,nameZ={},linecolor=gray,linewidth=.7pt]
{
\psset{linecolor=gray,linewidth=.3pt,linejoin=1,linestyle=dashed,fillcolor=lightgray,fillstyle=none}
\pstThreeDLine(10,0,0)(0,10,0)\pstThreeDLine(0,10,0)(0,0,10)\pstThreeDLine(0,0,10)(10,0,0)
}
{
\psset{linecolor=black,linewidth=.7pt,linejoin=1,fillcolor=lightgray,fillstyle=none}
\pstThreeDLine(0,0,0)(0,0,10)
}
{
\psset{labelsep=2pt}
\uput[90](0,1.7){\psframebox*[framesep=0.3pt,framearc=1]{\darkgray\scriptsize$v_4$}}
}
{
\psset{linecolor=black,linewidth=.7pt,linejoin=1,fillcolor=lightgray,fillstyle=none}
\pstThreeDLine(0,0,0)(0,3.33,6.67)
\pstThreeDLine(0,0,0)(2.63,1.58,5.79)
\pstThreeDLine(0,0,0)(8.15,1.48,.370)
\pstThreeDLine(0,0,0)(0,6.67,3.33)
}
{
\psset{linecolor=darkgray,linewidth=.8pt,linejoin=1,arrows=->,arrowscale=1,fillcolor=lightgray,fillstyle=none}
\pstThreeDLine(0,0,0)(0,1.5,3)
\pstThreeDLine(0,0,0)(2.10,1.27,4.62)
\pstThreeDLine(0,0,0)(4.89,.888,.222)
\pstThreeDLine(0,0,0)(0,4,2)
\pstThreeDLine(0,0,0)(0,0,2)
}
{
\psset{linecolor=white,linewidth=2pt,linejoin=1,fillcolor=lightgray,fillstyle=none}
\pstThreeDLine(2.24,1.84,5.92)(1.32,2.45,6.24)
\pstThreeDLine(2.37,2.09,5.54)(1.05,4.63,4.32)
}
{
\psset{linecolor=black,linewidth=.7pt,linejoin=1,fillcolor=lightgray,fillstyle=none}
\pstThreeDLine(0,0,10)(0,3.33,6.67)
\pstThreeDLine(0,0,10)(2.63,1.58,5.79)
\pstThreeDLine(0,6.67,3.33)(0,3.33,6.67)(2.63,1.58,5.79)(8.15,1.48,.370)
}
{
\psset{linecolor=black,linewidth=.7pt,linejoin=1,linestyle=dashed,fillcolor=lightgray,fillstyle=none}
\pstThreeDLine(0,0,10)(8.15,1.48,.370)
\pstThreeDLine(8.15,1.48,.370)(0,6.67,3.33)
\pstThreeDLine(2.63,1.58,5.79)(0,6.67,3.33)
}
{
\psset{labelsep=2pt}
\uput[-30](.7,1.4){\darkgray\scriptsize$v_0$}
\uput[-60](-1.3,-.9){\darkgray\scriptsize$v_2$}
\uput[105](1.96,.224){\darkgray\scriptsize$v_3$}
}
{
\psset{labelsep=1.5pt}
\uput[202](-.3,1.4){\darkgray\scriptsize$v_1$}
}
{
\psset{labelsep=3.8pt}
\uput[30](2.35,4.58){$\Delta_{\tau_0}$}
\uput[30](3.54,2.55){$\Delta_l$}
\uput[30](4.73,.52){$\Delta_{\tau_3}$}
}
{
\psset{labelsep=3.0pt}
\rput(-6.1,-3.0){\psframebox*[framesep=0.6pt,framearc=1]{\phantom{$\Delta$}}}
\uput[180](-4.7,-3.1){$\Delta_{\tau_2}$}
}
\rput(.7,5.2){\psframebox*[framesep=0.3pt,framearc=1]{\footnotesize$\Delta_A$}}
\rput(-1.43,3.5){\footnotesize$\delta_B$}
\rput(1.9,3.3){\psframebox*[framesep=0.7pt,framearc=1]{\footnotesize$\delta_1$}}
\rput(1.98,2.02){\psframebox*[framesep=0.3pt,framearc=.3]{\footnotesize$\delta_2$}}
\rput(-1.3,.25){\footnotesize$\delta_3$}
\pstThreeDNode(2.63,1.58,5.79){dt1}
\pstThreeDNode(1.32,.790,7.90){AB}
\pstThreeDNode(1.32,2.45,6.24){AC}
\pstThreeDNode(5.40,1.53,3.08){BC}
\rput[Bl](-7.967,2.675){\rnode{BClabel}{$\Delta_{[BC]}$}}
\rput[Bl](-7.967,9.075){\rnode{dt1label}{$\Delta_{\tau_1}$}}
\rput[Bl](-4.88,9.075){\rnode{ABlabel}{$\Delta_{[AB]}$}}
\rput[B](0,9.075){$z,\Delta_{\tau_4}$}
\rput[Br](8.01,9.075){\rnode{AClabel}{$\Delta_{[AC]}$}}
\ncline[linewidth=.3pt,nodesepB=2.5pt,nodesepA=2pt]{->}{dt1label}{dt1}
\ncline[linewidth=.3pt,nodesepB=2pt,nodesepA=1pt]{->}{ABlabel}{AB}
\ncline[linewidth=.3pt,nodesepB=3.5pt,nodesepA=1pt]{->}{AClabel}{AC}
\ncline[linewidth=.3pt,nodesepB=2pt,nodesepA=2.5pt]{->}{BClabel}{BC}
}
\end{pspicture}
}
\caption{Case V: the only facets contributing to $s_0$ are the non-compact $B_1$-facet $\tau_0$ and the $B_1$-simplex $\tau_1$}
\label{figcase5}
\end{figure}
%
%

Let us put
\begin{gather*}
\begin{alignedat}{10}
&\overrightarrow{AC}&&(x_C&&-x_A&&,y_C&&-y_A&&,1&&)&&=(\aA&&,\bA&&,1),\\
&\overrightarrow{BC}&&(x_C&&-x_B&&,y_C&&-y_B&&,1&&)&&=(\aB&&,\bB&&,1),
\end{alignedat}\\
\text{and}\qquad\,\fAB=\gcd(x_B-x_A,y_B-y_A)
\end{gather*}
as before. The unique primitive vector $v_0\in\Zplus^3$ perpendicular to $\tau_0$ is given by $v_0(0,1,-\bA)$; such vectors for the other relevant facets $\tau_1,\tau_2,\tau_3,\tau_4$ will be denoted
\begin{equation*}
v_1(a_1,b_1,c_1),\quad v_2(a_2,b_2,c_2),\quad v_3(0,b_3,c_3),\quad v_4(0,0,1),
\end{equation*}
respectively. Equations for the affine supports of $\tau_i$; $i=0,\ldots,4$; are given by
\begin{alignat*}{3}
\aff(\tau_0)&\leftrightarrow&\;        y&\;-\;&\bA z&=y_A,\\
\aff(\tau_1)&\leftrightarrow&\;a_1x+b_1y&\;+\;& c_1z&=m_1,\\
\aff(\tau_2)&\leftrightarrow&\;a_2x+b_2y&\;+\;& c_2z&=m_2,\\
\aff(\tau_3)&\leftrightarrow&\;     b_3y&\;+\;& c_3z&=m_3,\\
\aff(\tau_4)&\leftrightarrow&\;         &     &    z&=0
\end{alignat*}
for certain $m_1,m_2,m_3\in\Zplus$, and to these facets we associate the respective numerical data
\begin{equation*}
(y_A,1-\bA),\quad (m_1,\sigma_1),\quad (m_2,\sigma_2),\quad (m_3,\sigma_3),\quad (0,1),
\end{equation*}
with $\sigma_i=a_i+b_i+c_i$; $i=1,2$; and $\sigma_3=b_3+c_3$.

Since we assume that $\tau_0$ and $\tau_1$ both contribute to the candidate pole $s_0$, we have that $p^{1-\bA+y_As_0}=p^{\sigma_1+m_1s_0}=1$; hence
\begin{equation*}
\Re(s_0)=\frac{\bA-1}{y_A}=-\frac{\sigma_1}{m_1}\quad\ \text{and}\ \quad\Im(s_0)=\frac{2n\pi}{\gcd(y_A,m_1)\log p}\quad\ \text{for some $n\in\Z$.}
\end{equation*}

\subsection{Contributions to the candidate pole $s_0$}
The goal of this section is again to show that both
\begin{align*}
R_2&=\lim_{s\to s_0}\left(p^{1-\bA+y_As}-1\right)\left(p^{\sigma_1+m_1s}-1\right)\Zof(s)\qquad\text{and}\\
R_1&=\lim_{s\to s_0}\frac{d}{ds}\left[\left(p^{1-\bA+y_As}-1\right)\left(p^{\sigma_1+m_1s}-1\right)\Zof(s)\right]
\end{align*}
equal zero. The compact faces of \Gf\ contributing to $s_0$ are again the (seven) compact subfaces $A,B,C,[AB],[AC],[BC],$ and $\tau_1$ of the two contributing facets $\tau_0$ and $\tau_1$. Only three of them also contribute to the \lq residue\rq\ $R_2$: $A,C,$ and $[AC]$.

If we consider the nine simplicial cones
\begin{align*}
\Dteen&=\cone(v_1),        &\delta_1&=\cone(v_0,v_1,v_3),&\DAB&=\cone(v_1,v_4),\\
   \DA&=\cone(v_0,v_1,v_4),&\delta_2&=\cone(v_1,v_3),    &\DAC&=\cone(v_0,v_1),\\
   \dB&=\cone(v_1,v_2,v_4),&\delta_3&=\cone(v_1,v_2,v_3),&\DBC&=\cone(v_1,v_2),
\end{align*}
the same approach as in Cases~III and IV leads to the following expressions for $R_2$ and $R_1$:
\begin{gather*}
\begin{aligned}
R_2&=L_A(s_0)\frac{\Sigma(\Delta_A)(s_0)}{p-1}
+L_C(s_0)\frac{\Sigma(\delta_1)(s_0)}{p^{\sigma_3+m_3s_0}-1}
+L_{[AC]}(s_0)\Sigma(\Delta_{[AC]})(s_0),\\
R_1&=L_A'(s_0)\frac{\Sigma(\Delta_A)(s_0)}{p-1}
+L_A(s_0)\frac{\Sigma(\Delta_A)'(s_0)}{p-1}
+L_B(s_0)\frac{y_A(\log p)\Sigma(\delta_B)(s_0)}{\Ftwee(p-1)}
\end{aligned}\\
+L_C'(s_0)\frac{\Sigma(\delta_1)(s_0)}{p^{\sigma_3+m_3s_0}-1}
+L_C(s_0)\frac{\Sigma(\delta_1)'(s_0)}{p^{\sigma_3+m_3s_0}-1}\\
-L_C(s_0)\frac{m_3(\log p)p^{\sigma_3+m_3s_0}\Sigma(\delta_1)(s_0)}{\Fdrie^2}
+L_C(s_0)\frac{y_A(\log p)\Sigma(\delta_2)(s_0)}{p^{\sigma_3+m_3s_0}-1}\\
+L_C(s_0)\frac{y_A(\log p)\Sigma(\delta_3)(s_0)}{\Ftwee\Fdrie}
+L_{[AB]}(s_0)\frac{y_A(\log p)\Sigma(\Delta_{[AB]})(s_0)}{p-1}\\
+L_{[AC]}'(s_0)\Sigma(\Delta_{[AC]})(s_0)
+L_{[AC]}(s_0)\Sigma(\Delta_{[AC]})'(s_0)\\
+L_{[BC]}(s_0)\frac{y_A(\log p)\Sigma(\Delta_{[BC]})(s_0)}{p^{\sigma_2+m_2s_0}-1}
+L_{\tau_1}(s_0)y_A(\log p)\Sigma(\Delta_{\tau_1})(s_0).
\end{gather*}

\subsection{Towards simplified formulas for $R_2$ and $R_1$}
\subsubsection{The factors $L_{\tau}(s_0)$ and $L_{\tau}'(s_0)$}
In the usual way we obtain
\begin{align*}
L_A(s_0)=L_B(s_0)=L_C(s_0)&=\left(\frac{p-1}{p}\right)^3,\qquad L_A'(s_0)=L_C'(s_0)=0,\\
L_{[AB]}(s_0)&=\left(\frac{p-1}{p}\right)^3-\frac{(p-1)N}{p^2}\frac{p^{s_0}-1}{p^{s_0+1}-1},\\
L_{[AC]}(s_0)=L_{[BC]}(s_0)&=\left(\frac{p-1}{p}\right)^3-\left(\frac{p-1}{p}\right)^2\frac{p^{s_0}-1}{p^{s_0+1}-1},\\
L_{[AC]}'(s_0)&=-(\log p)\left(\frac{p-1}{p}\right)^3\frac{p^{s_0+1}}{\bigl(p^{s_0+1}-1\bigr)^2},\\
\text{and}\qquad L_{\tau_1}(s_0)&=\left(\frac{p-1}{p}\right)^3-\frac{(p-1)^2-N}{p^2}\frac{p^{s_0}-1}{p^{s_0+1}-1},
\end{align*}
with
\begin{equation*}
N=\#\left\{(x,y)\in(\Fpcross)^2\;\middle\vert\;\overline{f_{[AB]}}(x,y)=0\right\}.
\end{equation*}

\subsubsection{Cone multiplicities}
Let us investigate the multiplicities of the nine contributing simplicial cones. As we did before, we shall also consider the multiplicities $\mu_l$ and $\mu_1'$ of the respective simplicial cones $\Delta_l$ and $\delta_1'=\cone(v_0,v_1,v_2)$; the first cone is the cone associated to the half-line $l=\tau_0\cap\tau_3$ (see Figure~\ref{figcase5}), while the second one is a simplicial subcone of $\Delta_C$ that could have been chosen as a member of an alternative subdivision of $\Delta_C$. Proceeding as in the previous cases, we find
\begin{gather*}
\mult\Delta_{[AB]}=\mult\Delta_{\tau_1}=1,\\
\begin{alignedat}{5}
\mu_A&=\mult\Delta_A&&=&\;\#H(v_0,v_1,v_4)&=\mult\Delta_{[AC]}&&=\#H(v_0,v_1)&&=a_1,\\
\mu_B&=\mult\delta_B&&=&\;\#H(v_1,v_2,v_4)&=\mult\Delta_{[BC]}&&=\#H(v_1,v_2)&&=-
\begin{vmatrix}
a_1&b_1\\a_2&b_2
\end{vmatrix},\\
&&&&\;\mu_l&=\mult\Delta_l&&=\#H(v_0,v_3)&&=-
\begin{vmatrix}
1&-\bA\\b_3&c_3
\end{vmatrix},
\end{alignedat}\\
\begin{alignedat}{4}
\mu_1&=\mult\delta_1\ \;&&=\#H(v_0,v_1,v_3)&&=
\begin{vmatrix}
0&1&-\bA\\a_1&b_1&c_1\\0&b_3&c_3
\end{vmatrix}
&&=-a_1
\begin{vmatrix}
1&-\bA\\b_3&c_3
\end{vmatrix}=\mu_A\mu_l,\\
\mu_1'&=\mult\delta_1'&&=\#H(v_0,v_1,v_2)&&=
\begin{vmatrix}
0&1&-\bA\\a_1&b_1&c_1\\a_2&b_2&c_2
\end{vmatrix}
&&=\mu_A\mu_B\fAB.
\end{alignedat}\quad\ \,
\end{gather*}

We observe that
\begin{equation*}
\mu_1=\#H(v_0,v_1,v_3)=\mu_A\mu_l=\#H(v_0,v_1)\#H(v_0,v_3);
\end{equation*}
i.e., the factor $\phi_{Al}=\mu_1/\mu_A\mu_l$ equals one. Theorem~\ref{algfp}(v) now asserts that
\begin{equation*}
\mu_2=\mult\delta_2=\#H(v_1,v_3)=\gcd(\mu_A,\mu_l).
\end{equation*}

For
\begin{equation*}
\mu_3=\mult\delta_3=\#H(v_1,v_2,v_3)=
\begin{vmatrix}
a_1&b_1&c_1\\a_2&b_2&c_2\\0&b_3&c_3
\end{vmatrix}>0,
\end{equation*}
finally, we obtain in a similar way as in Case~III that
\begin{align*}
\mu_3&=-b_3
\begin{vmatrix}
a_1&c_1\\a_2&c_2
\end{vmatrix}
+c_3
\begin{vmatrix}
a_1&b_1\\a_2&b_2
\end{vmatrix}\\
&=
\begin{vmatrix}
a_1&b_1\\a_2&b_2
\end{vmatrix}
(b_3\bB+c_3)\\
&=-\mu_B\left(v_3\cdot\overrightarrow{BC}\right)\\
&=\mu_B\left(v_3\cdot\overrightarrow{AB}-v_3\cdot\overrightarrow{AC}\right)\\
&=\mu_B\bigl(v_3\cdot\fAB(-b_1,a_1,0)-v_3\cdot(\aA,\bA,1)\bigr)\\
&=\mu_B\bigl(\fAB\Psi+\mu_l\bigr)\\
&=\mu_B\mu_2\fBtwee,
\end{align*}
whereby
\begin{gather*}
\Psi=a_1b_3=b_3\mu_A,\qquad\fBtwee=\fAB\psi+\mu_l',\\
\psi=\frac{\Psi}{\mu_2}=b_3\mu_A'\in\Zplusnul,\qquad\mu_A'=\frac{\mu_A}{\mu_2}\in\Zplusnul,\qquad\text{and}\qquad\mu_l'=\frac{\mu_l}{\mu_2}\in\Zplusnul.
\end{gather*}
Note that the coprimality of $b_3$ and $c_3$ implies the coprimality of $b_3$ and $\mu_l$. Hence
\begin{equation*}
\mu_2=\gcd(\mu_A,\mu_l)=\gcd(\Psi,\mu_l)
\end{equation*}
and $\gcd(\psi,\mu_l')=\gcd(\psi,\fBtwee)=1$.

\subsubsection{The sums $\Sigma(\cdot)(s_0)$ and $\Sigma(\cdot)'(s_0)$}
We have of course that $\Sigma(\Delta_{[AB]})(s_0)=\Sigma(\Delta_{\tau_1})(s_0)=1$. As usual, we denote
\begin{alignat*}{3}
H_A&=H(v_0,v_1,v_4)&&=H(v_0,v_1),&\qquad\qquad\quad H_1&=H(v_0,v_1,v_3),\\
H_B&=H(v_1,v_2,v_4)&&=H(v_1,v_2),&                  H_2&=H(v_1,v_3),\\
H_l&=H(v_0,v_3),   &&            &                  H_3&=H(v_1,v_2,v_3),
\end{alignat*}
and $w=(1,1,1)+s_0(x_C,y_C,1)\in\C^3$, yielding
\begin{gather*}
\begin{alignedat}{3}
\Sigma_A &=\Sigma(\Delta_A)(s_0) &&=\Sigma(\Delta_{[AC]})(s_0) &&=\sum\nolimits_{h\in H_A}p^{w\cdot h};\\
\Sigma_B &=\Sigma(\delta_B)(s_0) &&=\Sigma(\Delta_{[BC]})(s_0) &&=\sum\nolimits_{h\in H_B}p^{w\cdot h};\\
\Sigma_A'&=\Sigma(\Delta_A)'(s_0)&&=\Sigma(\Delta_{[AC]})'(s_0)&&=(\log p)\sum\nolimits_{h\in H_A}m(h)p^{w\cdot h};
\end{alignedat}\\
\begin{alignedat}{2}
\Sigma_l &=\Sigma(\Delta_l)(s_0) &&=\sum\nolimits_{h\in H_l}p^{w\cdot h};\\
\Sigma_i &=\Sigma(\delta_i)(s_0) &&=\sum\nolimits_{h\in H_i}p^{w\cdot h};\qquad i=1,2,3;\\
\Sigma_1'&=\Sigma(\delta_1)'(s_0)&&=(\log p)\sum\nolimits_{h\in H_1}m(h)p^{w\cdot h}.
\end{alignedat}
\end{gather*}

\subsubsection{New formulas for the residues}
Let us put
\begin{gather*}
R_2=\left(\frac{p-1}{p}\right)^3R_2',\qquad R_1=(\log p)\left(\frac{p-1}{p}\right)^3R_1',\\
F_2=p^{\sigma_2+m_2s_0}-1,\qquad F_3=p^{\sigma_3+m_3s_0}-1,\qquad\text{and}\qquad q=p^{-s_0-1}.
\end{gather*}
Our findings so far lead to the following expressions for $R_2'$ and $R_1'$:
\begin{gather}
R_2'=\frac{\Sigma_A}{1-q}+\frac{\Sigma_1}{F_3},\label{formR2accasevijf}\\
\begin{multlined}[.8\textwidth]
R_1'=\frac{1}{1-q}\left(\frac{\Sigma_A'}{\log p}-\frac{\Sigma_A}{q^{-1}-1}+\frac{y_A\Sigma_B}{F_2}+y_A\right)\\
+\frac{\Sigma_1'}{(\log p)F_3}-\frac{m_3(F_3+1)\Sigma_1}{F_3^2}+\frac{y_A\Sigma_2}{F_3}+\frac{y_A\Sigma_3}{F_2F_3}.
\end{multlined}\label{formR1accasevijf}
\end{gather}

We prove that $R_2'=R_1'=0$.

\subsection{Investigation of the sums $\Sigma_{\bullet}$ and $\Sigma_{\bullet}'$}
\subsubsection{Vector identities and consequences}
The identities that will be useful to us in this case are
\begin{align}
b_3v_0-v_3&=(0,0,\mu_l),\label{vi1c5}\\
-\mu_Bv_0+a_2v_1-\mu_Av_2&=(0,0,\mu_1'),\qquad\text{and}\label{vi2c5}\\
\Theta v_1-\Psi v_2-\mu_Bv_3&=(0,0,\mu_3),\notag
\end{align}
whereby $\Theta=a_2b_3$ and $\Psi=a_1b_3$. These give rise to
\begin{equation}\label{interpbdrieyammdriec5}
y_Ab_3-m_3=\mu_l
\end{equation}
and to
\begin{align}
-b_3(w\cdot v_0)+w\cdot v_3&=\mu_l(-s_0-1),\label{dpi1c5}\\
\mu_B(w\cdot v_0)-a_2(w\cdot v_1)+\mu_A(w\cdot v_2)&=\mu_1'(-s_0-1),\qquad\text{and}\label{dpi2c5}\\
-\Theta(w\cdot v_1)+\Psi(w\cdot v_2)+\mu_B(w\cdot v_3)&=\mu_3(-s_0-1).\label{dpi3c5}
\end{align}
Moreover, they show that
\begin{equation}\label{bpsc5}
\frac{-b_3}{\mu_l}v_0+\frac{1}{\mu_l}v_3\in\Z^3\quad\text{ and }\quad\frac{-\Theta}{\mu_3}v_1+\frac{\psi}{\mu_B\fBtwee}v_2+\frac{1}{\mu_2\fBtwee}v_3\in\Z^3.
\end{equation}
(Recall that $\mu_3=\mu_B\mu_2\fBtwee$ and $\Psi=\psi\mu_2$.)

\subsubsection{Points of $H_A,H_B,H_l,H_1,H_2,$ and $H_3$}
The $\mu_A$ points of $H_A$ are given by
\begin{alignat}{2}
\left\{\frac{i\xi_A}{\mu_A}\right\}v_0+\frac{i}{\mu_A}v_1;&&\qquad i&=0,\ldots,\mu_A-1;\notag\\
\intertext{or, alternatively, by}
\frac{i}{\mu_A}v_0+\left\{\frac{i\xi_A'}{\mu_A}\right\}v_1;&&\qquad i&=0,\ldots,\mu_A-1;\label{pofAaccentcasevijf}\\
\intertext{for certain $\xi_A,\xi_A'\in\verA$ with $\xi_A\xi_A'\equiv1\bmod\mu_A$. By \eqref{bpsc5}, the $\mu_l$ points of $H_l$ are}
\left\{\frac{-jb_3}{\mu_l}\right\}v_0+\frac{j}{\mu_l}v_3;&&\qquad j&=0,\ldots,\mu_l-1;\label{poflcasevijf}\\
\intertext{while those of $H_B$ and $H_2$ are given by}
\left\{\frac{i\xi_B}{\mu_B}\right\}v_1+\frac{i}{\mu_B}v_2;&&\qquad i&=0,\ldots,\mu_B-1;\label{pofBcasevijf}\\
\shortintertext{and by}
\left\{\frac{j\xi_2}{\mu_2}\right\}v_1+\frac{j}{\mu_2}v_3;&&\qquad j&=0,\ldots,\mu_2-1;\label{pof2casevijf}
\end{alignat}
respectively, for unique $\xi_B\in\verB$ and $\xi_2\in\vertwee$, coprime to $\mu_B$ and $\mu_2$, respectively.

Since $\mu_1=\mu_A\mu_l$, the description of the $\mu_1$ points of $H_1$ is rather easy:
\begin{equation}\label{pof1casevijf}
\left\{\frac{i\xi_A\mu_l-jb_3\mu_A}{\mu_1}\right\}v_0+\frac{i}{\mu_A}v_1+\frac{j}{\mu_l}v_3;\qquad i=0,\ldots,\mu_A-1;\quad j=0,\ldots,\mu_l-1.
\end{equation}

Based on (\ref{pofBcasevijf}--\ref{pof2casevijf}) and \eqref{bpsc5}, we find in exactly the same way as in Case~III (Paragraph~\ref{ssspoH3c3}) a complete list of the $\mu_3=\mu_B\mu_2\fBtwee$ points of $H_3$:
\begin{multline}\label{pof3casevijf}
\left\{\frac{(i-\lfloor k\psi/\fBtwee\rfloor)\xi_B\mu_2\fBtwee+j\xi_2\mu_B\fBtwee-k\Theta}{\mu_3}\right\}v_1\\
\shoveright{+\frac{i\fBtwee+\{k\psi\}_{\fBtwee}}{\mu_B\fBtwee}v_2+\frac{j\fBtwee+k}{\mu_2\fBtwee}v_3;}\\
i=0,\ldots,\mu_B-1;\quad j=0,\ldots,\mu_2-1;\quad k=0,\ldots,\fBtwee-1.
\end{multline}

\subsubsection{Formulas for $\Sigma_A,\Sigma_A',\Sigma_B,\Sigma_l,\Sigma_1,\Sigma_1',\Sigma_2,$ and $\Sigma_3$}As in Case~III, for some of the sums $\Sigma_{\bullet}$ and $\Sigma_{\bullet}'$, we will have to distinguish between two cases. Let us put
\begin{gather*}
n^{\ast}=\frac{\gcd(y_A,m_1)}{\gcd(y_A,m_1,m(h^{\ast}))},\\
\text{with}\qquad m(h^{\ast})=\frac{\xi_Ay_A+m_1}{\mu_A}\in\Zplusnul\qquad\text{and}\qquad h^{\ast}=\frac{\xi_A}{\mu_A}v_0+\frac{1}{\mu_A}v_1\in\Z^3,
\end{gather*}
a generating element of the group $H_A$ (if $\mu_A>1$). Then we have that
\begin{equation*}
p^{\frac{\xi_A(w\cdot v_0)+w\cdot v_1}{\mu_A}}\qquad\qquad\text{and}\qquad\qquad\ p^{\frac{w\cdot v_0+\xi_A'(w\cdot v_1)}{\mu_A}}
\end{equation*}
both equal one if $n^{\ast}\mid n$, while they both differ from one if $n^{\ast}\nmid n$. Proceeding as in Paragraphs~\ref{sssSigmaBc3}--\ref{sssSigmaBaccentc3}, we obtain that
\begin{align}
\Sigma_A&=
\begin{dcases*}
\mu_A,&if $n^{\ast}\mid n$;\\
0,&otherwise;
\end{dcases*}\qquad\qquad\text{and}\label{formSigAc5}\\
\frac{\Sigma_A'}{\log p}&=
\begin{dcases*}
\frac{(y_A+m_1)(\mu_A-1)}{2},&if $n^{\ast}\mid n$;\\[+.4ex]
\frac{y_A}{p^{\frac{w\cdot v_0+\xi_A'(w\cdot v_1)}{\mu_A}}-1}+
\frac{m_1}{p^{\frac{\xi_A(w\cdot v_0)+w\cdot v_1}{\mu_A}}-1},&otherwise.
\end{dcases*}\label{formSigAaccentc5}
\end{align}

We continue as in Paragraph~\ref{sssSASCS2c3}. Based on \eqref{pofBcasevijf} and $p^{w\cdot v_1}=1$, we find that
\begin{equation}\label{formsigmaBcase5}
\Sigma_B=
\begin{dcases*}
\frac{F_2}{p^{\frac{\xi_B(w\cdot v_1)+w\cdot v_2}{\mu_B}}-1},&in any case;\\
\frac{F_2}{q^{\fAB}-1},&if $n^{\ast}\mid n$.
\end{dcases*}
\end{equation}
The special formula for $n^{\ast}\mid n$ arises from
\begin{equation}\label{sigbsigaidcasevijf}
p^{\frac{w\cdot v_0+\xi_A'(w\cdot v_1)}{\mu_A}}p^{\frac{\xi_B(w\cdot v_1)+w\cdot v_2}{\mu_B}}=p^{\fAB(-s_0-1)}=q^{\fAB},
\end{equation}
which in turn follows from $v_0+\xi_A'v_1\in\mu_A\Z^3$, $\xi_Bv_1+v_2\in\mu_B\Z^3$, \eqref{vi2c5}, and \eqref{dpi2c5}. For $\Sigma_l$ we use \eqref{poflcasevijf}, $p^{w\cdot v_0}=1$, and \eqref{dpi1c5} in order to conclude
\begin{equation*}
\Sigma_l=\sum_{h\in H_l}p^{w\cdot h}=\sum_{j=0}^{\mu_l-1}\Bigl(p^{\frac{-b_3(w\cdot v_0)+w\cdot v_3}{\mu_l}}\Bigr)^j=\sum_jq^j=\frac{F_3}{q-1}.
\end{equation*}

By \eqref{pofAaccentcasevijf}, \eqref{pof2casevijf}, and \eqref{vi1c5} we have that
\begin{equation*}
v_0+\xi_A'v_1\in\mu_A\Z^3,\qquad\xi_2v_1+v_3\in\mu_2\Z^3,\qquad\text{and}\qquad b_3v_0-v_3\in\mu_l\Z^3.
\end{equation*}
Since $\mu_2=\gcd(\mu_A,\mu_l)$, we obtain
\begin{equation*}
-b_3(v_0+\xi_A'v_1)+(\xi_2v_1+v_3)+(b_3v_0-v_3)=(-\xi_A'b_3+\xi_2)v_1\in\mu_2\Z^3,
\end{equation*}
and hence
\begin{equation*}
\frac{-\xi_A'b_3+\xi_2}{\mu_2}\in\Z.
\end{equation*}
Using $\psi\mu_2=\Psi=b_3\mu_A$, $p^{w\cdot v_1}=1$, \eqref{dpi1c5}, and $\mu_l/\mu_2=\mu_l'\in\Zplusnul$, it then follows that
\begin{multline}\label{extraidentiteitc5}
\Bigl(p^{\frac{w\cdot v_0+\xi_A'(w\cdot v_1)}{\mu_A}}\Bigr)^{-\psi}p^{\frac{\xi_2(w\cdot v_1)+w\cdot v_3}{\mu_2}}
=p^{\frac{-\Psi(w\cdot v_0)+(-\xi_A'\Psi+\xi_2\mu_A)(w\cdot v_1)+\mu_A(w\cdot v_3)}{\mu_A\mu_2}}\\
=p^{\frac{-b_3(w\cdot v_0)+(-\xi_A'b_3+\xi_2)(w\cdot v_1)+w\cdot v_3}{\mu_2}}
=p^{\frac{-b_3(w\cdot v_0)+w\cdot v_3}{\mu_2}}=p^{\frac{\mu_l(-s_0-1)}{\mu_2}}=q^{\mu_l'}.
\end{multline}
In this way \eqref{pof2casevijf} yields
\begin{equation}\label{formsigma2case5}
\Sigma_2=\sum_{h\in H_2}p^{w\cdot h}=\sum_{j=0}^{\mu_2-1}\Bigl(p^{\frac{\xi_2(w\cdot v_1)+w\cdot v_3}{\mu_2}}\Bigr)^j=
\begin{dcases*}
\frac{F_3}{p^{\frac{\xi_2(w\cdot v_1)+w\cdot v_3}{\mu_2}}-1},&in any case;\\
\frac{F_3}{q^{\mu_l'}-1},&if $n^{\ast}\mid n$.
\end{dcases*}
\end{equation}

Keeping in mind that $p^{w\cdot v_0}=1$, we easily find from the description \eqref{pof1casevijf} of the points of $H_1$ that
\begin{align}
\Sigma_1=\sum_{h\in H_1}p^{w\cdot h}
&=\sum_{i=0}^{\mu_A-1}\sum_{j=0}^{\mu_l-1}p^{\frac{i\xi_A\mu_l-jb_3\mu_A}{\mu_1}(w\cdot v_0)+\frac{i}{\mu_A}(w\cdot v_1)+\frac{j}{\mu_l}(w\cdot v_3)}\notag\\
&=\sum_i\Bigl(p^{\frac{\xi_A(w\cdot v_0)+w\cdot v_1}{\mu_A}}\Bigr)^i\sum_j\Bigl(p^{\frac{-b_3(w\cdot v_0)+w\cdot v_3}{\mu_l}}\Bigr)^j\notag\\
&=\Sigma_A\Sigma_l=\begin{dcases*}
\frac{\mu_AF_3}{q-1},&if $n^{\ast}\mid n$;\\[+.3ex]
0,&otherwise.
\end{dcases*}\label{formsigma1case5}
\end{align}

To calculate $\Sigma_1'$ we follow the same process as in Case~III. Write $\psi$ as $\psi=t\mu_l'+\barpsi$ with $t\in\Zplus$ and $\barpsi=\{\psi\}_{\mu_l'}$. Clearly $\barpsi\in\{0,\ldots,\mu_l'-1\}$, and since $\gcd(\psi,\mu_l')=1$, we also have $\gcd(\barpsi,\mu_l')=1$. Hence $\barpsi=0$ occurs if and only if $\mu_l'=1$. Exclusively in the case that $\mu_l'>1$ we also introduce the numbers
\begin{equation*}
\kappa_{\rho}=\min\left\{\kappa\in\Zplus\;\middle\vert\;\left\lfloor\frac{\kappa\barpsi}{\mu_l'}\right\rfloor=\rho\right\}=\left\lceil\frac{\rho\mu_l'}{\barpsi}\right\rceil;\qquad\rho=0,\ldots,\barpsi.
\end{equation*}
Proceeding as in Subsection~\ref{studysigmaeenacc3} and applying \eqref{interpbdrieyammdriec5} in the end, we eventually obtain that
\begin{equation}\label{formsigma1accentcase5}
\frac{\Sigma_1'}{(\log p)F_3}=
\begin{dcases*}
\begin{multlined}[b][.68\textwidth]
\frac{1}{1-q}\Biggl(\frac{y_A}{q^{\mu_l'}-1}\sum_{\rho=1}^{\barpsi}q^{\kappa_{\rho}}-\frac{(y_A+m_1)(\mu_A-1)}{2}\\
+\frac{y_At-\mu_A}{1-q^{-1}}-\frac{m_3\mu_A(F_3+1)}{F_3}-y_A\Biggr)-\frac{y_A}{q^{\mu_l'}-1},
\end{multlined}
&\!\text{if $n^{\ast}\mid n$;}\\[+.5ex]
\begin{multlined}[b][.68\textwidth]
\frac{y_Ap^{\frac{w\cdot v_0+\xi_A'(w\cdot v_1)}{\mu_A}}}{\Bigl(p^{\frac{w\cdot v_0+\xi_A'(w\cdot v_1)}{\mu_A}}-1\Bigr)\Bigl(p^{\frac{\xi_2(w\cdot v_1)+w\cdot v_3}{\mu_2}}-1\Bigr)}\cdot\\
\shoveright{\sum_{\kappa=0}^{\mu_l'-1}q^{\kappa}\Bigl(p^{\frac{w\cdot v_0+\xi_A'(w\cdot v_1)}{\mu_A}}\Bigr)^{\left\lfloor\frac{\kappa\psi}{\mu_l'}\right\rfloor}}\\
-\frac{y_A}{p^{\frac{\xi_2(w\cdot v_1)+w\cdot v_3}{\mu_2}}-1}+\frac{m_1}{(q-1)\Bigl(p^{\frac{\xi_A(w\cdot v_0)+w\cdot v_1}{\mu_A}}-1\Bigr)},
\end{multlined}
&\!\text{if $n^{\ast}\nmid n$;}
\end{dcases*}
\end{equation}
thereby adopting the convention that the empty sum over $\rho$ equals zero if $\mu_l'=1$.

From the description \eqref{pof3casevijf} of the points of $H_3$, it is reasonable that also the calculation of $\Sigma_3$ is essentially not different from the one in Case~III; proceeding as in Paragraph~\ref{ssscalcsigmadriecasedrie}, thereby using Identity~\eqref{dpi3c5}, we find that
\begin{equation}\label{formuleSigmadriednc5}
\Sigma_3=\frac{F_2F_3}{\Bigl(p^{\frac{\xi_B(w\cdot v_1)+w\cdot v_2}{\mu_B}}-1\Bigr)\Bigl(p^{\frac{\xi_2(w\cdot v_1)+w\cdot v_3}{\mu_2}}-1\Bigr)}\sum_{k=0}^{\fBtwee-1}q^k\Bigl(p^{\frac{\xi_B(w\cdot v_1)+w\cdot v_2}{\mu_B}}\Bigr)^{-\left\lfloor\frac{k\psi}{\fBtwee}\right\rfloor}.
\end{equation}
A simplified version of this formula, valid in the case that $n^{\ast}\mid n$ and justified by Equalities \eqref{sigbsigaidcasevijf} and \eqref{extraidentiteitc5}, is given by
\begin{equation}\label{formuleSigmadriedeeltc5}
\Sigma_3=\frac{F_2F_3}{(q^{\fAB}-1)(q^{\mu_l'}-1)}\sum_{k=0}^{\fBtwee-1}q^{k-\fAB\left\lfloor\frac{k\psi}{\fBtwee}\right\rfloor}.
\end{equation}

\subsection{Proof of $R_2'=R_1'=0$}
On the one hand, it is clear from (\ref{formR2accasevijf}, \ref{formSigAc5}, and \ref{formsigma1case5}) that $R_2'=0$ in any case. On the other hand, if we fill in Formulas (\ref{formSigAc5}--\ref{formsigmaBcase5}, \ref{formsigma2case5}--\ref{formuleSigmadriedeeltc5}) for the $\Sigma(\cdot)(s_0)$ and the $\Sigma(\cdot)'(s_0)$ in Expression~\eqref{formR1accasevijf} for $R_1'$, we see that proving $R_1'=0$ comes down to verifying that
\begin{equation*}
\sum_{k=0}^{\fBtwee-1}q^{k-\fAB\left\lfloor\frac{k\psi}{\fBtwee}\right\rfloor}=\frac{q^{\fAB}-1}{q-1}\sum_{\rho=1}^{\barpsi}q^{\kappa_{\rho}}+\frac{q^{\mu_l'}-1}{q-1}\left(tq\frac{q^{\fAB}-1}{q-1}+1\right),
\end{equation*}
if $n^{\ast}\mid n$, and
\begin{multline*}
\Bigl(p^{\frac{w\cdot v_0+\xi_A'(w\cdot v_1)}{\mu_A}}-1\Bigr)\sum_{k=0}^{\fBtwee-1}q^k\Bigl(p^{\frac{\xi_B(w\cdot v_1)+w\cdot v_2}{\mu_B}}\Bigr)^{-\left\lfloor\frac{k\psi}{\fBtwee}\right\rfloor}\\*
+p^{\frac{w\cdot v_0+\xi_A'(w\cdot v_1)}{\mu_A}}\Bigl(p^{\frac{\xi_B(w\cdot v_1)+w\cdot v_2}{\mu_B}}-1\Bigr)\sum_{\kappa=0}^{\mu_l'-1}q^{\kappa}\Bigl(p^{\frac{w\cdot v_0+\xi_A'(w\cdot v_1)}{\mu_A}}\Bigr)^{\left\lfloor\frac{\kappa\psi}{\mu_l'}\right\rfloor}\\*
=\Bigl(p^{\frac{\xi_2(w\cdot v_1)+w\cdot v_3}{\mu_2}}-1\Bigr)\frac{q^{\fAB}-1}{q-1},
\end{multline*}
otherwise. In order to obtain this last equation, we need to apply Identity~\eqref{sigbsigaidcasevijf} at some point. Since the analogous relations between the variables hold, e.g.,
\begin{equation*}
\fBtwee=\fAB\psi+\mu_l'\qquad\text{and}\qquad\psi=t\mu_l'+\barpsi,
\end{equation*}
these final assertions can be proved in exactly the same way as in Subsection~\ref{ssfinalsscasedrie} of Case~III. Hence we conclude Case~V.

\section{Case~VI: at least three facets of \Gf\ contribute to $s_0$; all of them are $B_1$-facets (compact or not) with respect to a same variable and they are \lq connected to each other by edges\rq}
More precisely, we mean that we can denote the contributing $B_1$-facets by $\tau_0,\tau_1,\ldots,\tau_t$ with $t\geqslant2$ in such a way that facets $\tau_{i-1}$ and $\tau_i$ share an edge for all $i\in\{1,\ldots,t\}$. An example with $t=2$ is shown in Figure~\ref{figcase6}.

%
%
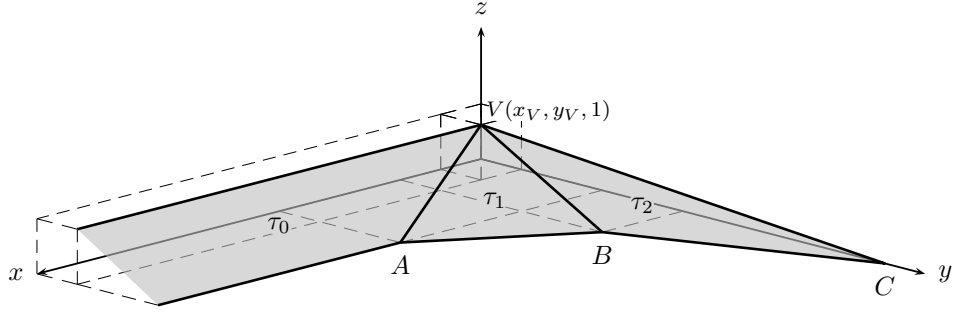
\begin{figure}
\psset{unit=.06\textwidth}
\centering
{
\psset{Beta=15}
\begin{pspicture}(-8.3,-2.9)(8.23,2.95)
\pstThreeDCoor[xMin=0,yMin=0,zMin=0,xMax=11,yMax=11,zMax=2.4,linecolor=black,linewidth=.7pt]
{
\psset{linecolor=black,linewidth=.3pt,linestyle=dashed,subticks=1}
\pstThreeDPlaneGrid[planeGrid=xy](0,0)(5,3)
\pstThreeDPlaneGrid[planeGrid=xy](0,0)(2,5)
\pstThreeDPlaneGrid[planeGrid=xz](0,0)(1,1)
\pstThreeDPlaneGrid[planeGrid=yz](0,0)(1,1)
\pstThreeDPlaneGrid[planeGrid=xy,planeGridOffset=1](0,0)(1,1)
\pstThreeDPlaneGrid[planeGrid=xz,planeGridOffset=1](0,0)(1,1)
\pstThreeDPlaneGrid[planeGrid=yz,planeGridOffset=1](0,0)(1,1)
\pstThreeDLine(1,0,1)(11,0,1)\pstThreeDLine(11,0,1)(11,0,0)\pstThreeDLine(11,0,0)(11,1,0)\pstThreeDLine(11,1,0)(11,1,1)\pstThreeDLine(11,1,1)(11,0,1)\pstThreeDLine(11,1,0)(11,3,0)\pstThreeDLine(11,1,0)(1,1,0)
}
\pstThreeDPut[pOrigin=c](7,2,0.5){\psframebox*[framesep=0.5pt,framearc=1]{\phantom{$\tau_0$}}}
\pstThreeDPut[pOrigin=c](2.66,3,0.33){\psframebox*[framesep=0.5pt,framearc=1]{\phantom{$\tau_1$}}}
\pstThreeDPut[pOrigin=c](1.1,5.05,0.33){\psframebox*[framesep=0.7pt,framearc=1]{\phantom{$\tau_2$}}}
{
\psset{dotstyle=none,dotscale=1,drawCoor=false}
\psset{linecolor=black,linewidth=1pt,linejoin=1}
\psset{fillcolor=lightgray,opacity=.6,fillstyle=solid}
\pstThreeDLine(1,1,1)(0,10,0)(2,5,0)
\pstThreeDLine(1,1,1)(2,5,0)(5,3,0)
\pstThreeDLine(11,3,0)(5,3,0)(1,1,1)(11,1,1)
}
\pstThreeDPut[pOrigin=t](5,3,-0.25){$A$}
\pstThreeDPut[pOrigin=t](2,5,-0.25){$B$}
\pstThreeDPut[pOrigin=t](0,10,-0.24){$C$}
\pstThreeDPut[pOrigin=c](7,2,0.5){$\tau_0$}
\pstThreeDPut[pOrigin=c](2.66,3,0.33){$\tau_1$}
\pstThreeDPut[pOrigin=c](1.1,5.1,0.33){$\tau_2$}
\pstThreeDPut[pOrigin=bl](.925,1.075,1.07){\psframebox*[framesep=-.3pt,framearc=1]{\footnotesize $V(x_V,y_V,1)$}}
\end{pspicture}
\psset{Beta=30}
}
\caption{Case VI: $B_1$-facets $\tau_0,\tau_1,$ and $\tau_2$ all contribute to $s_0$}
\label{figcase6}
\end{figure}
%
%

Let us assume that the contributing facets $\tau_0,\tau_1,\ldots,\tau_t$ are $B_1$ with respect to the variable $z$. Since the $\tau_i$ all contribute to the same candidate pole $s_0$, their affine supports intersect the diagonal of the first octant in the same point $(-1/s_0,-1/s_0,-1/s_0)$. As these affine supports share only one point, the aforementioned intersection point must be the contributing facets' common vertex $V$ at \lq height\rq\ one:
\begin{equation*}
\left(-\frac{1}{s_0},-\frac{1}{s_0},-\frac{1}{s_0}\right)=(x_V,y_V,1).
\end{equation*}
We conclude that $x_V=y_V=1$ and $s_0=-1$. Hence under the conditions of Theorem~\ref{maintheoartdrie}, Case~VI cannot occur.

\section{General case: several groups of $B_1$-facets contribute to $s_0$; every group is separately covered by one of the previous cases, and the groups have pairwise at most one point in common}\label{secgeval7art3}
As the different \lq clusters\rq\ of contributing $B_1$-facets pairwise share not more than one point, we can decompose each cone associated to a vertex of \Gf\ into simplicial cones in such a way that the relevant residues in $s_0$ split up into parts, each part corresponding to one of the preceding cases. In this way the general case follows immediately from the previous ones. Figure~\ref{figcase7} shows two possible configurations of $B_1$-facets that fall under the general case.\label{eindegrbew}

%
%
\begin{figure}
\centering
\subfigure[Non-compact $B_1$-facet $\tau_0$ and $B_1$-simplex $\tau_1$ both contribute to $s_0$. As they have only one point in common, they form two separate clusters.]{
\psset{unit=.07067137809\textwidth}
\psset{Beta=8}\psset{Alpha=60}
\begin{pspicture}(-5.8,-1.8)(8.07,3.25)
\pstThreeDCoor[xMin=0,yMin=0,zMin=0,xMax=11,yMax=9,zMax=3,spotZ=180,linecolor=black,linewidth=.7pt]
{
\psset{linecolor=gray,linewidth=.3pt,linestyle=dashed,subticks=1}
\pstThreeDLine(2,0,0)(2,2,0)\pstThreeDLine(2,2,0)(0,2,0)\pstThreeDLine(3,0,0)(3,5,0)\pstThreeDLine(3,5,0)(0,5,0)\pstThreeDLine(0,0,1)(2,0,1)\pstThreeDLine(2,0,1)(2,0,0)\pstThreeDLine(0,0,1)(0,2,1)\pstThreeDLine(0,2,1)(0,2,0)\pstThreeDLine(2,0,0)(2,2,0)
\pstThreeDPlaneGrid[planeGrid=xy,planeGridOffset=1](0,0)(2,2)
\pstThreeDPlaneGrid[planeGrid=xz,planeGridOffset=2](0,0)(2,1)
\pstThreeDPlaneGrid[planeGrid=yz,planeGridOffset=2](0,0)(2,1)
\pstThreeDLine(2,0,1)(11,0,1)\pstThreeDLine(11,0,1)(11,0,0)\pstThreeDLine(11,0,0)(11,2,0)\pstThreeDLine(11,2,0)(11,2,1)\pstThreeDLine(11,2,1)(11,0,1)\pstThreeDLine(11,2,0)(11,4,0)\pstThreeDLine(11,2,0)(2,2,0)\pstThreeDLine(6,0,0)(6,4,0)
}
{
\psset{linecolor=black,linewidth=.3pt,linestyle=dashed,subticks=1}
\pstThreeDLine(6,4,0)(0,4,0)\pstThreeDLine(8,0,0)(3,5,0)\pstThreeDLine(4,4,0)(0,0,2)\pstThreeDLine(0,0,0)(3,3,3)
}
{
\psset{linecolor=black,dotsize=2.8pt}
\pstThreeDDot(4,4,0)\pstThreeDDot(1.3333,1.3333,1.3333)
}
\pstThreeDPut[pOrigin=c](2,5,0.33){\psframebox*[framesep=0.5pt,framearc=1]{\phantom{$\tau_1$}}}
{
\psset{dotstyle=none,dotscale=1,drawCoor=false}
\psset{linecolor=black,linewidth=1pt,linejoin=1}
\psset{fillcolor=lightgray,opacity=.6,fillstyle=solid}
\pstThreeDTriangle(2,2,1)(0,8,0)(3,5,0)
\pstThreeDLine(11,4,0)(6,4,0)(2,2,1)(11,2,1)
}
{
\psset{linecolor=gray,dotsize=2.8pt}
\pstThreeDDot(2,2,1)
}
\pstThreeDPut[pOrigin=t](6,4,-0.24){$A$}
\pstThreeDPut[pOrigin=t](3,5,-0.24){$B$}
\pstThreeDPut[pOrigin=t](0,8,-0.24){$C$}
\pstThreeDPut[pOrigin=bl](1.925,2.04,1.1){\psframebox*[framesep=-1pt,framearc=1]{$D=\tau_0\cap\tau_1$}}
\pstThreeDPut[pOrigin=c](7.5,3,0.5){$\tau_0$}
\pstThreeDPut[pOrigin=c](2,5,0.33){$\tau_1$}
\pstThreeDPut[pOrigin=br](0.06,-0.05,1.11){\footnotesize$1$}
\pstThreeDPut[pOrigin=bl](2.96,3.03,3.02){$d\leftrightarrow x=y=z$}
\pstThreeDNode(1.3333,1.3333,1.3333){ipunt}
\rput[r](8.25,1.7){\rnode[l]{ipuntlabel}{\footnotesize$\aff\tau_0\cap\aff\tau_1\cap d=(-1/s_0,-1/s_0,-1/s_0)$}}
\ncline[linewidth=.3pt,nodesepB=2.5pt,nodesepA=2pt]{->}{ipuntlabel}{ipunt}
\end{pspicture}
}\\\subfigure[$B_1$-Facets $\tau_0,\tau_1,\tau_2$ all contribute to $s_0$. We distinguish the clusters $\{\tau_0\}$ and $\{\tau_1,\tau_2\}$.]{
\psset{unit=.06426\textwidth}
\psset{Beta=8}\psset{Alpha=53}
\begin{pspicture}(-6.97,-1.84)(8.34,3.25)
\pstThreeDCoor[xMin=0,yMin=0,zMin=0,xMax=11,yMax=10,zMax=3,spotZ=180,linecolor=black,linewidth=.7pt]
{
\psset{linecolor=gray,linewidth=.3pt,linestyle=dashed,subticks=1}
\pstThreeDLine(7,0,0)(7,4,0)\pstThreeDLine(7,4,0)(0,4,0)
\pstThreeDLine(6,0,0)(6,5,0)\pstThreeDLine(6,5,0)(0,5,0)
\pstThreeDLine(4,0,0)(4,8,0)\pstThreeDLine(4,8,0)(0,8,0)
\pstThreeDLine(3,0,0)(3,2,0)\pstThreeDLine(3,2,0)(0,2,0)
\pstThreeDLine(2,0,0)(2,3,0)\pstThreeDLine(2,3,0)(0,3,0)
\pstThreeDLine(3,0,1)(3,2,1)\pstThreeDLine(3,2,1)(0,2,1)
\pstThreeDLine(2,0,1)(2,3,1)\pstThreeDLine(2,3,1)(0,3,1)
\pstThreeDLine(0,0,1)(2,0,1)\pstThreeDLine(2,0,1)(3,0,1)\pstThreeDLine(2,0,1)(2,0,0)\pstThreeDLine(3,0,1)(3,0,0)
\pstThreeDLine(0,0,1)(0,2,1)\pstThreeDLine(0,2,1)(0,3,1)\pstThreeDLine(0,2,1)(0,2,0)\pstThreeDLine(0,3,1)(0,3,0)
\pstThreeDLine(2,3,0)(2,3,1)\pstThreeDLine(3,2,0)(3,2,1)
\pstThreeDLine(3,0,1)(11,0,1)\pstThreeDLine(11,0,1)(11,0,0)\pstThreeDLine(11,0,0)(11,2,0)\pstThreeDLine(11,2,0)(11,2,1)\pstThreeDLine(11,2,1)(11,0,1)\pstThreeDLine(11,2,0)(11,4,0)\pstThreeDLine(11,2,0)(3,2,0)
\pstThreeDLine(0,3,1)(0,10,1)\pstThreeDLine(0,10,1)(0,10,0)\pstThreeDLine(0,10,0)(2,10,0)\pstThreeDLine(2,10,0)(2,10,1)\pstThreeDLine(2,10,1)(0,10,1)\pstThreeDLine(2,10,0)(4,10,0)\pstThreeDLine(2,10,0)(2,3,0)
}
{
\psset{linecolor=black,linewidth=.3pt,linestyle=dashed,subticks=1}
\pstThreeDLine(2,3,1)(.8,0,1.6)\pstThreeDLine(0,0,0)(3,3,3)
}
{
\psset{linecolor=black,dotsize=2.8pt}
\pstThreeDDot(1.3333,1.3333,1.3333)
}
\pstThreeDPut[pOrigin=c](4,5.3,0.33){\psframebox*[framesep=-1pt,framearc=1]{\phantom{$\tau_1$}}}
\pstThreeDPut[pOrigin=c](3,7.72,0.5){\psframebox*[framesep=0.5pt,framearc=1]{\phantom{$\tau_2$}}}
{
\psset{dotstyle=none,dotscale=1,drawCoor=false}
\psset{linecolor=black,linewidth=1pt,linejoin=1}
\psset{fillcolor=lightgray,opacity=.6,fillstyle=solid}
\pstThreeDLine(4,8,0)(6,5,0)(2,3,1)
\pstThreeDLine(4,10,0)(4,8,0)(2,3,1)(2,10,1)
\pstThreeDLine(11,4,0)(7,4,0)(3,2,1)(11,2,1)
}
{
\psset{linecolor=black,linewidth=.3pt,linestyle=dashed,subticks=1}
\pstThreeDLine(8,4,0)(0,.8,1.6)
}
{
\psset{linecolor=gray,dotsize=2.8pt}
\pstThreeDDot(8,4,0)\pstThreeDDot(3,2,1)\pstThreeDDot(4,8,0)\pstThreeDDot(2,3,1)
}
\pstThreeDPut[pOrigin=t](7,4,-0.24){$A$}
\pstThreeDPut[pOrigin=t](6,5,-0.24){$B$}
\pstThreeDPut[pOrigin=t](4,8,-0.24){$C$}
\pstThreeDPut[pOrigin=br](3.075,1.96,1.1){\psframebox*[framesep=-.85pt,framearc=1]{$E$}}
\pstThreeDPut[pOrigin=bl](1.925,3.04,1.1){\psframebox*[framesep=-1pt,framearc=1]{$D$}}
\pstThreeDPut[pOrigin=c](8,3,0.51){$\tau_0$}
\pstThreeDPut[pOrigin=c](4,5.33,0.33){$\tau_1$}
\pstThreeDPut[pOrigin=c](3,7.75,0.5){$\tau_2$}
\pstThreeDPut[pOrigin=br](0.06,-0.05,1.11){\psframebox*[framesep=0.2pt,framearc=0]{\footnotesize$1$}}
\pstThreeDPut[pOrigin=bl](2.96,3.02,3.02){$d\leftrightarrow x=y=z$}
\pstThreeDNode(1.3333,1.3333,1.3333){ipunt}
\rput[l](2.7,1.68){\rnode[l]{ipuntlabel}{\parbox{3.4cm}{\footnotesize$\aff\tau_0\cap\aff\tau_1\cap\aff\tau_2\cap d$\\$=(-1/s_0,-1/s_0,-1/s_0)$}}}
\ncline[linewidth=.3pt,nodesepB=2.5pt,nodesepA=2.5pt]{->}{ipuntlabel}{ipunt}
\end{pspicture}
}
\psset{Beta=30}\psset{Alpha=45}
\caption{General Case: several \lq clusters\rq\ of $B_1$-facets contribute to the candidate-pole $s_0$}
\label{figcase7}
\end{figure}
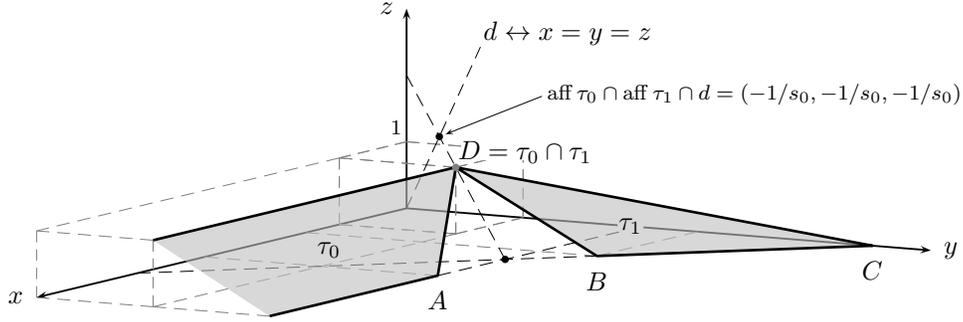
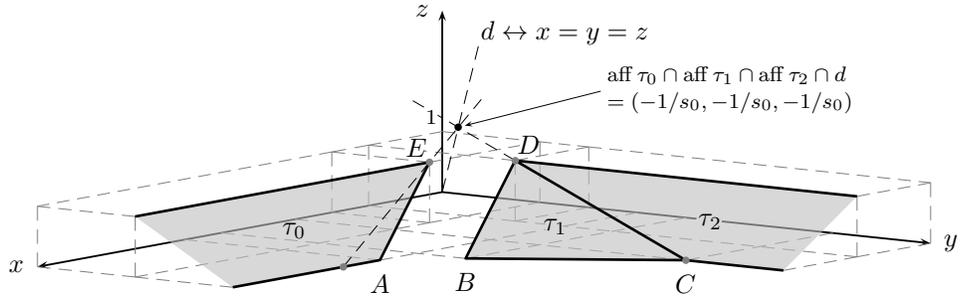
%
%

\section{The main theorem for a non-trivial character of \Zpx}\label{sectkarakter}
In this section we consider Igusa's zeta function of a polynomial $f(x_1,\ldots,x_n)\in\Zp[x_1,\ldots,x_n]$ and a character $\chi:\Zpx\to\Ccross$ of \Zpx, and we prove the analogue of Theorem~\ref{mcigusandss} for a non-trivial character. We start with the definition of this \lq twisted\rq\ $p$-adic zeta function.

Let $p$ be a prime number and $a\in\Qp$. We denote the $p$-adic order of $a$ by $\ord_pa\in\Z\cup\{\infty\}$; we write $\abs{a}=p^{-\ord_pa}$ for the $p$-adic norm of $a$ and $\ac a=\abs{a}a$ for its angular component. As before, we denote by $\abs{dx}=\abs{dx_1\wedge\cdots\wedge dx_n}$ the Haar measure on \Qpn, normalized in such a way that \Zpn\ has measure one.

\begin{definition}[local twisted Igusa zeta function]\label{defIZFkaraktartdrie} Cfr.\ \cite[Def.~1.1]{Hoo01}. Let $p$ be a prime number and $f(x)=f(x_1,\ldots,x_n)$ a polynomial in $\Zp[x_1,\ldots,x_n]$. Let $\chi:\Zpx\to\Ccross$ be a character of \Zpx, i.e., a multiplicative group homomorphism with finite image. We formally put $\chi(0)=0$. To $f$ and $\chi$ we associate the local Igusa zeta function
\begin{equation*}
Z_{f,\chi}^0:\{s\in\C\mid\Re(s)>0\}\to\C:s\mapsto\int_{p\Zpn}\chi(\ac f(x))\abs{f(x)}^s\abs{dx}.
\end{equation*}
\end{definition}

If $\chi$ is the trivial character, we obtain the usual local Igusa zeta function of $f$. In this section we will deal with the non-trivial characters. The rationality result of Igusa \cite{Igu74} and Denef \cite{Den84} holds for the above version of Igusa's zeta function as well. From now on, by $Z_{f,\chi}^0$ we mean the meromorphic continuation to \C\ of the function defined in Definition~\ref{defIZFkaraktartdrie}.

The goal is to verify the following analogue of Theorem~\ref{mcigusandss}.

\begin{theorem}[Monodromy Conjecture for Igusa's zeta function of a non-degenerated surface singularity and a non-trivial character of \Zpx]\label{mcigusandsskarakter}
Let $f(x,y,z)\in\Z[x,y,z]$ be a nonzero polynomial in three variables satisfying $f(0,0,0)=0$, and let $U\subset\C^3$ be a neighborhood of the origin. Suppose that $f$ is non-degenerated over \C\ with respect to all the compact faces of its Newton polyhedron, and let $p$ be a prime number such that $f$ is also non-degenerated over \Fp\ with respect to the same faces.\footnote{By Remark~\ref{verndcndfp}(i) this is the case for almost all prime numbers $p$.} Let $\chi:\Zpx\to\Ccross$ be a non-trivial character of \Zpx, and assume that $\chi$ is trivial on $1+p\Zp$. Suppose that $s_0$ is a pole of the local Igusa zeta function $Z_{f,\chi}^0$ associated to $f$ and $\chi$. Then $e^{2\pi i\Re(s_0)}$ is an eigenvalue of the local monodromy of $f$ at some point of $f^{-1}(0)\cap U$.
\end{theorem}

The reason that we restrict to characters $\chi$ that are trivial on $1+p\Zp$, is that in this case we have a nice analogue of Denef and Hoornaert's formula (Theorem~\ref{formdenhoor}) for $Z_{f,\chi}^0$. We give the formula below, but first we introduce a notation that simplifies the statement of the formula.

\begin{notation}\label{notchibarartdrieh}
Let $p$ be a prime number and $\chi:\Zpx\to\Ccross$ a character of \Zpx. Assume that $\chi$ is trivial on the (multiplicative) subgroup $1+p\Zp$ of \Zpx. We shall identify the quotient group $\Zpx/(1+p\Zp)$ with \Fpcross, and we shall denote by $\pi:\Zpx\to\Fpcross$ the natural surjective homomorphism. Since $1+p\Zp\subset\ker\chi$, there exists a unique homomorphism $\barchi:\Fpcross\to\Ccross$ such that $\chi=\barchi\circ\pi$. In order for $\barchi$ to be defined on the whole of \Fp, we shall formally put $\barchi(0)=0$.
\end{notation}

\begin{theorem}\label{formdenhoorkarakters}
\textup{\cite[Thm.~3.4]{Hoo01}}. Let $p$ be a prime number; let $f(x)=f(x_1,\ldots,x_n)$ be a nonzero polynomial in $\Zp[x_1,\ldots,x_n]$ satisfying $f(0)=0$. Suppose that $f$ is non-degenerated over \Fp\ with respect to all the compact faces of its Newton polyhedron \Gf. Let $\chi:\Zpx\to\Ccross$ be a non-trivial character of \Zpx, and assume that $\chi$ is trivial on $1+p\Zp$. Then the local Igusa zeta function associated to $f$ and $\chi$ is the meromorphic complex function
\begin{equation*}
Z_{f,\chi}^0:s\mapsto\sum_{\substack{\tau\mathrm{\ compact}\\\mathrm{face\ of\ }\Gf}}L_{\tau}S(\Dtu)(s),
\end{equation*}
with
\begin{gather*}
L_{\tau}=p^{-n}\sum_{x\in\Fpcrossn}\barchi\bigl(\fbart(x)\bigr)\\\shortintertext{and}
S(\Dtu)(s)=\sum_{k\in\Zn\cap\Dtu}p^{-\sigma(k)-m(k)s}
\end{gather*}
for every compact face $\tau$ of \Gf. Hereby $\ft,\fbart$, and \barchi\ are defined as in Notations~\ref{notftauart3}, \ref{notftaubarart3}, and \ref{notchibarartdrieh}, respectively; the definitions of $\sigma(k),m(k)$, and \Dtu\ can be found in Notation~\ref{notsigmakartdrieintro} and Definitions~\ref{def_mfad} and \ref{def_Dfad}, respectively. The sums $S(\Dtu)(s)$ can be calculated in the same way as in Theorem~\ref{formdenhoor}.
\end{theorem}

Note that, contrary to the trivial character case, the $L_{\tau}$ do not depend on the variable $s$. Consequently, $Z_{f,\chi}^0$ for a non-trivial character $\chi$, has \lq fewer\rq\ candidate poles than \Zof.

\begin{corollary}
Let $f$ and $\chi$ be as in Theorem~\ref{formdenhoorkarakters}. Let $\gamma_1,\ldots,\gamma_r$ be all the facets of \Gf, and let $v_1,\ldots,v_r$ be the unique primitive vectors in $\Zplusn\setminus\{0\}$ that are perpendicular to $\gamma_1,\ldots,\gamma_r$, respectively. From Theorem~\ref{formdenhoorkarakters} and the rational expression for $S(\Dtu)(s)$ obtained in Theorem~\ref{formdenhoor}, it follows that the poles of $Z_{f,\chi}^0$ are among the numbers
\begin{equation}\label{candpoleskarakterad}
-\frac{\sigma(v_j)}{m(v_j)}+\frac{2k\pi i}{m(v_j)\log p},
\end{equation}
with $j\in\{1,\ldots,r\}$ such that $m(v_j)\neq0$, and $k\in\Z$. We shall refer to these numbers as the candidate poles of $Z_{f,\chi}^0$.
\end{corollary}

Now suppose that $f,U,p,\chi$, and $s_0$ are as in Theorem~\ref{mcigusandsskarakter}. Then $s_0$ is one of the numbers \eqref{candpoleskarakterad}. Theorem~\ref{theoAenL} tells us that if $s_0$ is contributed (cfr.\ Definition~\ref{defcontrinlad}) by a facet of \Gf\ that is not a $B_1$-facet (cfr.\ Definition~\ref{defbeenfacetad}), then $e^{2\pi i\Re(s_0)}$ is an eigenvalue of monodromy of $f$ at some point of $f^{-1}(0)\cap U$. Proposition~\ref{propAenL} says that the same is true if $s_0$ is contributed by two $B_1$-facets of \Gf\ that are not $B_1$ for a same variable and that have an edge in common. Therefore, in order to obtain Theorem~\ref{mcigusandsskarakter}, it is sufficient to verify the following proposition.

\begin{proposition}[On candidate poles of $Z_{f,\chi}^0$ only contributed by $B_1$-facets]\label{maintheoartdriekarakt}
Let $p$ be a prime number and let $f(x,y,z)\in\Zp[x,y,z]$ be a nonzero polynomial in three variables with $f(0,0,0)=0$. Suppose that $f$ is non-degenerated over \Fp\ with respect to all the compact faces of its Newton polyhedron. Let $\chi:\Zpx\to\Ccross$ be a non-trivial character of \Zpx\ that is trivial on $1+p\Zp$. Suppose that $s_0$ is a candidate pole of $Z_{f,\chi}^0$ that is only contributed by $B_1$-facets of \Gf. Further assume that for any pair of contributing $B_1$-facets, we have that
\begin{itemize}
\item[-] either they are $B_1$-facets for a same variable,
\item[-] or they have at most one point in common.
\end{itemize}
Then $s_0$ is not a pole of $Z_{f,\chi}^0$.
\end{proposition}

Let $p,f,\chi$, and $s_0$ be as in the proposition. Let us consider the same seven cases as in the proof of Theorem~\ref{maintheoartdrie}. The three observations below show that in every case, the relevant terms in the formula for $Z_{f,\chi}^0$ from Theorem~\ref{formdenhoorkarakters}, are either zero or they cancel in pairs. In what follows we shall use the notations of Theorem~\ref{formdenhoor}.

First consider a vertex $V(x_V,y_V,1)$ of \Gf\ at \lq height\rq\ one. The corresponding polynomial $\overline{f_V}$ has the form $\overline{f_V}=\overline{a_V}x^{x_V}y^{y_V}z$ with $\overline{a_V}\in\Fpcross$. The factor $L_V$ is thus given by
\begin{align}
L_V&=p^{-3}\sum_{(x,y,z)\in(\Fpcross)^3}\barchi\bigl(\overline{a_V}x^{x_V}y^{y_V}z\bigr)\notag\\
&=p^{-3}\barchi(\overline{a_V})\sum_{x\in\Fpcross}\barchi^{x_V}(x)\sum_{y\in\Fpcross}\barchi^{y_V}(y)\sum_{z\in\Fpcross}\barchi(z).\label{lastsumkaraktad}
\end{align}
Since $\chi$ is non-trivial but trivial on $1+p\Zp$, the character $\barchi$ of \Fpcross\ is also non-trivial. It is well-known that in this case the last sum of \eqref{lastsumkaraktad} equals zero. Indeed, for any $u\in\Fpcross$ the map $\Fpcross\to\Fpcross:z\mapsto uz$ is a permutation. Consequently,
\begin{equation*}
\sum_{z\in\Fpcross}\barchi(z)=\sum_{z\in\Fpcross}\barchi(uz)=\barchi(u)\sum_{z\in\Fpcross}\barchi(z).
\end{equation*}
As $\barchi$ is non-trivial, there exists a $u\in\Fpcross$ with $\barchi(u)\neq1$, and for such $u$ the above equation implies that $\sum_{z\in\Fpcross}\barchi(z)=0$. We conclude that $L_V=0$ and the term associated to $V$ vanishes.

Next let us consider a $B_1$-simplex $\tau_0$, say for the variable $z$. Let $A$ and $B$ be the two vertices of $\tau_0$ in the plane $\{z=0\}$, and let $C(x_C,y_C,1)$ be the vertex of $\tau_0$ at distance one of this plane. For $L_{[AB]}$ we find
\begin{equation*}
L_{[AB]}=p^{-3}\sum_{(x,y,z)\in(\Fpcross)^3}\barchi\bigl(\overline{f_{[AB]}}(x,y)\bigr)=p^{-3}(p-1)\sum_{(x,y)\in(\Fpcross)^2}\barchi\bigl(\overline{f_{[AB]}}(x,y)\bigr),
\end{equation*}
while $L_{\tau_0}$ is given by
\begin{equation}\label{ltnulkarakad}
L_{\tau_0}=p^{-3}\sum_{(x,y)\in(\Fpcross)^2}\sum_{z\in\Fpcross}\barchi\bigl(\overline{f_{[AB]}}(x,y)+\overline{a_C}x^{x_C}y^{y_C}z\bigr)
\end{equation}
for some $\overline{a_C}\in\Fpcross$. Fix $(x,y)\in(\Fpcross)^2$. If $z$ runs through \Fpcross, the argument of $\barchi$ in \eqref{ltnulkarakad} runs through all elements of the set $\Fp\setminus\bigl\{\overline{f_{[AB]}}(x,y)\bigr\}$. Consequently,
\begin{align*}
L_{\tau_0}&=p^{-3}\sum_{(x,y)\in(\Fpcross)^2}\Bigl(\sum\nolimits_{u\in\Fp}\barchi(u)-\barchi\bigl(\overline{f_{[AB]}}(x,y)\bigr)\Bigr)\\
&=-p^{-3}\sum_{(x,y)\in(\Fpcross)^2}\barchi\bigl(\overline{f_{[AB]}}(x,y)\bigr).
\end{align*}
Together with the fact that
\begin{equation*}
S(\Delta_{\tau_0})(s)=\frac{1}{p^{\sigma_0+m_0s}-1}\qquad\text{and}\qquad S(\Delta_{[AB]})(s)=\frac{1}{(p^{\sigma_0+m_0s}-1)(p-1)}
\end{equation*}
(with $(m_0,\sigma_0)$ the numerical data associated to $\tau_0$), we now easily find that the sum of the terms associated to $\tau_0$ and $[AB]$ equals zero.

Finally, consider any $B_1$-facet $\tau_0$ (compact or not), and assume that $\tau_0$ is $B_1$ for the variable $z$. Let $A$ be a vertex of $\tau_0$ in the plane $\{z=0\}$ and $C(x_C,y_C,1)$ the vertex of $\tau_0$ at \lq height\rq\ one. Denote by $\tau_1$ the other facet of \Gf\ that contains the edge $[AC]$, and let $\tau_2$ be the facet in $\{z=0\}$. Denote by $\delta_A$ the simplicial subcone of $\Delta_A$ strictly positively spanned by the primitive vectors $v_0,v_1,v_2\in\Zplus^3\setminus\{0\}$ that are perpendicular to $\tau_0,\tau_1,\tau_2$, respectively. In the same way as in the previous paragraph we find that
\begin{equation*}
L_A=-(p-1)L_{[AC]}.
\end{equation*}
If we combine this identity with the expressions
\begin{align*}
S(\Delta_{[AC]})(s)&=\frac{\Sigma(\Delta_{[AC]})(s)}{(p^{\sigma_0+m_0s}-1)(p^{\sigma_1+m_1s}-1)}\qquad\text{and}\\
S(\delta_A)(s)&=\frac{\Sigma(\Delta_{[AC]})(s)}{(p^{\sigma_0+m_0s}-1)(p^{\sigma_1+m_1s}-1)(p-1)}
\end{align*}
(where $(m_0,\sigma_0)$ and $(m_1,\sigma_1)$ denote the numerical data of $\tau_0$ and $\tau_1$, re\-spec\-tive\-ly), we find again that the terms associated to $[AC]$ and $\delta_A$ cancel out.

This concludes (the sketch of) the proof of Proposition~\ref{maintheoartdriekarakt} and Theorem~\ref{mcigusandsskarakter}.

\section{The main theorem in the motivic setting}\label{sectmotivisch}
\subsection{The local motivic zeta function and the motivic Monodromy Conjecture}
The theory of motivic integration was invented by Kontsevich and further developed by a.o.\ Denef--Loeser \cite{DLmoteen,DLmottwee,DLmotdrie}, Loeser--Sebag \cite{LSmot,Sebag}, and Cluckers--Loeser \cite{CLmot}. Denef and Loeser introduced the motivic zeta function and the corresponding monodromy conjecture in \cite{DL98}. For an introduction to motivic integration, motivic zeta functions, and the (motivic) Monodromy Conjecture, we refer to \cite{nicaisemot} and \cite{veysmot}. In this section we will only give the definitions that are needed to state the results.

In motivic integration theory, one associates to each algebraic variety $X$ over \C, and to each $l\in\Zplus$, a space $\mathcal{L}_l(X)$ of so-called $l$-jets on $X$. Informally speaking, this jet space $\mathcal{L}_l(X)$ is an algebraic variety over \C\ whose points with coordinates in \C\ correspond to points of $X$ with coordinates in $\C[t]/(t^{l+1})$, and vice versa. For all $l'\geqslant l$, there are natural \emph{truncation maps} $\pi_l^{l'}:\mathcal{L}_{l'}(X)\to\mathcal{L}_l(X)$, sending $l'$-jets to their reduction modulo $t^{l+1}$.

Next one obtains the space $\mathcal{L}(X)$ of \emph{arcs} on $X$ as the inverse limit $\varprojlim\mathcal{L}_l(X)$ of the system $\bigl((\mathcal{L}_l(X))_{l\geqslant0},(\pi_l^{l'})_{l'\geqslant l\geqslant0}\bigr)$. The arc space $\mathcal{L}(X)$ should be thought of as an \lq algebraic variety of infinite dimension\rq\ over \C\ whose points with coordinates in \C\ agree with the points of $X$ with coordinates in $\C[[t]]$. It comes together with natural truncation maps $\pi_l:\mathcal{L}(X)\to\mathcal{L}_l(X)$, sending arcs to their reduction modulo $t^{l+1}$.

In this section, the only algebraic variety we will consider, is the $n$-dimensional affine space $X=\mathbf{A}^n(\C)$. In this case, $\mathcal{L}_l(\mathbf{A}^n(\C))\cong\mathbf{A}^{n(l+1)}(\C)$ and $\mathcal{L}(\mathbf{A}^n(\C))$ can be identified with $(\C[t]/(t^{l+1}))^n$ and $(\C[[t]])^n$, respectively. We will use these identifications throughout the section. The truncation maps are as expected:
\begin{align*}
\pi_l^{l'}:&\ (\C[t]/(t^{l'+1}))^n\to(\C[t]/(t^{l+1}))^n:\bigl(\phi_{\rho}+(t^{l'+1})\bigr)_{\rho}\mapsto\bigl(\phi_{\rho}+(t^{l+1})\bigr)_{\rho},\\
\pi_l:&\ (\C[[t]])^n\to(\C[t]/(t^{l+1}))^n:\bigl({\textstyle\sum_{\kappa}}\phi_{\rho,\kappa}t^{\kappa}\bigr)_{\rho}\mapsto\bigl({\textstyle\sum_{\kappa=0}^{l}}\phi_{\rho,\kappa}t^{\kappa}+(t^{l+1})\bigr)_{\rho}.
\end{align*}
In motivic integration, the discrete valuation ring $\C[[t]]$ and its uniformizer $t$ play the role that $\Zp$ and $p$ play in $p$-adic integration.

The Grothendieck group of (complex) algebraic varieties is the abelian group $K_0(Var_{\C})$ generated by the isomorphism classes $[X]$ of algebraic varieties $X$, modulo the relations $[X]=[X\setminus Y]+[Y]$ if $Y$ is Zariski-closed in $X$. The Grothendieck group is turned into a Grothendieck ring by putting $[X]\cdot[Y]=[X\times Y]$ for all algebraic varieties $X$ and $Y$. The class of a (complex) algebraic variety in the Grothendieck ring is the universal invariant of an algebraic variety with respect to the additive and multiplicative relations above; it is a refinement of the topological Euler characteristic.

We call a subset $C$ of an algebraic variety $X$ constructible if it can be written as a finite disjoint union of locally closed\footnote{w.r.t.\ the Zariski-topology on $X$} subvarieties $Y_1,\ldots,Y_r$ of $X$. For such a constructible subset $C=\bigsqcup_jY_j$, the class $[C]=\sum_j[Y_j]$ of $C$ in the Grothendieck ring is well-defined, i.e., is independent of the chosen decomposition. We denote the class of a point by $1$ and the class of the affine line $\mathbf{A}^1(\C)$ by \LL. Finally, we denote by $\MC=K_0(Var_{\C})[\LL^{-1}]$ the localization of $K_0(Var_{\C})$ with respect to \LL. It is known that $K_0(Var_{\C})$ is not a domain \cite{Poonen}; however, it is still an open question whether \MC\ is a domain or not.

We shall call a subset $A$ of $(\C[[t]])^n$ cylindric if $A=\pi_l^{-1}(C)$ for some $l\in\Zplus$ and some constructible subset $C$ of $(\C[t]/(t^{l+1}))^n$. For such a cylindric subset $A=\pi_l^{-1}(C)$, one has that
\begin{gather*}
\pi_{l'}(A)\cong C\times\mathbf{A}^{n(l'-l)}(\C)\qquad\text{for all }l'\geqslant l;\\\shortintertext{therefore,}
\mu(A)=[C]\LL^{-n(l+1)}=\lim_{l'\to\infty}[\pi_{l'}(A)]\LL^{-n(l'+1)}\in\MC
\end{gather*}
is independent of $l$. We call $\mu(A)$ the naive motivic measure of $A$. Its definition and in particular the chosen normalization are inspired by the $p$-adic Haar measure; note that $\mu((t^l\C[[t]])^n)=\LL^{-nl}$ for all $l\in\Zplus$.

For $\phi_{\rho}=\phi_{\rho,0}+\phi_{\rho,1}t+\phi_{\rho,2}t^2+\cdots\in\C[[t]]\setminus\{0\}$, we define $\ord_t\phi_{\rho}$ as the smallest $\kappa\in\Zplus$ such that $\phi_{\rho,\kappa}\neq0$; additionally, we agree that $\ord_t0=\infty$. If $\phi=(\phi_1,\ldots,\phi_n)\in(\C[[t]])^n$, then we put
\begin{equation*}
\ord_t\phi=(\ord_t\phi_1,\ldots,\ord_t\phi_n)\in(\Zplus\cup\{\infty\})^n.
\end{equation*}
Let us recall the definition of the local $p$-adic zeta function. If $f(x)=f(x_1,\ldots,x_n)$ is a nonzero polynomial in $\Zp[x_1,\ldots,x_n]$ with $f(0)=0$, then
\begin{align*}
\Zofs&=\int_{p\Zpn}\abs{f(x)}^s\abs{dx}\\
&=\sum_{l\geqslant1}\mu(\{x\in p\Zpn\mid\ord_pf(x)=l\})p^{-ls}\\
&=p^{-n}\sum_{l\geqslant1}\#\bigl\{x+p^{l+1}\Zpn\in\bigl(p\Zp/p^{l+1}\Zp\bigr)^n\mid\ord_pf(x)=l\bigr\}\cdot(p^{-n}p^{-s})^l,
\end{align*}
with $\mu(\cdot)$ the Haar measure on \Qpn, so normalized that $\mu(\Zpn)=1$. This is the motivation for the following definition.

\begin{definition}[Local motivic zeta function]\label{deflocmotzf}
Let $f(x)=f(x_1,\ldots,x_n)$ be a nonzero polynomial in $\C[x_1,\ldots,x_n]$ satisfying $f(0)=0$. Put
\begin{equation*}
\mathcal{X}_l^0=\bigl\{\phi+\bigl(t^{l+1}\C[t]\bigr)^n\in\bigl(t\C[t]/t^{l+1}\C[t]\bigr)^n\mid\ord_tf(\phi)=l\bigr\}
\end{equation*}
for $l\in\Zplusnul$. Then the local motivic zeta function $\Zmotof(s)$ associated to $f$ is by definition the following element of $\MC[[\LL^{-s}]]$:
\begin{align*}
\Zmotof(s)&=\sum_{l\geqslant1}\mu(\{\phi\in(t\C[[t]])^n\mid\ord_tf(\phi)=l\})(\LL^{-s})^l\\
&=\LL^{-n}\sum_{l\geqslant1}[\mathcal{X}_l^0](\LL^{-n}\LL^{-s})^l\in\MC[[\LL^{-s}]].
\end{align*}
Here $\LL^{-s}$ should be seen as a formal indeterminate. In what follows we shall always denote $\LL^{-s}$ by $T$; i.e., we define the local motivic zeta function \Zmotoft\ of $f$ as
\begin{equation*}
\Zmotoft=\LL^{-n}\sum_{l\geqslant1}[\mathcal{X}_l^0](\LL^{-n}T)^l\in\MC[[T]].
\end{equation*}
\end{definition}

The (local) motivic zeta function \Zmotoft\ is thus by definition a formal power series in $T$ with coefficients in \MC. By means of resolutions of singularities, Denef and Loeser proved that it is also a rational function in $T$. More precisely, they proved that there exists a finite set $S\subset\Zplusnul^2$ such that
\begin{equation*}
\Zmotoft\in\MC\left[\frac{\LL^{-\sigma}T^m}{1-\LL^{-\sigma}T^m}\right]_{(m,\sigma)\in S}\subset\MC[[T]].
\end{equation*}

Denef and Loeser also formulated a motivic version of the Monodromy Conjecture. One should be careful, however, when translating the $p$-adic (or topological) statement of the conjecture to the motivic setting; since it is not known whether \MC\ is a domain or not, the notion of pole of \Zmotoft\ is not straightforward.

\begin{conjecture}[Motivic Monodromy Conjecture]
Let $f(x)=f(x_1,\ldots,x_n)$ be a non\-zero polynomial in $\C[x_1,\ldots,x_n]$ satisfying $f(0)=0$. Then there exists a finite set $S\subset\Zplusnul^2$ such that
\begin{equation*}
\Zmotoft\in\MC[T]\left[\frac{1}{1-\LL^{-\sigma}T^m}\right]_{(m,\sigma)\in S}\subset\MC[[T]],
\end{equation*}
and such that, for each $(m,\sigma)\in S$, the complex number $e^{-2\pi i\sigma/m}$ is an eigenvalue of the local monodromy of $f$ at some point of the complex zero locus $f^{-1}(0)\subset\C^n$ close to the origin.
\end{conjecture}

The goal of this section is to prove the motivic Monodromy Conjecture for a polynomial in three variables that is non-degenerated over \C\ with respect to its Newton polyhedron, i.e., to prove the following motivic version of Theorem~\ref{mcigusandss}.

\begin{theorem}[Monodromy Conjecture for the local motivic zeta function of a non-degenerated surface singularity]\label{mcmotivndss}
Let $f(x,y,z)\in\C[x,y,z]$ be a nonzero polynomial in three variables satisfying $f(0,0,0)=0$, and let $U\subset\C^3$ be a neighborhood of the origin. Suppose that $f$ is non-degenerated over \C\ with respect to all the compact faces of its Newton polyhedron. Then there exists a finite set $S\subset\Zplusnul^2$ such that
\begin{equation*}
\Zmotoft\in\MC[T]\left[\frac{1}{1-\LL^{-\sigma}T^m}\right]_{(m,\sigma)\in S},
\end{equation*}
and such that, for each $(m,\sigma)\in S$, the complex number $e^{-2\pi i\sigma/m}$ is an eigen\-value of the local monodromy of $f$ at some point of $f^{-1}(0)\cap U\subset\C^3$.
\end{theorem}

We discuss a proof of Theorem~\ref{mcmotivndss} in Subsection~\ref{finsspmtms}. The essential formula for this proof is treated in the next subsection.

\subsection{A formula for the local motivic zeta function of a non-degenerated polynomial}
We will prove a combinatorial formula \`{a} la Denef--Hoornaert \cite{DH01} for the local motivic zeta function associated to a polynomial that is non-degenerated over the complex numbers. This was also done (in less detail) by Guibert \cite{guibert}. We state the formula below, but first we recall the precise notion of non-degeneracy we will be dealing with.

\begin{definition}[Non-degenerated over \C]
Let $f(x)=f(x_1,\ldots,x_n)$ be a nonzero polynomial in $\C[x_1,\ldots,x_n]$ satisfying $f(0)=0$. We say that $f$ is non-degenerated over \C\ with respect to all the compact faces of its Newton polyhedron \Gf, if for every compact face $\tau$ of \Gf, the zero locus $\ft^{-1}(0)\subset\C^n$ of \ft\ has no singularities in \Ccrossn\ (cfr.\ Notation~\ref{notftauart3}).
\end{definition}

Looking for an analogue for the motivic zeta function of Denef and Hoornaert's formula for Igusa's $p$-adic zeta function, one roughly expects to recover their formula with $p$, $p^{-s}$, and $N_{\tau}$ replaced by \LL, $T$, and the class of $\{x\in\Ccrossn\mid\ft(x)=0\}$ in the Grothendieck ring of complex varieties, respectively. We have to be careful however. Neither $T^{-1}$, nor $(1-\LL^{-1})^{-1}$ are elements in $\MC[[T]]$; especially, whereas $\sum_{\lambda=0}^{\infty}p^{-\lambda}=(1-p^{-1})^{-1}$ in \R, the corresponding $\sum_{\lambda=0}^{\infty}\LL^{-\lambda}=(1-\LL^{-1})^{-1}$ does not make sense in $\MC[[T]]$. To avoid the appearance of $T^{-1}$ in the formula, we adopt a slightly different notion of fundamental parallelepiped; to avoid dividing by $1-\LL^{-1}$, we have to treat compact faces lying in coordinate hyperplanes differently.\footnote{The reason is that for such a face $\tau$, at least one $v$ among the primitive vectors spanning \Dtu, has numerical data $(m(v),\sigma(v))=(0,1)$.}

\begin{theorem}\label{formlocmotzf}
Let $f(x)=f(x_1,\ldots,x_n)$ be a nonzero polynomial in $\C[x_1,\ldots,x_n]$ satisfying $f(0)=0$. Suppose that $f$ is non-degenerated over \C\ with respect to all the compact faces of its Newton polyhedron \Gf. Then the local motivic zeta function associated to $f$ is given by
\begin{multline*}
\Zmotoft=\\
\sum_{\substack{\tau\mathrm{\ compact\ face\ of\ }\Gf,\\\tau\nsubseteq\{x_{\rho}=0\}\mathrm{\ for\ all\ }\rho}}\!L_{\tau}S(\Dtu)+\sum_{\substack{\tau\mathrm{\ compact\ face\ of\ }\Gf,\\\tau\subset\{x_{\rho}=0\}\mathrm{\ for\ some\ }\rho}}\!L_{\tau}'S(\Dtu)'\in\MC[[T]],
\end{multline*}
where the $L_{\tau},S(\Dtu),L_{\tau}',S(\Dtu)'$ are as defined below.

For $\tau$ not contained in any coordinate hyperplane, we have
\begin{gather*}
L_{\tau}=\bigl(1-\LL^{-1}\bigr)^n-\LL^{-n}[\mathcal{X}_{\tau}]\frac{1-T}{1-\LL^{-1}T}\in\MC[[T]],\\\shortintertext{with}
\mathcal{X}_{\tau}=\left\{x\in\Ccrossn\;\middle\vert\;\ft(x)=0\right\},\\\shortintertext{and}
S(\Dtu)=\sum_{k\in\Zn\cap\Delta_{\tau}}\LL^{-\sigma(k)}T^{m(k)}\in\MC[[T]].
\end{gather*}
By $S(\Dtu)\in\MC[[T]]$ we mean, more precisely, the following. First choose a decomposition $\{\delta_i\}_{i\in I}$ of the cone \Dtu\ into simplicial cones $\delta_i$ without introducing new rays, and put $S(\Dtu)=\sum_{i\in I}S(\delta_i)$, with
\begin{equation*}
S(\delta_i)=\sum_{k\in\Zn\cap\delta_i}\LL^{-\sigma(k)}T^{m(k)}\in\MC[[T]]
\end{equation*}
for all $i\in I$. Then assuming that the cone $\delta_i$ is strictly positively spanned by the linearly independent primitive vectors $v_j$, $j\in J_i$, in $\Zplusn\setminus\{0\}$, the element $S(\delta_i)\in\MC[[T]]$ is defined as\,\footnote{Since $\tau$ is not contained in any coordinate hyperplane, all $m(v_j)$ are positive integers. Hence $\left(1-\LL^{-\sigma(v_j)}T^{m(v_j)}\right)^{-1}=\sum_{\lambda=0}^{\infty}\left(\LL^{-\sigma(v_j)}T^{m(v_j)}\right)^{\lambda}\in\MC[[T]]$ for all $j\in\bigcup_{i\in I}J_i$.}
\begin{equation*}
S(\delta_i)=\frac{\tilde{\Sigma}(\delta_i)}{\prod_{j\in J_i}\bigl(1-\LL^{-\sigma(v_j)}T^{m(v_j)}\bigr)}\in\MC[[T]],
\end{equation*}
with
\begin{equation*}
\tilde{\Sigma}(\delta_i)=\sum_h\LL^{-\sigma(h)}T^{m(h)}\in\MC[T],
\end{equation*}
where $h$ runs through the elements of the set
\begin{equation*}
\tilde{H}(v_j)_{j\in J_i}=\Z^n\cap\tilde{\lozenge}(v_j)_{j\in J_i},
\end{equation*}
with
\begin{equation*}
\tilde{\lozenge}(v_j)_{j\in J_i}=\left\{\sum\nolimits_{j\in J_i}h_jv_j\;\middle\vert\;h_j\in(0,1]\text{ for all }j\in J_i\right\}
\end{equation*}
the fundamental parallelepiped\,\footnote{with opposite boundaries as before} spanned by the vectors $v_j$, $j\in J_i$.

Suppose now that the compact face $\tau$ of \Gf\ is contained in at least one coordinate hyperplane. Define $P_{\tau}\subset\{1,\ldots,n\}$ such that $\rho\in P_{\tau}$ if and only if $\tau\subset\{(x_1,\ldots,x_n)\in\Rplusn\mid x_{\rho}=0\}$. Note that $1\leqslant\abs{P_{\tau}}\leqslant n-1$ and that \ft\ only depends on the variables $x_{\rho}$, $\rho\not\in P_{\tau}$. If we put
\begin{equation*}
\mathcal{X}_{\tau}'=\left\{(x_{\rho})_{\rho\not\in P_{\tau}}\in(\C^{\times})^{n-\abs{P_{\tau}}}\;\middle\vert\;\ft(x_{\rho})_{\rho\not\in P_{\tau}}=0\right\},
\end{equation*}
then we have
\begin{equation*}
L_{\tau}'=\bigl(1-\LL^{-1}\bigr)^{n-\abs{P_{\tau}}}-\LL^{-(n-\abs{P_{\tau}})}[\mathcal{X}_{\tau}']\frac{1-T}{1-\LL^{-1}T}\in\MC[[T]].
\end{equation*}
Denoting the standard basis of \Rn\ by $(e_{\rho})_{1\leqslant\rho\leqslant n}$, it follows that \Dtu\ is strictly positively spanned by the vectors $e_{\rho}$, $\rho\in P_{\tau}$, and one or more other primitive vectors $v_j$, $j\in J_{\tau}$, in $\Zplusn\setminus\{0\}$.\footnote{Indeed, as $\tau$ is compact and contained in $\bigcap_{\rho\in P_{\tau}}\{x_{\rho}=0\}$, we have that $\dim\tau\leqslant n-\abs{P_{\tau}}-1$; hence $\dim\Dtu\geqslant\abs{P_{\tau}}+1$.} Choose a decomposition $\{\delta_i\}_{i\in I}$ of the cone \Dtu\ into simplicial cones $\delta_i$ without introducing new rays, and assume that $\delta_i$ is strictly positively spanned by the linearly independent primitive vectors $e_{\rho},v_j$; $\rho\in P_i,j\in J_i$; with $\emptyset\subset P_i\subset P_{\tau}$ and $\emptyset\varsubsetneq J_i\subset J_{\tau}$. For $i\in I$, put
\begin{align*}
\delta_i'&=\tilde{\lozenge}(e_{\rho})_{\rho\in P_i}+\cone(v_j)_{j\in J_i}\\
&=\left\{\sum\nolimits_{\rho\in P_i}h_{\rho}e_{\rho}+\sum\nolimits_{j\in J_i}\lambda_jv_j\;\middle\vert\;h_{\rho}\in(0,1],\lambda_j\in\Rplusnul\text{ for all }\rho,j\right\}\subset\delta_i.
\end{align*}
Then $S(\Dtu)'$ is given by\,\footnote{Note again that all $m(v_j)$ are positive; therefore, $\left(1-\LL^{-\sigma(v_j)}T^{m(v_j)}\right)^{-1}\in\MC[[T]]$ for all $j\in J_{\tau}=\bigcup_{i\in I}J_i$.}
\begin{align*}
S(\Dtu)'&=\sum_{i\in I}\bigl(1-\LL^{-1}\bigr)^{\abs{P_{\tau}}-\abs{P_i}}\sum_{k\in\Zn\cap\delta_i'}\LL^{-\sigma(k)}T^{m(k)}\\
&=\sum_{i\in I}\bigl(1-\LL^{-1}\bigr)^{\abs{P_{\tau}}-\abs{P_i}}\frac{\tilde{\Sigma}(\delta_i)}{\prod_{j\in J_i}\bigl(1-\LL^{-\sigma(v_j)}T^{m(v_j)}\bigr)}\in\MC[[T]],
\end{align*}
with
\begin{equation*}
\tilde{\Sigma}(\delta_i)=\sum_h\LL^{-\sigma(h)}T^{m(h)}\in\MC[T],
\end{equation*}
where $h$ runs through the elements of the set
\begin{equation*}
\tilde{H}(e_{\rho},v_j)_{\rho,j}=\Z^n\cap\left\{\sum\nolimits_{\rho\in P_i}h_{\rho}e_{\rho}+\sum\nolimits_{j\in J_i}h_jv_j\;\middle\vert\;h_{\rho},h_j\in(0,1]\text{ for all }\rho,j\right\}.
\end{equation*}
\end{theorem}

The formula as stated above is obtained from Denef and Hoornaert's formula by first replacing $p$, $p^{-s}$, and $N_{\tau}$ by their proper analogues and then rewriting the formula in such a way that everything lives in $\MC[[T]]$. The proof is naturally similar to its $p$-adic counterpart, but we have to make adaptations due to some restrictions in comparison with the $p$-adic case.

One important restriction is that the (naive) motivic measure is not $\sigma$-additive. As mentioned earlier, we can no longer give meaning to countable sums of measures as $\sum_{\lambda=0}^{\infty}\LL^{-\lambda}$ in $\MC$. This results in a necessary different treatment of compact faces that are contained in coordinate hyperplanes. It also makes that we have---in some sense---less measurable subsets and therefore less freedom in the way we calculate things. For example, where in the $p$-adic case we start the proof by splitting up the integration domain $p\Zpn$ according to the $p$-order of its elements, we cannot copy this approach in the present setting, as it would give rise to unmeasurable sets. Another example is the following.

In the $p$-adic case, when calculating $\int_{p\Zpn}\abs{f(x)}^s\abs{dx}$, we could ignore the $x\in p\Zpn$ with one or more coordinates equal to zero, because this part of the integration domain has measure zero. In the motivic setting, working with the naive motivic measure, we don't have this luxury; the corresponding $\{(\phi_1,\ldots,\phi_n)\in(t\C[[t]])^n\mid\phi_{\rho}=0\text{ for some }\rho\}$ is not a cylindric subset of $(t\C[[t]])^n$, hence is not measurable. In what follows, we adapt some familiar notions to better describe this new situation.

We consider the extended non-negative real numbers $\Rplusbar=\Rplus\cup\{\infty\}$ with the usual order \lq$\leqslant$\rq\ and addition \lq$+$\rq. We extend the usual multiplication in \Rplus\ to a multiplication in \Rplusbar\ by putting $\infty\cdot0=0\cdot\infty=0$ and $\infty\cdot x=x\cdot\infty=\infty$ for $x\in\Rplusnulbar=\Rplusnul\cup\{\infty\}$. This allows us to also extend the dot product on \Rplusn\ to a dot product
\begin{equation*}
\cdot:\Rplusbarn\times\Rplusbarn\to\Rplusbar:\bigl((x_{\rho})_{\rho},(y_{\rho})_{\rho}\bigr)\mapsto(x_{\rho})_{\rho}\cdot(y_{\rho})_{\rho}=\sum\nolimits_{\rho=1}^nx_{\rho}y_{\rho}
\end{equation*}
on \Rplusbarn. The motivation for this definition is that, in this way,
\begin{equation*}
\ord_t\phi^{\omega}=\ord_t\phi_1^{\omega_1}\cdots\phi_n^{\omega_n}=(\ord_t\phi_1,\ldots,\ord_t\phi_n)\cdot(\omega_1,\ldots,\omega_n)=(\ord_t\phi)\cdot\omega
\end{equation*}
for $\phi=(\phi_1,\ldots,\phi_n)\in(t\C[[t]])^n$ and $\omega=(\omega_1,\ldots,\omega_n)\in\Rplusn$, even if $\phi_{\rho}=0$ for some $\rho$.

Next we extend $m(\cdot)$ and $F(\cdot)$ to \Rplusbarn\ in the expected way:
\begin{equation*}
m(k)=\inf_{x\in\Gf}k\cdot x=\min_{\omega\in\supp(f)}k\cdot\omega\in\Rplusbar,\qquad F(k)=\{x\in\Gf\mid k\cdot x=m(k)\}
\end{equation*}
for $k\in\Rplusbarn$. We have the following properties.

\begin{proposition}
Let $k\in\Rplusbarn$ and put $P_k=\{\rho\mid k_{\rho}=\infty\}\subset\{1,\ldots,n\}$.
\begin{enumerate}
\item If $k=(0,\ldots,0)$ or $m(k)=\infty$, then $F(k)=\Gf$, otherwise $F(k)$ is a proper face of \Gf;
\item the face $F(k)$ is compact if and only if $k\in\Rplusnulbarn$ and $m(k)<\infty$;
\item if $P_k\neq\emptyset$ and $m(k)<\infty$, then $F(k)$ is contained in $\bigcap_{\rho\in P_k}\{x_{\rho}=0\}$.
\end{enumerate}
\end{proposition}

The map $F:\Rplusbarn\to\{\text{faces of }\Gf\}$ induces an equivalence relation on \Rplusbarn. For every face $\tau$ of \Gf, we put $\Dti=F^{-1}(\tau)$ and call it the (extended) cone associated to $\tau$. These equivalence classes are subject to the following properties.

\begin{proposition}
Let $\tau$ be a face of \Gf\ and put $\emptyset\subset P_{\tau}=\{\rho\mid\tau\subset\{x_{\rho}=0\}\}\subsetneq\{1,\ldots,n\}$. Suppose that \Dtu\ is strictly positively spanned by the primitive vectors $e_{\rho},v_j$; $\rho\in P_{\tau},j\in J_{\tau}$; in $\Zplusn\setminus\{0\}$.\footnote{We agree that $J_{\tau}=\emptyset$ if $\tau=\Gf$; this is, $\Delta_{\Gf}=\{(0,\ldots,0)\}$ is strictly positively spanned by the empty set.} Then we have
\begin{enumerate}
\item $\Dtu=\Dti\cap\Rplusn$;
\item if $\tau=\Gf$, then $\Dti=\{(0,\ldots,0)\}\cup\{k\in\Rplusbarn\mid m(k)=\infty\}$;
\item if $\tau$ is a proper face of \Gf, then
\begin{equation*}
\Dti=\left\{\sum\nolimits_{\rho\in P_{\tau}}\bar{\lambda}_{\rho}e_{\rho}+\sum\nolimits_{j\in J_{\tau}}\lambda_jv_j\;\middle\vert\;\bar{\lambda}_{\rho}\in\Rplusnulbar,\lambda_j\in\Rplusnul\text{ for all }\rho,j\right\};
\end{equation*}
\item in particular, if $\tau$ is a proper face not contained in any coordinate hyperplane, then $\Dti=\Dtu$.
\end{enumerate}
Furthermore,
\begin{enumerate}
\setcounter{enumi}{4}
\item the family $\{\Dti\mid\tau\text{ is a face of }\Gf\}$ of all extended cones forms a partition of \Rplusbarn, while
\item $\{\Dti\mid\tau\text{ is a compact face of }\Gf\}$ partitions $\{k\in\Rplusnulbarn\mid m(k)<\infty\}$.
\end{enumerate}
\end{proposition}

Let us do some more preliminary work to facilitate the actual proof of the theorem. In the lemmas and corollaries that follow we calculate the (naive) motivic measure of some cylindric subsets of $(t\C[[t]])^n$, but first we introduce a notation.

\begin{notation}
For $K\subset\Zplusnulbarn=(\Zplusnul\cup\{\infty\})^n$ and $l\in\Zplusnul$, we put
\begin{equation*}
X_{K,l}=\{\phi\in(t\C[[t]])^n\mid\ord_t\phi\in K\text{ and }\ord_tf(\phi)=l\}.
\end{equation*}
If $k\in\Zplusnuln$, then we usually write $X_{k,l}$ instead of $X_{\{k\},l}$.
\end{notation}

\begin{lemma}\label{motlemmaeen}
Let $f$ be as in Theorem~\ref{formlocmotzf}. Suppose that $\tau$ is a compact face of \Gf, and put $\mathcal{X}_{\tau}=\left\{x\in\Ccrossn\;\middle\vert\;\ft(x)=0\right\}$. Let $k\in\Zn\cap\Dtu$ and $l\in\Zplusnul$. Then $k\in\Zplusnuln$, $m(k)\in\Zplusnul$, and
\begin{equation*}
\mu(X_{k,l})=
\begin{cases}
{\ds0,}&\text{if \,$l<m(k)$};\\
{\ds\bigl((\LL-1)^n-[\mathcal{X}_{\tau}]\bigr)\LL^{-n-\sigma(k)},}&\text{if \,$l=m(k)$};\\
{\ds[\mathcal{X}_{\tau}](\LL-1)\LL^{-n+m(k)-\sigma(k)-l},}&\text{if \,$l>m(k)$}.
\end{cases}
\end{equation*}
\end{lemma}

\begin{proof}
Let $\phi=(\phi_1,\ldots,\phi_n)\in(t\C[[t]])^n$ with $\ord_t\phi=k=(k_1,\ldots,k_n)$, and let $\psi=(\psi_1,\ldots,\psi_n)\in(\C[[t]]^{\times})^n$ be such that $\phi_{\rho}=t^{k_{\rho}}\psi_{\rho}$ for all $\rho$. Then, for $\omega=(\omega_1,\ldots,\omega_n)\in\Zplusn$, we have that
\begin{equation*}
\phi^{\omega}=\phi_1^{\omega_1}\cdots\phi_n^{\omega_n}=t^{k_1\omega_1}\psi_1^{\omega_1}\cdots t^{k_n\omega_n}\psi_n^{\omega_n}=t^{k\cdot\omega}\psi^{\omega}.
\end{equation*}
Write
\begin{equation*}
f(x)=\sum_{\omega\in\Zplusn}a_{\omega}x^{\omega}\qquad\text{and}\qquad\ft(x)=\sum_{\omega\in\Zn\cap\tau}a_{\omega}x^{\omega}.
\end{equation*}
It follows from $k\in\Dtu$ that $k\cdot\omega=m(k)$ for all $\omega\in\supp(f)\cap\tau$,\footnote{Recall that $\supp(f)=\{\omega\in\Zplusn\mid a_{\omega}\neq0\}$.} whereas $k\cdot\omega\geqslant m(k)+1$ for $\omega\in\supp(f)\setminus\tau$. Hence we can write $f(\phi)$ as
\begin{gather*}
f(\phi)=t^{m(k)}\bigl(\ft(\psi)+t\tilde{f}_{\tau,k}(t,\psi)\bigr),\\\shortintertext{with}
\tilde{f}_{\tau,k}(t,\psi)=\sum_{\omega\in\supp(f)\setminus\tau}a_{\omega}t^{k\cdot\omega-m(k)-1}\psi^{\omega}.
\end{gather*}

First of all, we see that $\ord_tf(\phi)\geqslant m(k)$; hence $\mu(X_{k,l})=0$ for $l<m(k)$. Secondly, we observe that $\ord_tf(\phi)=m(k)$ if and only if $\ord_t\ft(\psi)=0$. If we write $\psi=(\psi_{\rho,0}+\psi_{\rho,1}t+\psi_{\rho,2}t^2+\cdots)_{1\leqslant\rho\leqslant n}$, then $\ft(\psi)\in\ft(\psi_{1,0},\ldots,\psi_{n,0})+t\C[[t]]$. Consequently, the set
\begin{align*}
\widetilde{X}_{k,m(k)}&=\{\psi\in(\C[[t]]^{\times})^n\mid\ord_tf(\phi)=m(k)\}\\
&=\{\psi\in(\C[[t]]^{\times})^n\mid\ft(\psi_{1,0},\ldots,\psi_{n,0})\neq0\}\\
&=\pi_0^{-1}\bigl(\pi_0\bigl(\widetilde{X}_{k,m(k)}\bigr)\bigr)
\end{align*}
is a cylindric subset of $(\C[[t]])^n$ of motivic measure
\begin{equation*}
\mu\bigl(\widetilde{X}_{k,m(k)}\bigr)=\bigl[\pi_0\bigl(\widetilde{X}_{k,m(k)}\bigr)\bigr]\LL^{-n}=[\Ccrossn\setminus\mathcal{X}_{\tau}]\LL^{-n}=\bigl((\LL-1)^n-[\mathcal{X}_{\tau}]\bigr)\LL^{-n}.
\end{equation*}
The corresponding set $X_{k,m(k)}$ has motivic measure
\begin{equation*}
\mu\bigl(X_{k,m(k)}\bigr)=\LL^{-\sigma(k)}\mu\bigl(\widetilde{X}_{k,m(k)}\bigr)=\bigl((\LL-1)^n-[\mathcal{X}_{\tau}]\bigr)\LL^{-n-\sigma(k)}.
\end{equation*}

Suppose now that $l>m(k)$. Let us first calculate the measure of
\begin{equation*}
X_{k,\geqslant l}=\{\phi\in(t\C[[t]])^n\mid\ord_t\phi=k\text{ and }\ord_tf(\phi)\geqslant l\}.
\end{equation*}
From our expression for $f(\phi)$ we see that $\ord_tf(\phi)\geqslant l$ if and only if
\begin{equation*}
\ord_t\bigl(\ft(\psi)+t\tilde{f}_{\tau,k}(t,\psi)\bigr)\geqslant l'=l-m(k)\geqslant1,
\end{equation*}
or, equivalently, if and only if
\begin{equation}\label{motcondle}
\ft(\psi)+t\tilde{f}_{\tau,k}(t,\psi)\equiv0\mod t^{l'}\C[[t]].
\end{equation}

Whether $\psi$ satisfies the above condition, only depends on the complex numbers
\begin{equation*}
\psi_{\rho,\kappa};\qquad\rho=1,\ldots,n;\ \kappa=0,\ldots,l'-1.
\end{equation*}
Clearly, for $\psi$ to satisfy \eqref{motcondle}, it is necessary that $\ft(\psi_{1,0},\ldots,\psi_{n,0})=0$. Fix such an $n$-tuple $(\psi_{1,0},\ldots,\psi_{n,0})$. Since $f$ is non-degenerated over \C\ with respect to $\tau$, there exists a $\rho_0$ such that $(\partial\ft/\partial x_{\rho_0})(\psi_{1,0},\ldots,\psi_{n,0})\neq0$. Therefore, Hensel's lifting lemma returns, for every free choice of $(n-1)(l'-1)$ complex numbers
\begin{equation*}
\psi_{\rho,\kappa};\qquad\rho=1,\ldots,\rho_0-1,\rho_0+1,\ldots,n;\ \kappa=1,\ldots,l'-1;
\end{equation*}
unique $\psi_{\rho_0,1},\ldots,\psi_{\rho_0,l'-1}\in\C$ such that $\psi$ satisfies \eqref{motcondle}. It follows that
\begin{align*}
\widetilde{X}_{k,\geqslant l}&=\{\psi\in(\C[[t]]^{\times})^n\mid\ord_tf(\phi)\geqslant l\}\\
&=\{\psi\in(\C[[t]]^{\times})^n\mid\ft(\psi)+t\tilde{f}_{\tau,k}(t,\psi)\equiv0\bmod t^{l'}\C[[t]]\}\\
&=\pi_{l'-1}^{-1}\bigl(\pi_{l'-1}\bigl(\widetilde{X}_{k,\geqslant l}\bigr)\bigr)
\end{align*}
is a cylindric subset of $(\C[[t]])^n$ of motivic measure
\begin{multline*}
\mu\bigl(\widetilde{X}_{k,\geqslant l}\bigr)=\bigl[\pi_{l'-1}\bigl(\widetilde{X}_{k,\geqslant l}\bigr)\bigr]\LL^{-nl'}\\
=\bigl[\mathcal{X}_{\tau}\times\C^{(n-1)(l'-1)}\bigr]\LL^{-nl'}=[\mathcal{X}_{\tau}]\LL^{-n+m(k)-l+1}.
\end{multline*}
The corresponding set $X_{k,\geqslant l}$ therefore has motivic measure
\begin{equation*}
\mu(X_{k,\geqslant l})=\LL^{-\sigma(k)}\mu\bigl(\widetilde{X}_{k,\geqslant l}\bigr)=[\mathcal{X}_{\tau}]\LL^{-n+m(k)-\sigma(k)-l+1}.
\end{equation*}

By additivity of the motivic measure, finally, we obtain that
\begin{gather*}
\mu(X_{k,l})=\mu(X_{k,\geqslant l}\setminus X_{k,\geqslant l+1})=\mu(X_{k,\geqslant l})-\mu(X_{k,\geqslant l+1})=\\
[\mathcal{X}_{\tau}]\LL^{-n+m(k)-\sigma(k)-l+1}-[\mathcal{X}_{\tau}]\LL^{-n+m(k)-\sigma(k)-l}=[\mathcal{X}_{\tau}](\LL-1)\LL^{-n+m(k)-\sigma(k)-l},
\end{gather*}
which concludes the proof of the lemma.
\end{proof}

\begin{corollary}\label{motcoroleen}
Let $f$ be as in Theorem~\ref{formlocmotzf} and suppose that $\tau$ is a com\-pact face of \Gf\ that is not contained in any coordinate hyperplane. Let $l\in\Zplusnul$. Then $X_{\Zn\cap\Dtu,l}$ is a cylindric subset of $(t\C[[t]])^n$; i.e., $\mu(X_{\Zn\cap\Dtu,l})$ exists.
\end{corollary}

\begin{proof}
Clearly, $X_{\Zn\cap\Dtu,l}$ equals the disjoint union
\begin{equation}\label{disjuncoreen}
X_{\Zn\cap\Dtu,l}=\bigsqcup_{k\in\Zn\cap\Dtu}X_{k,l}.
\end{equation}
We know that $X_{k,l}=\emptyset$ for $k\in\Zplusnuln$ with $m(k)>l$; hence we may restrict the above union to $k$ satisfying $m(k)\leqslant l$. Choose $x\in\tau\cap\Rplusnuln\neq\emptyset$. Then $m(k)=k\cdot x$ for all $k\in\Dtu$. Moreover, $\{k\in\Rplusn\mid k\cdot x\leqslant l\}$ is a closed and bounded subset of \Rn, containing finitely many integral points. The union \eqref{disjuncoreen} so boils down to a finite disjoint union of sets $X_{k,l}$ that, by Lemma~\ref{motlemmaeen}, are cylindric subsets of $(t\C[[t]])^n$. Consequently,
\begin{equation*}
\mu(X_{\Zn\cap\Dtu,l})=\sum_{\substack{k\in\Zn\cap\Dtu,\\m(k)\leqslant l}}\mu(X_{k,l})
\end{equation*}
is well-defined.
\end{proof}

\begin{lemma}\label{motlemmatwee}
Let $f$ be as in Theorem~\ref{formlocmotzf} and suppose that $\tau$ is a compact face of \Gf\ that is contained in at least one coordinate hyperplane. Define $P_{\tau}\subset\{1,\ldots,n\}$ such that $\rho\in P_{\tau}$ if and only if $\tau\subset\{x_{\rho}=0\}$, and denote
\begin{equation*}
\mathcal{X}_{\tau}'=\left\{(x_{\rho})_{\rho\not\in P_{\tau}}\in(\C^{\times})^{n-\abs{P_{\tau}}}\;\middle\vert\;\ft(x_{\rho})_{\rho\not\in P_{\tau}}=0\right\}.
\end{equation*}
Let $k\in\Zn\cap\Dtu$, $\emptyset\subset P\subset P_{\tau}$, and put
\begin{equation*}
k\vee P=k+\sum_{\rho\in P}\Zplusbar e_{\rho}\subset\Zplusnulbarn\cap\Dti.\footnote{Hereby $(e_{\rho})_{1\leqslant\rho\leqslant n}$ denotes the standard basis of \Rn, and $\Zplusbar=\Zplus\cup\{\infty\}\subset\Rplusbar$.}
\end{equation*}
Note that $m(k')=m(k)\in\Zplusnul$ for all $k'\in k\vee P$. Finally, let $l\in\Zplusnul$. Then
\begin{multline*}
\mu(X_{k\vee P,l})=\\
\begin{cases}
{\ds0,}&\text{if \,$l<m(k)$};\\
{\ds\bigl((\LL-1)^{n-\abs{P_{\tau}}}-[\mathcal{X}_{\tau}']\bigr)(\LL-1)^{\abs{P_{\tau}}-\abs{P}}\LL^{-n+\abs{P}-\sigma(k)},}&\text{if \,$l=m(k)$};\\
{\ds[\mathcal{X}_{\tau}'](\LL-1)^{\abs{P_{\tau}}-\abs{P}+1}\LL^{-n+\abs{P}+m(k)-\sigma(k)-l},}&\text{if \,$l>m(k)$}.
\end{cases}
\end{multline*}
\end{lemma}

\begin{proof}
The proof is analogous to the proof of Lemma~\ref{motlemmaeen}. Essential is that \ft\ only depends on the variables $x_{\rho}$, $\rho\not\in P_{\tau}$. The measure of
\begin{gather*}
X_{k\vee P,\geqslant l}=\{\phi\in(t\C[[t]])^n\mid\ord_t\phi\in k\vee P\text{ and }\ord_tf(\phi)\geqslant l\}\\\shortintertext{equals}
\mu(X_{k\vee P,\geqslant l})=[\mathcal{X}_{\tau}'](\LL-1)^{\abs{P_{\tau}}-\abs{P}}\LL^{-n+\abs{P}+m(k)-\sigma(k)-l+1}
\end{gather*}
for $l>m(k)$.
\end{proof}

\begin{corollary}\label{motcoroltwee}
Let $f$ be as in Theorem~\ref{formlocmotzf} and suppose that $\tau$ is a compact face of \Gf\ that is contained in at least one coordinate hyperplane. Let $l\in\Zplusnul$. Then $X_{\Zplusnulbarn\cap\Dti,l}$ is a cylindric subset of $(t\C[[t]])^n$; i.e., $\mu(X_{\Zplusnulbarn\cap\Dti,l})$ exists.
\end{corollary}

\begin{proof}
Put $\emptyset\subsetneq P_{\tau}=\{\rho\mid\tau\subset\{x_{\rho}=0\}\}\subsetneq\{1,\ldots,n\}$ as usual, and suppose that \Dtu\ is strictly positively spanned by the primitive vectors $e_{\rho},v_j$; $\rho\in P_{\tau},j\in J_{\tau}\neq\emptyset$; in $\Zplusn\setminus\{0\}$. Choose a decomposition $\{\delta_i\}_{i\in I}$ of the cone \Dtu\ into simplicial cones $\delta_i$ without introducing new rays, and assume that $\delta_i$ is strictly positively spanned by the linearly independent primitive vectors $e_{\rho},v_j$; $\rho\in P_i,j\in J_i$; with $\emptyset\subset P_i\subset P_{\tau}$ and $\emptyset\varsubsetneq J_i\subset J_{\tau}$. Then the \emph{extended simplicial cones}
\begin{equation*}
\delta_i^{\infty}=\left\{\sum\nolimits_{\rho\in P_i}\bar{\lambda}_{\rho}e_{\rho}+\sum\nolimits_{j\in J_i}\lambda_jv_j\;\middle\vert\;\bar{\lambda}_{\rho}\in\Rplusnulbar,\lambda_j\in\Rplusnul\text{ for all }\rho,j\right\},\quad i\in I,
\end{equation*}
clearly partition \Dti, and so we are looking at the finite disjoint union
\begin{equation*}
X_{\Zplusnulbarn\cap\Dti,l}=\bigsqcup_{i\in I}X_{\Zplusnulbarn\cap\delta_i^{\infty},l}.
\end{equation*}

Next we decompose $\Zplusnulbarn\cap\delta_i^{\infty}$, and subsequently $X_{\Zplusnulbarn\cap\delta_i^{\infty},l}$, as
\begin{equation}\label{disjuncortwee}
\Zplusnulbarn\cap\delta_i^{\infty}=\bigsqcup_{k\in\Zn\cap\delta_i'}k\vee P_i\quad\text{ and }\quad X_{\Zplusnulbarn\cap\delta_i^{\infty},l}=\bigsqcup_{k\in\Zn\cap\delta_i'}X_{k\vee P_i,l},
\end{equation}
with
\begin{equation*}
\delta_i'=\left\{\sum\nolimits_{\rho\in P_i}h_{\rho}e_{\rho}+\sum\nolimits_{j\in J_i}\lambda_jv_j\;\middle\vert\;h_{\rho}\in(0,1],\lambda_j\in\Rplusnul\text{ for all }\rho,j\right\}\subset\delta_i.
\end{equation*}

Recall that $X_{k\vee P_i,l}=\emptyset$ for $k\in\Zplusnuln$ with $m(k)>l$. We may therefore restrict the second union of \eqref{disjuncortwee} to $k$ satisfying $m(k)\leqslant l$. Choose $x=(x_1,\ldots,x_n)\in\tau$ with $x_{\rho}>0$ for all $\rho\not\in P_{\tau}$. Then $m(k)=k\cdot x$ for all $k\in\delta_i'\subset\Dtu$, and $\{k\in\delta_i'\mid k\cdot x\leqslant l\}$ is a bounded subset of \Rn, containing finitely many integral points. It follows that the second union of \eqref{disjuncortwee} is actually a finite disjoint union of cylindric\footnote{See Lemma~\ref{motlemmatwee}.} subsets $X_{k\vee P_i,l}$ of $(t\C[[t]])^n$. We conclude that the finite sum
\begin{equation*}
\mu\bigl(X_{\Zplusnulbarn\cap\Dti,l}\bigr)=\sum_{i\in I}\sum_{\substack{k\in\Zn\cap\delta_i',\\m(k)\leqslant l}}\mu(X_{k\vee P_i,l})
\end{equation*}
is well-defined.
\end{proof}

\begin{proof}[Proof of Theorem~\ref{formlocmotzf}]
By definition we have\footnote{Note the difference between $\mathcal{X}_l^0$ (see Definition~\ref{deflocmotzf}) and $X_l^0$.}
\begin{gather*}
\Zmotoft=\LL^{-n}\sum_{l\geqslant1}[\mathcal{X}_l^0](\LL^{-n}T)^l=\sum_{l\geqslant1}\mu(X_l^0)T^l,\\\shortintertext{with}
X_l^0=\{\phi\in(t\C[[t]])^n\mid\ord_tf(\phi)=l\},\qquad l\in\Zplusnul.
\end{gather*}
If $\phi\in X_l^0$, then $\ord_t\phi\in\Zplusnulbarn$ and $m(\ord_t\phi)\leqslant\ord_tf(\phi)=l<\infty$. Further, $\{\Dti\mid\tau\text{ is a compact face of }\Gf\}$ forms a partition of $\{k\in\Rplusnulbarn\mid m(k)<\infty\}$. Hence we may write each $X_l^0$ as the finite disjoint union
\begin{equation*}
X_l^0=\bigsqcup_{\tau}X_{\Zplusnulbarn\cap\Dti,l}=\bigsqcup_{\substack{\tau,\\P_{\tau}=\emptyset}}X_{\Zn\cap\Dtu,l}\sqcup\bigsqcup_{\substack{\tau,\\P_{\tau}\neq\emptyset}}X_{\Zplusnulbarn\cap\Dti,l},
\end{equation*}
where all unions are over compact faces $\tau$ of \Gf, and $P_{\tau}=\{\rho\mid\tau\subset\{x_{\rho}=0\}\}$ as usual. By Corollaries~\ref{motcoroleen} and \ref{motcoroltwee}, all $X_{\Zn\cap\Dtu,l}$ and $X_{\Zplusnulbarn\cap\Dti,l}$ are cylindric subsets of $(t\C[[t]])^n$, which allows us to write $\mu(X_l^0)$ as the finite sum
\begin{equation*}
\mu(X_l^0)=\sum_{\substack{\tau,\\P_{\tau}=\emptyset}}\mu\bigl(X_{\Zn\cap\Dtu,l}\bigr)+\sum_{\substack{\tau,\\P_{\tau}\neq\emptyset}}\mu\bigl(X_{\Zplusnulbarn\cap\Dti,l}\bigr).
\end{equation*}
This leads to
\begin{equation*}
\Zmotoft=\sum_{\substack{\tau,\\P_{\tau}=\emptyset}}\sum_{l\geqslant1}\mu\bigl(X_{\Zn\cap\Dtu,l}\bigr)T^l+\sum_{\substack{\tau,\\P_{\tau}\neq\emptyset}}\sum_{l\geqslant1}\mu\bigl(X_{\Zplusnulbarn\cap\Dti,l}\bigr)T^l.
\end{equation*}

If $\tau$ is not contained in any coordinate hyperplane, then by Corollary~\ref{motcoroleen}, we have
\begin{align}
\sum_{l\geqslant1}\mu\bigl(X_{\Zn\cap\Dtu,l}\bigr)T^l&\notag\\
&\hspace{-1cm}=\sum_{l\geqslant1}\,\sum_{\substack{k\in\Zn\cap\Dtu,\\m(k)\leqslant l}}\mu(X_{k,l})T^l\in\MC[[T]]\label{motexpeen}\\
&\hspace{-1cm}=\sum_{k\in\Zn\cap\Dtu}\,\sum_{l\geqslant m(k)}\mu(X_{k,l})T^l\label{motexptwee}\\
&\hspace{-1cm}=\sum_{k\in\Zn\cap\Dtu}\mu\bigl(X_{k,m(k)}\bigr)T^{m(k)}+\sum_{k\in\Zn\cap\Dtu}\,\sum_{l\geqslant m(k)+1}\mu(X_{k,l})T^l.\notag
\end{align}
Replacing the motivic measures $\mu(\cdot)$ by the expressions found in Lemma~\ref{motlemmaeen}, we obtain
\begin{align*}
\sum_{l\geqslant1}\mu\bigl(X_{\Zn\cap\Dtu,l}\bigr)T^l&=\bigl((\LL-1)^n-[\mathcal{X}_{\tau}]\bigr)\LL^{-n}\sum_{k\in\Zn\cap\Dtu}\LL^{-\sigma(k)}T^{m(k)}\\
&\quad\,+[\mathcal{X}_{\tau}](\LL-1)\LL^{-n}\sum_{k\in\Zn\cap\Dtu}\LL^{-\sigma(k)}\sum_{l\geqslant m(k)+1}\LL^{m(k)-l}T^l,
\end{align*}
and since
\begin{equation*}
\sum_{l\geqslant m(k)+1}\LL^{m(k)-l}T^l=T^{m(k)}\sum_{l\geqslant1}\LL^{-l}T^l=T^{m(k)}\frac{\LL^{-1}T}{1-\LL^{-1}T}\in\MC[[T]],
\end{equation*}
we eventually find
\begin{multline*}
\sum_{l\geqslant1}\mu\bigl(X_{\Zn\cap\Dtu,l}\bigr)T^l\\
=\left(\bigl(1-\LL^{-1}\bigr)^n-\LL^{-n}[\mathcal{X}_{\tau}]\frac{1-T}{1-\LL^{-1}T}\right)\sum_{k\in\Zn\cap\Dtu}\LL^{-\sigma(k)}T^{m(k)}.
\end{multline*}

This last sum, denoted $S(\Dtu)$, can be calculated as follows. First choose a decomposition $\{\delta_i\}_{i\in I}$ of the cone \Dtu\ into simplicial cones $\delta_i$ without introducing new rays. Then $S(\Dtu)=\sum_{i\in I}S(\delta_i)$, whereby
\begin{equation*}
S(\delta_i)=\sum_{k\in\Zn\cap\delta_i}\LL^{-\sigma(k)}T^{m(k)}
\end{equation*}
for all $i\in I$. Next assume that the cone $\delta_i$ is strictly positively spanned by the linearly independent primitive vectors $v_j$, $j\in J_i$, in $\Zplusn\setminus\{0\}$. Then $\Zn\cap\delta_i$ equals the finite disjoint union $\bigsqcup_hh+\sum_{j\in J_i}\Zplus v_j$, where $h$ runs through the elements of
\begin{equation*}
\tilde{H}(v_j)_{j\in J_i}=\Z^n\cap\tilde{\lozenge}(v_j)_{j\in J_i}=\Z^n\cap\left\{\sum\nolimits_{j\in J_i}h_jv_j\;\middle\vert\;h_j\in(0,1]\text{ for all }j\in J_i\right\}.
\end{equation*}
Consequently,
\begin{align*}
S(\delta_i)&=\adjustlimits\sum_{h\in\tilde{H}(v_j)_j}\sum_{(\lambda_j)_j\in\Zplus^{\abs{J_i}}}\LL^{-\sigma\left(h+\sum_j\lambda_jv_j\right)}T^{m\left(h+\sum_j\lambda_jv_j\right)};\\\intertext{then exploiting the linearity\footnotemark\ of $m(\cdot)$ on $\overbar{\Dtu}\supset\delta_i$, we find}
S(\delta_i)&=\sum_{h\in\tilde{H}(v_j)_j}\LL^{-\sigma(h)}T^{m(h)}\prod_{j\in J_i}\sum_{\lambda_j\geqslant0}\left(\LL^{-\sigma(v_j)}T^{m(v_j)}\right)^{\lambda_j}\\
&=\frac{\sum_{h\in\tilde{H}(v_j)_j}\LL^{-\sigma(h)}T^{m(h)}}{\prod_{j\in J_i}\bigl(1-\LL^{-\sigma(v_j)}T^{m(v_j)}\bigr)}\in\MC[[T]].
\end{align*}\footnotetext{Recall that for any $x\in\tau$ we have that $m(k)=k\cdot x$ for all $k\in\overbar{\Dtu}$.}%
Note that since all $m(v_j)$ are positive, we indeed obtain an element of $\MC[[T]]$.

The eventual formula for $\sum_{l\geqslant1}\mu\bigl(X_{\Zn\cap\Dtu,l}\bigr)T^l$ is thus
\begin{equation*}
\left(\bigl(1-\LL^{-1}\bigr)^n-\LL^{-n}[\mathcal{X}_{\tau}]\frac{1-T}{1-\LL^{-1}T}\right)\sum_{i\in I}\frac{\sum_{h\in\tilde{H}(v_j)_{j\in J_i}}\LL^{-\sigma(h)}T^{m(h)}}{\prod_{j\in J_i}\bigl(1-\LL^{-\sigma(v_j)}T^{m(v_j)}\bigr)}\in\MC[[T]],
\end{equation*}
as announced in the theorem. To rigorously prove that \eqref{motexpeen} equals this last expression in $\MC[[T]]$, in particular to defend the change of summation order in going from \eqref{motexpeen} to \eqref{motexptwee}, one compares the coefficients of $T^l$ in both elements and finds twice the same finite sum in $\MC$.

From now suppose that $\tau$ is contained in at least one coordinate hyperplane; i.e., $P_{\tau}\neq\emptyset$. Let \Dtu\ be strictly positively spanned by the primitive vectors $e_{\rho},v_j$; $\rho\in P_{\tau},j\in J_{\tau}\neq\emptyset$; in $\Zplusn\setminus\{0\}$. Choose a decomposition $\{\delta_i\}_{i\in I}$ of \Dtu\ into simplicial cones $\delta_i$ without introducing new rays, and assume that $\delta_i$ is strictly positively spanned by the linearly independent primitive vectors $e_{\rho},v_j$; $\rho\in P_i,j\in J_i$; with $\emptyset\subset P_i\subset P_{\tau}$ and $\emptyset\varsubsetneq J_i\subset J_{\tau}$. Finally, put $\delta_i'=\tilde{\lozenge}(e_{\rho})_{\rho\in P_i}+\cone(v_j)_{j\in J_i}$ as before.

We proceed as in the $P_{\tau}=\emptyset$ case. Corollary~\ref{motcoroltwee} yields
\begin{align*}
\sum_{l\geqslant1}\mu\bigl(X_{\Zplusnulbarn\cap\Dti,l}\bigr)T^l&\\
&\hspace{-2.6cm}=\sum_{l\geqslant1}\sum_{i\in I}\sum_{\substack{k\in\Zn\cap\delta_i',\\m(k)\leqslant l}}\mu(X_{k\vee P_i,l})T^l\in\MC[[T]]\\
&\hspace{-2.6cm}=\sum_{i\in I}\sum_{k\in\Zn\cap\delta_i'}\mu\bigl(X_{k\vee P_i,m(k)}\bigr)T^{m(k)}+\sum_{i\in I}\sum_{k\in\Zn\cap\delta_i'}\,\sum_{l\geqslant m(k)+1}\mu(X_{k\vee P_i,l})T^l.
\end{align*}
Then applying Lemma~\ref{motlemmatwee}, we find
\begin{align*}
&\,\sum_{l\geqslant1}\mu\bigl(X_{\Zplusnulbarn\cap\Dti,l}\bigr)T^l\\
&=\bigl((\LL-1)^{n-\abs{P_{\tau}}}-[\mathcal{X}_{\tau}']\bigr)\LL^{-n}\sum_{i\in I}(\LL-1)^{\abs{P_{\tau}}-\abs{P_i}}\LL^{\abs{P_i}}\sum_{k\in\Zn\cap\delta_i'}\LL^{-\sigma(k)}T^{m(k)}\\
&\hspace{.7cm}+[\mathcal{X}_{\tau}']\LL^{-n}\sum_{i\in I}(\LL-1)^{\abs{P_{\tau}}-\abs{P_i}+1}\LL^{\abs{P_i}}\sum_{k\in\Zn\cap\delta_i'}\LL^{-\sigma(k)}\sum_{l\geqslant m(k)+1}\LL^{m(k)-l}T^l\\
&=\left(\bigl(1-\LL^{-1}\bigr)^{n-\abs{P_{\tau}}}-\LL^{-(n-\abs{P_{\tau}})}[\mathcal{X}_{\tau}']\frac{1-T}{1-\LL^{-1}T}\right)\\
&\hspace{4.87cm}\cdot\sum_{i\in I}\bigl(1-\LL^{-1}\bigr)^{\abs{P_{\tau}}-\abs{P_i}}\sum_{k\in\Zn\cap\delta_i'}\LL^{-\sigma(k)}T^{m(k)}.
\end{align*}
This last double sum, which we denote by $S(\Dtu)'$, can be calculated in the same way as we calculated $S(\Dtu)$ in the $P_{\tau}=\emptyset$ case. We obtain
\begin{equation*}
S(\Dtu)'=\sum_{i\in I}\bigl(1-\LL^{-1}\bigr)^{\abs{P_{\tau}}-\abs{P_i}}\frac{\sum_h\LL^{-\sigma(h)}T^{m(h)}}{\prod_{j\in J_i}\bigl(1-\LL^{-\sigma(v_j)}T^{m(v_j)}\bigr)}\in\MC[[T]],
\end{equation*}
where $h$ runs through the elements of the set
\begin{equation*}
\tilde{H}(e_{\rho},v_j)_{\rho,j}=\Z^n\cap\left\{\sum\nolimits_{\rho\in P_i}h_{\rho}e_{\rho}+\sum\nolimits_{j\in J_i}h_jv_j\;\middle\vert\;h_{\rho},h_j\in(0,1]\text{ for all }\rho,j\right\}.
\end{equation*}
Note again that since all $m(v_j)$ are positive, we indeed find an element of $\MC[[T]]$. This concludes the proof of Theorem~\ref{formlocmotzf}.
\end{proof}

\subsection{A proof of the main theorem in the motivic setting}\label{finsspmtms}
In this final subsection we explain why (and how) Theorem~\ref{mcmotivndss} can be proved in the same way as Theorem~\ref{mcigusandss}. Let us start with a small overview.

Let $f$ be as in Theorem~\ref{mcmotivndss}. By the general rationality result of Denef--Loeser, we know that there exists a finite set $\widetilde{S}\subset\Zplusnul^2$ such that
\begin{equation*}
\Zmotoft\in\MC\left[\frac{\LL^{-\sigma}T^m}{1-\LL^{-\sigma}T^m}\right]_{(m,\sigma)\in\widetilde{S}}\subset\MC[[T]].
\end{equation*}
Our formula for non-degenerated $f$ (Theorem~\ref{formlocmotzf}), on the other hand, yields
\begin{equation*}
\Zmotoft\in\MC[T]\left[\frac{1}{1-\LL^{-\sigma}T^m}\right]_{(m,\sigma)\in S}\subset\MC[[T]],
\end{equation*}
whereby
\begin{multline}\label{defsetS}
S=\{(1,1)\}\cup\{(m(v),\sigma(v))\mid\text{$v$ is the primitive vector associated to a}\\
\text{facet of \Gf\ that is not contained in any coordinate hyperplane}\}\subset\Zplusnul^2.
\end{multline}
Now we want to prove that there exists a subset $S'\subset S$ such that
\begin{equation}\label{tbpfinmotmaintheo}
\Zmotoft\in\MC[T]\left[\frac{1}{1-\LL^{-\sigma}T^m}\right]_{(m,\sigma)\in S'},
\end{equation}
and such that $e^{-2\pi i\sigma/m}$ is an eigenvalue of monodromy (in the sense of the theorem) for each $(m,\sigma)\in S'$.

Let us introduce some notations and terminology. Consider the set $S$ from \eqref{defsetS}, and put $Q=\{\sigma/m\mid(m,\sigma)\in S\}\subset\Qplusnul$. Let $q\in Q$ and let $\tau$ be a facet of \Gf\ that is not contained in any coordinate hyperplane. We say that $\tau$ contributes to $q$ if $\sigma(v)/m(v)=q$, with $v$ the unique primitive vector in $\Zplusn\setminus\{0\}$ perpendicular to $\tau$. We shall call a ratio $q\in Q$ \emph{good} if
\begin{itemize}
\item $q=1$,
\item or $q$ is contributed by a facet of \Gf\ that is not a $B_1$-facet,
\item or $q$ is contributed by two $B_1$-facets of \Gf\ that are \underline{not} $B_1$ for a same variable and that have an edge in common.
\end{itemize}
We shall call $q\in Q$ \emph{bad} if $q$ is not good, i.e., if
\begin{itemize}
\item $q\neq1$;
\item and $q$ is only contributed by $B_1$-facets of \Gf;
\item and for any pair of contributing $B_1$-facets, we have that
\begin{itemize}
\item either they are $B_1$-facets for a same variable,
\item or they have at most one point in common.
\end{itemize}
\end{itemize}
Finally, we shall call a facet $\tau$ of \Gf\ \emph{bad} if it contributes to a bad $q\in Q$. This implies that $\tau$ is a $B_1$-facet.

Let us now define
\begin{equation*}
S'=\{(m,\sigma)\in S\mid\sigma/m\text{ is good}\}\subset S.
\end{equation*}
Then by Theorem~\ref{theoAenL} and Proposition~\ref{propAenL} by Lemahieu and Van Proeyen, we know that $e^{-2\pi i\sigma/m}$ is an eigenvalue of monodromy for each $(m,\sigma)\in S'$. It remains to prove that \eqref{tbpfinmotmaintheo} holds for the $S'$ proposed above.

The formula for \Zmotoft\ in Theorem~\ref{formlocmotzf} associates a term to every compact face $\tau$ of \Gf. If $\tau$ is not contained in any bad facet, then its associated term clearly belongs to
\begin{equation*}
\MC[T]\left[\frac{1}{1-\LL^{-\sigma}T^m}\right]_{(m,\sigma)\in S'}.
\end{equation*}
Hence it suffices to consider the sum of the terms associated to bad $B_1$-simplices or compact subfaces of bad $B_1$-facets. We will refer to this sum as the relevant part of \Zmotoft.

If we look at the formula carefully, we see that it is a rational expression (with integer coefficients) in \LL\ and $T$, except for the presence of $[\mathcal{X}_{\tau}]$ and $[\mathcal{X}_{\tau}']$ in $L_{\tau}$ and $L_{\tau}'$, respectively. Fortunately, for the relevant faces, these classes have a fairly simple form. For any vertex $V$, we have $[\mathcal{X}_V]=[\mathcal{X}_V']=0$. If $[CD]$ is any edge with one vertex in a coordinate hyperplane and the other vertex at distance one of this hyperplane, then $[\mathcal{X}_{[CD]}]=(\LL-1)^2$ if $[CD]$ is not contained in any coordinate hyperplane, and $[\mathcal{X}_{[CD]}']=\LL-1$ otherwise.

Lastly, let $\tau_0$ be a $B_1$-simplex with a base\footnote{If a $B_1$-simplex $\tau_0$ has two vertices $A$ and $B$ in a coordinate hyperplane and one vertex at distance one of this hyperplane, then we shall call $[AB]$ a base of $\tau_0$. A $B_1$-simplex has by definition at least one base, but can have several.} $[AB]$. Then we have the relation
\begin{equation}\label{beensconid}
[\mathcal{X}_{\tau_0}]=(\LL-1)^2-[\mathcal{X}_{[AB]}'].
\end{equation}
Let us write down the contributions of $\tau_0$ and $[AB]$ to \Zmotoft. If we denote by $v_0$ the unique primitive vector in $\Zplusn\setminus\{0\}$ perpendicular to $\tau_0$, then
\begin{align}
&L_{\tau_0}S(\Delta_{\tau_0})+L_{[AB]}'S(\Delta_{[AB]})'\notag\\
&\ \ \;=\left[\bigl(1-\LL^{-1}\bigr)^3-\LL^{-3}\bigl((\LL-1)^2-[\mathcal{X}_{[AB]}']\bigr)\frac{1-T}{1-\LL^{-1}T}\right]\frac{\LL^{-\sigma(v_0)}T^{m(v_0)}}{1-\LL^{-\sigma(v_0)}T^{m(v_0)}}\notag\\
&\qquad\hspace{1.692cm}+\left[\bigl(1-\LL^{-1}\bigr)^2-\LL^{-2}[\mathcal{X}_{[AB]}']\frac{1-T}{1-\LL^{-1}T}\right]\frac{\LL^{-\sigma(v_0)-1}T^{m(v_0)}}{1-\LL^{-\sigma(v_0)}T^{m(v_0)}}\notag\\
&\ \ \;=\frac{\bigl(1-\LL^{-1}\bigr)^3}{1-\LL^{-1}T}\,\frac{\LL^{-\sigma(v_0)}T^{m(v_0)}}{1-\LL^{-\sigma(v_0)}T^{m(v_0)}}.\label{conaftcanc}
\end{align}
Like we observed in the $p$-adic case, Identity~\eqref{beensconid}, together with the fact that $\mult\Delta_{[AB]}=1$, causes the cancellation of $[\mathcal{X}_{[AB]}']$. We shall call \eqref{conaftcanc} the contribution of $\tau_0$ and $[AB]$ to \Zmotoft\ after cancellation.

After these cancellations (one for every bad $B_1$-simplex), the relevant part of \Zmotoft\ is indeed a rational expression in \LL\ and $T$. More precisely, it is an element of the ring
\begin{equation}\label{motringeen}
\Z[\LL,\LL^{-1}][T]\left[\frac{1}{1-\LL^{-\sigma}T^m}\right]_{(m,\sigma)\in S}\subset\MC[T]\left[\frac{1}{1-\LL^{-\sigma}T^m}\right]_{(m,\sigma)\in S},
\end{equation}
whereby $\Z[\LL,\LL^{-1}]\subset\MC$ denotes the smallest subring of \MC\ containing $\Z,\LL$, and $\LL^{-1}$. We can now replace \LL\ by a new indeterminate $S$ and study the relevant part of \Zmotoft\ in the ring
\begin{equation}\label{motringtwee}
\Z[S,S^{-1}][T]\left[\frac{1}{1-S^{-\sigma}T^m}\right]_{(m,\sigma)\in S}\subset\Z[S,S^{-1}](T),
\end{equation}
where $\Z[S,S^{-1}]$ is the ring of formal Laurent polynomials over \Z. The advantage is that the coefficients of $T$ now live in the unique factorization domain $\Z[S,S^{-1}]$. There clearly exists a surjective ring morphism from \eqref{motringtwee} to \eqref{motringeen}; so if we can prove equality in \eqref{motringtwee}, equality in \eqref{motringeen} follows.

The goal is now to prove that the relevant part of \Zmotoft\ (seen as an element in this new ring) also belongs to
\begin{equation*}
\Z[S,S^{-1}][T]\left[\frac{1}{1-S^{-\sigma}T^m}\right]_{(m,\sigma)\in S'}.
\end{equation*}
The advantage of working in a unique factorization domain is that we may now choose a bad $q\in Q$ randomly and restrict ourselves to proving that the relevant part of \Zmotoft\ is an element of
\begin{equation*}
\Z[S,S^{-1}][T]\left[\frac{1}{1-S^{-\sigma}T^m}\right]_{(m,\sigma)\in S\setminus S_q},
\end{equation*}
with $S_q=\{(m,\sigma)\in S\mid\sigma/m=q\}\subset S$ and $S'\subset S\setminus S_q\subset S$. Indeed, if $\sigma_1/m_1\neq\sigma_2/m_2$, then $1-S^{-\sigma_1}T^{m_1}$ and $1-S^{-\sigma_2}T^{m_2}$ have no common irreducible factors in $\Z[S,S^{-1}][T]$.

So from now on, $q$ is a fixed bad ratio in $Q$. We define a \emph{$q$-cluster} as a family $\mathcal{C}$ of (bad $B_1$-) facets contributing to $q$, such that for any two facets $\tau,\tau'\in\mathcal{C}$, there exists a chain $\tau=\tau_0,\tau_1,\ldots,\tau_t=\tau'$ of $B_1$-facets in $\mathcal{C}$ with the property that $\tau_{j-1}$ and $\tau_j$ share an edge for all $j\in\{1,\ldots,t\}$. A \emph{maximal $q$-cluster} is a $q$-cluster that is not contained in a strictly bigger one. Note that every facet contributing to $q$ is contained in precisely one maximal $q$-cluster. Also note that the supports\footnote{By the support of a $q$-cluster we mean the union of its facets.} of two distinct maximal $q$-clusters may share a vertex of \Gf, but never share an edge.

%
Let $V$ be a vertex of \Gf, and let $\tau_j$, $j\in J$, be all the facets of \Gf\ that contain $V$. Denote for each $j\in J$, by $v_j$ the unique primitive vectors in $\Zplusn\setminus\{0\}$ perpendicular to $\tau_j$. Then $\Delta_V$ is the cone strictly positively spanned by the vectors $v_j$, $j\in J$. Let $\{\delta_i\}_{i\in I}$ be a decomposition of $\Delta_V$ into simplicial cones $\delta_i$ without introducing new rays, and assume that $\delta_i$ is strictly positively spanned by the vectors $v_j$, $j\in J_i$. We shall say that a cone $\delta_i$ \emph{meets} a $q$-cluster $\mathcal{C}$ if $\{\tau_j\mid j\in J_i\}\cap\mathcal{C}\neq\emptyset$. We shall call $\{\delta_i\}_{i\in I}$ a \emph{nice} decomposition if every $\delta_i$ meets at most one maximal $q$-cluster. By construction of the maximal $q$-clusters, a nice decomposition of $\Delta_V$ always exists.

Let us now choose a nice decomposition $\{\delta_{V,i}\}_{i\in I_V}$ of $\Delta_V$ for every relevant vertex $V$ of \Gf. The relevant part of \Zmotoft\ contains a term for every such $V$. According to the formula in Theorem~\ref{formlocmotzf}, this term can be split up into terms, one for each simplicial cone $\delta_{V,i}$ in the decomposition of $\Delta_V$. Let $\mathcal{C}$ be a maximal $q$-cluster. We define the part of \Zmotoft\ associated to $\mathcal{C}$ as the sum of the following terms:
\begin{itemize}
\item for each $B_1$-simplex $\tau\in\mathcal{C}$ with chosen base $b_{\tau}$, the contribution of $\tau$ and $b_{\tau}$ to \Zmotoft\ after cancellation;
\item the terms associated to the other compact edges of the $B_1$-facets in $\mathcal{C}$;
\item the terms associated to the simplicial cones $\delta_{V,i}$ that meet $\mathcal{C}$.
\end{itemize}
Note that in this way no term is assigned to more than one maximal $q$-cluster.

It follows that the relevant part of \Zmotoft\ is given by
\begin{equation*}
\sum_{\substack{\text{$\mathcal{C}$ maximal}\\\text{$q$-cluster}}}(\text{part of \Zmotoft\ associated to $\mathcal{C}$})+(\text{sum of remaining terms}).
\end{equation*}
By construction the sum of the remaining terms is certainly an element of
\begin{equation}\label{motringdrie}
\Z[S,S^{-1}][T]\left[\frac{1}{1-S^{-\sigma}T^m}\right]_{(m,\sigma)\in S\setminus S_q}.
\end{equation}
The problem is therefore reduced to proving that for every maximal $q$-cluster $\mathcal{C}$, the part of \Zmotoft\ associated to $\mathcal{C}$ belongs to \eqref{motringdrie}.

A maximal $q$-cluster contains no more than two $B_1$-facets, otherwise $q$ would equal one (see Case~VI). Moreover, two $B_1$-facets belonging to the same maximal $q$-cluster, are always $B_1$ for a same variable, otherwise $q$ would be good. This leaves us five possible configurations of a maximal $q$-cluster $\mathcal{C}$; it consists of
\begin{enumerate}
\item one $B_1$-simplex $\tau_0$,
\item or one non-compact $B_1$-facet $\tau_0$,
\item or two $B_1$-simplices $\tau_0$ and $\tau_1$ for a same variable,
\item or two non-compact $B_1$-facets $\tau_0$ and $\tau_1$ for a same variable,
\item or one non-compact $B_1$-facet $\tau_0$ and one $B_1$-simplex $\tau_1$ for a same variable.
\end{enumerate}
Pictures can be found in Figures~\ref{figcase1}, \ref{figcase2}, \ref{figcase3}, \ref{figcase4}, and \ref{figcase5}, respectively.

In Cases~(i) and (ii), the part of \Zmotoft\ associated to $\mathcal{C}$ has the form
\begin{align*}
\frac{N_1(T)}{(1-S^{-1}T)F_0F_1F_2}&\in\Z[S,S^{-1}][T]\left[\frac{1}{1-S^{-\sigma}T^m}\right]_{(m,\sigma)\in S}\subset\Z[S,S^{-1}](T),\\\intertext{while in Cases~(iii)--(v), it has the form}
\frac{N_2(T)}{(1-S^{-1}T)F_0F_1F_2F_3}&\in\Z[S,S^{-1}][T]\left[\frac{1}{1-S^{-\sigma}T^m}\right]_{(m,\sigma)\in S}\subset\Z[S,S^{-1}](T).
\end{align*}
Hereby $N_1(T)$ and $N_2(T)$ are polynomials in $T$ with coefficients in $\Z[S,S^{-1}]$, and so are
\begin{equation*}
F_j=1-S^{-\sigma_j}T^{m_j}\subset\Z[S,S^{-1}][T];\qquad j=0,\ldots,3.
\end{equation*}

In Cases~(i) and (ii), the factor $F_0$ corresponds to $\tau_0$, while $F_1$ and $F_2$ correspond to neighbor facets\footnote{By a neighbor facet we mean a facet sharing an edge. A factor will appear in the denominator for every neighbor facet that does not lie in a coordinate hyperplane.}\addtocounter{footnote}{-1} of $\tau_0$. It follows that $\sigma_0/m_0=q$ and $\sigma_j/m_j\neq q$ for $j=1,2$. In Cases~(iii)--(v), factors $F_0$ and $F_1$ correspond to $\tau_0$ and $\tau_1$, where $F_2$ and $F_3$ come from neighbor facets\footnotemark\ of $\tau_0$ and $\tau_1$. We have $\sigma_0/m_0=\sigma_1/m_1=q$ and $\sigma_j/m_j\neq q$ for $j=2,3$.

Finally everything boils down to proving that (depending on the case)
\begin{equation}\label{motdivcond}
F_0\mid N_1(T)\qquad\text{or}\qquad F_0F_1\mid N_2(T)
\end{equation}
in the polynomial ring $\Z[S,S^{-1}][T]$. As $F_0$ and $F_1$ are monic polynomials (in the sense that their leading coefficients are units of $\Z[S,S^{-1}]$), the divisibility conditions \eqref{motdivcond} can be investigated equivalently over the fraction field $\Q(S)$ of $\Z[S,S^{-1}]$. Now we can decide divisibility by looking at the roots of $F_0$ and $F_1$ in some algebraic closure of the coefficient field $\Q(S)$. We shall consider the field $\overbar{\Q}\{\{S\}\}$ of formal Puiseux series over the field $\overbar{\Q}$ of algebraic numbers.

The polynomial $F_j=1-S^{-\sigma_j}T^{m_j}$ has $m_j$ distinct roots
\begin{equation*}
T_k^{(j)}=S^{\frac{\sigma_j}{m_j}}e^{\frac{2k\pi i}{m_j}};\qquad k=0,1,\ldots,m_j-1;
\end{equation*}
in $\overbar{\Q}\{\{S\}\}$ for $j=0,1$. Hence $F_0$ divides $N_1(T)$ if and only if $N_1\bigl(T_k^{(0)}\bigr)=0$ in $\overbar{\Q}\{\{S\}\}$ for all $k$. In Cases~(iii)--(v), we may conclude that $F_0F_1\mid N_2(T)$ as soon as $N_2\bigl(T_k^{(j)}\bigr)=0$ for all $k$ and $j=0,1$ \emph{and} $N_2'(T)$ vanishes in all common roots
\begin{equation*}
T_k^{(0,1)}=S^qe^{\frac{2k\pi i}{\gcd(m_0,m_1)}};\qquad k=0,1,\ldots,\gcd(m_0,m_1)-1;
\end{equation*}
of $F_0$ and $F_1$ in $\overbar{\Q}\{\{S\}\}$.

The proof of each of the identities $N_1\bigl(T_k^{(0)}\bigr)=0$, $N_2\bigl(T_k^{(j)}\bigr)=0$, $N_2'\bigl(T_k^{(0,1)}\bigr)=0$ is identical to one of the \lq residue vanishing proofs\rq\ in Cases~I--V. For example, in Case~(iii) of the current proof, the proof of $N_2\bigl(T_k^{(j)}\bigr)=0$ for a simple root $T_k^{(j)}$ of $F_0F_1$, corresponds to the proof of $R_1=0$ in Case~I. For a double root $T_k^{(0,1)}$ of $F_0F_1$, the proofs of $N_2\bigl(T_k^{(0,1)}\bigr)=0$ and $N_2'\bigl(T_k^{(0,1)}\bigr)=0$ are completely analogous to the proofs of $R_2=0$ and $R_1=0$, respectively, in Case~III. This ends the sketch of the proof of the main theorem in the motivic setting.

\bibliography{bartbib}
\bibliographystyle{amsplain}
\end{document}